%% file: CEP.tex
\documentclass[a4paper]{amsart}
\usepackage[T1]{fontenc}
\usepackage{amsmath, amsthm, amsfonts, mathtools, tikz-cd, amssymb}
\usepackage{enumitem}
\usepackage{caption}
\usepackage[all,cmtip]{xy}
\usepackage{multirow}
\usepackage{pifont}
\usepackage{ifthen}
\usepackage{soul}

\makeatletter
\def\part{\@startsection{part}{0}%
  \z@{\linespacing\@plus\linespacing}{.5\linespacing}%
  {\clearpage\centering\normalfont\bfseries}}
\def\@addtocvspace{}
\newboolean{withpagenum}
\setboolean{withpagenum}{true}
\def\@tocline#1#2#3#4#5#6#7{\relax
  \ifnum #1>\c@tocdepth %
  \else
    \par\addpenalty\@secpenalty\addvspace{#2}%
    \begingroup \hyphenpenalty\@M
    \@ifempty{#4}{%
      \@tempdima\csname r@tocindent\number#1\endcsname\relax
    }{%
      \@tempdima#4\relax
    }%
    \parindent\z@ \leftskip#3\relax \advance\leftskip\@tempdima\relax
    \rightskip\@pnumwidth plus4em \parfillskip-\@pnumwidth
    #5\leavevmode\hskip-\@tempdima #6\nobreak\relax
    \ifthenelse{\boolean{withpagenum}}{%
      \hfil\hbox to\@pnumwidth{\@tocpagenum{#7}}}{}%
    \par
    \@addtocvspace%
    \nobreak
    \endgroup
    \def\@addtocvspace{}%
    \setboolean{withpagenum}{true}%
  \fi}
\def\l@part{\def\@addtocvspace{\vspace{1.5ex}}%
  \setboolean{withpagenum}{false}%
  \@tocline{-1}{12pt plus2pt}{0pt}{}{\bfseries\centering}}
  
\def\l@section{\@tocline{1}{.5ex}{1pc}{}{}}
\def\l@subsection{\@tocline{2}{0pt}{2.5pc}{5pc}{}}
\makeatother

\usepackage[pagebackref,pdfborder={0 0 0},urlcolor=black]{hyperref}
\hypersetup{pdfauthor={Kevin Li, Clara L\"oh, Marco Moraschini},%
  pdftitle={The cheap embedding property}}

\title[The cheap embedding principle]
      {The cheap embedding principle:\\ Dynamical upper bounds for homology growth}
\author[K.~Li]{Kevin Li}
\address{Fakult\"at f\"ur Mathematik, Universit\"at Regensburg, 93040 Regensburg, Germany}
\email{kevin.li@ur.de}

\author[C.~L\"oh]{Clara L\"oh}
\address{Fakult\"at f\"ur Mathematik, Universit\"at Regensburg, 93040 Regensburg, Germany}
\email{clara.loeh@ur.de}

\author[M.~Moraschini]{Marco Moraschini}
\address{Dipartimento di Matematica, Universit{\`a} di Bologna, 40126 Bologna, Italy}
\email{marco.moraschini2@unibo.it}

\author[R.~Sauer]{Roman Sauer}
\address{Fakult\"at f\"ur Mathematik, Karlsruher Institut f\"ur Technologie, 76131 Karlsruhe, Germany}
\email{roman.sauer@kit.edu}

\author[M.~Uschold]{Matthias Uschold}
\address{Fakult\"at f\"ur Mathematik, Universit\"at Regensburg, 93040 Regensburg, Germany}
\email{matthias.uschold@ur.de}

\date{\today.\ \copyright{\ K.~Li, C.~L\"oh, M.~Moraschini, R.~Sauer, M.~Uschold 2025}.}

\makeatletter
\@namedef{subjclassname@2020}{%
  \textup{2020} Mathematics Subject Classification}
\makeatother

\newcounter{commentcounter}

\theoremstyle{definition}
\newtheorem{defn}{Definition}[section]
\newtheorem{ex}[defn]{Example}
\newtheorem{question}[defn]{Question}
\newtheorem{setup}[defn]{Setup}

\newtheorem{rem}[defn]{Remark}

\theoremstyle{plain}
\newtheorem{thm}[defn]{Theorem}
\newtheorem{lem}[defn]{Lemma}
\newtheorem{prop}[defn]{Proposition}
\newtheorem{conjecture}[defn]{Conjecture}
\newtheorem{cor}[defn]{Corollary}

\numberwithin{equation}{section}

\newcommand{\IN}{\ensuremath\mathbb{N}}
\newcommand{\IZ}{\ensuremath\mathbb{Z}}
\newcommand{\IQ}{\ensuremath\mathbb{Q}}
\newcommand{\IR}{\ensuremath\mathbb{R}}
\newcommand{\IF}{\ensuremath\mathbb{F}}

\let\N\IN
\let\Z\IZ
\let\Q\IQ
\let\R\IR
\let\C\IC
\def\hyp{\mathbb{H}}

\newcommand{\sfFP}{\ensuremath\mathrm{FP}}

\newcommand{\sfCERP}{\ensuremath\mathsf{CERP}}

\newcommand{\enum}{\rm{(\roman*)}}
\newcommand{\spann}[1]{{\ensuremath \langle{#1}\rangle}}
\newcommand{\into}{\ensuremath\hookrightarrow}
\newcommand{\onto}{\ensuremath\twoheadrightarrow}
\DeclareMathOperator{\Isom}{Isom}
\DeclareMathOperator{\Hom}{Hom}
\DeclareMathOperator{\im}{im}
\DeclareMathOperator{\ind}{Ind}
\DeclareMathOperator{\id}{id}
\DeclareMathOperator{\res}{Res}
\DeclareMathOperator{\map}{map}
\DeclareMathOperator{\vol}{vol}
\newcommand{\linf}{\ensuremath\ell^\infty}

\newcommand{\aonlyk}[3]{\ensuremath{%
a_{#1*,#2}^{#3}%
}}
\def\mrk{\rk}

\DeclareMathOperator{\medim}{medim}
\DeclareMathOperator{\mevol}{mevol}
\DeclareMathOperator{\Aug}{A}
\def\epi{\twoheadrightarrow}
\DeclareMathOperator{\EMD}{EMD}
\DeclareMathOperator{\Indop}{Ind}
\def\Ind#1#2{\Indop_{#1}^{#2}}

\DeclareMathOperator{\cost}{cost}
\DeclareMathOperator{\Prob}{Prob}

\DeclareMathOperator{\GH}{GH}
\def\dgh#1#2{d_{\GH}^{#1}(#2)}

\usepackage{mathabx}
\makeatletter
\newcommand\incircbin
{\mathpalette\@incircbin}
\newcommand\@incircbin[2]
{\mathbin%
  {\ooalign{\hidewidth$#1#2$\hidewidth\crcr$#1\ovoid$}}
}
\newcommand{\msymmdiff}{\incircbin{\scalebox{0.76}{\raisebox{0.06em}{$\triangle$}}}}
\newcommand{\mcompl}{\circleddash}
\makeatother

\def\ltb#1#2{%
  b^{(2)}_{#1}(#2)}

\def\args{\,\cdot\,}

\def\qand{\quad\text{and}\quad}
\def\qor{\quad\text{or}\quad}
\def\fa#1{%
  \forall_{#1}\;\;\;}
\def\exi#1{%
  \exists_{#1}\;\;\;}

\def\actson{%
  \curvearrowright}

\def\symmdiff{\mathbin{\triangle}}
\def\linf#1{L^\infty(#1)}
\def\Rrel{\mathcal{R}}
\DeclareMathOperator{\supp}{supp}
\def\gen#1{\langle #1 \rangle}
\DeclareMathOperator{\tors}{tors}
\DeclareMathOperator{\rk}{rk}
\DeclareMathOperator{\FP}{\ensuremath\mathrm{FP}}
\DeclareMathOperator{\size}{size}
\DeclareMathOperator{\proj}{proj}
\DeclareMathOperator{\constop}{const}
\def\const#1{\constop_{#1}\!}

\makeatletter
\newcommand\norm{\bBigg@{0.8}}
\makeatother

\newcommand{\ifsv}[2][norm]{\csname #1l\endcsname\bracevert\!#2\!%
                            \csname #1r\endcsname\bracevert}
\newcommand{\ifsvp}[3][norm]{\csname #1l\endcsname\bracevert\!#2\!%
                            \csname #1r\endcsname\bracevert\!^{#3}}

\newcommand{\Nbasic}{\ensuremath{N_1}}
\newcommand{\Nmax}{\ensuremath{\underline{N}_1}}
\newcommand{\Nsum}{\ensuremath{N_1}}
\newcommand{\Ntbasic}{\ensuremath{N_2}}
\newcommand{\Ntmax}{\ensuremath{\underline{N}_2}}
\newcommand{\Ntsum}{\ensuremath{N_2}}

\newcommand{\numax}{\ensuremath{\underline{\nu}}}
\newcommand{\nusum}{\ensuremath{\nu}}

\DeclareMathOperator{\pt}{pt}
\DeclareMathOperator{\Cone}{Cone}
\newcommand{\calR}{\ensuremath\mathcal{R}}

\DeclareMathOperator{\logp}{\log_+}
\DeclareMathOperator{\lognorm}{lognorm}

\newcommand{\pmp}{probability measure preserving}
\newcommand{\wkcont}{\ensuremath{\prec}}
\newcommand{\Linf}[1]{\ensuremath{\linf #1}}

\newcommand{\Linfalpha}{\ensuremath{R_\alpha}}
\newcommand{\Linfbeta}{\ensuremath{R_\beta}}
\newcommand{\Linfbetap}{\ensuremath{R_{\beta'}}}

\newcommand{\indexw}[1]{}
\newcommand{\indexnot}[3]{}

\newcommand{\highlight}[1]{#1}

\renewcommand{\epsilon}{\ensuremath{\varepsilon}}

\newcommand{\LinftyX}{\ensuremath{L^\infty(\alpha)}}
\newcommand{\LinftyXLambda}{\ensuremath{L^\infty(\alpha|_\Lambda)}}
\newcommand{\LinftyXZ}{\ensuremath{L^\infty(\alpha,Z)}}
\renewcommand{\subseteq}{\ensuremath{\subset}}
\newcommand{\almostcc}[2]{$#1$-almost $#2$-chain complex}
\newcommand{\almostccs}[2]{$#1$-almost $#2$-chain complexes}
\newcommand{\almostcm}[2]{$#1$-almost $#2$-chain map}
\newcommand{\almostcms}[2]{$#1$-almost $#2$-chain maps}
\newcommand{\LinftyXwc}{\ensuremath{L^\infty(\alpha)}}
\newcommand{\LinftyYwc}{\ensuremath{L^\infty(\beta)}}

\begin{document}

\begin{abstract}
  We provide upper bounds for logarithmic torsion homology
  growth and Betti number growth of groups, phrased in the
  language of measured group theory.
\end{abstract}

\maketitle

\input{intro}

\clearpage
\tableofcontents

\part{Dynamical upper bounds for homology growth}

The main goal is to prove the dynamical upper bounds for Betti number
gradients and logarithmic torsion homology gradients in terms of measured
embedding dimension and measured embedding volume, respectively. We first
develop a framework of quantitative homological algebra over rings
associated with standard actions and prove strictification and
deformation results (Sections~\ref{sec:basics}--\ref{sec:lognorm}).
We then explain the passage to finite index subgroups (Section~\ref{sec:passing})
and prove the dynamical upper bounds (Section~\ref{sec:proofthmCEP}).

\input{GH.tex}

\addtocontents{toc}{\clearpage}
\part{Examples}

\input{examples}

\part{Dynamical inheritance properties}
\label{part:dyn}

We establish the following dynamical inheritance and computational
properties of measured embedding dimension and measured embedding
volume: monotonicity under weak containment (Section~\ref{sec:weak_containment}),
a disintegration estimate and reduction to ergodic actions (Section~\ref{sec:disint}),
estimates via the equivalence relation ring (Section~\ref{sec:overZR}), 
invariance under weak bounded orbit equivalence (Section~\ref{sec:wbOE}),
comparison with cost (Section~\ref{sec:cost}),
comparison with integral foliated simplicial volume (Section~\ref{sec:ifsv}).
In particular, this will also allow us to compute further examples. 

\input{wc}
\input{disintegration.tex}
\input{eqrelring.tex}
\input{wbOE.tex}
\input{cost.tex}
\input{simvol.tex}

\noindent
\textbf{Acknowledgements.}
K.L., C.L., and M.U.\ were supported by the CRC~1085 ``Higher Invariants'' (Universit\"at Regensburg, funded by the DFG). The material on monotonicity under weak containment is part of M.U.'s 
PhD project.

M.M.\ was supported by the ERC ``Definable Algebraic Topology'' DAT -- Grant Agreement number~101077154. 

M.M.\ was funded by the European Union -- NextGenerationEU under the National Recovery and Resilience Plan (PNRR) -- Mission 4 Education and research -- Component 2 From research to business -- Investment~1.1 Notice Prin~2022 -- DD~N.~104 del 2/2/2022, from title ``Geometry and topology of manifolds'', proposal code 2022NMPLT8 -- CUP~J53D23003820001.

R.S.~was funded by the Deutsche Forschungsgemeinschaft (DFG, German Research Foundation) -- project number~338540207.

\clearpage
\bibliographystyle{amsalphaabbrv}
\bibliography{bib}

\setlength{\parindent}{0cm}

\end{document}

%% file: intro.tex
\section{Introduction}

\def\pone{\ding{192}}
\def\ptwo{\ding{193}}

We provide new upper bounds for logarithmic torsion homology
growth and Betti number gradients along systems of finite index
normal subgroups via the following two basic principles:
\begin{enumerate}
\item[\pone]
  If $A$ ``embeds'' into~$B$ and $B$ is ``small'',
  then also $A$ is ``small''.
\item[\ptwo]
  The asymptotic behaviour along finite index normal subgroups
  is encoded in the dynamical system given
  by the profinite completion.
\end{enumerate}
We apply these two principles in the setting of chain complexes over
crossed product rings and orbit equivalence relation rings associated
with dynamical systems. Typically, ``embeddings'' refer to chain maps
that admit homotopy retractions and ``smallness'' will be measured in
terms of ``dimensions'' or ``determinants'' over various rings.

The principle~\pone\ was used to prove vanishing results for
$L^2$-Betti numbers and homology gradients in the presence of amenable
covers with small multiplicity~\cite{Sauer09,
  sauer:volume:homology:growth}.  The principle~\ptwo\ was previously
established for $L^2$-Betti numbers~\cite{gaboriaul2}, rank gradients
of groups~\cite{abertnikolov}, and stable integral simplicial
volume~\cite{loeh_pagliantini,loeh2020ergodic}. In these cases, the dynamical
point of view was the key to proving novel types of inheritance
results, leading to a deeper understanding of the invariants and
concrete calculations and
estimates~\cite{gaboriaul2,gaboriaucost,FLPS,FLMQ}.

Similarly, our approach provides new perspectives on calculations and
estimates for logarithmic torsion homology growth and Betti number
gradients over finite fields. 
We will now describe the setting and
method in more detail:

\subsection{Setup and dynamical sizes}\label{subsec:introsetup}

Let $\Gamma$ be a countable group. Let $\alpha\colon \Gamma \actson
(X,\mu)$ be a \emph{standard $\Gamma$-action}, i.e., an essentially
free probability measure preserving action on a standard Borel
probability space.  As coefficients, we consider $Z$ to be $\Z$ (with
the usual norm) or a finite field (with the trivial norm).  We write
$\linf{\alpha,Z}$ for the $Z\Gamma$-module of essentially bounded
measurable functions~$X \to Z$ up to equality $\mu$-almost everywhere;
i.e., elements of~$\linf {\alpha,Z}$ are represented by finite
$Z$-linear combinations of characteristic functions on measurable
subsets of~$X$. This leads to the crossed product ring~$R \coloneqq \linf
{\alpha,Z} * \Gamma$ of the action~$\alpha$.

We assume that $\Gamma$ is \emph{of type $\FP_{n+1}$}, i.e., the trivial $\IZ\Gamma$-module~$\IZ$ admits a projective resolution that is finitely generated in degrees~$\le n+1$.
 
We fix a free $Z\Gamma$-resolution~$C_* \epi Z$ of 
the trivial $Z\Gamma$-module~$Z$. By elementary
homological algebra, the subsequent definitions will be independent
of the choice of the resolution of the group. 
On the dynamical side, we consider marked projective augmented
chain complexes over~$R$
(see Section~\ref{sec:basics} for precise definitions).

\begin{defn}[$\alpha$-embedding]
	\label{def:alpha-emb}
  In this situation, an \emph{$\alpha$-embedding (up to degree~$n$)} is
  a pair that consists of a marked projective augmented $R$-chain
  complex~$D_* \epi \linf {\alpha,Z}$ and
  a $Z \Gamma$-chain map~$C_* \to D_*$ up to degree~$n+1$ extending
  the inclusion~$Z \to \linf{\alpha,Z}$ as constant functions.
  We write~$\Aug_n(\alpha)$ for the class of all augmented
  complexes arising in $\alpha$-embeddings up to degree~$n$.
\end{defn}

For every $\R_{\geq 0}$-valued isomorphism invariant~$\Delta$
of marked projective augmented $R$-chain
complexes, we may define
\[ \Delta(\alpha)
\coloneqq \inf_{(D_* \epi \linf {\alpha,Z}) \in \Aug_n(\alpha)}
\Delta\bigl(D_* \epi \linf {\alpha,Z}\bigr).
\]
For example, for~$n \in \N$, we obtain the following invariants
of~$\Gamma \actson (X,\mu)$, still under the assumption that $\Gamma$ is \emph{of type~$\FP_{n+1}$}. 
\begin{itemize}
\item The \emph{measured embedding dimension~$\medim_n^Z (\alpha)$
  over~$Z$ in degree~$n$}:
  Here, we take
  \[ \Delta\bigl(D_* \epi \linf {\alpha,Z}\bigr)
  \coloneqq \dim_{R} (D_n).
  \]
\item The \emph{measured embedding volume~$\mevol_n(\alpha)$
  in degree~$n$}: 
  Here, we take $Z = \Z$ and 
  \[ \Delta\bigl(D_* \epi \linf {\alpha,\Z}\bigr)
  \coloneqq \lognorm (\partial^{D}_{n+1}).
  \]
  The quantity~$\lognorm$ is a crude approximation of the logarithmic
  determinant, introduced in Section~\ref{sec:lognorm}.
\end{itemize}

In fact, this setting also extends to the equivalence relation ring~$Z \Rrel$
over the orbit relation~$\Rrel \coloneqq \Rrel_{\alpha} \coloneqq \{
(x,\gamma \cdot x) \mid x \in X,\ \gamma \in \Gamma\}$ of~$\alpha$.
This leads to the same values of measured embedding dimension and
of measured embedding volume (Corollary~\ref{cor:overZR}). However,
it is not clear whether one may obtain suitable orbit equivalence invariants
in this way because $Z\Rrel$ in general might not be flat over~$Z \Gamma$ and
because $Z \Rrel \otimes_{Z \Gamma} Z$ in general is not isomorphic
to~$\linf {\alpha,Z}$. 

\subsection{Upper bounds for gradient invariants}

Let $\Gamma$ be a countable residually finite group (satisfying suitable
finiteness properties).
If $\Gamma_* = (\Gamma_i)_{i \in I}$ is a directed system of finite index normal
subgroups of~$\Gamma$,
we may study $\Gamma$ through~$\Gamma_*$ via
the residually finite point of view and the asymptotic behaviour
of invariants of the subgroups in~$\Gamma_*$: Let $F$ be an
$\R_{\geq0}$-valued isomorphism invariant of residually finite groups
(satisfying suitable finiteness properties).
Then we obtain an associated (upper) gradient invariant via
\[ \widehat F (\Gamma, \Gamma_*)
\coloneqq \limsup_{i \in I}
\frac{F (\Gamma_i)}{[\Gamma:\Gamma_i]}.
\]

For example, for~$n \in \N$ and $\Gamma$ of type~$\FP_{n+1}$, we have
the following invariants of~$\Gamma_*$:
\begin{itemize}
\item The \emph{(upper) Betti number gradient over~$Z$ in degree~$n$}:
  \[ \widehat b_n (\Gamma,\Gamma_*; Z)
  \coloneqq \bigl( \rk_Z H_n(\args;Z) \bigr)\widehat{\phantom{l}}
  = \limsup_{i \in I}
  \frac{\rk_Z H_n(\Gamma_i;Z)}
       {[\Gamma:\Gamma_i]}.
  \]
\item The \emph{(upper) logarithmic torsion homology gradient in degree~$n$}:
  \[ \widehat t_n (\Gamma,\Gamma_*)
  \coloneqq \bigl( \log \# \tors H_n(\args;\Z) \bigr)\widehat{\phantom{l}}
  = \limsup_{i \in I}
  \frac{\log \# \tors H_n(\Gamma_i;\Z)}
       {[\Gamma:\Gamma_i]}.
  \]
\end{itemize}

On the dynamical side, we can consider the corresponding profinite
completion~$\widehat \Gamma_* \coloneqq \varprojlim_{i \in I} \Gamma/\Gamma_i$.
Then the left translation action~$\Gamma \actson \widehat\Gamma_*$
is measure preserving with respect to the Haar measure. If this
action is essentially free, we are in the dynamical setting
of Section~\ref{subsec:introsetup} and we can compare the gradient
invariants with dynamical invariants.

\begin{thm}[dynamical upper bounds; Theorem~\ref{thm:dynupperproofsec}]\label{thm:dynupper}
  Let $n \in \N$ and let $\Gamma$ be a residually finite group
  of type~$\FP_{n+1}$. 
  Let $(\Gamma_i)_{i \in I}$ be a directed system of finite index
  normal subgroups of~$\Gamma$ with~$\bigcap_{i \in I} \Gamma_i = 1$
  (e.g., a residual chain in~$\Gamma$
  or the system of all finite index normal subgroups).
  Then:
  \begin{align*}
    \widehat b_n(\Gamma,\Gamma_*; Z)
    & \leq \medim_n^Z (\Gamma \actson \widehat\Gamma_*)
     \quad \text{if $Z$ is~$\Z$ or a finite field}
    \\
    \widehat t_n(\Gamma,\Gamma_*)
    & \leq \mevol_n (\Gamma \actson \widehat \Gamma_*).
  \end{align*}
\end{thm}

The proof of Theorem~\ref{thm:dynupper} relies on the fundamental 
observation that the subalgebra of cylinder sets is dense in~$\widehat
\Gamma_*$ with respect to the Haar measure. Therefore, we can
approximate augmented complexes over the crossed product
ring~$R \coloneqq \linf{\widehat \Gamma_*,Z} * \Gamma$ with arbitrary precision by
augmented complexes that only involve single, deep enough,
subgroups~$\Gamma_i$. We can then apply classical homological algebra
to interpret the embeddings as chain homotopy retracts and use
standard estimates for Betti numbers and logarithmic torsion.

On a technical level, the underlying approximation result is
challenging and is obtained through a balanced sequence of
deformations and strictifications of (almost) chain complexes and
(almost) chain maps (Section~\ref{sec:strictification} and 
Section~\ref{sec:deformation}). We thus develop a quantitative
homological algebra over~$R$. In particular, this includes various
norms on marked projective modules over~$R$, norms for homomorphisms
between such modules, and the introduction of a Gromov--Hausdorff
distance between marked projective chain complexes over~$R$
(Section~\ref{sec:GH}). The proof of Theorem~\ref{thm:dynupper} is
completed in Section~\ref{sec:proofthmCEP}.

The measured embedding dimension is also an upper bound
for the $L^2$-Betti numbers:

\begin{thm}[Theorem~\ref{thm:l2upper}]\label{thm:l2upperintro}
  Let $n \in \N$, 
  let $\Gamma$ be a group of type~$\FP_{n+1}$,
  and let~$\alpha$  be a standard $\Gamma$-action. 
  Then:
  \begin{align*}
    \ltb n \Gamma
    \leq 
    \medim_n^\Z (\alpha).
  \end{align*}
\end{thm}

\subsection{Examples}

As in the case of $L^2$-Betti numbers,
we can use a Rokhlin lemma to show that amenable
groups are homologically dynamically small (Section~\ref{sec:ame:have:CEP}).  In
particular, we obtain: 

\begin{thm}[amenable groups; Theorem~\ref{thm:amenable}]
\label{thm:amenable_intro}
  Let $n \in \N$, let $\Gamma$ be a countable infinite amenable group
  of type~$\FP_{n+1}$, and let $\alpha$ be a standard
  $\Gamma$-action. Then:
  \begin{align*}
     \medim_n^Z (\alpha)
     & = 0
     \quad \text{if $Z$ is $\Z$ or a finite field}
     \\
     \mevol_n(\alpha)
     & = 0.
  \end{align*}    
\end{thm}

Theorem~\ref{thm:dynupper} and Theorem~\ref{thm:amenable_intro}
refine the well-known results that amenable residually finite groups
of type~$\FP_\infty$ have vanishing Betti gradients over every
field~\cite{Cheeger-Gromov86,LLS11} and vanishing logarithmic torsion
growth~\cite{Kar-Kropholler-Nikolov} in all degrees.

For free groups, we obtain
(Section~\ref{subsec:freegroups}):

\begin{prop}[free groups; Proposition~\ref{prop:freegroups}]
  Let $d\in \N_{>0}$, let $F_d$ be a free group of rank~$d$,
  and let~$\alpha$ be a standard
  action of~$F_d$. 
  Let~$Z$ be~$\IZ$ or a finite field.
  Then:
  \begin{enumerate}[label=\enum]
  \item For all~$n \in \N$, we have~$\mevol_n (\alpha) = 0$;
  \item For all~$n \in \N \setminus \{1\}$, we have~$\medim_n^Z (\alpha) = 0$;
  \item We have~$\medim^\Z_1 (\alpha) \geq d-1$;
  \item If $\alpha$ is the profinite completion with respect to a directed system~$(\Gamma_i)_{i\in I}$ of finite index normal subgroups of~$F_d$ with $\bigcap_{i\in I}\Gamma_i=1$,
    then
    \[ \medim^Z_1 (\alpha) = d - 1.
    \]
  \end{enumerate}
\end{prop}

We prove estimates for amalgamated products over~$\Z$ (Section~\ref{subsec:amalg:prod}).
As a consequence of the calculation for free groups, we may then also treat surface groups (Section~\ref{sec:surface}). 
Moreover, we obtain inheritance results for direct products with an amenable factor (Section~\ref{sec:amenable_factor}) and for finite index subgroups (Section~\ref{sec:finindex}).

\begin{rem}[relation with the cheap rebuilding property]
	The definition of the measured embedding dimension and the
measured embedding volume as well as the corresponding
dynamical upper bounds is inspired
by the work of Ab\'ert, Bergeron, Fr\k{a}czyk, and Gaboriau~\cite{ABFG21} on the cheap rebuilding property.
	This property of groups implies the vanishing of both Betti number growth and logarithmic torsion homology growth.
	The present authors introduced an algebraic version 	of this property~\cite{LLMSU}.
	We show that an equivariant algebraic version implies the vanishing of~$\medim$ and~$\mevol$ (Section~\ref{sec:CRP}).
	
	Via the measured embedding volume, the logarithmic torsion
growth estimates are decoupled from the homology gradient
estimates. In particular, one may use the measured
embedding volume to prove vanishing of the logarithmic torsion
homology growth where previous methods are not applicable
because of non-vanishing $L^2$-Betti numbers.

In the future, we plan to combine the algebraic techniques developed in our previous paper~\cite{LLMSU} with the dynamical approach in order to obtain bootstrapping theorems for measured embedding dimensions and volumes.
\end{rem}

\subsection{Dynamical inheritance properties}

By design, the measured embedding dimension and the measured
embedding volume are of dynamical nature. More precisely, we
establish the following concrete instances of dynamical behaviour:
\begin{itemize}
\item monotonicity under weak containment;
\item reduction to ergodic actions; 
\item invariance under weak bounded orbit equivalence;
\item comparison with cost;
\item comparison with integral foliated simplicial volume.
\end{itemize}

As in the fixed price problem for cost and integral foliated
simplicial volume, in general, it is not clear how the measured
embedding dimension/volume depend on the underlying dynamical
system. Similarly to the case of cost and integral foliated simplicial
volume, we show that one can always restrict to ergodic dynamical
systems (Corollary~\ref{cor:ergsuffices}) and 
monotonicity under weak containment.

\begin{thm}[weak containment; Theorem~\ref{thm:wc}]\label{thm:wcintro}
  Let $n \in \N$, 
  let $\Gamma$ be a countable group of type~$\FP_{n+1}$, 
  and let $\alpha,\beta$ be standard $\Gamma$-actions 
  with~$\alpha \prec \beta$.
  Then, we have 
  \begin{align*}
     \medim_n^Z (\beta)
     & \leq \medim_n^Z(\alpha)
     \quad \text{if $Z$ is~$\Z$ or a finite field}
     \\
     \mevol_n (\beta)
     & \leq \mevol_n(\alpha).
  \end{align*}      
\end{thm}

In particular, for groups with property~$\EMD^*$ (Definition~\ref{def:EMD}), we obtain in combination with 
Theorem~\ref{thm:dynupper}:

\begin{cor}[Corollary~\ref{cor:EMDreduction}]
  Let $n \in \N$, let $\Gamma$ be a residually finite group of type~$\FP_{n+1}$
  that satisfies~$\EMD^*$, 
  and let
  $\alpha$ be a standard $\Gamma$-action.
  Then
  \begin{alignat*}{4}
    \widehat b_n(\Gamma,\Gamma_*; Z) &\leq 
    \medim^Z_n (\Gamma \actson \widehat \Gamma)
    &\leq& \medim^Z_n (\alpha)
     &&\quad \text{if $Z$ is~$\Z$ or a finite field}
    \\
    \widehat t_n(\Gamma,\Gamma_*) &\leq
    \mevol_n (\Gamma \actson \widehat \Gamma)
    &\leq& \mevol_n (\alpha).
  \end{alignat*}
\end{cor}

It is an open problem to determine how (vanishing of)
homology gradients over finite fields or torsion homology
growth behaves under orbit equivalence. As a step towards
this problem, we show that measured embedding dimension
and measured embedding volume are compatible with weak
bounded orbit equivalence (Definition~\ref{def:wbOE}). In particular, these invariants
provide upper bounds for homology growth over finite finite fields
and for torsion homology growth that are multiplicative under 
weak bounded orbit equivalences.

\begin{thm}[weak bounded orbit equivalence; Theorem~\ref{thm:wbOE}]
\label{thm:wbOE:intro}
  Let $n \in \N$, 
  let $\Gamma$ and $\Lambda$ be groups of type~$\FP_{n+1}$, and let
  $\alpha$ and $\beta$ be standard actions of $\Gamma$ and $\Lambda$,
  respectively, that are weakly bounded orbit equivalent of index~$c$.
  Then, we have
  \begin{align*}
    \medim^Z_n (\alpha)
    & = c \cdot \medim^Z_n (\beta)
    \quad\text{if $Z$ is $\Z$ or a finite field}
    \\
    \mevol_n (\alpha)
    & = c \cdot \mevol_n (\beta).
  \end{align*}
\end{thm}

We deduce a proportionality result for hyperbolic 3-manifolds:

\begin{thm}[Theorem~\ref{thm:mevolhyp3}]
  Let $M$ and $N$ be oriented closed connected hyperbolic $3$-manifolds
  and let $\Gamma \coloneqq \pi_1(M)$, $\Lambda \coloneqq \pi_1(N)$. Then
  \begin{align*}
      \frac{\medim^Z_1 (\Gamma \actson \widehat \Gamma)}{\vol (M)}
      & = \frac{\medim^Z_1 (\Lambda \actson \widehat \Lambda)}{\vol (N)}
      \quad\text{if $Z$ is $\Z$ or a finite field}
      \\
            \frac{\mevol_1 (\Gamma \actson \widehat \Gamma)}{\vol (M)}
      & = \frac{\mevol_1 (\Lambda \actson \widehat \Lambda)}{\vol (N)}.
  \end{align*}
\end{thm}

``Small'' resolutions over the equivalence relation ring lead to upper
bounds for the measured embedding dimension/volume, but at the moment,
the case of general orbit equivalence is out of reach, because the
equivalence relation rings do not exhibit the same level of exactness
and finiteness properties as the crossed product rings.

\begin{prop}[small resolutions over the equivalence relation ring;
    Proposition~\ref{prop:smallres}]
  Let $Z$ denote~$\Z$ (with the standard norm) or a finite field
  (with the trivial norm). 
  Let $\Rrel$ be a measured standard equivalence relation on
  a standard Borel probability space~$(X,\mu)$, let $n \in \N$,
  and let $D_*$ be a marked projective $Z\Rrel$-resolution
  of~$\linf {\alpha,Z}$ (up to degree~$n+1$). Then:
  If $\Gamma$ is a countable group of type~$\FP_{n+1}$ and if $\alpha$
  is standard probability action of~$\Gamma$ on~$(X,\mu)$ that
  induces~$\Rrel$, then
  \begin{align*}
    \medim^Z_n(\alpha)
    & \leq \dim_{Z\Rrel} (D_n)
    \\
    \mevol_n(\alpha)
    & \leq \lognorm (\partial^{D}_{n+1})
    \quad\text{if $Z = \Z$}.
  \end{align*}
\end{prop}

Similarly to the estimates for $L^2$-Betti numbers via resolutions
over dynamical rings, we obtain upper bounds for the measured
embedding dimension/volume through cost and the integral foliated
simplicial volume:

\begin{thm}[cost estimate; Theorem~\ref{thm:cost}]
  Let $\Gamma$ be an infinite group of type~$\FP_2$ and let $\alpha$ be
  a standard action of~$\Gamma$. Then
  \[ \medim^\Z_1 (\alpha) \leq \cost (\alpha) - 1.
  \]
\end{thm}

\begin{thm}[integral foliated simplicial volume estimate; Theorem~\ref{thm:simvolestimate}]
\label{thm:simvol_intro}
  Let~$M$ be an oriented closed connected aspherical
  $n$-manifold with fundamental group~$\Gamma$, let $\alpha$
  be a standard $\Gamma$-action, and let $k \in \{0,\dots, n\}$.
  Then
  \begin{align*}
    \medim^\Z_k (\alpha)
    & \leq {{n+1} \choose {k+1}} \cdot \ifsv M^\alpha
    \\
    \mevol_k (\alpha)
    & \leq \log(k+2) \cdot {{n+1} \choose {k+1}} \cdot \ifsv M^\alpha.
  \end{align*}
\end{thm}

The integral foliated simplicial volumes is zero for aspherical manifolds admitting an amenable open cover of small multiplicity~\cite{loehmoraschinisauer}.
By Theorem~\ref{thm:simvol_intro}, in this case also the measured embedding dimensions and volumes are zero.
This applies to many geometrically interesting situations including, e.g., aspherical manifolds with amenable fundamental group,
aspherical graph $3$-manifolds and smooth aspherical manifolds admitting a smooth circle action without fixed points~(Example~\ref{ex:amenable_cover}).
Moreover, Theorem~\ref{thm:simvol_intro} also provides upper bounds for the measured embedding volumes
of aspherical $3$-manifolds and, more generally, for Riemannian manifolds in terms of Riemannian volumes (Example~\ref{exa:ifsvhyp3} and Example~\ref{ex:riem:ifsv}).

\subsection{Open problems and further motivation}

The following beautiful result is proved by Ab\'ert--Bergeron--Fr\k{a}czyk--Gaboriau. In our previous paper~\cite{LLMSU} we revisit this result and put it into a larger homological context. 

\begin{thm}[\cite{ABFG21}]
Let $\Gamma=\mathrm{SL}_d(\IZ)$ with $d\ge 3$. Let $p$ be a prime. Then $\widehat t_n (\Gamma,\Gamma_*)=0$ and $\widehat b_n(\Gamma, \Gamma_\ast; \IF_p)=0$ for every~$n \in \{0,\dots,d-2\}$ and every residual chain~$\Gamma_\ast$ of $\Gamma$. 
\end{thm}

We also have $\widehat b_n(\Gamma, \Gamma_\ast; \IZ)=0$ in the same range but this follows easily from L\"uck's approximation theorem and the computation of $L^2$-Betti numbers. 

Their result is more general than stated above and provides a vanishing result for arithmetic lattices in all degress less than the $\IQ$-rank. 
Conjecturally, one should be able to replace the $\IQ$-rank by the $\IR$-rank and allow for more general, Benjamini--Schramm convergent, sequences of lattices. 

\begin{conjecture}\label{conj:lattices}
Let $G$ be a semisimple Lie group with finite center and without compact factors. Let $r$ be the $\IR$-rank of~$G$, and assume that $r\ge 2$. 
Let $(\Gamma_i)_{i\in I}$ be a sequence of irreducible lattices in~$G$ whose covolumes tend to infinity. Then 
\[\lim_{i \in I}
  \frac{\rk_{\IF_p} H_n(\Gamma_i;\IF_p)}
       {\vol(G/\Gamma_i)}=0
       \qand
       \lim_{i \in I}
  \frac{\log \# \tors H_n(\Gamma_i;\Z)}
       {\vol(G/\Gamma_i)}=0
\]
for every~$n \in \{0,\dots, r-1\}$.

In particular, $\widehat t_n (\Gamma,\Gamma_*)=0$ and $\widehat b_n(\Gamma, \Gamma_\ast; \IF_p)=0$ for every lattice $\Gamma$ in $G$, every~$n \in \{0,\dots,r-1\}$ and every residual chain~$\Gamma_\ast$ of $\Gamma$.  
\end{conjecture}

Some evidence comes from the following breakthrough result of Fr\k{a}czyk--Mellick--Wilkens~\cite[Theorem~B]{FMW} in degree~$1$. 

\begin{thm}[\cite{FMW}]
Let $(\Gamma_i)_{i\in I}$ and $G$ and $r\ge 2$ be as in Conjecture~\ref{conj:lattices}. Then 
\[\lim_{i \in I}
  \frac{\rk_{\IF_p} H_1(\Gamma_i;\IF_p)}
       {\vol(G/\Gamma_i)}=0.\]
\end{thm}

Part of the motivation for writing this foundational paper is the authors' programme to tackle the above conjecture by a dynamical-homological approach that hopefully allows 
to extend the result by Ab\'ert--Bergeron--Fr\k{a}czyk--Gaboriau to other lattices (of lower $\IQ$-rank) via orbit equivalence techniques.

%% file: GH.tex
\section{Basic notions}\label{sec:basics}

We recall basic notions on rings and modules associated
with dynamical systems: crossed product rings and equivalence
relation rings. As we are interested in quantitative
aspects, we will also introduce corresponding norms, sizes,
and dimensions.

\begin{setup}\label{setup:rings}
  Let $\Gamma$ be a countable group. We consider a standard
  $\Gamma$-action~$\alpha \colon \Gamma \actson (X,\mu)$, i.e., an
  essentially free measure preserving action of~$\Gamma$ on a standard
  Borel probability space~$(X,\mu)$. 

  Moreover, let $S$ be an algebra of measurable sets of~$X$
  with~$\Gamma \cdot S \subset S$.
  
  Let~$Z$ be the ring of integers~$\Z$ (with the usual absolute value) or a 
  finite field with the trivial norm
  \[ x \mapsto
\begin{cases}
  0 & \text{if $x = 0$};
  \\
  1 & \text{if $x \neq 0$}.
\end{cases}
\]
\end{setup}

\begin{rem}
	Most of our results can be generalised to the setting of principal ideal domains~$Z$ with 
	a \emph{discrete norm}, i.e.,
  a function~$|\cdot|\colon Z \to \R_{\geq 0}$ satisfying:
  \begin{itemize}
  \item $|0| = 0$ and $|x| \geq 1$ for all~$x \in Z \setminus \{0\}$;
  \item $|x + y| \leq |x| + |y|$ for all~$x, y \in Z$;
  \item $|x \cdot y| \leq |x| \cdot |y|$ for all~$x, y \in Z$.
  \end{itemize}
  However, for simplicity, we state all results for the case~$Z=\Z$ or $Z$ being
  a finite field only.
\end{rem}

\begin{rem}\label{rem:ae}
  When dealing with $L^\infty$-function spaces, we take the
  liberty of using pointwise notation. All corresponding
  notions such as equalities, defining equalities, suprema/infima,
  estimates, etc.\ are to be interpreted in the ``almost every'' sense. 
\end{rem}

\subsection{Rings}\label{subsec:rings}

We consider the following rings: 
\begin{itemize}
\item The ring~$\LinftyX \coloneqq L^\infty(X,\mu,Z)$ of essentially bounded measurable functions $X\to Z$ up to equality almost everywhere. 
The $\Gamma$-action on~$X$
  induces a left $\Gamma$-action on~$\LinftyX$:
  \[ \fa{\lambda \in \LinftyX} \fa{\gamma \in \Gamma}
  \gamma \cdot \lambda \coloneqq \bigl(x \mapsto \lambda(\gamma^{-1} \cdot x)\bigr).
  \]
\item The subring~$L$ of~$\LinftyX$ generated by~$S$.
  (Every element of~$L$ is a finite $Z$-linear combination
  of characteristic functions over members of~$S$).
\item The \emph{crossed product ring}~$L * \Gamma \subset \LinftyX * \Gamma$,
  i.e., the free $L$-module with basis~$\Gamma$,
  endowed with the multiplication given by 
  \[ (\lambda,\gamma) \cdot (\lambda' ,\gamma')
  \coloneqq \bigl(\lambda \cdot (\gamma \cdot \lambda'), \gamma \cdot \gamma'\bigr).
  \]
  Sometimes we also write~$\lambda\cdot\gamma$ instead of~$(\lambda, \gamma)$ for 
  an element in the crossed product ring.
\item Let $\Rrel \coloneqq \{(\gamma\cdot x,x)\mid x\in X, \gamma\in \Gamma\} \subset X
  \times X$ be the orbit relation of~$\alpha$.
  Let~$\nu$ be the non-negative measure on the Borel $\sigma$-algebra of~$\Rrel$ defined by
  \[
  \nu(A) \coloneqq \int_X \# \bigl(A \cap (\{x\} \times X)\bigr) \; d\mu(x).
  \]
  The \emph{equivalence relation ring}~$Z \Rrel$ is defined as
  \begin{align*}
    Z \Rrel \coloneqq \bigl\{\lambda \in L^\infty (\Rrel, \nu, Z)
    \bigm|
    &\sup_{x \in X}  \# \{y \mid \lambda(x, y) \neq 0 \} < \infty, \\
    &\sup_{y \in X} \# \{x \mid \lambda(x, y) \neq 0\} < \infty
    \bigr\}
  \end{align*}
  equipped with the convolution product
  \[
  (\lambda \cdot \lambda') (x, y) \coloneqq \sum_{w \in [x]_{\Rrel}} \lambda(x, w) \cdot \lambda'(w, y).
  \]
\end{itemize}
We have a commutative diagram of canonical inclusions
of rings (because the action of~$\Gamma$ on~$X$ is essentially free):
\[
\begin{tikzcd}
  \LinftyX
  \ar{r}
  & \LinftyX * \Gamma
  \ar{r}
  & Z\Rrel
  \\
  L
  \ar{r}
  \ar{u}
  & L * \Gamma
  \ar{u}
  & Z \Gamma
  \ar{l}
  \ar{u}
\end{tikzcd}
\]
Under the ring inclusion $\LinftyX * \Gamma \to Z \Rrel$, the element~$(\lambda,\gamma)\in \LinftyX\ast \Gamma$ corresponds to the function
\[
(\gamma'\cdot x,x)\mapsto \begin{cases}\lambda(\gamma \cdot x) &\mbox{ if } \gamma' = \gamma; \\ 0 &\mbox{ otherwise}.\end{cases} 
\]

A recurring theme will be that we need to control
$\ell^1$-norms or supports.

\begin{defn}
  Let $\lambda \colon \Rrel \to Z$ be a function. Let 
  \begin{align*}
    \Nbasic(\lambda,\args) \colon X & \to \N \cup \{\infty\}
    \\
    y & \mapsto
    \# \bigl\{ x \in X \bigm| \lambda(x,y) \neq 0 \bigr\}
  \end{align*}
  and $\Nbasic(\lambda) \coloneqq \sup_{y \in X} \Nbasic(\lambda,y)$. Symmetrically, we define $\Ntbasic(\lambda,\args)$
  and~$\Ntbasic(\lambda)$.
\end{defn}

If $\lambda \in Z\Rrel$, then $\Nbasic(\lambda) < \infty$ and $\Ntbasic(\lambda) < \infty$; per
our convention in Remark~\ref{rem:ae}, this is interpreted in
the ``almost everywhere'' sense.
  
\begin{lem}\label{lem:l1prod}
  Let $\lambda, \lambda'\in Z\Rrel$. Then 
  \[ |\lambda \cdot \lambda'|_1 \leq \Ntbasic(\lambda') \cdot |\lambda'|_\infty \cdot |\lambda|_1. 
  \]
  \begin{proof}
We have
  \begin{align*}
  	|\lambda \cdot \lambda'|_1 &=\int_\calR |\lambda\cdot \lambda'| \;d\nu
  	\\
	&= \int_X\sum_{y\in [x]_\calR}|\lambda\cdot \lambda'(x,y)| \;d\mu(x)
	\\
	&= \int_X\sum_{y\in [x]_\calR}\Bigl|\sum_{w\in [x]_\calR}\lambda(x,w)\cdot\lambda'(w,y)\Bigr| \;d\mu(x)
	\\
	&\le \int_X\sum_{y\in [x]_\calR}\sum_{w\in [x]_\calR}|\lambda(x,w)|\cdot |\lambda'(w,y)| \;d\mu(x)
	\\
	&= \int_X\sum_{w\in [x]_\calR}|\lambda(x,w)|\cdot \sum_{y\in [x]_\calR}|\lambda'(w,y)| \;d\mu(x)
	\\
	&\le \int_X\sum_{w\in [x]_\calR}|\lambda(x,w)|\cdot \Ntbasic(\lambda',w)\cdot |\lambda'|_\infty \;d\mu(x)
	\\
	&\le \Ntbasic(\lambda')\cdot |\lambda'|_\infty\cdot \int_X\sum_{w\in [x]_\calR}|\lambda(x,w)| \;d\mu(x)
	\\
	&= \Ntbasic(\lambda')\cdot |\lambda'|_\infty\cdot \int_\calR |\lambda| \;d\nu 
	\\
	&=\Ntbasic(\lambda') \cdot |\lambda'|_\infty\cdot |\lambda|_1. \qedhere
  \end{align*}
  \end{proof}
\end{lem}

\begin{defn}[support]
  Let $\lambda \in Z\Rrel$. We define (up to measure~$0$)
  \begin{itemize}
  \item the \emph{support of~$\lambda$} by 
    \[ \supp (\lambda) \coloneqq \bigl\{(x,y) \in \Rrel \bigm| \lambda(x,y) \neq 0\bigr\}
    \subset X \times X
    \]
  \item and the \emph{$1$-support of~$\lambda$} by
    \[ \supp_1 (\lambda)
    \coloneqq \proj_1 (\supp (\lambda))
    \subset X,
    \]
    where $\proj_1 \colon X \times X \to X$ denotes the projection onto the first factor.
  \end{itemize}
\end{defn}

\begin{rem}\label{rem:support:est}
  Let $\lambda \in Z \Rrel$.
  We have~$\lambda(x, y) = \chi_{\supp (\lambda)}(x, y) \cdot \lambda(x, y)$ for all $(x,
  y) \in \Rrel$.
  Moreover, $\lambda = \chi_{\supp_1 (\lambda)} \cdot \lambda$ (with respect to the convolution 
  product); 
  indeed, under the inclusion $\LinftyX
  \to Z \Rrel$, the element~$\chi_{\supp_1 (\lambda)}\in \LinftyX$ corresponds to the function
  \[
  (\gamma\cdot x, x) \mapsto 
  \chi_{\supp_1 (\lambda)} (\gamma\cdot x, x) =
  \begin{cases}
    \chi_{\supp_1 (\lambda)}(x) & \text{if~$\gamma = e$};
    \\
    0 &\text{otherwise}.
  \end{cases}
  \]  
  Furthermore, we have that $\mu(\supp_1 (\lambda)) \leq \nu(\supp (\lambda)) \leq |\lambda|_1$. 
\end{rem}

\begin{rem}
  For~$(\lambda,\gamma)\in\LinftyX * \Gamma\subset Z \Rrel$, we have:
  \begin{align*}
    \Nbasic( (\lambda, \gamma) )
    & \leq 1
    \\
    \Ntbasic( (\lambda, \gamma))
    & \leq 1
    \\
    \supp_1 ((\lambda, \gamma))
    & = \gamma^{-1} \cdot \supp (\lambda).
  \end{align*}
  Hence, in general we have $\Nbasic(\sum_{j=1}^k (\lambda_j, \gamma_j)) \leq k$, $\Ntbasic(\sum_{j=1}^k (\lambda_j, \gamma_j)) \leq k$, and $\supp_1(\sum_{j=1}^k (\lambda_j, \gamma_j)) \subset \bigcup_{j = 1}^k \gamma_j^{-1} \cdot \supp (\lambda)$. 
\end{rem}

\subsection{Base changes}

Using the inclusion relations between our basic rings, we can
also consider the associated base change and induction functors.

\begin{rem}
	\label{rem:linfX-ZR}
	We view $\LinftyX$ as a $Z \Rrel$-module via the following 
	scalar multiplication: Let $\iota\colon \LinftyX \hookrightarrow 
	\LinftyX * \Gamma \hookrightarrow Z\Rrel$ be the canonical inclusion. We define
	$\varepsilon\colon Z\Rrel \to \LinftyX$ by 
	\[
		\varepsilon (\lambda) (x) \coloneqq \sum_{y\in  [x]_\calR} \lambda (x,y)
	\]
	for all~$\lambda\in Z\Rrel$ and~$x\in X$.
	For all $\lambda\in Z\Rrel$ and $\lambda'\in \LinftyX$, we set
	\[
		\lambda\cdot \lambda'\coloneqq \varepsilon \bigl(\lambda\cdot \iota(\lambda')\bigr) \in \LinftyX,
	\]
	where the multiplication on the right hand side is the multiplication in~$Z\Rrel$.

        More explicitly, if $A \subset X$ is a measurable subset,
        $\lambda \in \LinftyX$, and~$\gamma\in \Gamma$, then this action amounts to $(\lambda,
        \gamma) \cdot \chi_A = \lambda \cdot \chi_{\gamma \cdot A}$,
        where multiplication on the right hand side is the usual
        pointwise multiplication of $L^\infty$-functions. 
\end{rem}

\begin{prop}\label{prop:flat}
  The modules~$\LinftyX * \Gamma$ and $L* \Gamma$ are flat over~$Z \Gamma$.
\end{prop}
\begin{proof}
  For a $Z\Gamma$-module~$M$, there is a canonical isomorphism
  \[
  	(\LinftyX*\Gamma)\otimes_{Z\Gamma} M \cong_Z \LinftyX \otimes_Z M.
  \]
  The claim follows because~$\LinftyX$ is flat over~$Z$: If $Z$ is a finite
  field, then this is clear. If $Z$ is $\Z$, then it is also known that
  $\LinftyX$ is free abelian~\cite{Steprans-free-abelian}. 
  For~$L$, the
  same argument applies.
\end{proof}

\begin{rem}
	In general, the module $Z\Rrel$ might not be flat over the group ring~$Z\Gamma$.
\end{rem}

\subsection{Modules}

We will mainly be interested in a simple type of projective
modules and the base ring~$\LinftyX\ast \Gamma$. 

\begin{setup}\label{setup:subsecmodules}
  Let $L$ be the subring of~$\LinftyX$ generated by~$S$ and
  let $R$ be an $L$-subalgebra of~$Z \Rrel$, e.g., 
  one of the rings~$\LinftyX$, $\LinftyX * \Gamma$, or $Z \Rrel$.
\end{setup}

\begin{defn}[marked projective module]
	\label{def:marked-proj-mod}
  A \emph{marked projective $R$-module} is a triple~$(M, (A_i)_{i \in
    I}, \varphi)$, consisting of
  \begin{itemize}
  \item an $R$-module~$M$,
  \item a finite family~$(A_i)_{i \in I}$ of measurable subsets of~$X$, and
  \item an $R$-isomorphism~$\varphi \colon M \to \bigoplus_{i \in I} R \cdot \chi_{A_i}$.
  \end{itemize}
  In the following, we will abbreviate
  \[ \gen {A_i} \coloneqq R \cdot \chi_{A_i}.
  \]
  The \emph{dimension} of the marked projective $R$-module~$(M, (A_i)_{i \in I}, \varphi)$ is given by
  \[ \dim (M) \coloneqq \sum_{i \in I} \mu(A_i) \in \IR_{\ge 0}.
  \]
\end{defn}

To simplify notation, we will leave $\varphi$ implicit
and call a description as above a \emph{marked presentation}.
Then, ~$\rk (M) \coloneqq \#I$ is the \emph{rank} of this marked presentation.
Finiteness is built into this definition of marked projective
modules as this is the only case that we will consider.

\begin{defn}[marked homomorphism]
\label{defn:marked_hom}
	Let $f\colon M\to N$ be an $R$-homomorphism between marked projective $R$-modules and let $M=\bigoplus_{i\in I}\spann{A_i}$, $N=\bigoplus_{j\in J}\spann{B_j}$ be the marked presentations.
	We say that~$f$ is
	\begin{itemize}
		\item a \emph{marked inclusion} if there exists an injective function $\sigma\colon I\to J$ such that $A_i\subset B_{\sigma(i)}$ and $f(\chi_{A_i}\cdot e_i)=\chi_{A_i}\cdot \chi_{B_{\sigma(i)}}\cdot e_{\sigma(i)}$;
		\item a \emph{marked projection} if there exists an injective function $\tau\colon J\to I$ such that $B_j\subset A_{\tau(j)}$ and
	 	\[
			f(\chi_{A_i}\cdot e_i)=\begin{cases}
				\chi_{A_i}\cdot \chi_{B_{\tau^{-1}(i)}}\cdot e_{\tau^{-1}(i)}
				& \text{if }i\in \im(\tau);
				\\
				0
				& \text{otherwise};
			\end{cases}
		\]
		\item a \emph{marked} $R$-homomorphism if~$f$ is a composition of marked inclusions and marked projections.
	\end{itemize}
\end{defn}

\begin{rem}[canonical hull/complement, canonical inclusion, canonical projection]
  Marked projective modules are projective $R$-modules, as witnessed
  by the projections~$R \to R$ of the form~$\lambda \mapsto \lambda \cdot \chi_{A_i}$.

  More precisely: Let $M = \bigoplus_{i \in I} \gen{A_i}$ be a marked
  projective $R$-module. Then the \emph{canonical hull} of~$M$ is~$F
  \coloneqq \bigoplus_I R$.  The marked inclusion~$M \to F$ is the
  \emph{canonical inclusion}. The $R$-linear marked projection~$F \to M$ given
  on the standard basis~$(e_i)_{i \in I}$ by
  \[ e_i \mapsto \chi_{A_i} \cdot e_i 
  \]
  is the \emph{canonical projection to~$M$}. We call
  \[ M' \coloneqq \bigoplus_{i \in I} \gen{X \setminus A_i}
  \]
  the \emph{canonical complement of~$M$}. By construction, the
  canonical inclusions and projections of $M$ and~$M'$ combine into
  an isomorphism~$M \oplus M' \cong_R F$.
\end{rem}

\begin{defn}[support]
  Let $M = \bigoplus_{i \in I} \gen{A_i}$ be a marked projective $R$-module. 
  For~$z = \sum_{i \in I} \lambda_i \cdot \chi_{A_i} \cdot e_i \in M$, we define
  the supports of~$z$ by 
  \begin{align*}
    \supp (z)
    & \coloneqq \bigcup_{i \in I} \supp (\lambda_i \cdot \chi_{A_i})
    \subset X \times X
    \\
    \supp_1 (z)
    & \coloneqq \bigcup_{i \in I} \supp_1 (\lambda_i \cdot \chi_{A_i})
    \subset X.
  \end{align*}
\end{defn}

\begin{rem}\label{rem:restricttosupport}
  Let $M$ be a marked projective $R$-module, let $z \in M$,
  and let $B \coloneqq \supp_1(z) \subset X$. Then
  \[ \chi_{X \setminus B} \cdot z = 0.
  \]
  Indeed, $\chi_B \cdot z = z$ (which follows from Remark~\ref{rem:support:est})
  and so~$\chi_{X \setminus B} \cdot z = \chi_X \cdot z - \chi_B \cdot z = z - z = 0$.
\end{rem}

\begin{rem}[defining homomorphisms out of marked projective modules]\label{rem:def:homo:marked:proj}
  Let $M = \bigoplus_{i \in I} \gen{A_i}$ be a marked projective $R$-module.
  Let $A \subset X$ be a measurable subset and let $z \in M$
  with~$\supp_1 (z) \subset A$. By Remark~\ref{rem:support:est} we have $z = \chi_{\supp_1 (z)} \cdot z$.
  Hence, 
  \begin{align*}
    f \colon 
    \gen A & \to M
    \\
    \lambda \cdot \chi_A 
    & \mapsto \lambda \cdot z
  \end{align*}
  is a well-defined $R$-homomorphism: Indeed, $1 \mapsto z$
  describes a well-defined $R$-homo\-morphism~$\widetilde f \colon R \to M$
  and $\supp_1 (z) \subset A$ shows that 
  \[ \widetilde f (\lambda \cdot \chi_A)
  = \lambda \cdot \chi_A \cdot \widetilde f (1)
  = \lambda \cdot \chi_A \cdot z
  = \lambda \cdot z
  \]
  holds for all~$\lambda \in R$. Therefore, $f$ is obtained
  from~$\widetilde f$ by composition with the canonical
  inclusion~$\gen A \to R$.

  If $A, B \subset X$ and $f \colon \gen A \to \gen B$ is an
  $R$-homomorphism, then $f$ is given by right multiplication with~$z
  \in R$ and evaluation at~$\chi_A$ shows that we may choose~$z$
  always in such a way that~$\supp (z) \subset A \times B$.
\end{rem}

\subsection{Chain complexes}

The main objects will be chain complexes consisting of marked
projective modules. 
We continue to work in Setup~\ref{setup:subsecmodules}.  

\begin{defn}[marked projective chain complex]
  A \emph{marked projective $R$-chain complex} is a pair~$(D_*, \eta)$,
  consisting of
  \begin{itemize}
  \item an $R$-chain complex~$D_*$ of marked projective $R$-modules and 
  \item a surjective $R$-homomorphism~$\eta \colon D_0 \to \LinftyX$,
    called \emph{augmentation}.
  \end{itemize}
  We also write~$\eta \colon D_* \epi \LinftyX$ for such a marked
  projective $R$-chain complex. 

  An \emph{$R$-chain map} between marked projective $R$-chain
  complexes is an $R$-chain map between the underlying chain complexes
  that is compatible with the augmentations. A chain map
  \emph{extends~$\id_{\LinftyX}$} if the map in degree~$-1$
  is~$\id_{\LinftyX}$.
\end{defn}

\begin{rem}
  Marked projective $R$-chain complexes admit canonical
  inclusions and projections into/from free $R$-chain complexes.
\end{rem}

\begin{rem}[induction of resolutions]\label{rem:fromZGtoZR}
  Let $R$ contain~$L*\Gamma$ and let $(C_*,\zeta)$ be a free
  $Z\Gamma$-resolution of the trivial $Z\Gamma$-module~$Z$. Let $r
  \in \N$. Choosing a $Z\Gamma$-basis of~$C_r$, we can view~$\Ind{Z\Gamma}{R}C_r\coloneqq R
  \otimes_{L * \Gamma} L \otimes_Z C_r$ as a marked projective
  $R$-module (possibly of infinite type). Hence, applying the functor~$\Ind{Z\Gamma}{R}\coloneqq R
  \otimes_{L * \Gamma} L\otimes_Z \args$ to~$(C_*,\zeta)$ leads to a
  marked projective $R$-chain complex (possibly of infinite type in
  each degree). In such situations, we will always consider this
  marked structure. In particular, also norms of elements and
  homomorphisms are interpreted in this way.
\end{rem}

\subsection{Norms}

We use the $\ell^1$-norm to measure the size of elements in
marked projective modules and consider the associated operator
norm for homomorphisms between marked projective modules.
We continue to work in Setup~\ref{setup:subsecmodules}.

\begin{defn}
  Let $M = \bigoplus_{i \in I} \gen{A_i}$ be a marked projective
  $R$-module. Then $M$ carries the norm~$\|\cdot\|_1$, inherited
  from the corresponding ``norm''
  \begin{align*}
    \|\cdot\|_1 \colon
    \bigoplus_{i \in I} R & \to \R_{\geq 0}
    \\
    \sum_{i \in I} \lambda_i \cdot e_i
    & \mapsto
    \sum_{i \in I} |\lambda_i|_1
  \end{align*}
  on the canonical hull of~$M$.
\end{defn}

\begin{defn}
  Let $f \colon M \to N$ be an $R$-homomorphism between
  marked projective $R$-modules. Then the \emph{norm~$\|f\|$ of~$f$}
  is the least real number~$c$ with
  \[ \fa{z \in M} \bigl\|f(z)\bigr\|_1 \leq c \cdot \|z\|_1.
  \]
\end{defn}

\begin{rem}\label{rem:finnorm}
  All $R$-homomorphisms~$f \colon M \to N$ between marked projective
  $R$-modules have finite norm:
  Because the marked presentations are compatible with the $\ell^1$-norms,
  it suffices to show that maps of the form~$f \colon R \to R,\
  z \mapsto z \cdot \lambda$
  for some~$\lambda \in R$ have finite norm: Let $z \in R$. Then, by Lemma~\ref{lem:l1prod},
  \[ \bigl\| f(z)\bigr\|_1
  = \bigl\| z \cdot f(1) \bigr\|_1
  = \bigl\| z \cdot \lambda \bigr\|_1
  \leq \Ntbasic(\lambda) \cdot |\lambda|_\infty \cdot \|z\|_1
  \]
  and hence
  \[
  	\|f\|\le \Ntbasic(\lambda)\cdot |\lambda|_\infty.
  \]
  By definition of~$Z\Rrel$ (which contains~$R$), we know that
  $|\lambda|_\infty$ and $\Ntbasic(\lambda)$ are indeed finite.  However, one should
  note that in general this norm of~$f$ cannot be controlled directly
  in terms of~$|\lambda|_1$.  In particular, to compute norms of $R$-maps,
  it will in general not be sufficient to just compute the
  $\ell^1$-norms of the values on the canonical basis. This will
  cause some unpleasant detours later on. An alternative
  description of the operator norm is provided in
  Section~\ref{subsec:opnormalt}.  
\end{rem}

In order to generalise the estimate from Remark~\ref{rem:finnorm}
to homomorphisms between general marked projective modules,
we introduce the following additional norm on homomorphisms.
The said generalisation will be given in Lemma~\ref{lem:normN1est}. 

\begin{defn}
\label{defn:infty-norm}
  Let $f \colon M \to N$ be an $R$-homomorphism
  between marked projective $R$-modules $M = \bigoplus_{i \in I} \gen{A_i}$
  and $N=\bigoplus_{j \in J} \gen{B_j}$. We set
  \[ \|f\|_\infty
   \coloneqq \max_{(i,j) \in I \times J} |\lambda_{ij} \cdot \chi_{B_j}|_\infty,
  \]
  where $(\lambda_{ij})_{(i,j) \in I \times J} \in M_{I \times J}(R)$
  is the matrix that describes~$f$ (through right multiplication by
  this matrix).
  
  Similarly, if $\eta\colon M\to \LinftyX$ is an $R$-homomorphism, we set
   \begin{align*}
   	\|\eta\|_\infty &\coloneqq \max_{i\in I} |\eta(\chi_{A_i} \cdot e_i)|_\infty.
   \end{align*}
\end{defn}

\subsection{Support and size estimates}

In order to handle the deformation and strictification of elements
and homomorphisms, we use $1$-supports and $1$-sizes.
We continue to work in Setup~\ref{setup:subsecmodules}.

\begin{rem}\label{rem:supportl1}
  If $M$ is a marked projective
  $R$-module and $z \in M$, then (Remark~\ref{rem:support:est})
  \[ \mu(\supp_1 (z)) \leq \nu (\supp (z)) \leq \|z\|_1.
  \]
\end{rem}

However, in general, such norm estimates will be too coarse.
Therefore, we introduce the following size notions: 

\begin{defn}[size]
	\label{def:size}
  Let $M = \bigoplus_{i \in I} \gen{A_i}$ be a marked projective $R$-module. 
  \begin{itemize}
  \item If $z \in M$, we abbreviate
    \[ \size_1 (z) \coloneqq \mu\bigl(\supp_1 (z)\bigr)
    \in [0,1].
    \]
    If $z = \sum_{i \in I} \lambda_i \cdot \chi_{A_i} \cdot e_i$, 
    we write
    \begin{align*}
    \Nbasic(z) &\coloneqq \sum_{i \in I} \Nbasic(\lambda_i \cdot \chi_{A_i}) \in \IN
    \\
    |z|_\infty &\coloneqq \sum_{i\in I} |\lambda_i\cdot \chi_{A_i}|_\infty \in \IR_{\ge 0}.
    \end{align*}
  \item If $N$ is a marked projective $R$-module and $f \colon M \to N$
    is an $R$-homomorphism, we set
    \begin{align*}
    \size_1 f
    &\coloneqq \sum_{i \in I} \size_1 (f (\chi_{A_i} \cdot e_i) ) \in \R_{\geq 0}
    \\  
    \Nsum(f) &\coloneqq \sum_{i\in I} \Nbasic(f(\chi_{A_i}\cdot e_i)) \in \IN
    \\
    \Nmax(f) &\coloneqq \max_{i \in I} \Nbasic(f (\chi_{A_i} \cdot e_i)) \in \N.
    \end{align*}
    Clearly, $\Nmax(f)\le \Nsum(f)\le \rk(M)\cdot \Nmax(f)$.
    Similarly, we define~$\Ntsum(f)$ and~$\Ntmax(f)$.
  \end{itemize}
\end{defn}

\begin{rem}
	\label{rem:mhom-N1}
	Let~$f\colon M \to N$ be a marked $R$-homomorphism (Definition~\ref{defn:marked_hom}) between marked projective $R$-modules and let~$z\in M$. Then,
	\[
		\Nbasic(f(z)) \le \Nbasic(z), \quad 
		\Ntbasic(f(z)) \le \Ntbasic(z), \qand
		|f(z)|_\infty \le |z|_\infty.
	\]
\end{rem}

\begin{lem}[support estimates]\label{lem:supp1:est}
  Let $f \colon M \to N$ be an $R$-homomorphism between marked
  projective $R$-modules $M$ and~$N$ and let $z \in M$.
  \begin{enumerate}[label=\enum]
  \item For all~$\lambda \in R$, we have~$\size_1(\lambda \cdot z) 
    \leq \Nbasic(\lambda) \cdot \size_1 (z)$;
  \item\label{i:easy} 
    $ \supp_1 (f(z))
    \subset \supp_1 (z)  
    $;
  \item\label{i:N1}
    $ \size_1 (f(z))
    \leq \Nbasic(z) \cdot \size_1(f);
    $
  \item\label{i:size comp}
  Let~$g\colon L\to M$ be an $R$-homomorphism between marked projective $R$-modules.
  Then $\size_1(f\circ g)\le \Nsum(g)\cdot \size_1(f)$;
  \item\label{i:size comp embedding}
  Let~$g\colon L\to M$ be a marked $R$-homomorphism between marked projective $R$-modules.
  Then $\size_1(f\circ g)\le \size_1(f)$.
  \end{enumerate}
\end{lem}
\begin{proof}
  Let $M = \bigoplus_{i \in I} \gen{A_i}$ be the marked presentation
  of~$M$ and $z = \sum_{i \in I} \lambda_i \cdot \chi_{A_i} \cdot e_i$
  with~$\lambda_i \in R$. 
  
  (i)
  Let $A \coloneqq \supp (\lambda) \subset X \times X$
  and $B \coloneqq \supp_1 (z) \subset X$.
  If $x \in \supp_1(\lambda \cdot z)$, then by definition of the
  convolution product there exists a~$w \in [x]_{\Rrel}$ with
  $(x,w) \in A$ and~$w \in B$. 
  Therefore,
  \begin{align*}
    \size_1 (\lambda \cdot z) = \mu\bigl(\supp_1(\lambda \cdot z)\bigr)
    & \leq
    \mu \bigl(
    \bigl\{ x\in X
    \bigm| A \cap \proj_2^{-1}(B) \cap \{x\} \times X \neq \emptyset
    \bigr\}
    \bigr)
    \\
    & \leq
    \int_X \# \bigl( A \cap \proj_2^{-1}(B) \cap \{x\} \times X\bigr) \; d\mu(x)
    \\
    & =
    \nu \bigl(A \cap \proj_2^{-1}(B)\bigr).
  \end{align*}
  As $\nu$ can also be computed through~$\proj_2$,
  we obtain
  \begin{align*}
    \size_1(\lambda \cdot z) & \leq  \nu \bigl(A \cap \proj_2^{-1}(B)\bigr)
    \\
    & =
    \int_X \# \bigl( A \cap \proj_2^{-1}(B) \cap X \times \{y\}\bigr) \; d\mu(y)
    \\
    & = \int_B \# \bigl(A \cap X \times \{y\}\bigr) \;d\mu(y)
    \\
    & \leq \Nbasic(\lambda) \cdot \mu(B)
    & \text{(because~$A = \supp (\lambda)$)}
    \\
    & = \Nbasic(\lambda) \cdot \size_1 (z).
  \end{align*}
  
  (ii)
  We compute 
  \[ \supp_1 \bigl(f(z)\bigr)
  = \supp_1 \bigl(f( \chi_{\supp_1 z} \cdot z)\bigr)
  = \supp_1 \bigl( \chi_{\supp_1 z} \cdot f(z)\bigr)
  \subset \supp_1 (z).
  \]

  (iii) We use part~(i):
  \begin{align*}
    \size_1 \bigl(f(z)\bigr)
    & =
    \size_1 \biggl(
    \sum_{i \in I} \lambda_i \cdot f(\chi_{A_i} \cdot e_i)\biggr)
    \\
    & =\mu \biggl(\supp_1 \bigl(
    \sum_{i \in I} \lambda_i \cdot f(\chi_{A_i} \cdot e_i)\bigr) \biggr)
    \\
    &\leq  \sum_{i \in I} \mu \biggl(\supp_1 \bigl(
    \lambda_i \cdot f(\chi_{A_i} \cdot e_i)\bigr) \biggr)
    \\
    & =
    \sum_{i \in I} \mu \biggl(\supp_1 \bigl(
    \lambda_i \cdot \chi_{A_i} \cdot f(\chi_{A_i} \cdot e_i)\bigr) \biggr)
      & \text{(by Remark~\ref{rem:def:homo:marked:proj})}
    \\
    & =
    \sum_{i \in I} \size_1 \bigl(\lambda_i \cdot \chi_{A_i} \cdot f(\chi_{Ai} \cdot e_i) \bigr)
    \\
    & \leq
    \sum_{i \in I} \Nbasic(\lambda_i \cdot \chi_{A_i} ) \cdot \size_1\bigl( f(\chi_{A_i} \cdot e_i)\bigr)
    & \text{(by part~(i))}
    \\
    & \leq \Nbasic(z) \cdot \size_1(f).
  \end{align*}
  
  (iv)
  Let $L=\bigoplus_{j\in J}\spann{B_j}$ be the marked presentation of~$L$.
  We have
  \begin{align*}
  	\size_1(f\circ g)
	&= \sum_{j\in J} \size_1\bigl(f(g(\chi_{B_j}\cdot e_j))\bigr)
	\\
	&\le \sum_{j\in J} \Nbasic\bigl(g(\chi_{B_j}\cdot e_j)\bigr)\cdot \size_1(f)
	&\text{(by part~(iii))}
	\\
	&= \Nsum(g)\cdot \size_1(f).
  \end{align*}
  
  (v)
  By adding trivial summands to the marked presentations, we may assume that the marked $R$-homomorphism~$g$ is of the form \[g\colon L=\bigoplus_{i\in I}\spann{B_i}\to \bigoplus_{i\in I}\spann{A_i}=M,\] \[g(\chi_{B_i}\cdot e_i)=\chi_{B_i}\cdot \chi_{A_i}\cdot e_i.\]
  We have
  \begin{align*}
  	\size_1(f\circ g)
	&= \sum_{i\in I}\size_1\bigl(f(g(\chi_{B_i}\cdot e_i))\bigr)
	\\
	&= \sum_{i\in I}\size_1\bigl(f(\chi_{B_i}\cdot \chi_{A_i}\cdot e_i)\bigr)
	\\
	&= \sum_{i\in I}\size_1\bigl(\chi_{B_i}\cdot f(\chi_{A_i}\cdot e_i)\bigr)
	\\
	&\le \sum_{i\in I}\Nbasic(\chi_{B_i})\cdot \size_1\bigl(f(\chi_{A_i}\cdot e_i)\bigr)
	& \text{(by part~(i))}
	\\
	&\le \sum_{i\in I}\size_1\bigl(f(\chi_{A_i}\cdot e_i)\bigr)
	\\
	&= \size_1(f).
  \end{align*}
  This completes the proof.
\end{proof}

\begin{lem}[]
	\label{lem:etaz}
	Let~$M = \bigoplus_{i\in I} \gen{A_i}$ be a marked projective~$R$-module and let $\eta\colon M\to \LinftyX$
	be an $R$-homomorphism. Let~$z\in M$. Then,
	\[
		|\eta(z)|_\infty \le \Ntbasic(z) \cdot |z|_\infty \cdot \|\eta\|_\infty.
	\]
\end{lem}

\begin{proof}
	Without loss of generality, we may assume that $M =
        \gen{A}$. Since~$\eta$ is~$R$-linear and by
        Remark~\ref{rem:linfX-ZR}, we have
	\[
	\eta(z)(x) = \bigl(z \cdot \eta(\chi_A \cdot e) \bigr)(x)
        = \sum_{y\in [x]_\Rrel} z(x,y) \cdot \eta(\chi_A \cdot e)(y). 
	\]
	Thus,
	\begin{align*}
		|\eta(z)|_\infty &= \sup_{x\in X} \Bigl| \sum_{y\in [x]_\Rrel} z(x,y) \cdot \eta(\chi_A \cdot e)(y)\Bigr|\\
		&\le \sup_{x\in X} \Bigl| \sum_{y\in [x]_\Rrel} z(x,y)\Bigr| \cdot \|\eta\|_\infty\\
		&\le \sup_{x\in X}  \Ntbasic(z, x) \cdot |z|_\infty \cdot \|\eta\|_\infty\\
		&= \Ntbasic(z)\cdot |z|_\infty \cdot \|\eta\|_\infty. \qedhere
	\end{align*}
\end{proof}

\begin{lem}[norm estimate]\label{lem:normN1est}
  Let $f \colon M \to N$ be an $R$-homomorphism
  between marked projective $R$-modules. Then
  \[ \|f\| \leq \Ntmax(f) \cdot \|f\|_\infty.
  \]
  We abbreviate~$K_f \coloneqq \Ntmax(f) \cdot \|f\|_\infty$.
\end{lem}
\begin{proof}
  This is a straightforward calculation: Let $(\lambda_{ij})_{(i,j)\in
    I\times J}$ be the matrix describing~$f$ (through right
  multiplication by this matrix).
  Let $z \in M$, written as~$z = \sum_{i \in I} \lambda_i \cdot \chi_{A_i} \cdot e_i$.
  Then we obtain
  \begin{align*}
    \bigl\| f(z) \bigr\|_1
    & =
    \sum_{j \in J} \ \biggl|
    \sum_{i \in I}
    \lambda_i \cdot \chi_{A_i} \cdot \lambda_{ij} \cdot \chi_{B_j}\biggr|_1
    \\
    & \leq
    \sum_{j \in J}
    \sum_{i \in I}
    |\lambda_i \cdot \chi_{A_i} \cdot \lambda_{ij} \cdot \chi_{B_j}|_1
    \\
    & \leq
    \sum_{i \in I}
    \sum_{j \in J}
    \Ntbasic(\lambda_{ij} \cdot \chi_{B_j}) \cdot |\lambda_{ij} \cdot \chi_{B_j}|_\infty
    \cdot |\lambda_i \cdot \chi_{A_i} |_1
    & \text{(Lemma~\ref{lem:l1prod})}
    \\
    & \leq
    \|f\|_\infty \cdot \biggl(\sum_{i \in I} \Ntbasic\bigl(\sum_{j \in J} \lambda_{ij} \cdot \chi_{B_j} \cdot e_j\bigr) \cdot |\lambda_i \cdot \chi_{A_i} |_1 \biggr)
    \\
     & =
    \|f\|_\infty \cdot \biggl(\sum_{i \in I} \Ntbasic\bigl(f(\chi_{A_i} \cdot e_i)\bigr) \cdot |\lambda_i \cdot \chi_{A_i} |_1\biggr)
    \\
    & \leq
    \Ntmax(f) \cdot \|f\|_\infty \cdot \|z\|_1,
  \end{align*}
  which shows that $\|f\| \leq \Ntmax(f) \cdot \|f\|_\infty$. 
\end{proof}

\begin{rem}\label{rem:rk1infN1}
  Let $A \subset X$ be measurable, let $\lambda \in R$
  with~$\supp_1(\lambda) \subset A$, and let $f \colon \gen A \to R$
  be the $R$-homomorphism given by right multiplication with~$\lambda$.
  Then 
  \[ \|f\|_\infty = |\lambda|_\infty
  \qand
  \Ntsum(f) = \Ntmax(f) = \Ntbasic(\lambda).
  \]
\end{rem}

\begin{rem}\label{rem:norms:marked:homo:bounded:by:1}
	Let~$f\colon M\to N$ be a marked $R$-homomorphism between marked projective $R$-modules, i.e.,
	a composition of marked inclusions and marked projections.
	Then 
	\[
		\Nmax(f)\le 1, \quad \Ntmax(f)\le 1, \quad \|f\|_\infty\le 1, \quad \|f\|\le 1.
	\]
\end{rem}

\subsection{An explicit description of the operator norm}\label{subsec:opnormalt}

We provide an explicit description of the operator norm for
homomorphisms between marked projective modules. 

\begin{setup}[]
	\label{setup:opnormalt}
	Let~$\alpha\colon \Gamma \actson (X,\mu)$ be a standard action
        of a countable group and 
	let~$R\coloneqq \LinftyX *\Gamma$ be the crossed product ring.
	Let $f\colon \bigoplus_{i\in I} \gen{A_i} \to \bigoplus_{j\in J} \gen{B_j}$
	be an $R$-linear map. Then,~$f$ is given by right multiplication with a matrix~
	$z \coloneqq \bigl( z_{i,j}\bigr)_{i,j}$ over the crossed product ring. There is 
	a finite family~$(U_k)_{k\in K}$ of pairwise disjoint measurable subsets of~$X$
	and a finite subset~$F\subset \Gamma$ such that for all $i\in I, j\in J$, we have
	\[
		z_{i,j} = \sum_{(k,\gamma)\in K\times F} a_{i,j,k,\gamma} \cdot (\chi_{\gamma U_k}, \gamma),
	\]
	where~$a_{i,j,k,\gamma} \in Z$.
	
	Moreover, we call such a presentation \emph{reduced} if for 
	all~$i\in I, j\in J, k\in K,\gamma\in F$ with~$a_{i,j,k,\gamma} \neq 0$, the following hold:
	\begin{enumerate}
	\item $U_k \subseteq B_j$;
	\item $\gamma U_k \subseteq A_i$.
	\end{enumerate}
	Note that in particular, this implies that $\gamma^{-1} A_i\cap U_k = U_k \subseteq B_j$.
	It is straightforward to verify that we can always find a reduced presentation.
\end{setup}

\begin{prop}[]
  \label{prop:opnormalt}
  In the situation of Setup~\ref{setup:opnormalt}, 
	let~$f\colon M\to N$ be an $R$-homomorphism between marked projective $R$-modules.
	Let~$z$ be the matrix representing~$f$ as in Setup~\ref{setup:opnormalt}. Then,
	\[
		\|f\| = \max_{i\in I} \max\Bigl\{
			\sum_{j\in J, (k,\gamma)\in L} |a_{i,j,k,\gamma}| \Bigm| L\subseteq K\times F \text{ with }
			\mu\Bigl( \bigcap_{(k,\gamma)\in L} \gamma U_k \Bigr) > 0
		\Bigr\}.
	\]
	In particular, we have $\|f\|\in \IN$.
\end{prop}

As a first step, we show that we can restrict to modules of the form~$\Linf A$.

\begin{lem}[]
	\label{lem:norm-expl-base}
	Let~$A\subseteq X$ and 
	$f\colon \gen{A} \to N$ be an $R$-homomorphism between marked projective $R$-modules.
	Then, 
	\[
		\|f\| = \|f|_{\Linf A}\|.
	\]
	Here, $f|_{\Linf A}$ denotes the precomposition of~$f$ with the canonical inclusion~$\Linf A \hookrightarrow
	R$.
\end{lem}

\begin{proof}
	For~$\gamma\in \Gamma$, we define~$f_\gamma\colon \Linf A \to N$ by
	\[ 
		f_\gamma(g) \coloneqq f \bigl(
			(\chi_X, \gamma)\cdot (g, 1)
		\bigr).
	\]
	By definition of the norm as an~$\ell^1$-norm over the~$\Linf A$-summands, it is clear that
	\[
		\|f\| = \sup_{\gamma\in \Gamma} \|f_\gamma\|.
	\]
	It thus suffices to show that for~$\gamma\in \Gamma$, we have~$\|f_{\gamma}\| \le \|f_1\|$.
	Indeed, because~$f$ is $R$-linear, for all~$g\in \Linf A$, we have
	\begin{align*}
		\|f_{\gamma}(g)\|_1 &= \bigl\| f \bigl(
			(\chi_X, \gamma)\cdot (g, 1)
		\bigr)\bigr\|_1\\
		&= \bigl\| (\chi_X, \gamma)\cdot f ((g, 1))\bigr\|_1\\
		&= \| f ((g, 1))\|_1\\
		&\le \|f_1\|\cdot |g|_1,
	\end{align*}
	where in the penultimate step, we use that multiplication with~$(\chi_X, \gamma)$ defines an isometry
	because the action of~$\gamma$ preserves the probability measure~$\mu$.
\end{proof}

As a second step, we compute the value of the homomorphism~$f$
on specific small building blocks:

\begin{lem}[]
	\label{lem:UL-compu}
	Let~$f\colon\gen A \to \gen B$ be an $R$-homomorphism,
	given as in Setup~\ref{setup:opnormalt} by 
	right multiplication with
	\[
		z \coloneqq \sum_{(k,\gamma)\in K\times F} a_{k,\gamma}\cdot (\chi_{\gamma U_k}, \gamma),
	\]
	where~$a_{k,\gamma}\in Z$, $F\subseteq \Gamma$ is a finite set, and~$(U_k)_{k\in K}$ are pairwise
	disjoint measurable subsets of~$B$ with~$\gamma U_k \subseteq A$ whenever~$a_{k,\gamma}\neq 0$ 
	(see Remark~\ref{rem:def:homo:marked:proj}).
	For~$L\subseteq K\times F$, we define 
	\[
		U(L) \coloneqq \bigcap_{(k,\gamma)\in L} \gamma U_k \cap \bigcap_{(k,\gamma)\in (K\times F)\backslash L}
		 A\backslash \gamma U_k. 
	\]
	Then, the~$(U(L))_{L\subseteq K\times F}$ are pairwise disjoint subsets of~$A$ and 
	\[
		f\bigl((\chi_{U(L)}, 1)\bigr) = \sum_{\gamma\in \Gamma} \aonlyk{}{\gamma}{L} \cdot (\chi_{U(L)}, \gamma),
	\] 
	where 
	\[
		\aonlyk{}{\gamma}{L} \coloneqq \begin{cases}
			0 & \text{if }U(L) = \emptyset;\\
			a_{k,\gamma} & \text{if there exists }k\in K \text{ with }(k,\gamma)\in L;\\
			0 & \text{otherwise}.
		\end{cases}
	\]
	Because the~$U_k$ are pairwise disjoint, there is at most one~$k\in K$ with~$(k,\gamma)\in L$
	unless~$U(L) = \emptyset$.
\end{lem}

\begin{proof}
        By construction, the~$U(L)$ are pairwise disjoint and
	for~$(k,\gamma)\in K\times F$, we have
	\begin{align}
		U(L) \cap \gamma U_k = \begin{cases}
		U(L) & \text{if }(k,\gamma)\in L; \\
		\emptyset & \text{if }(k,\gamma)\not\in L.
		\end{cases}
                \label{eq:ULUk}
	\end{align}
	Thus, with Equation~\eqref{eq:ULUk}, we obtain
	\begin{align*}
		f\bigl((\chi_{U(L)}, 1)\bigr)
		&= (\chi_{U(L)}, 1) \cdot \sum_{(k,\gamma)\in K\times F} a_{k,\gamma}\cdot (\chi_{\gamma U_k}, \gamma)\\
		&= \sum_{(k,\gamma)\in K\times F} a_{k,\gamma}\cdot (\chi_{U(L)\cap \gamma U_k}, \gamma)\\
		&= \sum_{(k,\gamma)\in L} a_{k,\gamma}\cdot (\chi_{U(L)}, \gamma) 
                \\
		&= \sum_{\gamma\in \Gamma} \aonlyk{}{\gamma}{L} \cdot (\chi_{U(L)}, \gamma).
                \qedhere
	\end{align*}
\end{proof}

\begin{proof}[Proof of Proposition~\ref{prop:opnormalt}]
	Because the operator norm is defined with respect to the~$\ell^1$-norm on~$M$, we may assume without loss
	of generality that~$M = \gen{A}$. We will therefore drop the index~$i$ from the notation.
	Let~$m$ be the maximum on the right hand side of the claim.
	In order to prove that~$\|f\|\ge m$, let~$L_m\subseteq K\times F$ be a maximal subset realising the maximum, i.e., 
	\[m = \sum_{j\in J, (k,\gamma)\in L_m} |a_{j,k,\gamma}| \qand \mu \Bigl(
		\bigcap_{(k,\gamma)\in L_m} \gamma U_k
	\Bigr) > 0.
	\]
	Then, also~$\mu(U(L_m)) > 0$, by maximality of~$L_m$. 
	We consider the characteristic function~$x \coloneqq (\chi_{U(L_m)},1)$ as a witness and compute
	\begin{align*}
		|f(x)|_1
		&= \sum_{j\in J} |f_j(\chi_{U(L_m)})|_1\\
		&= \sum_{j\in J} \;\biggl|
			\sum_{\gamma\in \Gamma} \aonlyk{j,}{\gamma}{L_m} \cdot (\chi_{U(L_m)}, \gamma)
		\biggr|_1 
		& (\text{Lemma~\ref{lem:UL-compu}})\\
		&= \sum_{j\in J}\sum_{\gamma \in \Gamma} |\aonlyk{j,}{\gamma}{L_m}| \cdot \mu(U(L_m))\\
		&= m\cdot |x|_1.
	\end{align*}
	Since~$|x|_1 = \mu(U(L_m)) > 0$, this proves that~$\|f\| \ge m$.
	
	To show the converse inequality, we use the canonical $R$-isomorphism
	\[
		\gen A \cong \bigoplus_{L\subseteq K\times F} \gen{U(L)},
	\]
	where~$U(L)$ is defined as in Lemma~\ref{lem:UL-compu}. Equipping the right hand side with the~$\ell^1$-norms
	of the summands, the canonical isomorphism of $R$-modules is an isometry. Thus, it suffices 
	to prove that~$|f(x)|_1 \le m \cdot |x|_1$ for~$x\in \gen{U(L)}$. By Lemma~\ref{lem:norm-expl-base},
	it suffices to consider~$g\in L^\infty(U(L))$. 
	We can write
	\[
		g = \sum_{s\in S} g_s\cdot \chi_{V_s}
	\]
	for a finite set~$S$, $g_S\in Z$, and pairwise disjoint subsets~$V_s$ of~$U(L)$. The calculation in Lemma~\ref{lem:UL-compu}
	shows that 
	\begin{align*}
		|f(g)|_1 &= \sum_{j\in J} |f_j(g)|_1\\
		&= \sum_{j\in J} \;\biggl|
			\sum_{\gamma\in \Gamma}\aonlyk{j,}{\gamma}{L}\cdot  \sum_{s\in S}  g_s \cdot (\chi_{V_s},\gamma)
			\biggr|_1\\
		&\le \biggl(\sum_{j\in J} \sum_{\gamma\in \Gamma} |\aonlyk{j,}{\gamma}{L}|\biggr) \cdot |g|_1\\
		&\le m\cdot |g|_1.
                \qedhere        
	\end{align*} 
\end{proof}

The notion of $S$-adapted modules and morphisms used below will be introduced later (Definition~\ref{defn:adapted}).

\begin{cor}[adaptation with the same norm]
	\label{cor:adaptsamenorm}
	Let~$f \colon M\to N$ be an $S$-adapted morphism between marked 
	projective $R$-modules. Let $M'$ be a marked projective summand of~$M$.
	Then, there exists a marked projective summand~$M''$ of~$M$ that is $S$-adapted such that 
	\[
		M' \subseteq M'' \qand ||f|_{M''}|| = ||f|_{M'}||.
	\]
\end{cor}

\begin{proof}
	As it suffices to prove the claim componentwise in the 
	domain, suppose that~$M=\gen A$ and $M' = \gen{A'}$ with~$A' \subset A$. Set
	\[ A''\coloneqq \bigcup_{\substack{L\subseteq K\times F \text{ s.t.}\\\mu (A'\cap U(L) ) >0}} U(L),\]
	where~$U(L)$ is defined as in Lemma~\ref{lem:UL-compu}.
	
	Then, since~$f$ is $S$-adapted, so is~$M''\coloneqq \gen{A''} $ and~$M''$ contains $\gen{A'}$.
	Moreover, by the explicit description of the norm (Proposition~\ref{prop:opnormalt}), we have 
	$\|f|_{M''}\| = \|f|_{M'}\|$ as the maximum ranges
	over the same sums of coefficients.
\end{proof}

\section{Almost equality}\label{sec:GH}

We introduce quantitative notions of ``almost equality'' for
homomorphisms between marked projective modules and marked projective
chain complexes. Almost equality for homomorphisms requires that the
homomorphisms are equal except on a marked summand of small dimension
and that the norm on this exceptional summand is uniformly controlled.
For the comparison of homomorphisms with different domains/targets, we
introduce a controlled Gromov--Hausdorff distance. This admits a 
straightforward generalisation to chain complexes. In particular, we
will be able to speak of marked projective chain complexes that are
``almost equal''.

\begin{setup}\label{setup:GH}
  Let $\Gamma$ be a countable group and let $Z$ denote $\Z$ (with
  the usual norm) or a finite field (with the trivial norm).
  We consider a standard $\Gamma$-action~$\alpha \colon \Gamma \actson (X,\mu)$. 
  Moreover, let $\Rrel$ be the associated orbit relation, and
  let $R \subset Z \Rrel$ be a subring that
  contains~$\linf {\alpha,Z} * \Gamma$.
\end{setup}

\subsection{Almost equality}
\label{subsec:eqdelta}

We begin with a notion of almost equality that only requires
the homomorphisms to be equal except on a marked summand
of small dimension.

\begin{defn}[marked decomposition, almost equality]
  \label{def:marked-decomp}
  In Setup~\ref{setup:GH}, 
  let $M = \bigoplus_{i \in I} \gen{A_i}$ be a marked projective $R$-module.
  \begin{itemize}
  \item A \emph{marked decomposition} of~$M$
    is the canonical $R$-isomorphism
    \[ M \cong_R \bigoplus_{i \in I} \gen{A_i \setminus B_i}
    \oplus \bigoplus_{i \in I} \gen{B_i}
    \]
    induced by a family~$(B_i)_{i \in I}$ of (possibly empty)
    measurable sets~$B_i \subset A_i$.   
  \item Let $\delta \in \R_{>0}$ and let $f, f' \colon M \to N$
    be $R$-homomorphisms between marked projective $R$-modules.
    We write~$f =_\delta f'$ if there exists a
    marked decomposition~$M \cong_R M_0 \oplus M_1$
    with
    \[ f|_{M_0} = f'|_{M_0}
    \qand \dim (M_1) < \delta.
    \]
  \item
    If $z,z' \in R$, then we write $z =_\delta z'$
    if~$\mu(\supp_1(z-z')) < \delta$.
  \item 
    If $z,z' \in \LinftyX$, then we write $z =_\delta z'$
    if~$\mu(\supp(z-z')) < \delta$.
  \end{itemize}
\end{defn}

\begin{ex}[]
	\label{ex:alm-eq}
	Let~$\gamma\in \Gamma$ and $U,V\subset X$ be measurable subsets.
	Let~$f_U,f_V\colon R\to R$ be the $R$-linear maps
	 given by right multiplication with~$(\chi_{\gamma U}, \gamma)$ and~$(\chi_{\gamma V}, \gamma)$, respectively.
	Set~$\delta\coloneqq \mu(U\symmdiff V)$, where~$\symmdiff$ denotes the symmetric difference.
	Then, $f_U =_\delta f_V$. 
\end{ex}

\begin{rem}[]
	Let~$M$ be a marked projective~$R$-module.
	The notion of almost equality can also be defined in the same way for $R$-linear maps
	$M\to \LinftyX$. The following lemmas hold for such maps to~$\LinftyX$ in an analogous way.
\end{rem}

\begin{lem}\label{lem:almosteqchar}
  In the situation of Setup~\ref{setup:GH}, 
  let $f,f' \colon M \to N$ be $R$-homomorphisms
  between marked projective $R$-modules and let $\delta \in \R_{>0}$.
  Then,
  \[ f=_\delta f'
  \iff
  \size_1(f-f') < \delta.
  \]
\end{lem}
\begin{proof}
  Let $M = \bigoplus_{i \in I} \gen{A_i}$ be the marked presentation
  of~$M$. 

  We first assume that $f =_\delta f'$. Let $M \cong_R \bigoplus_{i \in
    I} \gen{A_i \setminus B_i} \oplus \bigoplus_{i \in I} \gen{B_i}$ be a
  corresponding marked decomposition. In particular, we have
  $(f-f')(\chi_{A_i} \cdot e_i) = (f-f')(\chi_{B_i} \cdot e_i)$
  for all~$i \in I$.
  Then, we obtain
  \begin{align*}
    \size_1 (f-f')
    &
    = \sum_{i \in I} \size_1 \bigl( (f-f') (\chi_{B_i} \cdot e_i)\bigr)
    \\
    &
    \leq \sum_{i \in I} \size_1 (\chi_{B_i} \cdot e_i)
    & \text{(Lemma~\ref{lem:supp1:est}~\ref{i:easy})}
    \\
    & = \sum_{i \in I} \mu(B_i) = \dim \Bigl( \bigoplus_{i \in I} \gen{B_i} \Bigr)
    < \delta.
  \end{align*}

  Conversely, let $\size_1(f-f') < \delta$.
  For~$i \in I$, we set $B_i \coloneqq \supp_1( (f-f') (\chi_{A_i} \cdot e_i) )$.
  We consider
  \[ M_0 \coloneqq \bigoplus_{i \in I} \gen{A_i \setminus B_i}
  \qand
  M_1 \coloneqq \bigoplus_{i \in I} \gen{B_i}.
  \]
  Then $M_0 \oplus M_1$ is a marked decomposition of~$M$. 
  Because of~$\size_1(f-f') < \delta$ and the definition of~$B_i$, we have
  \[ \dim M_1
  = \sum_{i \in I} \mu(B_i)
  = \sum_{i \in I} \size_1 \bigl( (f-f')(\chi_{A_i} \cdot e_i)\bigr)
  = \size_1 (f-f') < \delta.
  \]
  Moreover, by construction,
  \[ f(\chi_{A_i \setminus B_i} \cdot e_i) = f'(\chi_{A_i \setminus B_i} \cdot e_i)
  \]
  for all~$i \in I$ (Remark~\ref{rem:restricttosupport}).
  Hence, $f|_{M_0} = f'|_{M_0}$.
\end{proof}

\begin{lem}\label{lem:almosteqinherit}
  In the situation of Setup~\ref{setup:GH}, 
  let $L,M,N$
  be marked projective $R$-modules, let $f,f' \colon M \to N$
  be $R$-homomorphisms, and let $\delta,\delta' \in \R_{>0}$.
  We assume
  that $f =_\delta f'$.
  \begin{enumerate}[label=\enum]
  \item
  \label{itm:composition:additivity}
   If $f''\colon M\to N$ is an $R$-homomorphism with $f'=_{\delta'} f''$, then $f=_{\delta+\delta'}f''$.
  \item 
	\label{itm:postcomp}  
  If $h \colon N \to L$ is an $R$-homomorphism, then
    \[ h \circ f =_\delta h \circ f'.
    \]
  \item 
	\label{itm:precomp}    
  If $g \colon L \to M$ is an $R$-homomorphism,
    then 
    \[ f\circ g =_{\Nsum(g) \cdot \delta} f' \circ g.
    \] 
    \item If $g\colon L\to M$ is a marked $R$-homomorphism, then
    \[
    	f\circ g=_\delta f'\circ g.
    \]
    \item If $g,g'\colon L\to M$ are $R$-homomorphisms with $g=_{\delta'}g'$, then
    \[
    	f\circ g =_{\Nsum(g)\cdot \delta+\delta'} f'\circ g'.
    \]
    \item \label{itm:almeq-sum}
    If $g,g'\colon M\to N$ are $R$-homomorphisms with $g =_{\delta'}g'$, then
    \[
    	f+g =_{\delta+\delta'} f'+g'
    \]
  \end{enumerate}
\end{lem}
\begin{proof}
  Parts~(i), ~(ii) and~(vi) are immediate from the definition. Parts~(iii) and~(iv)
  follow from the characterisation in Lemma~\ref{lem:almosteqchar}
  and the estimate in Lemma~\ref{lem:supp1:est}~\ref{i:size comp} and~\ref{i:size comp embedding}, respectively.
  Part~(v) follows from parts~(i), (ii), and~(iii).
\end{proof}

We obtain the following consequences.

\begin{lem}[]
	\label{lem:alm-eq-split}
	Let~$M = \bigoplus_{i\in I} M_i$ and $N = \bigoplus_{j\in J} N_j$ be 
	marked projective $R$-modules. 
	Let~$f,g\colon M\to N$ be $R$-homomorphisms.	
	For~$i\in I$ and~$j\in J$, let $f_{i,j}, g_{i,j}\colon  M_i \to N_j$ denote the
	restrictions to the specified summands in the domain and codomain. 
	Let~$\delta\in \R_{>0}$ such that for all~$i\in I, j\in J$, we have~$f_{i,j} =_\delta g_{i,j}$.
	Then, we have $f =_{\# I\cdot \# J \cdot \delta} g$.
\end{lem}
\begin{proof}
This follows directly from Lemma~\ref{lem:almosteqinherit}.
\end{proof}

\begin{lem}
\label{lem:almosteq rectangle}
	Let~$\delta,\delta'\in \IR_{>0}$ and let the following be a diagram of projective $R$-modules:
	\[\begin{tikzcd}
		M\ar{r}{g_1}\ar{d}{\partial}
		& M'\ar{r}{f_1}\ar{d}{\partial'}
		& M''\ar{d}{\partial''}
		\\
		N\ar{r}{g_0}
		& N'\ar{r}{f_0}
		& N''
	\end{tikzcd}\]
	If $\partial'\circ g_1=_\delta g_0\circ \partial$ and $\partial''\circ f_1=_{\delta'} f_0\circ \partial'$, then
	\[
		\partial''\circ f_1\circ g_1=_{\delta+\Nsum(g_1)\cdot \delta'} f_0\circ g_0\circ \partial.
	\]
	\begin{proof}
          In view of Lemma~\ref{lem:almosteqinherit}, we have
		\begin{align*}
			\partial''\circ f_1\circ g_1
			&=_{\Nsum(g_1)\cdot \delta'} f_0\circ \partial'\circ g_1,
			\\
			f_0\circ \partial'\circ g_1
			&=_{\delta} f_0\circ g_0\circ \partial,
		\end{align*}
          which combines to the claimed almost equality.
	\end{proof}
\end{lem}

\begin{rem}\label{rem:rk1almosteq}
  Let $A\subset X$ be measurable, let $\lambda, \lambda' \in R$
  with~$\supp_1(\lambda) \subset A$, $\supp_1(\lambda') \subset A$,
  and let $f, f' \colon \gen A \to R$ be the $R$-homomorphisms
  given by right multiplication by~$\lambda$ and~$\lambda'$, respectively.
  Then the following are equivalent:
  \begin{enumerate}[label=\enum]
   	\item $\lambda =_\delta \lambda'$;
	\item $\mu(\supp_1(\lambda-\lambda'))<\delta$;
	\item $f =_\delta f'$.
   \end{enumerate}
  Moreover, we recall that (Remark~\ref{rem:supportl1}) 
  \[ \mu \bigl(\supp_1(\lambda-\lambda')\bigr)
  \leq \nu \bigl(\supp(\lambda-\lambda')\bigr)
  \leq |\lambda - \lambda'|_1.
  \]
\end{rem}

\subsection{Controlled almost equality}

Almost equality of maps is not robust with respect to norm estimates.
We thus proceed to a controlled version of the notion of almost
equality from Section~\ref{subsec:eqdelta}:

\begin{defn}
  In the situation of Setup~\ref{setup:GH}, let $M$ and $N$ be marked
  projective $R$-modules, let $f,f' \colon M \to N$ be
  $R$-homomorphisms, and let $\delta, K \in \R_{>0}$.
  We then say that $f$ is \emph{$(\delta,K)$-almost equal to~$f'$}
  if there exists a marked decomposition~$M \cong_R M_0 \oplus M_1$
  with
  \[ f|_{M_0} = f'|_{M_0}
  \qand
  \dim (M_1) < \delta
  \qand
  \| f|_{M_1} - f'|_{M_1} \| \leq K.
  \]
  In this case, we write~$f =_{\delta, K} f'$.
\end{defn}

\begin{prop}\label{prop:eqdeltaK}
  In the situation of Setup~\ref{setup:GH}, let $M,N$ be marked
  projective $R$-modules, let $f,f',f'' \colon M \to N$ be $R$-homomorphisms,
  and let $\delta, \delta', K, K' \in \R_{>0}$.
  Then the following hold:
  \begin{enumerate}[label=\enum]
  \item
    We have $f =_{\delta, K} f$.
  \item
    We have $f = _{\delta,K} f'$ if and only if $f - f' =_{\delta,K} 0$.
  \item
    We have $f =_{\delta,K} f'$ if and only if $f =_\delta f'$ and $\|f-f'\| \leq K$.
  \item
    If $f =_{\delta,K} f'$ and $\delta \leq \delta'$, $K \leq K'$,
    then~$f =_{\delta',K'} f'$.
  \item
    If $f = _{\delta,K} f'$ and $f' = _{\delta',K'} f''$, then $f =_{\delta + \delta', K + K'} f''$.
  \item\label{i:eqdeltaK triangle}
    Let $g, g' \colon M \to N$ be $R$-homomorphisms with $g=_{\delta',K'} g'$. If $f =_{\delta, K} f'$,
    then $ f+ g =_{\delta +\delta', K + K'} f'+g'$.
  \item\label{i:eqdeltaK comp}
    Let $L$ and $P$ be marked projective $R$-modules and let $h \colon N \to P$,
    $g \colon L \to M$ be $R$-homomorphisms. 
    If $f=_{\delta,K} f'$, then
    \[ h \circ f =_{\delta, \|h\| \cdot K} h \circ f'
    \qand
    f\circ g =_{\Nsum(g) \cdot \delta, \|g\| \cdot K} f' \circ g.
    \]
    Moreover, if~$g\colon L\to M$ is a marked $R$-homomorphism, then $f\circ g=_{\delta,K} f'\circ g$.
    \item
    Let~$L$ and~$P$ be marked projective $R$-modules and let $g,g'\colon L\to M$ be $R$-homomorphisms with $g=_{\delta',K'}g'$.
    If $f=_{\delta,K} f'$, then 
    \[
    	f\circ g=_{\Nsum(g)\cdot \delta+\delta',\|g\|\cdot K+\|f'\|\cdot K'} f'\circ g'.
    \]
  \end{enumerate}
\end{prop}
\begin{proof}
	(i)--(vi)
	These properties are straightforward.
	
	(vii)
	This follows from Lemma~\ref{lem:almosteqinherit} and Remark~\ref{rem:norms:marked:homo:bounded:by:1}.
	
	(viii)
	This follows from parts~(vii) and~(v).
\end{proof}

\subsection{A Gromov--Hausdorff distance for homomorphisms}

We introduce a notion of Gromov--Hausdorff distance for
homomorphisms between marked projective modules. As in
the case of metric spaces, the Gromov--Hausdorff distance
is defined by inclusions into joint ambient objects. 

\begin{defn}[marked symmetric difference]
  Let $M = \bigoplus_{i \in I} \gen {A_i}$ be a marked projective $R$-module
  and let $N = \bigoplus_{i \in I} \gen{B_i}$, $N' = \bigoplus_{i \in I} \gen{B'_i}$
  be marked projective summands of~$M$.
  Then we define the \emph{marked symmetric  difference
    of~$N$ and~$N'$} by 
  \[ N \msymmdiff N' \coloneqq \bigoplus_{i \in I} \gen{B_i \symmdiff B'_i}.
  \]
\end{defn}

\begin{defn}[Gromov--Hausdorff distance for homomorphisms]
  In the situation of Setup~\ref{setup:GH}, let $M,N,M',N'$
  be marked projective $R$-modules, let $f \colon M \to N$ and $f' \colon M' \to N'$
  be $R$-homomorphisms, and let $\delta, K \in \R_{>0}$. We then say
  that $\dgh K {f,f'} < \delta$ if
  there exist marked projective $R$-modules~$L$,~$P$ and
  marked inclusions
  $\varphi \colon M \to L$, $\varphi' \colon M' \to L$,
  $\psi \colon N \to P$, $\psi' \colon N' \to P$ with the following
  properties: 
  \begin{itemize}
  \item $\dim \bigl(\varphi(M) \msymmdiff \varphi'(M') \bigr) < \delta$
  \item $\dim \bigl(\psi(N) \msymmdiff \psi'(N') \bigr) < \delta$
  \item $F =_{\delta,K} F'$, where $F \coloneqq \psi \circ f \circ
    \pi_\varphi$, $F' \coloneqq \psi' \circ f' \circ \pi_{\varphi'}$ and
    $\pi_\varphi$, $\pi_{\varphi'}$ are the marked projections
    associated with the marked inclusions~$\varphi$ and $\varphi'$,
    respectively.
  \end{itemize}
  \[ \begin{tikzcd}
    M
    \ar{r}{f}
    \ar[shift right,hookrightarrow]{d}[swap]{\varphi}
    & N
    \ar[shift right,hookrightarrow]{d}[swap]{\psi}
    \\
    L
    \ar[shift right]{u}[swap]{\pi_{\varphi}}
    \ar[shift left]{d}{\pi_{\varphi'}}
    \ar[shift left, dashed]{r}{F}
    \ar[shift right, dashed]{r}[swap]{F'}
    & P
    \ar[shift right]{u}[swap]{\pi_{\psi}}
    \ar[shift left]{d}{\pi_{\psi'}}
    \\
    M'
    \ar{r}[swap]{f'}
    \ar[shift left,hookrightarrow]{u}{\varphi'}
    & N'
    \ar[shift left,hookrightarrow]{u}{\psi'}
    \end{tikzcd}
  \]
\end{defn}

\begin{prop}\label{prop:dGH}
  In the situation of Setup~\ref{setup:GH}, let $M, M', M''$, $N, N', N''$ be marked
  projective $R$-modules, let $f \colon M \to N$, $f' \colon M' \to N'$, $f'' \colon M'' \to N''$
  be $R$-homomorphisms, and let $\delta, \delta', K, K' \in \R_{>0}$.
  Then the following hold:
  \begin{enumerate}[label=\enum]
  \item 
	\label{itm:dGH-dim}  
  If $\dgh K {f,f'} < \delta$, then
    \[ | \dim M - \dim M'| < \delta
    \qand
    | \dim N - \dim N'| < \delta.
    \]
  \item If $M=M'$ and $N = N'$ and $f=_{\delta,K} f'$, then $\dgh K {f,f'} < \delta$.
  \item If $\dgh K {f,f'} < \delta$ and $\delta \leq \delta'$, $K\leq K'$,
    then $\dgh {K'} {f,f'} < \delta'$.
  \item
    If $\dgh K {f,f'} < \delta$, then there exist marked
    $R$-homomorphisms $\Phi \colon M \to M'$ and $\Psi \colon N \to N'$
    with
    $\Psi \circ f =_{\delta,K} f' \circ \Phi$.
    \[
    \begin{tikzcd}
      M
      \ar{r}{f}
      \ar[dashed]{d}[swap]{\Phi}
      & N
      \ar[dashed]{d}{\Psi}
      \\
      M'
      \ar{r}[swap]{f'}
      & N'
    \end{tikzcd}
    \]
  \item 
	\label{itm:dGH5}  
  If $\dgh K {f,f'} < \delta$ and $\dgh {K'} {f',f''} < \delta'$,
    then
    \[ \dgh {K+K'} {f,f''} < \delta + \delta'.
    \]
  \end{enumerate}
\end{prop}
\begin{proof}    
    (i)
    This follows from
    \[
    	|\dim M-\dim M'|=|\dim \varphi(M)-\dim \varphi'(M')|\le \dim\bigl(\varphi(M)\msymmdiff \varphi'(M')\bigr) <\delta
    \]
    and similarly for~$N$ and~$N'$.
    
    (ii) and (iii)
    These are clear.
    
    (iv)
    Take $\Phi\coloneqq \pi_{\varphi'}\circ \varphi$ and $\Psi\coloneqq \pi_{\psi'}\circ \psi$.
    Then
    \begin{align*}
    	\Psi\circ f
	&= \pi_{\psi'}\circ \psi\circ f\circ \pi_{\varphi}\circ \varphi
	= \pi_{\psi'}\circ F\circ \varphi
	\\
	&=_{\delta,K} \pi_{\psi'}\circ F'\circ \varphi
	&\text{(Proposition~\ref{prop:eqdeltaK}~\ref{i:eqdeltaK comp})}
	\\
	&= \pi_{\psi'}\circ \psi'\circ f'\circ \pi_{\varphi'}\circ \varphi
	= f'\circ \Phi.
    \end{align*}
    
    (v)
    Suppose $\dgh K {f,f'} < \delta$ and $\dgh {K'} {f',f''} < \delta'$ witnessed by
    \[ \begin{tikzcd}
    M
    \ar{r}{f}
    \ar[shift right,hookrightarrow]{d}[swap]{\varphi}
    & N
    \ar[shift right,hookrightarrow]{d}[swap]{\psi}
    \\
    L
    \ar[shift right]{u}[swap]{\pi_{\varphi}}
    \ar[shift left]{d}{\pi_{\varphi'}}
    \ar[shift left, dashed]{r}{F}
    \ar[shift right, dashed]{r}[swap]{F'}
    & P
    \ar[shift right]{u}[swap]{\pi_{\psi}}
    \ar[shift left]{d}{\pi_{\psi'}}
    \\
    M'
    \ar{r}[swap]{f'}
    \ar[shift left,hookrightarrow]{u}{\varphi'}
    \ar[shift right,hookrightarrow]{d}[swap]{\rho}
    & N'
    \ar[shift left,hookrightarrow]{u}{\psi'}
    \ar[shift right,hookrightarrow]{d}[swap]{\theta}
    \\
    L'
    \ar[shift right]{u}[swap]{\pi_{\rho}}
    \ar[shift left]{d}{\pi_{\rho'}}
    \ar[shift left, dashed]{r}{G}
    \ar[shift right, dashed]{r}[swap]{G'}
    & P'
    \ar[shift right]{u}[swap]{\pi_{\theta}}
    \ar[shift left]{d}{\pi_{\theta'}}
    \\
    M''
    \ar{r}[swap]{f''}
    \ar[shift left,hookrightarrow]{u}{\rho'}
    & N''
    \ar[shift left,hookrightarrow]{u}{\theta'}
    \end{tikzcd}
  \]
  with $F=_{\delta,K}F'$ and $G=_{\delta',K'}G'$.
  We may assume that $L=\bigoplus_{i\in I}\spann{A_i}$, $L'=\bigoplus_{i\in I}\spann{A_i'}$, and $M'=\bigoplus_{i\in I}\spann{B_i'}$ with $B_i'\subset A_i$ and $B_i'\subset A_i'$.
  We define
  \[
  	L''\coloneqq \bigoplus_{i\in I}\spann{A_i\cup A_i'}
  \]    
  and similarly, we define~$P''$.
  We have
  \[\begin{tikzcd}
  	L
	\ar[shift left, dashed]{r}{F}
    	\ar[shift right, dashed]{r}[swap]{F'}
	& P
	\ar{d}
	\\
	L''
	\ar{u}
	\ar{d}
	& P''
	\\
	L'
	\ar[shift left, dashed]{r}{G}
    	\ar[shift right, dashed]{r}[swap]{G'}
	& P'
	\ar{u}
  \end{tikzcd}\]
  We denote the compositions $L''\to P''$ by $\widetilde{F},\widetilde{F'},\widetilde{G},\widetilde{G'}$, respectively.
  Then we have $\widetilde{F}=_{\delta,K}\widetilde{F'}$ and $\widetilde{G}=_{\delta',K'}\widetilde{G'}$ by Proposition~\ref{prop:eqdeltaK}~\ref{i:eqdeltaK comp},
  and $\widetilde{F'}=\widetilde{G}$.
  By Proposition~\ref{prop:eqdeltaK}~\ref{i:eqdeltaK triangle}, we conclude $\widetilde{F}=_{\delta+\delta',K+K'} \widetilde{G'}$, witnessing that $\dgh {K+K'} {f,f''} < \delta + \delta'$.
\end{proof}

\subsection{A Gromov--Hausdorff distance for chain complexes}

We extend the notion of Gromov--Hausdorff distance to chain complexes and, more generally, to sequences of homomorphisms.

\begin{defn}[marked projective sequence]
\label{defn:mp sequence}
  In the situation of Setup~\ref{setup:GH}, let $n \in \N$.
    A \emph{marked projective $n$-sequence (over~$R$)}
    is a sequence~$(D_*,\eta)$ of the form 
    \[ \xymatrix{%
      D_{n+1} \ar[r]^-{\partial_{n+1}}
      & D_n \ar[r]
      & \cdots \ar[r]
      & D_0 \ar[r]^-{\partial_0 = \eta}
      & \LinftyX,
    }
    \]
    consisting of marked projective $R$-modules~$D_0,\dots, D_{n+1}$
    and $R$-ho\-mo\-mor\-phisms $\partial_0 \coloneqq \eta, \partial_1, \dots, \partial_{n+1}$.
\end{defn}

Clearly, marked projective $R$-chain complexes (up to degree~$n+1$) are marked projective $n$-sequences.

For the Gromov--Hausdorff distance between sequences,
we require the inclusions into a common ambient module to
exist simultaneously for all degrees in the given range: 

\begin{defn}[Gromov--Hausdorff distance for sequences]
\label{defn:dgh chain}
  In the situation of Setup~\ref{setup:GH}, let $n \in \N$, 
  let $(D_*,\eta)$, $(D'_*, \eta')$ be marked projective $n$-sequences over~$R$, and let $\delta, K \in \R_{>0}$.
  We then say that $\dgh K {D_*,D'_*,n} < \delta$ if
  there exist marked projective $R$-modules $P_0,\dots, P_{n+1}$
  and marked inclusions~$\varphi_r \colon D_r \to P_r$,
  $\varphi'_r \colon D'_r \to P_r$ for all~$r \in \{0,\dots, n+1\}$
  with the following properties: 
  \begin{itemize}
  \item For all~$r \in \{0,\dots, n+1\}$, we have
    \[ \dim\bigl( \varphi_r(D_r) \msymmdiff \varphi'_r(D'_r) \bigr) < \delta.
    \]
  \item For all~$r \in \{0,\dots, n+1\}$, we have
    \[ F_r = _{\delta,K} F'_r,
    \]
    where $F_r \coloneqq \varphi_{r-1} \circ \partial_r \circ \pi_{\varphi_r}$
    and $F'_r \coloneqq \varphi'_{r-1} \circ \partial'_r \circ \pi_{\varphi'_r}$.
    Here, $P_{-1} \coloneqq \LinftyX$ and $\varphi_{-1} \coloneqq \id_{\LinftyX} \eqqcolon \varphi'_{-1}$. 
  \end{itemize}
  \[ \begin{tikzcd}
    D_{n+1}
    \ar{r}{\partial_{n+1}}
    \ar[shift right,hookrightarrow]{d}[swap]{\varphi_{n+1}}
    & D_n
    \ar[shift right,hookrightarrow]{d}[swap]{\varphi_{n}}
    & \dots
    & D_0
    \ar{r}{\partial_0}
    \ar[shift right,hookrightarrow]{d}[swap]{\varphi_0}
    & \LinftyX
    \ar[equal]{d}
    \\
    P_{n+1}
    \ar[shift right]{u}[swap]{\pi_{\varphi_{n+1}}}
    \ar[shift left]{d}{\pi_{\varphi'_{n+1}}}
    \ar[shift left, dashed]{r}{F_{n+1}}
    \ar[shift right, dashed]{r}[swap]{F'_{n+1}}
    & P_n
    \ar[shift right]{u}[swap]{\pi_{\varphi_n}}
    \ar[shift left]{d}{\pi_{\varphi'_n}}
    & \dots
    & P_0
    \ar[shift right]{u}[swap]{\pi_{\varphi_0}}
    \ar[shift left]{d}{\pi_{\varphi'_0}}
    \ar[shift left, dashed]{r}{F_0}
    \ar[shift right, dashed]{r}[swap]{F'_0}
    & \LinftyX
    \\
    D'_{n+1}
    \ar{r}[swap]{\partial'_{n+1}}
    \ar[shift left,hookrightarrow]{u}{\varphi'_{n+1}}
    & D'_{n}
    \ar[shift left,hookrightarrow]{u}{\varphi'_n}
    & \dots
    & D'_0
    \ar{r}[swap]{\partial'_{0}}
    \ar[shift left,hookrightarrow]{u}{\varphi'_{0}}
    & \LinftyX
    \ar[equal]{u}
    \end{tikzcd}
  \]
\end{defn}

\begin{prop}\label{prop:dghchaincomplex}
  In the situation of Setup~\ref{setup:GH}, let~$n\in \IN$.
  Let~$(D_*,\eta)$, $(D'_*,\eta')$, $(D''_*,\eta'')$ be marked projective $n$-sequences and let~$\delta,\delta',K,K'\in \IR_{>0}$.
  If we have $\dgh{K}{D_*,D'_*,n}<\delta$ and $\dgh{K'}{D'_*,D''_*,n}<\delta'$, then
	\[
		\dgh{K+K'}{D_*,D''_*,n}<\delta+\delta'.
	\]
\end{prop}
\begin{proof}
	The proof is similar to that of Proposition~\ref{prop:dGH}~\ref{itm:dGH5}.
\end{proof}

\section{Strictification}\label{sec:strictification}

We prove two strictification results:
\begin{itemize}
\item Theorem~\ref{thm:strictifycomplex}:
  Every marked projective
  ``almost'' chain complex is ``close'' (in the Gromov--Hausdorff sense)
  to an actual chain complex.
\item
  Theorem~\ref{thm:strictifychmap}:
  Every ``almost'' chain map is ``close''
  to an actual chain map to a target complex that is ``close'' to the
  original target complex.  
\end{itemize}
The control on the constants is delicate, in particular, in degree~$0$.

\begin{setup}\label{setup:strict}
  Let $\Gamma$ be a countable group and let $Z$ denote $\Z$ (with
  the usual norm) or a finite field (with the trivial norm).
  We consider a standard $\Gamma$-action~$\alpha \colon \Gamma \actson (X,\mu)$. 
  Moreover, let $\Rrel$ be the associated orbit relation and
  let $R \subset Z \Rrel$ be a subring that
  contains~$\LinftyXZ * \Gamma$.
\end{setup}
  
\subsection{Almost chain complexes and almost chain maps}

Almost chain complexes are sequences that ``almost'' satisfy the
chain complex equations. Almost chain maps are sequences of homomorphisms
between almost chain complexes that ``almost'' satisfy the chain
map equations.

\begin{defn}[almost chain complex]
\label{defn:almost_chain_complex}
  In the situation of Setup~\ref{setup:strict}, let $n \in \N$ and
  $\delta \in \R_{>0}$.
     A \emph{marked projective \almostcc{\delta}{n} (over~$R$)} is a marked projective $n$-sequence~$(D_*,\eta)$ over~$R$ (Definition~\ref{defn:mp sequence})
    such that
    \[ \fa{r \in \{0,\dots,n\}} \partial_r \circ \partial_{r+1} =_\delta 0
    \]
    and such that $\eta$ is $\delta$-surjective, i.e., there exists~$z
    \in D_0$ with~$\eta(z) =_\delta 1$.    
\end{defn}

\begin{defn}[almost chain map]
  Let $\delta,\varepsilon \in \R_{>0}$, let $n \in \N$, and let $(C_*,\zeta)$
  and $(D_*,\eta)$ be marked projective \almostccs{\delta}{n}. An \emph{\almostcm{\varepsilon}{n}~$C_* \to D_*$ extending~$\id_{\LinftyX}$} is a
  sequence~$(f_r \colon C_r \to D_r)_{r \in \{0,\dots,n+1\}}$ of
  $R$-homomorphisms with
  \[ \eta \circ f_0 =_\varepsilon \zeta
  \qand 
  \fa{r \in \{1,\dots,n+1\}}
  \partial^D_r \circ f_r =_\varepsilon f_{r-1} \circ \partial^C_r.
  \]
  We say that~$f_*$ is \emph{marked} if every~$f_r$ is a marked $R$-homomorphism.
\end{defn}

\begin{lem}[compositions of almost chain maps]
	\label{lem:comp-almost-chain-maps}
	Let $\delta, \varepsilon, \varepsilon' \in \R_{>0}$ 
	and let $f_*\colon C_* \to D_*$ be an \almostcm{\varepsilon}{n} 
	and $g_*\colon D_* \to E_*$ be an \almostcm{\varepsilon'}{n} 
	between marked projective \almostccs{\delta}{n}
	extending the
	identity. Then, $g_*\circ f_*$ is an \almostcm{(\varepsilon+N\varepsilon')}{ n}
	extending the identity,
	where $N\coloneqq \max_{r \in \{0,\dots,n+1\}} \Nsum(f_r)$.
	\begin{proof}
		This follows from Lemma~\ref{lem:almosteq rectangle}.
	\end{proof}
\end{lem}

\begin{prop}\label{prop:dGHalmostchmap}
  In the situation of Setup~\ref{setup:strict}, let $n \in \N$, and
  let $(D_*,\eta)$ and $(D'_*,\eta')$ be marked projective $R$-chain
  complexes (up to degree~$n+1$) satisfying $\dgh K{D_*,D'_*,n} < \delta$.
  Then there exist marked \almostcms{\delta}{n} $\Phi_* \colon D_* \to D'_*$
  and $\Phi'_* \colon D'_* \to D_*$ 
  extending~$\id_{\LinftyX}$ with
  \begin{align*} 
    \Phi'_r \circ \Phi_r
    =_\delta \id_{D_r}
    , \quad
    &
    \Phi_r \circ \Phi'_r
    =_\delta \id_{D'_r}
  \end{align*}
  for all~$r \in \{0,\dots, n+1\}$.
\end{prop}
\begin{proof}
    Let~$P_*, \varphi_*, \varphi'_*$ be witnesses for $\dgh K{D_*,D'_*,n}<\delta$ as in Definition~\ref{defn:dgh chain}.
    Set $\Phi_r\coloneqq \pi_{\varphi'_r}\circ \varphi_r$ and $\Phi'_r\coloneqq \pi_{\varphi_r}\circ \varphi'_r$.
    Then~$\Phi_*$ and~$\Phi'_*$ are marked \almostcms{\delta}{n}, since we have
    \begin{align*}
    	\Phi_r\circ \partial_{r+1}
	&= \pi_{\varphi'_r}\circ \varphi_r\circ \partial_{r+1}
	\\
	&= \pi_{\varphi'_r}\circ \varphi_r\circ \partial_{r+1}\circ \pi_{\varphi_{r+1}}\circ \varphi_{r+1}
	\\
	&= \pi_{\varphi'_r}\circ F_{r+1}\circ \varphi_{r+1}
	\\
	&=_\delta \pi_{\varphi'_r}\circ F'_{r+1}\circ \varphi_{r+1}
	& \text{(Lemma~\ref{lem:almosteqinherit})}
	\\
	&= \pi_{\varphi'_r}\circ \varphi'_r\circ \partial'_{r+1}\circ \pi_{\varphi'_{r+1}}\circ \varphi_{r+1}
	\\
	&= \partial'_{r+1}\circ \Phi_{r+1}
    \end{align*}
    and similarly for~$\Phi'_*$.
    We may assume that $P_r=\bigoplus_{i\in I}\spann{A_i}$, $D_r=\bigoplus_{i\in I}\spann{B_i}$, $D'_r=\bigoplus_{i\in I}\spann{B'_i}$ with $B_i,B'_i\subset A_i$ and that $\varphi_r,\varphi'_r$ are the obvious marked inclusions. 
    Then $\Phi'_r\circ \Phi_r\colon D_r\to D_r$ is the marked $R$-homomorphism given coordinate-wise by the projection $\spann{B_i}\to \spann{B_i\cap B'_i}$.
    Hence
    \begin{align*}
    	\size_1(\Phi'_r\circ \Phi_r-\id_{D_r})
	&\le \sum_{i\in I}\mu\bigl(B_i\setminus (B_i\cap B'_i)\bigr)
	\\
	& \le \sum_{i\in I}\mu(B_i\symmdiff B'_i)
	= \dim\bigl(\varphi_r(D_r) \msymmdiff \varphi'_r(D_r')\bigr)
	\le \delta
    \end{align*}
    and similarly for $\Phi_r\circ \Phi'_r$.
\end{proof}

\begin{lem}\label{lem:GHclosealmostcomplex}
  In the situation of Setup~\ref{setup:strict}, let $n \in \N$ and
  let $\delta, K, \varepsilon \in \R_{>0}$.
  Let $(D_*,\eta)$
  be a marked projective \almostcc{\delta}{n},
  let $z \in D_0$ with~$\eta(z) =_\delta 1$, 
  and let $(\widehat D_*,\widehat \eta)$ be a marked projective
  $n$-sequence.
  If $\dgh K {\widehat D_*,D_*,n} < \varepsilon$,
  then~$(\widehat D_*, \widehat \eta)$ is a marked
  projective \almostcc{\widehat{\delta}}{n}, where
  \begin{align*}
    \widehat{\delta}
    & \coloneqq \max\bigl\{\delta+(1+\nusum_n(D_*))\cdot \varepsilon, \delta+\Nbasic(z)\cdot \varepsilon\bigr\}
    \\
    \nusum_n(D_*)
    & \coloneqq \max \bigl\{\|\eta\|_\infty, \Nsum(\partial_1^D), \dots, \Nsum(\partial_{n+1}^D) \bigr\}.
  \end{align*}
  Moreover, there exists a~$\widehat z \in \widehat D_0$
  with~$\widehat \eta(\widehat z)=_{\widehat \delta} 1$,
   $\Nbasic(\widehat z) \leq \Nbasic(z)$, $\Ntbasic(\widehat z) \leq \Ntbasic(z)$, and $|\widehat z|_\infty \leq |z|_\infty$.
\end{lem}
\begin{proof}
  Suppose that $\dgh{K}{\widehat{D}_*,D_*,n}<\varepsilon$ is witnessed by
  \[\begin{tikzcd}
  	D_{r+1}
	\ar{r}{\partial_{r+1}}
	\ar[shift right,hook,swap]{d}{\varphi_{r+1}}
	& D_r
	\ar{r}{\partial_r}
	\ar[shift right,hook,swap]{d}{\varphi_n}
	& D_{r-1}
	\ar[shift right,hook,swap]{d}{\varphi_{r-1}}
	\\
	P_{r+1}
	\ar[shift right,two heads,swap]{u}{\pi_{\varphi_{r+1}}}
	\ar[shift left,dashed]{r}{F_{r+1}}
	\ar[shift right,dashed,swap]{r}{\widehat{F}_{r+1}}
	\ar[shift left,two heads]{d}{\pi_{\widehat{\varphi}_{r+1}}}
	& P_r
	\ar[shift right,two heads,swap]{u}{\pi_{\varphi_{r}}}
	\ar[shift left,dashed]{r}{F_{r}}
	\ar[shift right,dashed,swap]{r}{\widehat{F}_{r}}
	\ar[shift left,two heads]{d}{\pi_{\widehat{\varphi}_{r}}}
	& P_{r-1}
	\ar[shift right,two heads,swap]{u}{\pi_{\varphi_{r-1}}}
	\ar[shift left,two heads]{d}{\pi_{\widehat{\varphi}_{r-1}}}
	\\
	\widehat{D}_{r+1}
	\ar[swap]{r}{\widehat{\partial}_{r+1}}
	\ar[shift left,hook]{u}{\widehat{\varphi}_{r+1}}
	& \widehat{D}_r
	\ar[swap]{r}{\widehat{\partial}_{r}}
	\ar[shift left,hook]{u}{\widehat{\varphi}_{r}}
	& \widehat{D}_{r-1}.
	\ar[shift left,hook]{u}{\widehat{\varphi}_{r-1}}
  \end{tikzcd}\]
  By Lemma~\ref{lem:almosteqinherit}, we have
  \[
  	\widehat{F}_r\circ \widehat{F}_{r+1}
	 =_{\varepsilon + \Nsum(F_{r+1})\cdot \varepsilon} F_r\circ F_{r+1}
	\]
	and
	\[
	F_r\circ F_{r+1}
	 = \varphi_{r-1}\circ \partial_r\circ \partial_{r+1}\circ \pi_{\varphi_{r+1}}
	=_\delta 0,\]
  since $\pi_{\varphi_{r+1}}$ is a marked $R$-homomorphism. Moreover, since $\Nsum(F_{r+1})=\Nsum(\partial_{r+1})$,
  the previous equalities together imply $\widehat{F}_r\circ \widehat{F}_{r+1}=_{\varepsilon+\Nsum(\partial_{r+1})\cdot \varepsilon + \delta} 0$ (Lemma~\ref{lem:almosteqinherit}~\ref{itm:composition:additivity}).
  We conclude that
  \[
  	\widehat{\partial}_r\circ \widehat{\partial}_{r+1} 
	= \pi_{\widehat{\varphi}_{r-1}}\circ \widehat{F}_r\circ \widehat{F}_{r+1}\circ \widehat{\varphi}_{r+1}
	=_{\varepsilon+\Nsum(\partial_{r+1})\cdot \varepsilon + \delta} 0.
  \]
  In degree~0, we consider the diagram
  \[\begin{tikzcd}
  	R\ar{d}[swap]{f_z}\ar{r}{f_1}
	& \LinftyX\ar[equal]{d} 
	\\
	D_0\ar{r}{\eta}\ar{d}[swap]{\pi_{\widehat{\varphi}_0}\circ \varphi_0}
	& \LinftyX\ar[equal]{d}
	\\
	\widehat{D}_0\ar{r}{\widehat{\eta}}
	& \LinftyX
  \end{tikzcd}\]
  where the $R$-homomorphisms~$f_z,f_1$ are given by~$z\in D_0$ and~$1\in \LinftyX$, respectively.
  Since~$\eta(z)=_\delta 1$ in~$\LinftyX$, we have $\eta\circ f_z=_\delta f_1$.
  By Proposition~\ref{prop:dGHalmostchmap} and its proof, we have $\eta=_\varepsilon \widehat{\eta}\circ \pi_{\widehat{\varphi}_0}\circ \varphi_0$.
  Then Lemma~\ref{lem:almosteq rectangle} yields
  \[
  	f_1=_{\delta+ \Nbasic(z)\cdot \varepsilon} \widehat{\eta}\circ \pi_{\widehat{\varphi}_0}\circ \varphi_0 \circ f_z.
  \]
    Thus, the element $\widehat{z}\coloneqq \pi_{\widehat{\varphi}_0}\circ \varphi_0(z)\in \widehat{D}_0$ is as desired
    by Remark~\ref{rem:mhom-N1}.
\end{proof}

In particular, every sequence ``close'' to a chain complex is an
almost chain complex. Conversely, also every almost complex is
``close'' to a strict chain complex; this is the content of the
strictification theorem
(Theorem~\ref{thm:strictifycomplex}). Similarly, we establish
strictification for almost chain maps
(Theorem~\ref{thm:strictifychmap}).

\subsection{Strictification of almost chain complexes}

To formulate and prove the strictification theorems in the
appropriate uniformity, we bound the complexity of the input
data as follows:

\begin{defn}\label{def:crazyconsts}
  In the situation of Setup~\ref{setup:strict}, let $\delta \in \R_{>0}$,  
  let $n \in \N$, and let $(D_*,\eta)$ be a marked
  projective \almostcc{\delta}{n}. We set
  \begin{align*}
    \kappa_n(D_*)
    & \coloneqq \max \bigl\{\|\eta\|, \|\partial_1^D\|, \dots, \|\partial_{n+1}^D\|\bigr\}
    \\
    \numax_n(D_*)
    & \coloneqq \max \bigl\{\|\eta\|_\infty, \Nmax(\partial_1^D), \dots, \Nmax(\partial_{n+1}^D) \bigr\}
    \\
    \nusum_n(D_*)
    & \coloneqq \max \bigl\{\|\eta\|_\infty, \Nsum(\partial_1^D), \dots, \Nsum(\partial_{n+1}^D) \bigr\}
    .
  \end{align*}
  Moreover, if $\kappa \in \R_{>0}$, then we say that $\overline \kappa_n(D_*) < \kappa$
  if 
  \[
  \max \bigl\{
  \rk (D_1), \dots, \rk (D_{n+1}),
    \kappa_n(D_*),
    \numax_n(D_*)
    \bigr\}
    < \kappa
  \]
  and there exists a~$z \in D_0$ with
  \[ \eta(z)=_\delta 1
  , \quad
  \Nbasic(z) < \kappa
  , \quad
  \Ntbasic(z) < \kappa
  , \quad
  |z|_\infty < \kappa.
  \]
\end{defn}

\begin{thm}\label{thm:strictifycomplex}
  In the situation of Setup~\ref{setup:strict}, let $n\in \N$
  and let $\kappa \in \R_{>0}$.
  Then, there exists a~$K \in \R_{>0}$ such that:
  For every~$\delta \in \R_{>0}$ and every marked projective
  \almostcc{\delta}{n}~$(D_*,\eta)$ 
  with~$\overline \kappa_n(D_*) < \kappa$, there exists a marked projective
  $R$-chain complex~$(\widehat D_*,\widehat \eta)$ (up to degree~$n+1$)
  with
  \[ \dgh K {\widehat D_*, D_*, n} \leq K \cdot \delta.
  \]
  Moreover, $\widehat D_*$ can be chosen such that~$D_*$ is a subcomplex of~$\widehat{D}_*$
  and such that the inclusion map~$D_* \hookrightarrow \widehat D_*$ is a \almostcm{(K\cdot \delta)}{n}.
\end{thm}

\begin{rem}
  Moreover, if $(D_*,\eta)$ is $S$-adapted (Definition~\ref{defn:adapted}) and if there is an
  $S$-adapted~$z \in D_0$ with~$\eta(z) =_\delta 1$, $\Nbasic(z) < \kappa$, $\Ntbasic(z) < \kappa$, $|z|_\infty < \kappa$,
  we can choose~$(\widehat D_*,\widehat \eta)$ to be $S$-adapted.
\end{rem}

Before giving the proof of Theorem~\ref{thm:strictifycomplex},
we discuss the case of degree~$0$ separately:

\begin{lem}\label{lem:almostsurjective}
  In the situation of Setup~\ref{setup:strict}, let $n \in \N$,
  let $\delta \in \R_{>0}$, and let $(D_*,\eta)$ be a marked
  projective \almostcc{\delta}{n}. Moreover,
  let $z \in D_0$ with~$\eta(z) =_\delta 1$ and 
  let $K \coloneqq |1-\eta(z)|_\infty$.
  Then, there exists a marked projective \almostcc{\delta}{n}~$(\widehat D_*, \widehat \eta)$
  with 
  \[ \dgh K {\widehat D_*, D_*, n} < \delta
  \]
  and the following additional control:
  \begin{itemize}
  \item
    The $R$-homomorphism~$\widehat \eta \colon \widehat D_0 \to \LinftyX$ is surjective. 
    More precisely, there exists a~$\widehat z \in \widehat D_0$
    with~$\widehat \eta (\widehat z) = 1$ and
    \[ \Nbasic(\widehat z) \leq \Nbasic(z) + 1, \quad
    \Ntbasic(\widehat z) \leq \Ntbasic(z) + 1,
    \qand
    |\widehat z|_\infty \leq |z|_\infty + 1.
    \]
  \item
    We have 
    $\overline \kappa_n(\widehat D_*) \leq \overline \kappa_n(D_*) + K+1$.
  \end{itemize}
\end{lem}

\begin{rem}\label{rem:almostsurjectivecomments}
  By Lemma~\ref{lem:etaz}, we have $K = |1 - \eta(z)|_\infty \le \bigl(\overline\kappa_n(D_*)\bigr)^3+1$.
  We may assume~$K = 0$ if $\eta$ is already surjective.

  Moreover, if $(D_*,\eta)$ and $z$ are $S$-adapted (Definition~\ref{defn:adapted}), we may
  choose~$(\widehat D_*, \widehat \eta)$ and $\widehat z$
  also to be $S$-adapted. 
\end{rem}
  
\begin{proof}[Proof of Lemma~\ref{lem:almostsurjective}]
  We consider the error term~$B\coloneqq \supp(\eta(z) -1) \subset X$ and set
  \[ \widehat D_0 \coloneqq D_0 \oplus \gen B
  \qand \fa{r \in \{1,\dots, n+1\}} \widehat D_r \coloneqq D_r.
  \]
  Furthermore, we define
  $\widehat \partial_r \coloneqq \partial_r$ for all~$r \in \{1,\dots, n+1\}$
  and
  \begin{align*}
    \widehat \eta \colon \widehat D_0 = D_0 \oplus \gen B
    & \to \LinftyX
    \\
    D_0 \ni x
    & \mapsto \eta(x)
    \\
    \gen B \ni \chi_B \cdot e
    & \mapsto
    1 - \eta(z).
  \end{align*}
  Then, $\widehat \eta$ is a well-defined $R$-homomorphism.
  By construction, $\widehat \eta$ is surjective; indeed,
  for~$\widehat z \coloneqq z + \chi_B \cdot e$, we have~$\widehat\eta(\widehat z)
  = 1$,  $\Nbasic(\widehat z) \leq \Nbasic(z) + 1$, $\Ntbasic(\widehat z) \leq \Ntbasic(z) + 1$, and $|\widehat z|_\infty \leq |z|_\infty +1$. 
  Moreover, by hypotheses $\dim \gen B = \mu(B) < \delta$ and $\|\widehat \eta|_{\gen B}\| \leq K$
  (by definition of~$K$).
  Hence, $\widehat \eta =_{\delta,K} \eta$.

  In addition, we have $\widehat \eta \circ (D_0 \hookrightarrow \widehat D_0) = \eta$
  and $\widehat \eta \circ \widehat \partial_1 = \eta \circ \partial_1$.
  Thus, $(\widehat D_*, \widehat \eta)$ also is a \almostcc{\delta}{n} and~$\dgh K {\widehat D_*, D_*, n} < \delta$.

  We are left to show that $\overline \kappa_n(\widehat D_*)
  \leq \overline \kappa_n(D_*) + K + 1$. To this end, it is sufficient to estimate the quantities associated to $\widehat \eta$ and to the module $\widehat D_0$.
  We have \[\|\widehat \eta\|_\infty = \max \big\{\|\eta\|_\infty, |1-\eta(z)|_\infty\bigr\}
  \leq \numax_n(D_*) + K\]
  as
  well as \[\|\widehat \eta\| \le \|\eta\| + |1-\eta(z)|_1
  \leq \|\eta\| + |1-\eta(z)|_\infty \leq \kappa_n(D_*) + K.\]
  Moreover, by definition, $\rk (\widehat D_0) \leq \rk (D_0) + 1$. 
  Therefore, we conclude that $\overline \kappa_n(\widehat D_*)
  \leq \overline \kappa_n(D_*) + K + 1$.
\end{proof}

\begin{rem}
  In the upcoming proof and in the proof of
  Theorem~\ref{thm:strictifychmap} below, we will take the liberty of
  writing~``$\const{\kappa}$'' for constants that depend only
  on~$\kappa$ (and~$n$). For instance, $\kappa^2 + 1$ could be subsumed
  in~``$\const{\kappa}$", but $\dim D_1$ cannot.
\end{rem}
  
\begin{proof}[Proof of Theorem~\ref{thm:strictifycomplex}]
  In view of the preparation in Lemma~\ref{lem:almostsurjective} (and
  Remark~\ref{rem:almostsurjectivecomments},
  Proposition~\ref{prop:dghchaincomplex}), we may assume without loss of
  generality that $\eta \colon D_0 \to \LinftyX$ is surjective, which
  will simplify the notation in the proof below.

  Let $\widehat D_{-1} \coloneqq \LinftyX$, $\widehat \partial_{-1} \coloneqq 0$. 
  It suffices to show the following: For all~$r \in \{0,\dots,
  n+1\}$, there exist marked projective $R$-modules~$\widehat D_r =
  D_r \oplus E_r$ with~$\dim E_r < \const{\kappa} \cdot \delta$ and
  an $R$-homomorphism~$\widehat \partial_r \colon \widehat D_r \to \widehat D_{r-1}$
  with the following properties:
  \[ \widehat \partial_{r-1} \circ \widehat \partial_r = 0
  , \quad
  \widehat \partial_0 | _{D_0} = \eta
  , \quad
  \widehat \partial_r|_{D_r} =_{\const{\kappa}\cdot \delta, 1} \partial_r
  , \quad
  \| \widehat \partial_r|_{E_r} \| \leq \const{\kappa}.
  \]
  We proceed by induction over the degree, modifying the chain modules
  and the boundary operator (twice) in each degree:  

  For convenience, we set~$\widetilde \partial_0 \coloneqq \eta$
  and $E_{-1} \coloneqq 0$, $\partial_{-1} \coloneqq 0$.
  
  For the induction step, let $r \in \{0,\dots, n\}$ and suppose that
  we already constructed~$\widehat D_{0}, \dots, \widehat D_{r-1}$ and
  $R$-homomorphisms~$\widetilde \partial_r \colon D_r \to
  \widehat D_{r-1}$ as well as $\widehat \partial_j \colon \widehat D_j \to \widehat
  D_{j-1}$ for all~$j \in \{0,\dots, r-1\}$ subject to the following
  conditions:
  \[ 
     \widehat \partial_{r-1} \circ \widetilde \partial_r = 0
     \qand
     \fa{j \in \{0,\dots,r-1\}}
     \widehat \partial_{j-1} \circ \widehat \partial_j = 0.
  \]
  Moreover, we assume that $D_{r-1}$ is a marked projective summand
  in~$\widehat D_{r-1}$ of codimension~$< \const{\kappa} \cdot \delta$, that
  $\widetilde \partial_r =_{\const{\kappa} \cdot \delta, 1} \partial_r$,
  that $\widetilde \partial_r$
  and $\widehat \partial_{r-1}$ satisfy the claimed norm bounds,
  and that $\widehat \partial_{r-1}|_{D_{r-1}} =_{\const{\kappa} \cdot \delta, 1} \partial_{r-1}$. 
  We have:
  \[\begin{tikzcd}
  	D_{r+1}\ar{r}{\partial_{r+1}}
	& D_r\ar{r}{\partial_r}\ar{dr}[swap]{\widetilde{\partial}_r}
	& D_{r-1}\ar{r}{\partial_{r-1}}\ar[hook]{d}
	& \dotsm\ar{r}
	& D_0\ar[hook]{d}
	\\
	&& \widehat{D}_{r-1}\ar{r}{\widehat{\partial}_{r-1}}
	& \dotsm\ar{r}
	& \widehat{D}_0
  \end{tikzcd}\]

  Let $D_{r+1} = \bigoplus_{i \in I} \gen{A_i}$ be the marked presentation 
  of~$D_{r+1}$. For~$i \in I$, we consider the error term
  \[ B_i
  \coloneqq \supp_1 \bigl( \widetilde \partial_r \circ\partial_{r+1} (\chi_{A_i} \cdot e_i)\bigr)
  \subset X
  .
  \]
  Then $B_i \subset A_i$ and from~$\partial_r \circ \partial_{r+1} =_{\delta} 0$ and
  $\widetilde \partial_r =_{\const{\kappa} \cdot \delta} \partial_r$,
  we obtain 
  \begin{align*}
    \sum_{i \in I} \mu(B_i)
    & \leq
    \sum_{i \in I} 
    \size_1\bigl(\partial_r \circ \partial_{r+1} (\chi_{A_i} \cdot e_i)
    \bigr)
    +
    \sum_{i \in I}
    \size_1\bigl((\widetilde \partial_r - \partial_r) \circ \partial_{r+1} (\chi_{A_i} \cdot e_i)
    \bigr)
    \hspace{-5cm}
    \\
    & < \delta
    +
    \# I 
    \cdot \size_1(\widetilde \partial_r - \partial_r)
    \cdot \Nmax\bigl( \partial_{r+1} \bigr)
    & \text{(Lemmas~\ref{lem:almosteqchar} and~\ref{lem:supp1:est})}
    \\
    & \leq \delta
    +
    \rk (D_{r+1}) 
    \cdot \size_1(\widetilde \partial_r - \partial_r)
    \cdot \Nmax\bigl( \partial_{r+1} \bigr)
    \\
    & \leq
    \delta + 
    \const{\kappa} \cdot \delta
    \leq \const{\kappa} \cdot \delta. 
    & \text{(Lemma~\ref{lem:almosteqchar})}
  \end{align*}

  We set
  \[ E_r \coloneqq \bigoplus_{i \in I} \gen{B_i}
  \qand
  \widehat D_r \coloneqq D_r \oplus E_r.
  \]
  In particular, $\dim (E_r) < \const{\kappa} \cdot \delta$. Moreover, we define
  \begin{align*}
    \widetilde \partial_{r+1}
    \colon D_{r+1} & \to \widehat D_r
    \\
    \chi_{A_i} \cdot e_i
    & \mapsto
    \bigl(\partial_{r+1} (\chi_{A_i} \cdot e_i), 
    - \chi_{B_i} \cdot e_i\bigr)
  \end{align*}
  and
  \begin{align*}
    \widehat \partial_r
    \colon \widehat D_r = D_r \oplus E_r & \to \widehat D_{r-1}
    \\
    D_r \ni x
    & \mapsto \widetilde \partial_r(x) 
    \\
    E_r \ni 
    \chi_{B_i} \cdot e_i
    & \mapsto \widetilde \partial_r \circ \partial_{r+1}(\chi_{A_i} \cdot e_i);
  \end{align*}
  both $\widetilde \partial_{r+1}$ and $\widehat \partial_r$ are
  well-defined $R$-homomorphisms. The marked decomposition~$D_{r+1} \cong_R
  \bigoplus_{i \in I} \gen{A_i \setminus B_i} \oplus \bigoplus_{i \in I} \gen{B_i}$
  shows that $\widetilde\partial_{r+1} =_{\const{\kappa} \cdot \delta} \partial_{r+1}$. 
  By construction, we have
  \[ \widehat \partial_r \circ \widetilde \partial_{r+1} = 0
  \qand
  \widehat \partial_{r-1} \circ \widehat \partial_r = 0,
  \]
  To see the latter, we observe that
  $\widehat \partial_{r-1} \circ \widehat \partial_r|_{D_r}
  = \widehat \partial_{r-1} \circ \widetilde \partial_r = 0$
  and for all~$i \in I$
  \[ \widehat \partial_{r-1} \circ \widehat \partial_r(\chi_{B_i} \cdot e_i)
  = \widehat \partial_{r-1} \circ \widetilde \partial_r \circ \partial_{r+1}(\chi_{A_i} \cdot e_i)
  = 0,
  \]
  where we used the inductive
  property~$\widehat \partial_{r-1} \circ \widetilde \partial_r = 0$.
  Furthermore,
  by construction, we have $\|\widetilde \partial_{r+1} \| \leq \|\partial_{r+1}\| + 1$
  and
  $\widehat \partial_r | _{D_r} = \widetilde \partial_r =_{\const{\kappa} \cdot \delta} \partial_r$.
  Also, by induction,  we have the estimate
  $\|\widehat \partial_r\| \leq \|\widetilde \partial_r\| \cdot (\|\partial_{r+1}\| + 1)
  \leq (\|\partial_r\| + 1) \cdot (\|\partial_{r+1}\| + 1) \leq \const{\kappa}$.
  
  Finally, we set~$\widehat D_{n+1} \coloneqq D_{n+1}$ and $\widehat
  \partial_{n+1} \coloneqq \widetilde \partial_{n+1}$.  This concludes the
  construction. In particular, we obtain the claimed norm
  estimate for~$\widehat \partial_{n+1} = \widetilde \partial_{n+1}$.
\end{proof}

\begin{rem}
Clearly, the inductive construction in the above proof preserves
$S$-adaptedness (Definition~\ref{defn:adapted}) throughout if we start with an $S$-adapted~$z \in
D_0$ with~$\eta(z) =_\delta 1$.
\end{rem}

\subsection{Strictification of almost chain maps}

To formulate and prove strictification of chain maps, we
bound the complexity of the original chain map.

\begin{defn}
  In the situation of Setup~\ref{setup:strict}, let $\delta \in \R_{>0}$,
  let $n \in \N$, and 
  let $f_* \colon C_* \to D_*$ be a \almostcm{\delta}{n}
  between marked projective (almost) $R$-chain complexes. Then, we
  set 
  \[ \kappa_n(f_*) \coloneqq \max \bigl\{\|f_0\|, \dots, \|f_{n+1}\|\bigr\}.
  \]
\end{defn}

\begin{thm}\label{thm:strictifychmap}
  In the situation of Setup~\ref{setup:strict}, let $n \in \N$ and $\kappa \in \R_{>0}$.
  Then, there exists a~$K \in \R_{>0}$ such that: For all~$\delta \in \R_{>0}$, 
  if $(C_*,\zeta)$ and $(D_*,\eta)$ are marked projective $R$-chain
  complexes (up to degree~$n+1$) with~$\max\{\kappa_n(C_*), \nusum_n(C_*)\} < \kappa$,
  $\kappa_n(D_*) < \kappa$ and
  if~$f_* \colon C_* \to D_*$ is a \almostcm{\delta}{n}
  extending~$\id_{\LinftyX}$ with~$\kappa_n(f_*) \leq \kappa$, then
  there exists a marked projective $R$-chain complex~$(\widehat D_*,\widehat \eta)$
  (up to degree~$n+1$) and an $R$-chain map~$\widehat f_* \colon C_* \to \widehat D_*$
  with the following properties:
  \begin{itemize}
  \item We have
    $\dgh K {\widehat D_*,D_*,n} < K \cdot \delta$.
  \item For all~$r\in \{0,\dots,n+1\}$,
    we have~$\dgh K {\widehat f_r, f_r} < K \cdot \delta$.
  \end{itemize}
\end{thm}

\begin{rem}\label{rem:strictifychmap:adapted}
  Moreover, if $(C_*,\zeta)$ and $(D_*,\eta)$ as well as $f_*$ are
  $S$-adapted (Definition~\ref{defn:adapted}), we can choose~$(\widehat D_*,\widehat \eta)$ and $\widehat f_*$
  to be $S$-adapted.  
\end{rem}

\begin{proof}
  Let $\widehat D_{-1} \coloneqq \LinftyX$, $\widehat \partial_{-1} \coloneqq 0$,
  and $\widehat f_{-1} \coloneqq \id_{\LinftyX}$. 
  It suffices to show the following: For all~$r \in \{0,\dots,n+1\}$,
  there exist marked projective $R$-modules~$\widehat D_r = D_r \oplus E_r$
  with~$\dim (E_r) < \const{\kappa} \cdot \delta$ and $R$-homomorphisms 
  $\widehat \partial_r \colon \widehat D_r \to \widehat D_{r-1}$, 
  $\widehat f_r \colon C_r \to \widehat D_r$
  with the following properties:
  \[
  \widehat \partial_r \circ \widehat \partial_{r-1} = 0
  , \quad
  \widehat \partial_r |_{D_r} = \partial^D_r
  , \quad
  \| \widehat \partial_r |_{E_r} \| \leq \const{\kappa}
  \]
  and
  \[ \widehat f_r =_{ \const{\kappa} \cdot \delta, 1} f_r
  , \quad
  \|\widehat f_r\| \leq \|f_r\| + 1
  .
  \]
  We proceed by induction over the degree. Let $r \in \{-1,\dots,n\}$ and let us
  suppose that $\widehat f_*$ and $\widehat D_*$ are already constructed up to
  degree~$r$ with the claimed properties.  We extend the construction to
  degree~$r+1$: To this end, we consider the error function
  \[ \Delta
  \coloneqq \partial^D_{r+1} \circ f_{r+1} - \widehat f_r \circ \partial_{r+1}^C
  \colon C_{r+1} \to \widehat D_r. 
  \]
  Let $C_{r+1} = \bigoplus_{i \in I} \gen{A_i}$ be the marked
  presentation of~$C_{r+1}$; for~$i \in I$, we set
  \[ B_i \coloneqq \supp_1 \bigl(\Delta(\chi_{A_i} \cdot e_i )\bigr)
  \]
  and 
  \[ E_{r+1} \coloneqq \bigoplus_{i \in I} \gen{B_i}
  \qand \widehat D_{r+1} \coloneqq D_{r+1} \oplus E_{r+1}.
  \]
  We then consider the well-defined $R$-homomorphisms 
  \begin{align*}
    \widehat f_{r+1} \colon C_{r+1}
    & \to \widehat D_{r+1}
    \\
    \chi_{A_i} \cdot e_i
    & \to
    \bigl( f_{r+1} (\chi_{A_i} \cdot e_i), - \chi_{B_i} \cdot e_i\bigr)
  \end{align*}
  and
  \begin{align*}
    \widehat \partial_{r+1} \colon \widehat D_{r+1}
    & \to \widehat D_r
    \\
    D_{r+1} \ni x & \mapsto \partial^D_{r+1} (x)
    \\
    \chi_{B_i} \cdot e_i
    & \mapsto \Delta(\chi_{A_i} \cdot e_i).
  \end{align*}
  In particular, $\widehat \partial_{r+1}|_{D_{r+1}} = \partial^D_{r+1}$
  and $\widehat f_{r+1} =_{\dim (E_{r+1}), 1} f_{r+1}$. 
  It remains to show that this construction has also all the other 
  claimed, inductive, properties:

  \emph{Dimensions.}
  By construction, we have
  \begin{align*}
    \Delta
    & = \bigl(\partial^D_{r+1} \circ f_{r+1} - f_r \circ \partial^C_{r+1} \bigr)
    - (\widehat f_r - f_r) \circ \partial^C_{r+1}.
  \end{align*}
  The first difference is $\delta$-almost equal to~$0$,
  because $f_*$ is a \almostcm{\delta}{n}.
  Moreover, we know 
  $(\widehat f_r - f_r) \circ \partial^C_{r+1}
  =_{\const{\kappa} \cdot \Nsum(\partial^C_{r+1}) \cdot \delta} 0
  $ 
  from the inductive property~$\widehat f_r =_{\const{\kappa} \cdot
    \delta} f_r$ and Lemma~\ref{lem:almosteqinherit}~\ref{itm:precomp}.
  Unifying all constants, we conclude that~$\Delta =_{\const{\kappa}\cdot \delta} 0$ (Lemma~\ref{lem:almosteqinherit}~\ref{itm:composition:additivity}).
  In particular, with Lemma~\ref{lem:almosteqchar} we obtain
  \[ \dim (E_{r+1}) = \sum_{i \in I} \mu(B_i) = \size_1 (\Delta) < \const{\kappa} \cdot \delta.
  \]
  
  \emph{Chain complex property.}
  On the one hand, for all~$x \in D_{r+1}$, we have~$\widehat \partial_{r+1}(x)
  = \partial^D_{r+1}(x) \in D_r$ and so with the strict chain complex
  property of~$(D_*,\eta)$ we calculate
  \[ \widehat \partial_r \circ \widehat \partial_{r+1}(x)
  = \widehat \partial_r \circ \partial_{r+1}^D(x)
  = \partial_r^D(\partial_{r+1}^D(x))
  = 0.
  \]
  On the other hand, for all~$i \in I$, by construction, we have 
  \begin{align*}
    \widehat \partial_r \circ \widehat \partial_{r+1}(\chi_{B_i} \cdot e_i)
    &= \widehat\partial_r \circ \partial_{r+1}^D \circ f_{r+1}(\chi_{A_i}\cdot e_i) 
    - \widehat\partial_r \circ \widehat f_r \circ \partial^C_{r+1}(\chi_{A_i} \cdot e_i).
  \end{align*}
  Because $\partial^D_{r+1} \circ f_{r+1}(\chi_{A_i} \cdot e_i)$ lies
  in~$D_r$, the first term equals~$\partial_r^D \circ \partial_{r+1}^D
  \circ f_{r+1}(\chi_{A_i}\cdot e_i) $, which is zero.
  For the second term, by induction, we have
  \[\widehat\partial_r \circ \widehat f_r \circ \partial^C_{r+1}(\chi_{A_i} \cdot e_i)
  = \widehat f_{r-1} \circ \partial^C_r \circ \partial^C_{r+1}(\chi_{A_i} \cdot e_i), 
  \]
  which is also zero.
  Therefore, $\widehat \partial_r \circ \widehat \partial_{r+1} = 0$.

  \emph{Chain map property.}
  The fact that $\widehat \partial_{r+1} \circ \widehat f_{r+1}
  - \widehat f_r \circ \partial^C_{r+1} = 0$ is immediate from
  the construction.

  \emph{Norm estimates.}
  By construction,
  $\|\widehat f_{r+1}\| \leq \|f_{r+1}\| +1$
  and $\|\widehat \partial_{r+1}|_{E_{r+1}}\| \leq \|\Delta\|$.
  Moreover, $\|\Delta\|$ can be subsumed in~$\const{\kappa}$.
\end{proof}
  
\section{Deformation}\label{sec:deformation}

We explain how to adapt modules, maps, chain complexes, and
chain maps to a dense subalgebra of the measurable sets.
Our key example is the algebra of cylinder sets in profinite
completions along directed systems of finite index normal
subgroups.

We prove two deformation results:
\begin{itemize}
	\item Theorem~\ref{thm:adaptedGH}: Every marked projective chain complex is ``close'' (in the Gromov--Hausdorff sense) to an ``adapted'' chain complex.
	\item Theorem~\ref{thm:adaptedembgen}: Every chain map is ``close'' to an ``adapted'' chain map to a target complex that is ``close'' to the original target complex.
\end{itemize}

\begin{setup}\label{setup:adapt}
  Let $\Gamma$ be a countable group and let $Z$ denote $\Z$ (with
  the usual norm) or a finite field (with the trivial norm).
  We consider a standard $\Gamma$-action~$\alpha \colon \Gamma \actson (X,\mu)$. 
  Moreover, let $\Rrel$ be the associated orbit relation and
  let $R \subset Z \Rrel$ be a subring that
  contains~$\LinftyXZ * \Gamma$.

  Let $S$ be a subalgebra of all measurable sets of~$X$ 
  that is $\mu$-dense and that satisfies~$\Gamma \cdot S \subset S$.
  We write~$L$ for the subring of~$\LinftyX$ generated by~$S$.
\end{setup}

\subsection{Adapted objects/morphisms}

Elements, homomorphisms, or modules are adapted to the
algebra~$S$ if they involve only measurable subsets
in~$S$. 

\begin{defn}[adapted]
\label{defn:adapted}
	In the situation of Setup~\ref{setup:adapt}, we say that:
  \begin{itemize}
  \item
    A marked projective $R$-module~$M = \bigoplus_{i \in I} \gen{A_i}$
    is \emph{adapted to~$S$} if $A_i \in S$ holds for all~$i \in I$. 
  \item
    An element in~$Z\Rrel$ is \emph{adapted to~$S$}
    if it lies in~$L * \Gamma$.
  \item
    An $R$-homomorphism between marked projective $R$-modules
    is \emph{adapted} if it is defined over~$L * \Gamma$.
    These notions admit obvious extensions to the case where
    the target $R$-module is~$\LinftyX$.
  \item
    Marked projective chain complexes are \emph{adapted to~$S$} if the
    chain modules and the boundary operators are adapted to~$S$.  Chain
    maps between marked projective chain complexes are \emph{adapted
      to~$S$} if they consist of adapted homomorphisms.
  \end{itemize}
\end{defn}

\subsection{Adapting module homomorphisms}

Density of the subalgebra leads to a basic deformation observation,
which will be the foundation for all other deformation results:

\begin{lem}\label{lem:adaptbasicR}
  In the situation of Setup~\ref{setup:adapt}, let $A, B \in S$
  and let $\lambda \in R$ with~$\supp (\lambda) \subset A \times B$. 
  Then, for each~$\delta \in \R_{>0}$, 
  there exists an $S$-adapted element~$\widehat \lambda \in L * \Gamma \subset R$
  with~$\supp (\widehat{\lambda})\subset A \times B$ and:
  \begin{enumerate}[label=\enum]
  \item $\widehat \lambda =_\delta \lambda$;
  \item $|\widehat \lambda|_\infty \leq |\lambda|_\infty$;
  \item $\Nbasic(\widehat \lambda) \leq \Nbasic(\lambda)$ and $\Ntbasic(\widehat \lambda) \leq \Ntbasic(\lambda)$;
  \item $|\widehat \lambda - \lambda|_1 < \delta$.
  \end{enumerate}
\end{lem}
\begin{proof}
  We use that $\LinftyX * \Gamma$ is $L^1$-dense in~$R \subset Z\Rrel$
  and that $L*\Gamma$ is $L^1$-dense in~$\LinftyX *\Gamma$
  (because $S$ is $\mu$-dense).
  We can write~$\lambda$ in the form
  \[ \lambda
  = \sum_{k\in \N} \lambda_k \cdot \chi_{\Delta(A_k,\gamma_k)},
  \]
  with~$\lambda_k \in Z$,
  $A_k \subset X$ measurable,
  $\gamma_k \in \Gamma$, and
  $\Delta(A_k,\gamma_k) \coloneqq \{ (\gamma_k \cdot x, x) \mid x \in A_k\}.
  $
  Moreover, we may assume without loss of generality that
  this decomposition is reduced in the sense that
  the sets~$\Delta(A_k,\gamma_k)$ are pairwise disjoint. 

  We now make the following deformations: We truncate~$\lambda$ to
  a finite sum, approximate the~$A_k$ by elements in~$S$, and finally
  adjust the resulting functions to satisfy the upper bounds for $\Nbasic$
  and~$\Ntbasic$.

  \emph{Truncation.} 
  Because~$|\lambda|_1 < \infty$, we can find a finite (non-empty)
  subset~$K \subset \N$ such that
  \[ \lambda_K \coloneqq \sum_{k \in K} \lambda_{k} \cdot \chi_{\Delta(A_k,\gamma_k)}
  \in \LinftyX * \Gamma 
  \]
  satisfies~$| \lambda_K - \lambda|_1 < \delta$. 

  \emph{Approximation.} 
  Because $S$ is $\mu$-dense and $A, B \in S$,
  for each~$k \in K$, we find~$\widetilde A_k \in S$
  with $\gamma_k \cdot \widetilde A_k \times \widetilde A_k \subset A \times B$ and  
  \[ \mu (\widetilde A_k \symmdiff A_k) < \frac{\delta}{\# K \cdot |\lambda|_\infty +1}.
  \]
  In addition, by inductively refining the choice of the~$\widetilde
  A_k$, we may assume that the
  sets~$\Delta(\widetilde A_k,\gamma_k)$ are pairwise disjoint.
  Then
  $\widetilde \lambda \coloneqq \sum_{k \in K} \lambda_k \cdot
  \chi_{\Delta(\widetilde A_k,\gamma_k)}$ lies in~$L*\Gamma$ and since each
  $\lambda_k \in Z$ it
  satisfies~$|\widetilde \lambda|_\infty \leq |\lambda|_\infty$ as
  well as
  \begin{align*}
    \nu \bigl( \supp (\widetilde \lambda - \lambda) \bigr)
    & \leq |\widetilde \lambda - \lambda |_1
    \\
    & \leq |\widetilde \lambda - \lambda_K|_1 + |\lambda_K - \lambda|_1
    \\
    & <
    \Bigl| \sum_{k \in K} \lambda_k \cdot \chi_{\Delta(A_k,\gamma_k) \symmdiff \Delta(\widetilde A_k,\gamma_k)}
    \Bigr|_1
    + \delta
    \\
    &
    \leq
    |\lambda|_\infty \cdot \sum_{k \in K} \mu (A_k \symmdiff \widetilde A_k)
    + \delta
    \\
    & < 2 \cdot \delta.
  \end{align*}
  
  \emph{Controlling~$\Nbasic$ and~$\Ntbasic$.} 
  We consider the violating subset 
  $$E \coloneqq \bigl\{ (x,y) \in \supp(\widetilde \lambda)
  \bigm| \Nbasic(\widetilde \lambda,y) > \Nbasic(\lambda)
  \text{ or } \Ntbasic(\widetilde \lambda, x) > \Ntbasic(\lambda)
  \bigr\}.
  $$ 
  By definition, $E$ lies in the subalgebra~$S \otimes S$
  and $E \subset \supp(\widetilde
  \lambda) \symmdiff \supp(\lambda)$. In particular,
  \[ \nu(E)
  \leq \nu \bigl(\supp(\widetilde \lambda - \lambda)\bigr)
  < 2\cdot \delta
  .
  \]
  We finally consider the modified function
  \[ \widehat \lambda
  \coloneqq \chi_{\Rrel \setminus E} \cdot \widetilde \lambda.
  \]
  By construction, $\widehat \lambda \in L * \Gamma$ and $\widehat
  \lambda$ satisfies the following estimates:
  \begin{itemize}
  \item $\Nbasic(\widehat \lambda) \leq \Nbasic(\lambda)$
    and $\Ntbasic(\widehat \lambda) \leq \Ntbasic(\lambda)$
    (by construction of~$E$);
  \item $|\widehat \lambda|_\infty \leq |\widetilde \lambda|_\infty \leq |\lambda|_\infty$;
  \item $\mu(\supp_1(\widehat \lambda - \lambda)) \leq\nu (\supp(\widehat \lambda - \lambda))
    \leq \nu (E)
    + \nu (\supp(\widetilde \lambda - \lambda))
    \leq 2\cdot \delta + 2 \cdot \delta = 4 \cdot \delta$; 
  \item $|\widehat\lambda - \lambda|_1 \leq |\widehat \lambda - \lambda|_\infty
    \cdot \nu (\supp(\widehat \lambda- \lambda))
    \leq (|\widehat \lambda|_\infty + |\lambda|_\infty) \cdot 4 \cdot \delta
    = |\lambda|_\infty \cdot 8 \cdot \delta$.
  \end{itemize}
  In particular,
  $\widehat \lambda =_{4 \cdot \delta} \lambda$ (Remark~\ref{rem:rk1almosteq}). 
  Rescaling $\delta$ by the factor~$1/(4+|\lambda|_\infty \cdot 8)$,
  which depends only on~$\lambda$ but not on~$\delta$, finishes the proof.
\end{proof}

\begin{lem}\label{lem:adaptbasic}
  In the situation of Setup~\ref{setup:adapt},  
  let $f \colon M \to N$ be an $R$-homomorphism between $S$-adapted
  marked projective $R$-modules. 
  Then, for each~$\delta \in \R_{>0}$, 
  there exists an $S$-adapted $R$-homo\-morphism~$\widehat f \colon M
  \to N$ such that:
  \begin{enumerate}[label=\enum]
  \item $\widehat f =_\delta f$;
  \item $\|\widehat f\|_\infty \leq \|f\|_\infty$;
  \item $\Nmax(\widehat f) \leq \Nmax(f)$ and $\Ntmax(\widehat f) \leq \Ntmax(f)$.
  \end{enumerate}
  In particular, with $K_f \coloneqq \Ntmax(f) \cdot \|f\|_\infty$, we obtain~$\widehat f =_{\delta, 2\cdot K_f} f$.
\end{lem}
\begin{proof}
 Because of Lemma~\ref{lem:normN1est}, it suffices to prove the first three claims. Indeed,
 \[
 \|\widehat f - f\| \leq \|\widehat f\| + \|f\| \leq \Ntmax(\widehat f) \cdot \|\widehat f \|_\infty + \Ntmax(f) \cdot \|f \|_\infty \leq 2 \cdot K_f. 
 \]
  By definition of the involved norm and size invariants, it suffices
  to consider the case that $M$ and $N$ have rank~$1$; we
  thus consider the case that $M = \gen{A}$, $N=\gen{B}$ with~$A,B \in
  S$ and that $f \colon M \to N$ is given by~$\chi_A \mapsto \lambda
  \cdot \chi_B$ with~$\lambda \in Z\Rrel$
  and~$\supp (\lambda) \subset A \times B$ (Remark~\ref{rem:def:homo:marked:proj}).
  Applying the previous approximation result (Lemma~\ref{lem:adaptbasicR}),
  we find~$\lambda \in L * \Gamma$ that satisfies
  \[ \widehat \lambda =_\delta \lambda
  , \quad
  |\widehat \lambda|_\infty \leq |\lambda|_\infty
  , \quad
  \Nbasic(\widehat \lambda) \leq \Nbasic(\lambda)
  , \quad
  \Ntbasic(\widehat \lambda) \leq \Ntbasic(\lambda).
  \]
  Therefore, the $R$-linear map~$\widehat f \colon \gen A \to \gen B$
  given by~$\chi_A \mapsto \widehat\lambda \cdot \chi_B$ is well-defined
  and has the following properties (Remark~\ref{rem:rk1almosteq},
  Remark~\ref{rem:rk1infN1}):
  \begin{enumerate}[label=\enum]
  \item
    $\widehat f =_\delta f$;
  \item 
    $\|\widehat f\|_\infty = \bigl| \widehat f(\chi_A)\bigr|_\infty =
    |\widehat \lambda|_\infty\leq |\lambda|_\infty = \|f\|_\infty$;
  \item $\Nmax(\widehat f) = \Nbasic(\widehat \lambda) \leq \Nbasic(\lambda) = \Nmax(f)$ and \\
    $\Ntmax(\widehat f) = \Ntbasic(\widehat \lambda) \leq \Ntbasic(\lambda) = \Ntmax(f)$.
    \qedhere
  \end{enumerate}
\end{proof}

\subsection{Adapting almost chain maps}

As above, given an $R$-homomorphism~$f$ between marked projective $R$-modules, 
we set $K_f \coloneqq \Ntmax(f) \cdot \|f\|_\infty$. 

\begin{prop}\label{prop:adaptchmapalmost}
  In the situation of Setup~\ref{setup:adapt}, let $n \in \N$,
  let $\delta \in \R_{>0}$, and let $f_* \colon (C_*,\zeta) \to (D_*,\eta)$
  be a \almostcm{\delta}{n} between marked projective
  $S$-adapted $R$-chain complexes.
  Then, there exists an $S$-adapted \almostcm{(2\cdot \delta)}{n} $\widehat f_* \colon (C_*,\zeta) \to (D_*,\eta)$
  extending~$\id_{\LinftyX}$ with the following properties:
  For all~$r \in \{0,\dots, n+1\}$, we have
  \begin{enumerate}[label=\enum]
  \item  $\widehat f_r =_{\delta} f_r$;
  \item  $\|\widehat f_r \|_\infty \leq \|f_r\|_\infty$;
  \item  $\Nmax(\widehat f_r) \leq \Nmax(f_r)$ and $\Ntmax(\widehat{f}_r)\le \Ntmax(f_r)$.
  \end{enumerate}
  In particular, $\widehat f_r =_{\delta, 2\cdot K_{f_r}} f_r$. 
\end{prop}
\begin{proof}
  For~$\varepsilon \in \R_{>0}$, we apply Lemma~\ref{lem:adaptbasic}
  to~$f_0, \dots, f_{n+1}$ and to the parameter~$\varepsilon$ to obtain
  corresponding $S$-adapted~$\widehat f_0, \dots, \widehat f_{n+1}$
  with norm and multiplicity control.
  A straightforward computation (using the basic estimates from
  Lemma~\ref{lem:almosteqinherit}) shows that $\widehat f_*$ is a \almostcm{(\delta
  + (1 + \nusum_n(C_*)) \cdot \varepsilon)}{n}
  extending~$\id_{\LinftyX}$. 
  We then choose our initial~$\varepsilon$ small enough.
\end{proof}

\subsection{Adapting chain complexes (almost)}

We can approximate chain complexes by adapted almost chain complexes
by approximating the chain modules and boundary operators
by adapted modules/homomorphisms: 

\begin{prop}\label{prop:adaptGHalmost}
  In the situation of Setup~\ref{setup:adapt}, let $n \in \N$, 
  let $(D_*,\eta)$ be a marked projective $R$-chain complex
  (up to degree~$n+1$), and let $z \in D_0$ with~$\eta(z) = 1$.
  Then, there exists a~$K \in \R_{>0}$ such that:
  For every~$\delta \in \R_{>0}$, there exists a marked projective
  $S$-adapted \almostcc{\delta}{n}~$(\widehat D_*,\widehat \eta)$
  with
  \[ \dgh K {\widehat D_*, D_*, n} < \delta
  \]
  and the following additional control: 
  \begin{itemize}
  \item For all~$r \in \{0,\dots, n+1\}$,
    we have~$\rk (\widehat D_r) \leq \rk (D_r)$.
  \item For all~$r \in \{0,\dots, n+1\}$,
    we have~$\Nmax(\partial^{\widehat D}_r) \leq \Nmax(\partial^D_r)$.
  \item
    There exists an $S$-adapted~$\widehat z \in \widehat D_0$
    with~$\widehat \eta (\widehat z) =_\delta 1$, $\Nbasic(\widehat z)
    \leq \Nbasic(z)$, $\Ntbasic(\widehat z) \le \Ntbasic(z)$, and $|\widehat z|_\infty \leq |z|_\infty$.
  \end{itemize}
  In particular, $\overline \kappa_n(\widehat D_*) \leq \overline \kappa_n(D_*) + K$.
\end{prop}

As a preparation, we first adapt the chain modules:

\begin{lem}\label{lem:adaptGHalmostchainmodules}
  In the situation of Setup~\ref{setup:adapt}, let $n \in \N$, 
  let $(D_*,\eta)$ be a marked projective $R$-chain complex
  (up to degree~$n+1$), and let $z \in D_0$ with~$\eta(z) = 1$.
  Then, there exists a~$K \in \R_{>0}$ such that:
  For every~$\delta \in \R_{>0}$, there exists a marked projective
  \almostcc{\delta}{n}~$(\widehat D_*,\widehat \eta)$
  consisting of $S$-adapted chain modules (but not necessarily
  $S$-adapted boundary operators) 
  with
  \[ \dgh K {\widehat D_*, D_*, n} < \delta
  \]
  and the following additional control: 
  \begin{itemize}
  \item For all~$r \in \{0,\dots, n+1\}$,
    we have~$\rk (\widehat D_r) \leq \rk (D_r)$.
  \item For all~$r \in \{0,\dots, n+1\}$,
    we have~$\Nmax(\partial^{\widehat D}_r) \leq \Nmax(\partial^D_r)$.
  \item There exists a~$\widehat z \in \widehat D_0$ with $\widehat \eta(\widehat z) =_\delta 1$,
    $\Nbasic(\widehat z) \leq \Nbasic(z)$, $\Ntbasic(\widehat z) \le \Ntbasic(z)$, and $|\widehat z|_\infty \leq |z|_\infty$.
  \end{itemize}
\end{lem}
\begin{proof}
  Let $\delta \in \R_{>0}$. 
  Because $S$ is $\mu$-dense, we can efficiently adapt the chain
  modules: For~$r \in \N$ and the marked presentation~$D_r =
  \bigoplus_{i \in I} \gen{A_i}$, we choose~$\widehat A_i \in S$ in
  such a way that~$\sum_{i \in I} \mu (\widehat A_i \symmdiff A_i) <
  \delta$. We then consider the $S$-adapted ``sibling''
  \[ \widehat D_r \coloneqq \bigoplus_{i \in I} \gen{\widehat A_i}.
  \]
  In particular, $\rk (\widehat D_r) = \rk (D_r)$
  and $|\dim \widehat D_r - \dim D_r| < \delta$.

  We write~$\Phi_r \colon D_r \to \widehat D_r$ for
  the composition of the canonical inclusion/projection to/from
  the joint canonical hull~$P_r \coloneqq \bigoplus_{i \in I} \gen{A_i \cup \widehat A_i}$
  of~$D_r$ and~$\widehat D_r$; in
  particular, $\Nmax(\Phi_r) \leq 1$. Similarly,
  in the other direction, we write~$\Psi_r \colon \widehat D_r \to D_r$
  for the canonical $R$-homomorphism.
  Moreover, we set $\Phi_{-1} \coloneqq \id_{\LinftyX}$ and $\Psi_{-1} \coloneqq \id_{\LinftyX}$. 

  Regarding the boundary operators, we consider the compositions
  \[ \widehat \partial_r
  \coloneqq \Phi_{r-1} \circ \partial_r \circ \Psi_r
  \colon \widehat D_r \to \widehat D_{r-1}
  \]
  for~$r \in \{0,\dots, n+1\}$ and set $\widehat \eta \coloneqq \widehat\partial_0$. Hence we have the commutative diagram:
  \[
  \begin{tikzcd}
    D_r
    \ar{r}{\partial_r}
    & D_{r-1}
    \ar{d}{\Phi_{r-1}}
    \\
    \widehat D_r
    \ar[dashed]{r}[swap]{\widehat\partial_r}
    \ar{u}{\Psi_r}
    & \widehat D_{r-1}.
  \end{tikzcd}
  \]
  By construction, $\Nmax(\widehat \partial_r) \leq \Nmax(\partial_r)$
  and we claim that
  \[
  \dgh {2\cdot \kappa_n(D_*)} {\widehat D_*, D_*, n}
  < \bigl(1 + \nusum_n(D_*)\bigr) \cdot \delta.
  \]
  Indeed, this is witnessed by the following diagram:
  \[ \begin{tikzcd}
    D_{r}
    \ar{r}{\partial_{r}}
    \ar[shift right,hookrightarrow]{d}[swap]{\varphi_{r}}
    & D_{r-1}
    \ar[shift right,hookrightarrow]{d}[swap]{\varphi_{r-1}}
    \\
    P_{r}
    \ar[shift right]{u}[swap]{\pi_{\varphi_{r}}}
    \ar[shift left]{d}{\pi_{\widehat\varphi_{r}}}
    \ar[shift left, dashed]{r}{F_{r}}
    \ar[shift right, dashed]{r}[swap]{\widehat F_{r}}
    & P_{r-1}
    \ar[shift right]{u}[swap]{\pi_{\varphi_{r-1}}}
    \ar[shift left]{d}{\pi_{\widehat\varphi_{r-1}}}
    \\
    \widehat D_{r}
    \ar{r}[swap]{\widehat\partial_{r}}
    \ar[shift left,hookrightarrow]{u}{\widehat\varphi_{r}}
    & \widehat D_{r-1}
    \ar[shift left,hookrightarrow]{u}{\widehat\varphi_{r-1}}
    \end{tikzcd}
  \]
  By construction, we have
  \begin{align*}
  	\widehat{F}_r
	&= \widehat{\varphi}_{r-1}\circ \pi_{\widehat{\varphi}_{r-1}}\circ F_r\circ \widehat{\varphi}_r\circ \pi_{\widehat{\varphi}_r}
	\\
	\id_{P_r}
	&=_{\delta,1} \widehat{\varphi_r}\circ \pi_{\widehat{\varphi}_r}
	\\
	\id_{P_{r-1}}
	&=_{\delta,1} \widehat{\varphi}_{r-1}\circ \pi_{\widehat{\varphi}_{r-1}}.
  \end{align*}
  Using Proposition~\ref{prop:eqdeltaK}, we conclude
  \begin{align*}
  	\widehat{F}_r
	&=_{\delta,\|F_r\|} \widehat{\varphi}_{r-1}\circ \pi_{\widehat{\varphi}_{r-1}}\circ F_r
	\\
	\widehat{\varphi}_{r-1}\circ \pi_{\widehat{\varphi}_{r-1}}\circ F_r
	&=_{\Nsum(F_r)\cdot \delta,\|F_r\|} F_r
  \end{align*}
  and together, using $\Nsum(F_r)\le \Nsum(\partial_r)$ and $\|F_r\|\le \|\partial_r\|$,
  \[
  	\widehat{F}_r=_{(1+\Nsum(\partial_r))\cdot \delta,2\cdot \|\partial_r\|} F_r.
  \]
  This proves the claim. 
  Since all the estimates depend only on~$D_*$, we can now rescale
  our initial~$\delta$ appropriately and apply Lemma~\ref{lem:GHclosealmostcomplex}.
\end{proof}

\begin{proof}[Proof of Proposition~\ref{prop:adaptGHalmost}]
  We set $K \coloneqq 2 \cdot \max \{K_{\partial_0}, \dots, K_{\partial_{n+1}}\}$.  
  Let $\delta \in \R_{>0}$. In view of Lemma~\ref{lem:adaptGHalmostchainmodules}
  and the triangle inequality of the Gromov--Hausdorff distance between
  (almost) chain complexes (Proposition~\ref{prop:dghchaincomplex}),
  we may assume without loss of generality that~$(D_*,\eta)$ is a
  marked projective \almostcc{\delta}{n} with
  $S$-adapted chain modules of the same ranks; however, $z \in D_0$
  might not map to~$1$, but only satisfy  the almost
  equality~$\eta(z) =_\delta 1$
  (while keeping control on $\Nbasic$, $\Ntbasic$, and~$|\cdot|_\infty$).
  
  We then set~$\widehat D_r \coloneqq D_r$ for all~$r \in \{0,\dots,n+1\}$
  and apply the basic approximation lemma (Lemma~\ref{lem:adaptbasic})
  with accuracy~$\delta$ to~$\partial_0 = \eta,\partial_1 \dots,
  \partial_{n+1}$ to obtain $S$-adapted $R$-homomorphisms~$\widehat
  \partial_r \colon \widehat D_r \to \widehat D_{r-1}$ for all~$r \in
  \{0,\dots, n+1\}$ with
  \[ \widehat \partial_r =_{\delta, 2 \cdot K_{\partial_r}} \partial_r
    \qand
    \Nmax(\widehat \partial_r) \leq \Nmax(\partial_r)
    .
  \]
  Therefore, $\dgh K {\widehat D_*, D_*,n} < \delta$.
      
  Finally, approximating the carrier sets appearing in~$z$ well
  enough through elements of~$S$, we find an $S$-adapted~$\widehat z \in \widehat D_0$
  with~$\widehat \eta(\widehat z) =_{2\cdot \delta} 1$, $\Nbasic(\widehat z) \leq \Nbasic(z)$,
  $\Ntbasic(\widehat z) \le \Ntbasic(z)$, and 
  $|\widehat z|_\infty \leq |z|_\infty$ (using Lemma~\ref{lem:adaptbasicR} in each
  coordinate).
  Applying Lemma~\ref{lem:GHclosealmostcomplex} and 
  rescaling the initial parameter~$\delta$ beforehand completes the proof. 
\end{proof}

\subsection{Deformation of chain complexes}

We first approximate the chain complex by adapted almost
chain complexes (Proposition~\ref{prop:adaptGHalmost})
and then strictify these almost chain complexes
(Theorem~\ref{thm:strictifycomplex}): 

\begin{thm}\label{thm:adaptedGH}
  In the situation of Setup~\ref{setup:adapt}, let $n \in \N$ and let
  $(D_*,\eta)$ be a marked projective $R$-chain complex (up to
  degree~$n+1$).
  Then there exists a~$K\in \R_{>0}$ such that: 
  For every~$\delta \in \R_{>0}$, there exists an $S$-adapted marked
  projective chain complex~$\widehat D_*$ (up to degree~$n+1$) with
  \[ \dgh K {\widehat D_*, D_*, n} < \delta.
  \]
\end{thm}

\begin{proof}
  First, we fix some constants: Let $z \in D_0$
  with~$\eta(z) = 1$ suitable for~$\overline \kappa_n(D_*)$. 
  We apply Proposition~\ref{prop:adaptGHalmost} to~$(D_*,\eta)$, $z$, 
  and $n \in \N$ 
  and thus obtain a constant~$K \in \R_{>0}$ with the properties
  in Proposition~\ref{prop:adaptGHalmost}, controlling $S$-adapted
  almost chain complexes approximating~$(D_*,\eta)$. We set
  \[ \kappa \coloneqq \overline \kappa_n(D_*) + K + 1.
  \]
  Let $K' \in \R_{>0}$ be a constant as provided by
  Theorem~\ref{thm:strictifycomplex} when applied to the parameters $n$
  and~$\kappa$. 

  Now, let $\delta \in \R_{>0}$. 
  Let $(D'_*,\eta'_*)$ be an $S$-adapted
  \almostcc{\delta}{n} 
  as obtained from Proposition~\ref{prop:adaptGHalmost} for~$\delta$. 
  In particular, 
  \[ \overline \kappa_n(D'_*) \leq \overline \kappa_n(D_*) + K < \kappa \quad \mbox{ and } \quad \dgh K {D'_*, D_*, n} < \delta. 
  \]
  Thus, by the strictification theorem
  (Theorem~\ref{thm:strictifycomplex}), there exists an $S$-adapted 
  marked projective $R$-chain complex~$(\widehat D_*,\widehat \eta)$
  (up to degree~$n+1$) 
  with
  \[ \dgh {K'} {\widehat D_*, D'_*, n} < K' \cdot \delta.
  \]
  In total, we obtain (Proposition~\ref{prop:dghchaincomplex})
  \begin{align*}
    \dgh {K+K'} {\widehat D_*, D_*, n}
    & < \delta + K' \cdot \delta = (1+K') \cdot \delta
    .
  \end{align*}
  Unifying the constants and
  rescaling~$\delta$ beforehand, gives the desired result. 
\end{proof}

\subsection{Deformation of chain maps}

We apply the previously established deformation theorem for chain
complexes (Theorem~\ref{thm:adaptedGH}), the approximation of almost
chain maps (Proposition~\ref{prop:adaptchmapalmost}), and the 
strictification of chain maps (Theorem~\ref{thm:strictifychmap}) to prove
the following: 

\begin{thm}\label{thm:adaptedchmap}
  In the situation of Setup~\ref{setup:adapt}, let $n\in \IN$.
  Let~$(\widehat{C}_*,\widehat{\zeta})$ 
  and~$(D_*,\eta)$ be marked projective $R$-chain complexes (up to degree~$n+1$) with~$(\widehat{C}_*,\widehat{\zeta})$ being $S$-adapted. 
  Let $F_*\colon \widehat{C}_*\to D_*$ be an $R$-chain map extending~$\id_{\LinftyX}$.
  Then there exists a~$K\in \R_{>0}$ such that: For every~$\delta \in \R_{>0}$,
  there exists an $S$-adapted $R$-chain map $\widehat F_* \colon \widehat{C}_* \to \widehat D_*$
  with
  \[ \dgh K {\widehat D_*, D_*,n} < \delta
  \qand
  \fa{r \in \{0,\dots, n+1\}}
  \dgh K {\widehat F_r, F_r} < \delta 
  .
  \]
\end{thm}

\begin{proof}
  We first fix our set of constants: 
  Let $K \in \R_{>0}$ be a constant as provided by Theorem~\ref{thm:adaptedGH}
  when applied to~$n \in \N$ and the target complex~$(D_*,\eta)$. 
  
  We set
  $N
  \coloneqq \max \{ \Nsum(F_0), \dots, \Nsum(F_{n+1}), 1\}
  $
  and
  \[
  \kappa \coloneqq \max \bigl\{
  \kappa_n(\widehat{C}_*), \nusum_n(\widehat{C}_*),\kappa_n(D_*) + K,  
  \Ntsum(F_0) \cdot \|F_0\|_\infty, \dots, \Ntsum(F_{n+1}) \cdot \|F_{n+1}\|_\infty
  \bigr\}
  + 1
  .
  \]
  Let $K' \in \R_{>0}$ be a constant as provided by
  Theorem~\ref{thm:strictifychmap}
  when applied to~$n \in \N$
  and~$\kappa$. 

  Now, let $\delta \in \R_{>0}$.
  We proceed in the following steps:
  
  \begin{align*}
  	F_*\colon \widehat{C}_* 
	&\to D_*
	& (D_* \text{ chain complex},\ & F_* \text{ chain map})
	\\
	F'_*\colon \widehat{C}_*
	&\to D'_*
	& (D'_* \text{ adapted chain complex},\ & F' \text{ almost chain map})
	\\
	F''_*\colon \widehat{C}_*
	&\to D'_*
	& (D'_* \text{ adapted chain complex},\ & F''_* \text{ adapted almost chain map})
	\\
	\widehat{F}_*\colon \widehat{C}_*
	&\to \widehat{D}_*
	& (\widehat{D}_* \text{ adapted chain complex},\ & \widehat{F}_* \text{ adapted chain map})
\end{align*}

  \emph{Adapting the target complex.} 
  Theorem~\ref{thm:adaptedGH}
  yields an $S$-adapted marked projective $R$-chain complex~$(D'_*,\eta')$ 
  (up to degree~$n+1$) 
  with
  \[ \dgh K {D'_*,D_*,n} < \delta.
  \]
  In particular, $\kappa_n(D'_*) \leq \kappa_n(D_*) + K < \kappa$. 
  By Proposition~\ref{prop:dGHalmostchmap}, 
  there exists a marked \almostcm{\delta}{n} $\Phi_* \colon D_* \to D'_*$
  extending~$\id_{\LinftyX}$ with 
  \[ \fa{r \in \{0,\dots,n+1\}}
    \|\Phi_r\|_\infty \leq 1
    \qand
    \Nmax(\Phi_r) \leq 1 \quad \text{(Remark~\ref{rem:norms:marked:homo:bounded:by:1})}.
  \]
  We consider the composition 
  \[ F'_* \coloneqq \Phi_* \circ F_* \colon \widehat C_* \to D'_*.
  \]
  By construction, $F'_*$ is an \almostcm{(N \cdot \delta)}{n} extending~$\id_{\LinftyX}$ (Lemma~\ref{lem:comp-almost-chain-maps}) 
  and
  \[ \fa{r \in \{0,\dots, n+1\}}
  \Ntmax(F'_r) \leq \Ntmax(F_r),
  \quad
  \|F'_r\|_\infty \leq \|F_r\|_\infty,
  \quad
  \dgh {\|F_r\|} {F'_r, F_r} < N\cdot \delta.
  \]

  \emph{Adapting the chain map (almost).}
  We apply Proposition~\ref{prop:adaptchmapalmost} to obtain
  an $S$-adapted \almostcm{(2\cdot N \cdot \delta)}{n} $F''_* \colon \widehat C_* \to D'_*$ that satisfies
  \begin{align*}
    \kappa_n(F''_*)
    & \leq \max_{r \in \{0,\dots, n+1\}}
    \bigl( \Ntmax(F''_r) \cdot \|F''_r\|_\infty \bigr)
    & \text{(Lemma~\ref{lem:normN1est})}
    \\
    & \leq \max_{r \in \{0,\dots, n+1\}}
    \bigl( \Ntmax(F'_r) \cdot \|F'_r\|_\infty \bigr)
    & \text{(Proposition~\ref{prop:adaptchmapalmost})}
    \\
    & \leq \max_{r \in \{0,\dots, n+1\}}
    \bigl( \Ntmax(F_r) \cdot \|F_r\|_\infty \bigr)
    & \text{(construction of~$F'_*$)}
    \\
    & \leq \kappa
  \end{align*}
  and $\dgh {2\cdot K_{F_r}} {F''_r,F'_r} < N \cdot \delta$.
  
  \emph{Strictifying the chain map.} 
  Finally, we apply the strictification of chain maps
  (Theorem~\ref{thm:strictifychmap} and Remark~\ref{rem:strictifychmap:adapted}) to~$F''_* \colon \widehat C_* \to D'_*$
  and the accuracy~$2 \cdot N \cdot \delta$.
  This is possible because~$\max\{\kappa_n(\widehat{C}_*), \nusum_n(\widehat{C}_*)\} < \kappa$, $\kappa_n(D'_*) < \kappa$, 
  and $\kappa_n(F''_*)\leq \kappa$. 
  Hence, we obtain an $S$-adapted marked projective
  chain complex~$(\widehat D_*,\widehat \eta)$ (up to degree~$n+1$)
  and an $S$-adapted $R$-chain map~$\widehat F_* \colon \widehat C_*
  \to \widehat D_*$ (up to degree~$n+1$) extending~$\id_{\LinftyX}$
  with
  \begin{align*}
    \dgh{K'} {\widehat D_*, D'_*,n}
    & < K' \cdot 2 \cdot N \cdot \delta
    \\
    \fa{r \in \{0,\dots,n+1\}}
    \dgh{K'} {\widehat F_r, F''_r}
    & < K' \cdot 2 \cdot N \cdot \delta.
  \end{align*}
  Moreover, we have (Proposition~\ref{prop:dghchaincomplex} and Proposition~\ref{prop:dGH})
  \begin{align*} \dgh {K+K'} {\widehat D_*, D_*,n}
  & < \delta + K' \cdot 2\cdot N \cdot \delta
  = (1 + K' \cdot 2\cdot N) \cdot \delta
  \\
  \fa{r \in \{0,\dots, n+1\}}
  \dgh{\|F_r\| + 2\cdot K_{F_r} + K'} {\widehat F_r, F_r}
  & < (N + N + K' \cdot 2 \cdot N) \cdot \delta.
  \end{align*}
  Unifying the constants and rescaling~$\delta$ beforehand, gives the claimed result.
\end{proof}

\begin{thm}\label{thm:adaptedembgen}
  In the situation of Setup~\ref{setup:adapt},
  let $C_*$ be a free $Z\Gamma$-resolution of~$Z$ that has
  finite rank in degrees~$\leq n+1$ and let
  $f_* \colon C_* \to D_*$ be an $\alpha$-embedding (Definition~\ref{def:alpha-emb}).  
  Then there exists a~$K\in \R_{>0}$ such that: For every~$\delta \in \R_{>0}$,
  there exists an $S$-adapted $\alpha$-embedding~$\widehat f_* \colon C_* \to \widehat D_*$
  with
  \[ \dgh K {\widehat D_*, D_*,n} < \delta
  \qand
  \fa{r \in \{0,\dots, n+1\}}
  \dgh K {\Ind {Z\Gamma} R (\widehat f_r), \Ind {Z\Gamma} R (f_r)} < \delta 
  .
  \]
\end{thm}

\begin{proof}
	Let $(\widehat C_*,\widehat \zeta)$ be the result of applying
  the functor~$\Ind \Gamma R$ (Remark~\ref{rem:fromZGtoZR})
  to the given resolution~$(C_*,\zeta)$ and let $F_* \colon \widehat C_*
  \to D_*$ be the $R$-chain map (up to degree~$n+1$) induced by~$f_*$.
  	Theorem~\ref{thm:adaptedchmap} yields an $S$-adapted $R$-chain map $\widehat{F}_*\colon \widehat{C}_*\to \widehat{D}_*$.
    In particular, the restriction~$\widehat f_* \colon C_* \to \widehat D_*$
  of~$\widehat F_*$ to the ``$Z\Gamma$-subcomplex''~$C_*$ of~$\widehat C_*$
  is an $S$-adapted $\alpha$-embedding.
\end{proof}

Our main application will be in Section~\ref{sec:passing} to
approximate embeddings over profinite actions by chain complexes
related to individual finite index subgroups.

Moreover, for invariants with controlled behaviour with respect to the
Gromov--Hausdorff distance, there is no difference between considering
embeddings with target complexes over the equivalence relation ring or
the crossed product ring. In particular, this applies to the measured
embedding dimension and the measured embedding volume
(Section~\ref{sec:overZR}). 

\section{A logarithmic norm for morphisms}\label{sec:lognorm}

We introduce a refinement of the quantity~$\dim (N) \cdot
\logp \|f\|$ for homomorphisms $f \colon M \to N$ between marked
projective modules called~$\lognorm (f)$. 
The key properties of this invariant are that it
satisfies dimension estimates, subadditivity over marked decompositions
of the domain, compatibility with almost equality
(Section~\ref{subsec:lognormproperties}), compatibility with
adaptedness (Proposition~\ref{prop:lognormadapt}), and that it
provides an upper bound for the logarithmic torsion of cokernels
(Theorem~\ref{thm:logtorlognorm}). The ad-hoc construction is
given in Section~\ref{subsec:lognormdef}.

\subsection{Construction}\label{subsec:lognormdef}

We refine the expression ``$\dim \cdot \logp \|\cdot\|$'' by
allowing for marked decompositions of the domain and for
taking marked ranks of images. 

\begin{setup}\label{setup:lognorm}
  Let $\Gamma$ be a countable group and let $Z$ denote $\Z$ (with
  the usual norm) or a finite field (with the trivial norm).
  We consider a standard $\Gamma$-action~$\alpha \colon \Gamma \actson (X,\mu)$. 
  Moreover, let $\Rrel$ be the associated orbit relation and
  let $R \subset Z \Rrel$ be a subring that
  contains~$\LinftyXZ * \Gamma$.
\end{setup}

\begin{defn}[$\lognorm$]
  \label{def:lognorm}
  In the situation of Setup~\ref{setup:lognorm}, let $f \colon M \to N$
  be an $R$-homomorphism between marked projective $R$-modules.
  \begin{itemize}
  \item
    The \emph{marked rank of~$f$} is defined as
    \begin{align*}
      \rk (f)
      \coloneqq
      & \;
      \inf
      \bigl\{ \dim (N')
      \bigm|
      \text{$N' \subset N$ is a marked direct summand with~$f(M) \subset N'$}
      \bigr\}
      \\
      \in
      & \;
      [0,\dim (N)].
    \end{align*}
  \item
    We set
    \[ \lognorm' (f)
    \coloneqq \min \bigl\{
    \dim (M) \cdot \logp \|f\|
    , \rk(f) \cdot \logp \|f\|
    \bigr\} \in \IR_{\ge 0}.
    \]
  \item
    Let $D(M)$ denote the ``set'' of all finite marked decompositions
    of~$M$. For $(M_i)_{i \in I} \in D(M)$, we set
    \[
    \lognorm' (f,(M_i)_{i \in I})
    \coloneqq \sum_{i \in I} \lognorm' (f|_{M_i} \colon M_i \to N) \in \IR_{\ge 0}.
    \]
  \item
    Finally, we let
    \[ \lognorm (f)
    \coloneqq \inf_{M_* \in D(M)}
    \lognorm' (f, M_*)
    \in \R_{\geq 0}.
    \]
  \end{itemize}
\end{defn}

The value~$\lognorm (f)$ depends on the marked structure.
We can adapt the definition to subalgebras as follows.

\begin{rem}[adapted lognorm]
	\label{rem:lognormS}
	Let $S\subseteq R$ be a subalgebra and
	 $f \colon M \to N$
  	be an $S$-adapted homomorphism between $S$-adapted marked projective $R$-modules.
  	 We define $D_S(M)$ as the ``set'' of all finite marked $S$-adapted decompositions
    of~$M$ and 
    \[
  	\lognorm_S (f) \coloneqq \inf_{M_*\in D_S(M)} \lognorm'(f,M_*) \in \R_{\ge 0}.
  \]
\end{rem}

\subsection{Basic properties}\label{subsec:lognormproperties}

We collect some basic properties of~$\lognorm$.

\begin{prop}[$\lognorm$ properties]\label{prop:lognorm}
  In the situation of Setup~\ref{setup:lognorm}, let $f \colon M \to N$
  be an $R$-homomorphism between marked projective $R$-modules.
  Then the following hold:
  \begin{enumerate}[label=\enum]
  \item \emph{Dimension estimates.}
    We have
    \[ \lognorm (f) \leq \dim (M) \cdot \logp \|f\|
    \qand
    \lognorm (f) \leq \dim (N) \cdot \logp \|f\|.
    \]
  \item \emph{Subadditivity.}
    If $M \cong_R M_0 \oplus M_1$ is a marked decomposition, then
    \[ \lognorm (f)
    \leq \lognorm (f|_{M_0}) + \lognorm (f|_{M_1}). 
    \]
  \item \emph{Marked inclusion estimates.}
    If $i \colon M' \hookrightarrow M$ and $j \colon N \hookrightarrow N'$ are
    marked inclusions of marked projective $R$-modules, then
    \[ \lognorm (j\circ f) = \lognorm (f)
    \qand
    \lognorm (f \circ i) \leq \lognorm (f).
    \]
  \item
  \label{itm:lognorm-almeq}
  \emph{Almost equality.}
    If $\delta \in \R_{>0}$, $K \in \R_{\geq 0}$, and $g \colon M \to N$
    is an $R$-homomorphism with~$f =_{\delta,K} g$, then
    \[ \lognorm (f) \leq \lognorm (g) + \delta \cdot \logp K.
    \]
  \item 
	\label{itm:lognorm-GH}  
  \emph{Gromov--Hausdorff distance.}
    If $\delta, K \in \R_{>0}$, and $f' \colon M' \to N'$ is an $R$-homomorphism
    of marked projective $R$-modules with~$\dgh K {f,f'} < \delta$,
    then
    \[ \lognorm (f) \leq \lognorm (f') + \delta \cdot \log_+ K.
    \]
  \end{enumerate}
\end{prop}
\begin{proof}
  (i)
  The trivial decomposition of~$M$ into the single summand~$M$
  gives both dimension estimates.
  
  (ii)
  Combining marked decompositions of $M_0$ and~$M_1$ results
  in marked decompositions of~$M$. This leads to subadditivity.  

  (iii)
  For every marked decomposition~$M_* \in D(M)$, we
  see easily that
  \[ \lognorm' (j \circ f,M_*)
  = \lognorm' (f, M_*),
  \]
  because $j$ is a marked inclusion.  Taking the infimum over all~$M_*$ in~$D(M)$
  thus shows that~$\lognorm (j \circ f) = \lognorm (f)$.

  Every marked decomposition~$(M_k)_{k \in I} \in D(M)$ induces a
  marked decomposition~$(M'_k)_{k \in I} \in D(M')$
  such that for every~$k \in I$, the restriction~$i|_{M'_k} \colon M'_k \to M_k$
  is a marked inclusion; we have
  \[ \lognorm' (f \circ i|_{M'_k})
  \leq \lognorm' (f|_{M_k}).
  \]
  Therefore, $\lognorm' (f \circ i, M'_*) \leq \lognorm' (f,M_*)$.
  Taking the infimum over~$D(M)$ shows that $\lognorm (f \circ i) \leq
  \lognorm (f)$.
  
  (iv)
  This is a direct consequence of parts~(i)--(iii).

  (v)
  By part~(iv) and the definition of Gromov--Hausdorff distance, it suffices to show that
  $\lognorm (\psi\circ f\circ \pi_\varphi) = \lognorm (f)$, if $\varphi\colon M\to L$, $\psi\colon N\to P$
  are marked inclusions and~$\pi_\varphi\colon L\to M$ is the marked projection associated to~$\varphi$.
  By part~(iii), we are left to show that $\lognorm (f\circ \pi_\varphi) = \lognorm (f)$. 
  Since~$\pi_\varphi\circ\varphi = \id_M$, part~(iii) yields~$\lognorm (f) \le \lognorm (f\circ\pi_\varphi)$. 
  Conversely, part~(ii) yields
  \begin{align*}
  	\lognorm (f\circ\pi_\varphi) &\le 
  	\lognorm (f\circ \pi_\varphi|_{\varphi(M)}) 
  	= \lognorm f.
  \end{align*}    
  This finishes the proof.
\end{proof}

For a subgroup~$\Lambda$ of~$\Gamma$, we denote the restricted action by $\alpha|_\Lambda\colon \Lambda\actson (X,\mu)$.
The inclusion of crossed product rings $\LinftyXLambda\ast \Lambda\into \LinftyX\ast \Gamma$ induces induction and restriction functors~$\ind_{\LinftyXLambda\ast \Lambda}^{\LinftyX\ast \Gamma}$ and~$\res^{\LinftyX\ast \Gamma}_{\LinftyXLambda\ast \Lambda}$ between module categories.

\begin{lem}
\label{lem:lognorm_ind_res}
	In the situation of Setup~\ref{setup:lognorm}, let~$\Lambda$ be a subgroup of~$\Gamma$.
	\begin{enumerate}[label=\enum]
		\item\label{item:lognorm_ind}
		Let~$g$ be a map between marked projective $\LinftyXLambda\ast \Lambda$-modules.
		Then
		\[
			\lognorm\bigl(\ind_{\LinftyXLambda\ast \Lambda}^{\LinftyX\ast \Gamma}g\bigr)\le \lognorm(g).
		\]
		\item\label{item:lognorm_res} 
		Suppose that~$\Lambda$ has finite index in~$\Gamma$. 
		Let $f$ be a map between marked projective $\LinftyX\ast \Gamma$-modules.
		Then
		\[
			\lognorm\bigl(\res^{\LinftyX\ast \Gamma}_{\LinftyXLambda\ast \Lambda} f\bigr)\le [\Gamma:\Lambda]\cdot \lognorm(f).
		\]
	\end{enumerate}
	\begin{proof}
		(ii) Let~$\varphi\colon M\to N$ be a map between marked projective $\LinftyX\ast \Gamma$-modules.
		Then
		\begin{align*}
			\dim_{\LinftyXLambda\ast \Lambda}\bigl(\res^{\LinftyX\ast \Gamma}_{\LinftyXLambda\ast \Lambda} M\bigr)
			&\le [\Gamma:\Lambda]\cdot \dim_{\LinftyX\ast \Gamma}(M)
			\\
			\rk_{\LinftyXLambda\ast \Lambda}\bigl(\res^{\LinftyX\ast \Gamma}_{\LinftyXLambda\ast \Lambda}\varphi\bigr)
			&\le [\Gamma:\Lambda]\cdot \rk_{\LinftyX\ast \Gamma}(\varphi)
			\\
			\bigl\|\res^{\LinftyX\ast \Gamma}_{\LinftyXLambda\ast \Lambda} \varphi\bigr\|
			&\le \|\varphi\|.
		\end{align*}
		The claim follows, since a marked decomposition of~$M$ induces a marked decomposition of $\res^{\LinftyX\ast \Gamma}_{\LinftyXLambda\ast \Lambda} M$.
		
		Part~(i) is proved similarly.
	\end{proof}
\end{lem}

\subsection{Explicit description of the marked rank}
There is the following explicit description of the marked rank.

\begin{lem}[]
	\label{lem:mrk-expl}
	In the situation of Setup~\ref{setup:lognorm}, let $R=\LinftyX\ast \Gamma$.
	Let~$f\colon M\to N$ be an $(\LinftyX\ast \Gamma)$-homomorphism between marked
	 projective~$(\LinftyX\ast \Gamma)$-modules. 
	Fix a presentation of~$f$ as in Setup~\ref{setup:opnormalt}. Then, we have
	\[
		\mrk (f) = \sum_{j\in J} \mu \biggl(
			\bigcup_{\substack{(i, k,\gamma)\in I\times K\times F,\\ a_{i,j,k,\gamma} \neq 0}}
			\gamma^{-1} A_i \cap U_k
		\biggr).
	\]
\end{lem}

\begin{proof}
	Since the marked rank is defined via marked decompositions, we can assume without loss of generality that~
	$N = \gen{B}$. We will thus drop~$j\in J$ from the notation. Let
	\[
		B' \coloneqq \bigcup_{\substack{(i,k,\gamma)\in I\times K\times F,\\ a_{i,k,\gamma} \neq 0}}
			\gamma^{-1} A_i \cap U_k.
	\]
	It suffices to show that~$\gen{B'}$ is the smallest marked projective summand of~$N$ containing~$f(M)$.
	
	We first show that the image is contained in this summand. We denote by $\pi_{B\setminus B'}$
	the canonical projection to the marked summand~$\gen{B\setminus B'}$. For all~$i\in I$, we have that
	~$\pi_{B\setminus B'}\circ f|_{\gen{A_i}}$ is given by right multiplication with the element
	\begin{align*}
		&(\chi_{A_i}, 1) \cdot \sum_{\substack{(k,\gamma)\in K\times F,\\ a_{i,k,\gamma} \neq 0}} a_{i,k,\gamma} \cdot (\chi_{\gamma U_k}, \gamma)
		\cdot (\chi_{B\setminus B'}, 1)\\
		&= 
		\sum_{\substack{(k,\gamma)\in K\times F,\\ a_{i,k,\gamma} \neq 0}} a_{i,k,\gamma} \cdot (\chi_{A_i\cap \gamma U_k\cap \gamma B\setminus 
		\gamma B'}, \gamma)\\
		&=  0,
	\end{align*}
	because 
	\begin{align*}
	A_i\cap\gamma U_k\cap \gamma B\setminus \gamma B' &\subseteq \gamma (\gamma^{-1}A_i\cap U_k\setminus B')\\
	&\subseteq \gamma\bigl(\gamma^{-1}A_i\cap U_k\setminus (\gamma^{-1} A_i\cap U_k )\bigr)
	= \emptyset.
	\end{align*}
	This shows that~$f(M) \subseteq \gen{B'}$.
	
	Conversely, let~$(i_0, k_0,\gamma_0)\in I\times K\times F$ such that~$a_{i_0,k_0,\gamma_0}\neq 0$. It suffices to 
	show that~$\gen{\gamma_0^{-1}A_{i_0}\cap U_{k_0}}$ is contained in every marked projective summand
	containing~$f(M)$.
	Indeed, let~$\pi_1\colon \LinftyX\ast \Gamma \to \LinftyX$ be the~$\LinftyX$-linear projection to the summand
	indexed by~$1\in \Gamma$. We have
	\begin{align*}
		&\pi_1\circ f\bigl((\chi_{U_{k_0}}, \gamma_0^{-1})\cdot (\chi_{A_{i_0}},1)\bigr)\\
		&= \pi_1\circ f(\chi_{U_{k_0}\cap \gamma_0^{-1} A_{i_0}}, \gamma_0^{-1})\\
		&= \pi_1 \biggl(
			(\chi_{U_{k_0}\cap \gamma_0^{-1} A_{i_0}}, \gamma_0^{-1}) \cdot \sum_{(k,\gamma)\in K\times F}
			a_{i_0,k,\gamma} \cdot (\chi_{\gamma U_k},\gamma)
		\biggr)\\
		&= \pi_1 \biggl(
			 \sum_{(k,\gamma)\in K\times F}
			a_{i_0,k,\gamma} \cdot (\chi_{U_{k_0}\cap \gamma_0^{-1} A_{i_0}\cap \gamma_0^{-1}\gamma U_k},\gamma_0^{-1}\gamma)
			\biggr)\\
		&= \sum_{k\in K}
			a_{i_0,k,\gamma_0} \cdot \chi_{U_{k_0}\cap  \gamma_0^{-1} A_{i_0} \cap U_k}
			& (\text{projection})\\
		&= a_{i_0,k_0,\gamma_0} \cdot \chi_{U_{k_0}\cap \gamma_0^{-1} A_{i_0}}
			& ((U_k)_k \text{ pairwise disjoint})
	\end{align*}
	Since~$a_{i_0,k_0,\gamma_0}\neq 0$, we obtain from Lemma~\ref{lem:recog-summands} below
	that~$\gen{U_{k_0}\cap \gamma_0^{-1} A_{i_0}}$ is contained 
	in the marked direct summand generated by~$f(M)$.
\end{proof}

\begin{lem}[recognising marked summands]
	\label{lem:recog-summands}
	Let~$M= \gen{B}$ be a marked summand of~$\LinftyX\ast \Gamma$ and let~$x\in M$. Let~$\pi_1\colon \LinftyX\ast \Gamma\to \LinftyX$
	be the~$\LinftyX$-linear projection to the summand corresponding to the identity element~$1\in \Gamma$.
	Assume that~$a\in Z, a\neq 0$ and~$A\subseteq X$ such that~$\pi_1(x) = a\cdot\chi_A$.
	Then,~$\gen A \subseteq M$.
\end{lem}

\begin{proof} 
	It suffices to show that~$A\setminus B$ is a null-set. Because~$x\in M = (\LinftyX\ast \Gamma)\cdot (\chi_B,1)$,
	we can fix a reduced presentation (see Setup~\ref{setup:opnormalt}) for $x$, i.e., $x = \lambda\cdot (\chi_B,1)$ for some
	\[
		\lambda = \sum_{(k,\gamma)\in K\times F} a_{(k,\gamma)}\cdot (\chi_{\gamma U_k}, \gamma),
	\]
	with $K$ and~$F\subseteq \Gamma$ finite sets, $a_{(k,\gamma)}\in Z$, and~$U_k\subseteq X$. 
	Then, we have
	\begin{align*}
		a\cdot\chi_{A\setminus B} &= \chi_{A\setminus B} \cdot a\cdot\chi_A
		\\
		&= \chi_{A\setminus B}\cdot\pi_1( x)\\
		&= \pi_1(\chi_{A\setminus B}\cdot x)\\
		&= \pi_1\biggl(
			(\chi_{A\setminus B},1)\cdot
				\sum_{(k,\gamma)\in K\times F} a_{(k,\gamma)}\cdot (\chi_{\gamma U_k}, \gamma)
			\cdot (\chi_B,1)
		\biggr)\\
		&= \sum_{k\in K} a_{(k,1)}\cdot \chi_{A\setminus B}\cdot \chi_{U_k}\cdot \chi_B
		\\
		&= 0.
	\end{align*}
	Thus, we have~$a\cdot \chi_{A\setminus B} = 0$ almost everywhere with~$a\neq 0$, which is only possible if~$\mu(A\setminus B) = 0$.
\end{proof}

\begin{lem}\label{lem:homrkN1}
  In the situation of Setup~\ref{setup:lognorm}, let $R = \LinftyX\ast \Gamma$. 
  Let $f \colon M \to N$
  be an $(\LinftyX\ast \Gamma)$-homomorphism between marked projective $(\LinftyX\ast \Gamma)$-modules. 
  Fix a presentation of $f$ as in Setup~\ref{setup:opnormalt}.
  Then, in the notation of Setup~\ref{setup:opnormalt}, we have
  \[ \rk (f) \leq \dim (M) \cdot  \# I \cdot \# J \cdot \# K \cdot \# F.
  \]
\end{lem}

\begin{proof}
	Because $\rk (f)$ is subadditive, we can assume without loss of generality that 
	$f\colon \gen A \to \gen B$  is given by right multiplication with
	\[
		z\coloneqq a\cdot (\chi_{\gamma U}, \gamma),
	\]
	where $a\in Z$, $\gamma\in \Gamma$, and~$U\subseteq X$. We have 
	\[
		(\chi_A, 1)\cdot z = a \cdot (\chi_{\gamma (\gamma^{-1} A \cap U)}, \gamma) \in 
		\gen{\gamma^{-1} A \cap U}
	\]
	and thus,
	\[
		\rk (f) \le \mu \bigl(
			\gamma^{-1} A \cap U
		\bigr)
		\le \mu(A) = \dim \gen A. \qedhere
	\]
\end{proof}

\subsection{The logarithmic norm of adapted homomorphisms}

The logarithmic norm of adapted homomorphism can be computed through 
adapted decompositions:
	
\begin{prop}[$\lognorm$, adapted homomorphisms]\label{prop:lognormadapt}
	In the situation of Setup~\ref{setup:lognorm}, let $R=\LinftyX\ast \Gamma$.
  Let~$S$ be a dense subalgebra of the set of all measurable
  subsets of~$X$. 
  Let $f \colon M \to
  N$ be an $S$-adapted homomorphism between $S$-adapted marked
  projective $(\LinftyX\ast \Gamma)$-modules.  Then
  \[ \lognorm (f) = \lognorm_S (f).
  \]
\end{prop}
\begin{proof}
	First, we record the following observation:
	Fix a presentation of~$f$ as in Setup~\ref{setup:opnormalt} with $c\in \N$ summands. 
	 For every marked summand~$W$ of $M$, we obtain 
	a presentation of~$f|_W$ with the same number of summands. Thus, Lemma~\ref{lem:homrkN1} yields that
	\begin{equation}
		\label{eq:homrkN1}
		\rk (f|_W) \le c\cdot \dim (W).
	\end{equation}

  Clearly, $\lognorm (f) \leq \lognorm_S (f)$. For the converse
  estimate, let $M = \bigoplus_{i\in I} M_i$ be a marked 
  decomposition of~$M$ and let~$\varepsilon >0$. Without loss of generality,
  we may assume~$0\not\in I$.
  It suffices to show that there exists an $S$-adapted marked decomposition~$(W_i)_{i\in I\cup\{0\}}$
  such that 
  \[
  		\lognorm'\bigr(f, (W_i)_{i\in I\cup \{0\}}\bigr) \le 
  		\lognorm'\bigr(f, (M_i)_{i\in I}\bigr) + \varepsilon.
  \]
  Because $S$ is dense, for every $i\in I$, there exists an
  $S$-adapted marked summand~$U_i$ such that~$\mu(M_i \msymmdiff U_i) 
  <\epsilon$. Without loss of generality, we can assume that the $(U_i)_{i\in I}$ intersect pairwise trivially.
  
  By Corollary~\ref{cor:adaptsamenorm}, for every $i\in I$,
  there is an $S$-adapted marked summand~$V_i$ containing~$M_i$ such that
  \[
  \|f|_{M_i}\| = \|f|_{V_i}\|.
  \]
  For~$i \in I$, we set
  \begin{align*}
  		W_i &\coloneqq U_i \cap  V_i\\
  		W_0 & \coloneqq M \mcompl \bigoplus_{i\in I} W_i,
  \end{align*}
  where~$\mcompl$ denotes the ``marked complement''.
  By construction, the $(W_i)_{i\in I\cup\{0\}}$ 
  form an~$S$-adapted marked decomposition of~$M$.
  For all~$i\in I$, we have
  \begin{align*}
  		\dim (W_i) &\le \dim (M_i) + \varepsilon\\ 
  		\dim (W_0) &\le \# I \cdot \varepsilon\\ 
  		\|f|_{W_i}\| &\le \|f|_{M_i}\|\\ 
  		\|f|_{W_0}\| &\le \|f\|\\  
  		\rk (f|_{W_i}) &\le \rk (f|_{M_i}) + \rk (f|_{W_i\mcompl M_i})\\ 
  		&\le \rk (f|_{M_i}) +c \cdot \dim (W_i \mcompl M_i)
                & \text{(Estimate~\eqref{eq:homrkN1})}\\
  							&\le \rk (f|_{M_i}) +c\cdot\varepsilon. 
  \end{align*}
  Therefore, we obtain
  \begin{align*}
    &\lognorm'\bigl(f, (W_i)_{i\in I\cup\{0\}}) 
  		\\
		&= \sum_{i\in I} \logp \|f|_{W_i}\| \cdot \min \bigl\{\dim W_i, \rk (f|_{W_i})\bigr\}
  		+ \logp \|f|_{W_0}\|\cdot \min \bigl\{\dim W_0, \rk (f|_{W_0})\bigr\} \\
  		&\le 	\sum_{i\in I} \logp \|f|_{M_i}\| \cdot \min \bigl\{\dim M_i+\varepsilon, \rk (f|_{M_i})+c\cdot
  		\varepsilon \bigr\}
  		+ \logp \|f\|\cdot \# I \cdot \varepsilon \\
  		&\le \sum_{i\in I} \logp \|f|_{M_i}\| \cdot \min \bigl\{ \dim M_i, \rk (f|_{M_i})\bigr\}
  		 + \# I \cdot \logp \|f\| \cdot \varepsilon \cdot (c+1)\\
    &= \lognorm\bigl(f, (M_i)_{i\in I}\bigr) + \varepsilon
    \cdot \bigl(1 + \# I \cdot \logp \|f\| \cdot (c + 1)\bigr). 
  \end{align*}
  Rescaling $\varepsilon$ appropriately proves the claim.
\end{proof}

\section{Passing to finite index subgroups}\label{sec:passing}

We explain the passage from the dynamical view to finite index
subgroups. More precisely, for dynamical systems~$\Gamma \actson
\widehat \Gamma_*$ coming from systems~$\Gamma_*$ of finite index
normal subgroups, we reinterpret adapted modules and morphisms over
the crossed product ring as $Z[\Gamma/\Gamma_i] * \Gamma$-modules for large
enough~$i$ (where the multiplication on~$Z[\Gamma/\Gamma_i]$ is
pointwise multiplication of functions~$\Gamma/\Gamma_i \to Z$).

Adapting dynamical embeddings then leads to homotopy retracts
at the level of finite index subgroups and thus eventually to
the homological gradient bounds.

\begin{setup}\label{setup:pfc}
  Let $n \in \N$, 
  let $\Gamma$ be a residually finite group of type~FP$_{n+1}$,
  let $(\Gamma_i)_{i \in I}$ be a directed system of finite index
  normal subgroups of~$\Gamma$ with~$\bigcap_{i \in I} \Gamma_i = 1$,
  let $\alpha\colon \Gamma \actson \widehat \Gamma_*$ be the
  associated dynamical system,
  and let $Z$ be~$\Z$ (with the usual norm) or a finite field
  (with the discrete norm). We write~$(X,\mu)$
  for the probability space~$\widehat \Gamma_*$.
\end{setup}

\begin{setup}\label{setup:passage}
  In the situation of Setup~\ref{setup:pfc}, for each~$i \in I$, we denote by
  \[
S_i \coloneqq \bigl\{ \pi_i^{-1}(A) \bigm| A \subset \Gamma/\Gamma_i
  \bigr\}
  \]
  the \emph{cylindrical sets} in the $\sigma$-algebra of~$X = \widehat \Gamma_*$ corresponding to 
  the canonical projection map~$\pi_i \colon X \to \Gamma/\Gamma_i$. 
  Let $S$ be the union of all the~$S_i$, namely the subalgebra given by
  \[ S
  \coloneqq \bigcup_{i \in I} S_i
  = 
  \bigl\{ \pi_i^{-1}(A) \bigm| i \in I,\ A \subset \Gamma/\Gamma_i
  \bigr\}.
  \]
  We will use ``$\Gamma_*$-adapted'' as a synonym for ``$S$-adapted'' to
  emphasise the origin of the subalgebra~$S$.
  We write 
  $L_i$ for the subring of~$\linf {\alpha,Z}$ generated by~$S_i$.
  
  Furthermore, we abbreviate~$R \coloneqq \linf{\alpha,Z} * \Gamma$ and
  $R_i \coloneqq L_i * \Gamma$.
\end{setup}

Then $S$ is $\mu$-dense in the set of all measurable subsets of~$X$
and $\Gamma \cdot S \subset S$~\cite[Lemma~6.4.2]{loeh2020ergodic}. 
Therefore, the deformation arguments from the previous sections apply.

\subsection{Discretisation}

Let $i \in I$. We first establish a correspondence
between~$L_i$ and~$Z[\Gamma/\Gamma_i]$. The projection~$\pi_i \colon
X \to \Gamma/\Gamma_i$ induces mutually inverse $Z\Gamma$-iso\-morphisms
\begin{align*}
  \pi_i^* \colon Z[\Gamma/\Gamma_i]
  \to L_i,
  \quad
  & \text{given by
  $\gamma \cdot \Gamma_i
  \mapsto \chi_{\pi_i^{-1}(\gamma \cdot \Gamma_i)}$}
  \\
  \pi_{i,*} \colon L_i
  \to Z [\Gamma/\Gamma_i],
  \quad
  & \text{given by 
  $\chi_A
  \mapsto
  \chi_{\pi_i(A)}$ for~$A \in S_i$}.
\end{align*}
In this sense, we may view objects and morphisms over~$L_i$ (and
whence those over~$R_i$) as ``discrete''. Under the above $Z\Gamma$-isomorphism,
the multiplication on~$L_i$ translates into pointwise multiplication on~$Z[\Gamma/\Gamma_i]$
of functions~$\Gamma/\Gamma_i \to Z$. 
Moreover, we will use the following ``pre-induction'' construction:

\begin{defn}
  In the situation of Setup~\ref{setup:passage}, let $i \in I$.
  \begin{itemize}
  \item Let $M = \bigoplus_{j \in J} \gen{A_j}$ be a marked projective $R$-module
    that is adapted to~$S_i$. Then, we set
    \[ M(i) \coloneqq \bigoplus_{j \in J} \gen{A_j}_{R_i}
    \coloneqq \bigoplus_{j \in J} R_i \cdot (\chi_{A_j}, 1),
    \]
    which is a marked projective $R_i$-module (and $R \otimes_{R_i} M(i) \cong_R M$).
  \item If $f \colon M\to N$ is an $S_i$-adapted $R$-homomorphism
    between $S_i$-adapted marked projective $R$-modules, then we
    write~$f(i) \colon M(i) \to N(i)$ for the corresponding
    $R_i$-homomorphism (defined by the same matrix).
  \end{itemize}
\end{defn}

If $f$ and $g$ are $S_i$-adapted and composable, then $(g \circ f)(i)
= g(i) \circ f(i)$.  Consequently, if $(D_*,\eta)$ is a marked
projective $R$-chain complex that is adapted to~$S_i$, we naturally
obtain a corresponding $R_i$-chain complex~$(D_*(i), \eta(i))$, which
augments to~$L_i$. Adapted chain maps translate into corresponding
$R_i$-chain maps.

\subsection{Dimensions and norms}

This discretisation is compatible with taking
norms and dimensions:

\begin{rem}[compatibility of dimensions]\label{rem:dimcomp}
  In the situation of Setup~\ref{setup:passage}, let $i \in I$ 
  and let $M$ be a marked projective $R$-module
  that is $S_i$-adapted.
  Then the module~$M(i)_\Gamma$ of coinvariants is a free $Z$-module and 
  \[ \rk_Z \bigl(M(i)_\Gamma\bigr)
  = [\Gamma : \Gamma_i] \cdot \dim (M).
  \]
  Indeed, using compatibility with marked decompositions,
  it suffices to consider the case that $M= \gen A$ with~$A \in S_i$,
  say~$A = \pi_i^{-1} (A_i)$ with~$A_i \subset \Gamma /\Gamma_i$.
  Then, the map
  \[ e_{\gamma\Gamma_i}
  \mapsto  1 \otimes (\chi_{\pi_i^{-1}(\gamma\Gamma_i)},1) \cdot (\chi_A,1)
  \]
  induces a $Z$-isomorphism
  \[ \bigoplus_{A_i} Z
  \cong_Z L_i \cdot \chi_A
  \cong_Z Z \otimes_{Z \Gamma} (L_i * \Gamma) \cdot (\chi_A, 1)
  = M(i)_\Gamma.
  \]
  Therefore, 
  \[ \rk_Z \bigl(M(i)_\Gamma\bigr)
  = \# A_i
  = [\Gamma : \Gamma_i] \cdot \frac{1}{[\Gamma : \Gamma_i]} \cdot \# A_i
  = [\Gamma : \Gamma_i] \cdot \mu(A)
  = [\Gamma : \Gamma_i] \cdot \dim (M).
  \]
\end{rem}

\begin{rem}[compatibility of norms]\label{rem:normcomp}
  In the situation of Setup~\ref{setup:passage}, let $i \in I$
  and let $f \colon M \to N$ be an $S_i$-adapted $R$-homomorphism
  between $S_i$-adapted marked projective $R$-modules.
  Then
  \[ \bigl\| f(i)_\Gamma \bigr\| \leq \|f\|,
  \]
  where the operator norm on the left hand side is taken with respect
  to the $\ell^1$-norms induced by the canonical $Z$-bases on~$M(i)_\Gamma$
  and $N(i)_\Gamma$ (Remark~\ref{rem:dimcomp}).  Indeed, for
  notational simplicity let us consider the case of~$f \colon \gen A
  \to \gen B$ with~$A, B \in S_i$, given by right multiplication by $$z
  = \sum_{\lambda \in \Gamma} \sum_{\gamma\Gamma_i \in
    \Gamma/\Gamma_i} a_{\gamma\Gamma_i, \lambda} \cdot
  (\chi_{\pi_i^{-1} (\gamma\Gamma_i)}, \lambda).$$ 
  Without loss of generality, we may assume that
  $a_{\gamma\Gamma_i,\lambda} = 0$ whenever $\lambda^{-1}\gamma
  \Gamma_i \not \in \pi_i(B)$.
  Using the canonical $Z$-isomorphisms~$M(i)_\Gamma \cong_Z \bigoplus_{\pi_i(A)} Z$
  and $N(i)_\Gamma \cong_\Z \bigoplus_{\pi_i(B)} Z$ from Remark~\ref{rem:dimcomp},
  we obtain for all~$\gamma\Gamma_i \in \pi_i(A)$ that
  \begin{align*}
    \bigl\| f(i)_\Gamma (e_{\gamma\Gamma_i}) \bigr\|_1
    & = \biggl\| \sum_{\lambda \in \Gamma} a_{\gamma\Gamma_i,\lambda} \cdot e_{\lambda^{-1}\gamma\Gamma_i} \biggr\|_1
    = \sum_{\lambda \in \Gamma} |a_{\gamma\Gamma_i,\lambda}|
    \\
    & = [\Gamma:\Gamma_i] \cdot \bigl\| (\chi_{\pi_i^{-1}(\gamma\Gamma_i)},1) \cdot z \cdot (\chi_B, 1)\bigl\|_1
    \\
    & \leq [\Gamma:\Gamma_i] \cdot \|f\| \cdot \bigl\| (\chi_{\pi_i^{-1} (\gamma\Gamma_i)},1) \cdot (\chi_A, 1)\bigr\|_1
    \\
    & = [\Gamma:\Gamma_i] \cdot \|f\| \cdot \frac1{[\Gamma:\Gamma_i]}
     = \|f\|.
  \end{align*}
  Hence, $\|f(i)_\Gamma\| \leq \|f\|$.
\end{rem}

\subsection{From adapted embeddings to homology retracts}

The following theorem shows that adapted embeddings allow to construct suitable homology retracts
(and so estimates on the dimensions and the torsions).

\begin{thm}\label{thm:adaptedretract}
  In the situation of Setup~\ref{setup:pfc},
  let $C_*$ be a free $Z\Gamma$-resolution of~$Z$ that has
  finite rank in degrees~$\leq n+1$ and let
  $f_* \colon C_* \to D_*$ be a $\Gamma_*$-adapted 
  $\alpha$-embedding. Then, for all large enough~$i \in I$:
  \begin{enumerate}[label=\enum]
  \item
    The complex~$D_*(i)$ is defined up to degree~$n+1$ and
    the $Z$-module~$H_n(\Gamma_i;Z)$ is a $Z$-retract of~$H_n(D_*(i)_\Gamma)$.
  \item
    In particular, we have
    \begin{align*}
      \rk_Z H_n(\Gamma_i;Z)
      \leq [\Gamma : \Gamma_i] \cdot \dim (D_n)
    \end{align*}
    and, in the case of~$Z = \Z$, we additionally obtain
    \begin{align*}
      \log \# \tors H_n(\Gamma_i;\Z)
       & \leq \log \# \tors \bigl(D_n(i)_\Gamma / \im \partial^{D}_{n+1}(i)_\Gamma\bigr).
    \end{align*}
  \end{enumerate}
  \end{thm}
\begin{proof}
  Let $\delta \in \R_{>0}$ and let $f_* \colon C_* \to
  D_*$ be such a $\Gamma_*$-adapted $\alpha$-embedding. 
  Let $I' \subset I$ be the set of all~$i \in I$ such that all of the
  (finitely many!) cylinder sets appearing up to degree~$n+1$
  in~$D_*$, $\partial_*^{D}$, and $f_*$ come
  from~$\Gamma/\Gamma_i$. 
  Then $I'$ is non-empty and upwards-closed
  in~$I$. 

  For~$i \in I'$, the complex~$D_*(i)$ is defined up to degree~$n+1$
  and augments to~$L_i$. Moreover, we write $C_*(i)$ for the
  $R_i$-chain complex~$(R \otimes_{Z \Gamma} C_*)(i) \cong_{R_i} R_i
  \otimes_{Z \Gamma} C_*$.  As $R_i$ is flat over~$Z\Gamma$, this
  gives an $R_i$-resolution of~$R_i \otimes_{Z\Gamma} Z \cong_{R_i}
  L_i$.  Because $f_* \colon C_* \to D_*$ is $S_i$-adapted, we obtain
  a corresponding $R_i$-chain map~$f_*(i) \colon C_*(i) \to D_*(i)$
  extending~$\id_{L_i}$:
  \[
  \begin{tikzcd}
    C_*(i)
    \ar{r}{f_*(i)}
    \ar[two heads]{d}
    & D_*(i)
    \ar[two heads]{d}
    \\
    L_i
    \ar[equal]{r}
    & L_i
  \end{tikzcd}
  \]

  We now make use of the fact that the left hand side is a resolution
  to obtain the desired retract: By the fundamental lemma of
  homological algebra, there is an $R_i$-chain map~$g_*(i) \colon
  D_*(i) \to C_*(i)$ extending~$\id_{L_i}$ and $g_*(i) \circ f_*(i)
  \simeq_{R_i} \id_{C_*(i)}$.
  Taking $\Gamma$-coinvariants shows that hence $H_n(C_*(i)_\Gamma)$
  is a $Z$-retract of~$H_n(D_*(i)_\Gamma)$. Moreover,
  \begin{align*}
    H_n(\Gamma_i;Z)
    &
    \cong_Z H_n\bigl( (Z[\Gamma/\Gamma_i] \otimes_Z C_*)_\Gamma\bigr)
    &
    \text{(Shapiro lemma)}
    \\
    & \cong_Z H_n\bigl( (L_i \otimes_Z C_*)_\Gamma \bigr)
    \cong_Z H_n\bigl( (R_i \otimes_{Z\Gamma} C_*)_\Gamma \bigr)
    \\
    &
    \cong_Z H_n\bigl( C_*(i)_\Gamma\bigr). 
  \end{align*}
  This proves the first part.

  For the second part, we obtain from the retract in the first part
  and elementary properties of~$\rk_Z$ that
  \begin{align*}
    \rk_Z H_n(\Gamma_i;Z)
    \leq \rk_Z H_n\bigl(D_*(i)_\Gamma\bigr)
    \leq \rk_Z D_n(i)_\Gamma.
   \end{align*}
  Moreover, in the case~$Z = \Z$, we obtain
  \begin{align*}
    \log \# \tors H_n(\Gamma_i;\Z) 
    \leq \log \# \tors H_n\bigl(D_*(i)_\Gamma\bigr) 
    \leq \log \# \tors \bigl(D_n(i)_\Gamma / \im \partial^{D}_{n+1}(i)_\Gamma\bigr),
  \end{align*}
  as claimed.
\end{proof}

\subsection{Logarithmic torsion estimates}

The goal of this section is to establish a logarithmic torsion
growth estimate for cokernels in terms of the logarithmic norm:

\begin{thm}\label{thm:logtorlognorm}
  In the situation of Setup~\ref{setup:passage} (with~$Z = \Z$),
  let $f \colon M \to N$ be a $\Gamma_*$-adapted $R$-homomorphism
  between $\Gamma_*$-adapted marked projective $R$-modules. Then
  \begin{align*}
    \limsup_{i \in I}
    \frac{\log \# \tors \bigl( N(i)_\Gamma / \im f(i) _\Gamma\bigr)}
         {[\Gamma : \Gamma_i]}
         \leq \lognorm (f).     
  \end{align*}
\end{thm}

The proof relies on Gabber's estimate for the torsion in cokernels
and the fact that the logarithmic norm of adapted homomorphisms can
be computed via adapted decompositions. 
We will use the following version of Gabber's estimate
for the torsion part of cokernels.

\begin{prop}[{\cite[Lemma~1]{soule1999perfect}\cite[Lemma 3.1]{sauer:volume:homology:growth}}]\label{prop:torest}
  Let $M$ and $N$ be finitely generated marked free $\Z$-modules
  and let $f \colon M \to N$ be a $\Z$-homomorphism.
  Then
  \[ \log\#\tors (N/\im f)
  \leq \sum_{b \in B} \logp\; \bigl\| f(b)\bigr\|_1,
  \]
  whenever $B$ is a subset of the marked $\Z$-basis of~$M$ such that $\{ f(b)
  \mid b \in B\}$ generates~$\C \otimes_\Z \im f$ over~$\C$ and where
  $\| \cdot\|_1$ refers to the $\ell^1$-norm with respect to the
  chosen basis on~$N$.
\end{prop}
\begin{proof}
  The classical torsion estimate is 
  $ \log\#\tors (N/\im f)
  \leq \sum_{b \in B} \logp\; \bigl\| f(b)\bigr\|_2,
  $ 
  where $\|\cdot\|_2$ is the $\ell^2$-norm with respect to the
  chosen basis on~$N$~\cite[Lemma~3.1]{sauer:volume:homology:growth}. 
  Because of $\|\cdot\|_2 \leq \|\cdot\|_1$, the claim follows.
\end{proof}

\begin{cor}\label{cor:logtorgensplit}
  Let $M,M',N$ be marked finite rank free $\Z$-modules
  and let $f \colon M \oplus M' \to N$ be $\Z$-linear.
  Moreover, let $M \cong_\Z \bigoplus_{i \in I} M_i$
  and $M' \cong_\Z \bigoplus_{i \in I'} M'_i$ be
  marked decompositions; for each~$i \in I$ let~$N_i \subset N$
  be a marked direct summand with~$f(M_i) \subset N_i$.
  Then
  \[ \log\#\tors(N/\im f)
  \leq \sum_{i \in I} \rk_\Z (N_i) \cdot \logp \|f|_{M_i}\|
  + \sum_{i \in I'} \rk_\Z (M'_i) \cdot \logp \|f|_{M'_i}\|.
  \]
\end{cor}
\begin{proof}
  For each~$i \in I$, let $B_i \subset M_i$ be a subset
  of the marked basis such that~$\{f(b) \mid b \in B_i\}$
  generates~$\C \otimes_\Z f(M_i)$ over~$\C$. 
  For each~$i \in I$, let $B'_i \subset M'_i$ be the marked basis.  
  Then $B \coloneqq \bigcup_{i \in I} B_i \cup \bigcup_{i \in I'} B'_i$
  is a subset of the marked basis of~$M \oplus M'$ such that
  $\{f(b) \mid b\in B\}$ generates~$\C \otimes_\Z \im f$
  over~$\C$. Therefore, Proposition~\ref{prop:torest}
  shows that
  \begin{align*}
    \log\#\tors (N/\im f)
    & \leq \sum_{b \in B} \logp \|f(b)\|_1
    \\
    & \leq
      \sum_{i \in I} \sum_{b \in B_i} \logp \|f(b)\|_1
    + \sum_{i \in I'} \sum_{b \in B'_i} \logp \|f(b)\|_1
    \\
    & \leq
      \sum_{i \in I} \#B_i \cdot \logp \|f|_{M_i}\|
    + \sum_{i \in I'} \#B'_i \cdot \logp \|f|_{M'_i}\|
    \\
    & \leq
      \sum_{i \in I} \rk_\Z (N_i) \cdot \logp \|f|_{M_i}\|
    + \sum_{i \in I'} \rk_\Z (M'_i) \cdot \logp\|f|_{M'_i}\|,
  \end{align*}
  as desired.
\end{proof}

To prove Theorem~\ref{thm:logtorlognorm}, we show the
following version in which the $\limsup$ is unfolded
into an explicit statement: 

\begin{thm}
  In the situation of Setup~\ref{setup:passage} (with $Z=\IZ$), let $f \colon M \to N$
  be a $\Gamma_*$-adapted $R$-homomorphism between $\Gamma_*$-adapted
  marked projective $R$-modules.
  Then, for all~$\varepsilon \in \R_{>0}$, 
  we have for all large enough~$i \in I$: 
  \[ \frac{\log \# \tors \bigl(N(i)_\Gamma / \im f(i)_\Gamma\bigr)}
          {[\Gamma : \Gamma_i]}
  \leq \lognorm (f) + \varepsilon.
  \]
\end{thm}
\begin{proof}
  The subalgebra~$S$ is dense. Thus, we use the description of~$\lognorm$ in terms of cylinder sets
  (Proposition~\ref{prop:lognormadapt}) and the generic torsion estimate for
  cokernels (Corollary~\ref{cor:logtorgensplit}).
  
  Let $\varepsilon \in \R_{>0}$. By Proposition~\ref{prop:lognormadapt},
  we find an $S$-adapted marked decomposition~$(M_j)_{j \in J} \in D_S(M)$
  with
  \[ \lognorm' (f,M_*) \leq \lognorm (f) + \varepsilon.
  \]
  Then, for all large enough~$i \in I$, the decomposition~$M_*$,
  as well as $M$, $N$, and~$f$ all are $S_i$-adapted.
  We split~$J = J' \sqcup J''$ according to the branches
  of~$\lognorm'$: Let~$J'$ be the set of all~$j \in J$ with
  \[ \lognorm' (f|_{M_j}) = \rk (f|_{M_j}) \cdot \logp \| f|_{M_j} \|; 
  \]
  moreover, we let~$N_j \subset N$ be a marked summand
  satisfying~$f(M_j) \subset N_j$ and $\dim (N_j) \leq \rk (f) + \varepsilon/\#J$.
  Let $J'' \coloneqq J \setminus J'$. 

  Using Corollary~\ref{cor:logtorgensplit} and the dimension/norm
  compatibility (Remark~\ref{rem:dimcomp}, Remark~\ref{rem:normcomp}),
  we thus obtain
  \begin{align*}
   & \log \# \tors \bigl(N(i)_\Gamma / \im f(i)_\Gamma\bigr)
   \\
    & \leq
    \sum_{j \in J'} \rk_\Z N_j(i)_\Gamma \cdot \logp \|f(i)_\Gamma|_{M_j(i)_\Gamma}\|
     + \sum_{j \in J''} \rk_\Z M_j(i)_\Gamma \cdot \logp \|f(i)_\Gamma|_{M_j(i)_\Gamma}\|
    \\
    & \leq
    \sum_{j \in J'} [\Gamma : \Gamma_i] \cdot \dim N_j \cdot \logp\| f|_{M_j}\|
     + \sum_{j \in J''} [\Gamma : \Gamma_i] \cdot \dim M_j \cdot \logp \|f|_{M_j}\|
    \\
    & \leq
    \sum_{j \in J'} [\Gamma : \Gamma_i] \cdot (\rk f|_{M_j} + \varepsilon/\#J) \cdot \logp\| f|_{M_j}\|
     + \sum_{j \in J''} [\Gamma : \Gamma_i] \cdot \dim M_j \cdot \logp \|f|_{M_j}\|
    \\
    & \leq
    [\Gamma : \Gamma_i] \cdot \bigl( \lognorm' (f,M_*) + \varepsilon \cdot \logp \|f\| \bigr)      
    \\
    & \leq
    [\Gamma : \Gamma_i] \cdot \bigl( \lognorm (f) + \varepsilon + \varepsilon \cdot \logp \|f\| \bigr).
  \end{align*}
  Rescaling~$\varepsilon$ appropriately gives the claim.
\end{proof}

This completes the proof of Theorem~\ref{thm:logtorlognorm}.

\section{Proof of the dynamical upper bounds}\label{sec:proofthmCEP}

We give proofs for the upper bounds of homological invariants
in terms of measured embedding dimension and measured embedding
volume (Theorem~\ref{thm:dynupper} and Theorem~\ref{thm:l2upperintro}).
The key intermediate step is to write the homological terms in
question as retracts of the corresponding homology of suitably
adapted dynamical embeddings. 

For convenience of the reader we recall the definitions of 
measured embedding dimension and measured embedding volume from the introduction.
Let~$n \in \N$, 
  let~$\Gamma$ be a residually finite group of type~$\FP_{n+1}$,
  and let~$\alpha$ be a standard $\Gamma$-action. 
  We denote by $\Aug_n(\alpha)$ the class of all augmented complexes arising in $\alpha$-embeddings (up to degree $n$).
  Let $Z$ be $\Z$ or a finite field. We have the following:
\begin{itemize}
\item The \emph{measured embedding dimension~$\medim_n^Z (\alpha)$
  over~$Z$ in degree~$n$} is defined as:
  \[
  \medim_n^Z (\alpha) \coloneqq  \inf_{(D_* \epi \linf {\alpha,Z}) \in \Aug_n(\alpha)} \dim_{R} (D_n).
  \]
\item The \emph{measured embedding volume~$\mevol_n(\alpha)$
  in degree~$n$} for~$Z=\IZ$ is defined as: 
  \[
  \mevol_n(\alpha) \coloneqq \inf_{(D_* \epi \linf {\alpha,\Z}) \in \Aug_n(\alpha)} \lognorm (\partial^{D}_{n+1}).
  \]
  \end{itemize}
  
\subsection{Homology gradients}

In this section we prove Theorem~\ref{thm:dynupper} that we restate here:

\begin{thm}[dynamical upper bounds]\label{thm:dynupperproofsec}
  Let $n \in \N$, 
  let $\Gamma$ be a residually finite group of type~$\FP_{n+1}$,
  let $(\Gamma_i)_{i \in I}$ be a directed system of finite index
  normal subgroups of~$\Gamma$ with~$\bigcap_{i \in I} \Gamma_i = 1$
  (e.g., a residual chain in~$\Gamma$
  or the system of all finite index normal subgroups), 
  and let $Z$ be $\Z$ or a finite field.
  Then:
  \begin{align*}
    \widehat b_n(\Gamma,\Gamma_*; Z)
    & \leq \medim_n^Z (\Gamma \actson \widehat\Gamma_*)
    \\
    \widehat t_n(\Gamma,\Gamma_*)
    & \leq \mevol_n (\Gamma \actson \widehat \Gamma_*)
    \quad \text{(if $Z=\IZ$)}.
  \end{align*}
\end{thm}

The proof relies on the following input from previous sections:

\begin{setup}\label{setup:mainproof}
  Let $n \in \N$, 
  let $\Gamma$ be a residually finite group of type~FP$_{n+1}$,
  let $(\Gamma_i)_{i \in I}$ be a directed system of finite index
  normal subgroups of~$\Gamma$ with~$\bigcap_{i \in I} \Gamma_i = 1$,
  let $\alpha\colon \Gamma \actson \widehat \Gamma_*$ be the
  associated dynamical system,
  and let $Z$ be $\Z$ or a finite field.
\end{setup}

\begin{thm}\label{thm:adaptedemb}
  In the situation of Setup~\ref{setup:mainproof},
  let $C_*$ be a free $Z\Gamma$-resolution of~$Z$ that has
  finite rank in degrees~$\leq n+1$ and let
  $f_* \colon C_* \to D_*$ be an $\alpha$-embedding.  
  Then there exists a~$K\in \R_{>0}$ such that: For every~$\delta \in \R_{>0}$,
  there exists a $\Gamma_*$-adapted $\alpha$-embedding~$C_* \to \widehat D_*$
  with
  \[ \dgh K {\widehat D_*, D_*,n} < \delta.
  \]
\end{thm}
\begin{proof}
  This is the special case of Theorem~\ref{thm:adaptedembgen} for
  the action~$\Gamma \actson \widehat \Gamma_*$ and the subalgebra
  of all cylinder sets. 
\end{proof}

\begin{thm}[Theorem~\ref{thm:adaptedretract}]\label{thm:adaptedretractlocal}
  In the situation of Setup~\ref{setup:mainproof},
  let $C_*$ be a free $Z\Gamma$-resolution of~$Z$ that has
  finite rank in degrees~$\leq n+1$ and let
  $f_* \colon C_* \to D_*$ be a $\Gamma_*$-adapted
  $\alpha$-embedding.  
  Then for all large enough~$i \in I$,
  we have
  \[ \rk_Z H_n(\Gamma_i;Z) \leq [\Gamma : \Gamma_i] \cdot \dim (D_n)
  \]
  and, in the case $Z = \Z$,
  \begin{align*}
    \log \# \tors H_n(\Gamma_i;\Z)
    & 
    \leq \log \# \tors \bigl(D_n(i)_\Gamma / \im \partial^{D}_{n+1}(i)_\Gamma\bigr).
  \end{align*}
\end{thm}

\begin{thm}[Theorem~\ref{thm:logtorlognorm}]\label{thm:logtorlognormlocal}
  In the situation of Setup~\ref{setup:mainproof},
  let $f \colon M \to N$ be a $\Gamma_*$-adapted $R$-homomorphism
  between $\Gamma_*$-adapted marked projective $R$-modules. Then
  \begin{align*}
    \limsup_{i \in I}
    \frac{\log \# \tors \bigl( N(i)_\Gamma / \im f(i) _\Gamma\bigr)}
         {[\Gamma : \Gamma_i]}
         \leq \lognorm (f).     
  \end{align*}
\end{thm}

\begin{proof}[Proof of Theorem~\ref{thm:dynupperproofsec}]
  We spell out the proof for~$\widehat t_n$. The proof
  for the Betti gradients works basically in the same way. 
  Because $\Gamma$ is of type~FP$_{n+1}$, there exists a
  free $\Z\Gamma$-res\-o\-lu\-tion~$C_*$ of~$\Z$ that has
  finite rank in degrees~$\leq n+1$.
  By definition of the 
  measured embedding volume, it suffices to prove the following:
  If $f_* \colon C_* \to D_*$ is an $\alpha$-embedding, then
  \[ 
  \widehat t_n(\Gamma,\Gamma_*) \leq \lognorm (\partial^D_{n+1}).
  \]

  Thus, let $f_* \colon C_* \to D_*$ be an $\alpha$-embedding. 
  In view of Theorem~\ref{thm:adaptedemb},
  there exists a constant~$K\in \R_{>0}$ such that: For every~$\delta \in \R_{>0}$,
  there exists a $\Gamma_*$-adapted $\alpha$-embedding~$C_* \to \widehat D_*$
  with
  \begin{align*}
    \dgh K {\widehat D_*, D_*,n} < \delta.
  \end{align*}

  Let $\delta \in \R_{>0}$ and let $\widehat f_* \colon C_* \to
  \widehat D_*$ be such a $\Gamma_*$-adapted $\alpha$-embedding. 
  Combining the retracts for~$\widehat D_*$ from Theorem~\ref{thm:adaptedretractlocal}
  and the logarithmic norm estimates from Theorem~\ref{thm:logtorlognormlocal},
  we obtain
  \begin{align*}
    \widehat t_n(\Gamma,\Gamma_*)
    & = \limsup_{i \in I}
    \frac{\log \#\tors H_n(\Gamma_i;\Z)}
         {[\Gamma \colon \Gamma_i]}
    \\     
    & \leq
    \limsup_{i \in I}
    \frac{\log \# \tors \bigl( \widehat D_n(i)_\Gamma / \im \partial_{n+1}^{\widehat D}(i)_\Gamma \bigr)}
         {[\Gamma : \Gamma_i]}
    & \text{(Theorem~\ref{thm:adaptedretractlocal})}     
    \\     
    & \leq
    \lognorm (\partial_{n+1}^{\widehat D})
    & \text{(Theorem~\ref{thm:logtorlognormlocal})}
    \\
    & \leq
    \lognorm (\partial_{n+1}^{D})
    + \delta \cdot \log_+ K.
    & \text{(Proposition~\ref{prop:lognorm}~\ref{itm:lognorm-GH})}
  \end{align*}
  Taking~$\delta \to 0$, we get the desired estimate~$\widehat t_n(\Gamma,\Gamma_*)
  \leq \lognorm (\partial_{n+1}^D)$.
\end{proof}  

\subsection{\texorpdfstring{$L^2$-Betti numbers}{L2-Betti numbers}}

Analogously to the retraction argument for the gradient estimate,
we can apply the retraction argument also on the level of von~Neumann
algebras. This leads to the $L^2$-Betti number estimate (Theorem~\ref{thm:l2upperintro}).

\begin{thm}\label{thm:l2upper}
  Let $n \in \N$, 
  let $\Gamma$ be a group of type~$\FP_{n+1}$,
  and let $\alpha$  be a standard $\Gamma$-action. 
  Then:
  \begin{align*}
    \ltb n \Gamma
    \leq 
    \medim_n^\Z (\alpha).
  \end{align*}
\end{thm}

\begin{proof}
  We write~$\Rrel \coloneqq \Rrel_\alpha$ for the orbit relation of~$\alpha$
  and use the following description of the $L^2$-Betti numbers~\cite{sauergroupoids}:
  \[ \ltb n \Gamma
  = \dim_{N\Rrel} H_n(\Gamma; N\Rrel).
  \]  
  Let $C_*$ be a free $\Z\Gamma$-resolution of~$\Z$
  that has finite rank in degrees~$\leq n+1$ and let $f_* \colon C_* \to D_*$
  be an $\alpha$-embedding.
  It suffices to show that $\ltb n \Gamma \leq \dim_{\Linf{\alpha} *\Gamma} (D_n)$.
  Let $\widehat C_* \coloneqq (\linf \alpha * \Gamma) \otimes_{\Z \Gamma} C_*$
  and let $\widehat f_* \colon \widehat C_* \to D_*$ be the chain map
  induced by~$f_*$.

  Because $\linf \alpha * \Gamma$ is flat over~$\Z\Gamma$
  (Proposition~\ref{prop:flat}), $\widehat C_*$ is a free
  $\linf \alpha*\Gamma$-resolution of~$\linf \alpha$. Therefore,
  the fundamental lemma of homological algebra provides us
  with an $\linf \alpha * \Gamma$-chain map~$\widehat g_* \colon \widehat D_* \to \widehat C_*$
  (up to degree~$n+1$) that extends~$\id_{\linf \alpha}$ and that
  satisfies
  \[ \widehat g_* \circ \widehat f_*
  \simeq_{\linf \alpha * \Gamma} \id_{\widehat C_*}.
  \]

  Finally, we pass to the level of the von~Neumann algebra~$N\Rrel$:
  Let $\widetilde C_* \coloneqq N\Rrel \otimes_{\linf \alpha * \Gamma} \widehat C_*
  \cong_{N \Rrel} N\Rrel \otimes_{\Z\Gamma} C_*$, let $\widetilde D_* \coloneqq N\Rrel
  \otimes_{\linf \alpha * \Gamma} D_*$, and let
  \begin{align*}
    \widetilde f_*
    & \coloneqq \id_{N \Rrel} \otimes \widehat f_*
    \colon \widetilde C_* \to \widetilde D_*,
    \\
    \widetilde g_*
    & \coloneqq \id_{N \Rrel} \otimes \widehat g_*
    \colon \widetilde D_* \to \widetilde C_*.
  \end{align*}
  In particular, we obtain~$\widetilde g_* \circ \widetilde f_*
  \simeq_{N\Rrel} \id_{\widetilde C_*}$ from the corresponding relation
  between~$\widehat g_*$ and~$\widehat f_*$. Therefore, $H_n(\widetilde C_*)$
  is an $N\Rrel$-retract of~$H_n(\widetilde D_*)$ and the properties
  of~$\dim_{N\Rrel}$~\cite{sauergroupoids} show that 
  \begin{align*}
    \ltb n \Gamma
    & = \dim_{N\Rrel} H_n(\Gamma;N\Rrel)
    \\
    & = \dim_{N\Rrel} H_n(N\Rrel \otimes_{\Z\Gamma} C_*)
    = \dim_{N\Rrel} H_n(\widetilde C_*)
    & \text{(by definition)}
    \\
    & \leq \dim_{N\Rrel} H_n(\widetilde D_*)
    & \text{(by the retract)}
    \\
    & \leq \dim_{N\Rrel} \widetilde D_n
    & \text{(properties of~$\dim_{N\Rrel}$)}
    \\
    & = \dim_{\linf \alpha *\Gamma} D_n. 
  \end{align*}
  Taking the infimum over all $\alpha$-embeddings proves the claim.
\end{proof}

%% file: examples.tex
We provide examples and computations for the measured embedding dimension and volume:: in degree~$0$ (Section~\ref{sec:deg0}), for amenable groups (Section~\ref{sec:ame:have:CEP}), for amalgamated products (Section~\ref{sec:amalg}), for products with an amenable factor (Section~\ref{sec:amenable_factor}), and for finite index subgroups (Section~\ref{sec:finindex}).

Later we will give further examples, especially on hyperbolic $3$-manifolds  (Section~\ref{sec:ex_hyperbolic} and Section~\ref{sec:ex_simvol}), using dynamical inheritance properties established in Part~\ref{part:dyn}.

\section{Basic properties}

We collect some basic properties of~$\medim$ and~$\mevol$.

\begin{setup}
	Throughout, let~$Z$ be the integers (with the usual norm) or a finite field (with the trivial norm).
\end{setup}

\begin{lem}
\label{lem:integers_field}
	Let~$n\in \IN$ and let~$\Gamma$ be a group of type~$\sfFP_{n+1}$.
	Let~$\alpha$ be a standard $\Gamma$-action.
	Let~$Z$ be a finite field (with the trivial norm).
	Then
	\[
		\fa{r\in\{0,\ldots,n+1\}} \medim^Z_r(\alpha)\le \medim^\IZ_r(\alpha).
	\]
	\begin{proof}
		Let $C_*\to D_*$ be an $\alpha$-embedding over~$\IZ$.
		Then~$Z\otimes_\IZ C_*$ is a $Z\Gamma$-resolution of~$Z$ because~$C_*$ is contractible as a $\IZ$-chain complex.
		Hence $Z\otimes_\IZ C_*\to Z\otimes_\IZ D_*$ is an $\alpha$-embedding over~$Z$ with 
		\[
			\dim_{L^\infty(\alpha,Z)\ast \Gamma}(Z\otimes_\IZ D_r)=\dim_{L^\infty(\alpha,\IZ)}(D_r).
		\]
		This proves the claim.
	\end{proof}
\end{lem}

\begin{lem}
\label{lem:cd}
	Let~$n\in \IN$ and let~$\Gamma$ be a group such that there exists a finite free $\IZ\Gamma$-resolution of~$\IZ$ of length~$n$.
	 Let~$\alpha$ be a standard $\Gamma$-action.
	Then
	\begin{align*}
		\fa{r>n} & \medim^Z_r(\alpha)=0
		\\
		\fa{r\ge n} & \mevol_r(\alpha)=0.
	\end{align*}
	\begin{proof}
		By Lemma~\ref{lem:integers_field}, we may assume that~$Z=\IZ$.
		Let~$C_*$ be a finite free $\IZ\Gamma$-resolution of~$\IZ$ of length~$n$.
		Then~$D_*\coloneqq \ind_{\IZ\Gamma}^{\LinftyX\ast \Gamma}C_*$ is a marked projective chain complex augmented over~$\LinftyX$.
		The canonical $\IZ\Gamma$-map $C_*\to D_*$ is an $\alpha$-embedding.
		By construction, we have $D_r=0$ for all~$r>n$.
		Hence for all~$r>n$, we have
		\[
			\medim^\IZ_r(\alpha)\le \dim (D_r)=0.
		\]
		For all~$r\ge n$, we have
		\[
			\mevol_r(\alpha)\le \lognorm(\partial^{D}_{r+1})\le \dim(D_{r+1})\cdot \log_+\|\partial^{D}_{r+1}\|=0.
		\]
		This finishes the proof.
	\end{proof}
\end{lem}

\section{\texorpdfstring{Degree~$0$}{Degree 0}}
\label{sec:deg0}

We show that every standard action of an infinite group has~$\medim$ and~$\mevol$ equal to zero in degree~$0$.

\begin{prop}\label{prop:deg0infinite}
	Let~$\Gamma$ be a finitely generated group with finite generating set~$S$.
	Let~$(C_*,\zeta)$ be a free $\IZ\Gamma$-resolution of~$\IZ$ with $\partial_1\colon C_1\to C_0$ given by
	\[
		\partial_1\colon \bigoplus_{s\in S}\IZ\Gamma\cdot e_s\to \IZ\Gamma, \quad \partial_1(e_s)=1_\Gamma-s.
	\]
	Let~$\alpha\colon \Gamma\actson (X,\mu)$ be a standard action.
	Then the following are equivalent:
	\begin{enumerate}[label=\enum]
		\item The group~$\Gamma$ is infinite;
		\item For every~$\varepsilon\in \IR_{>0}$, there exists an $\alpha$-embedding $C_*\to D_*$ with $\dim(D_0)<\varepsilon$ and $\|\partial^D_1\|\le 2$;
		\item For every~$\varepsilon\in \IR_{>0}$, there exists an $\alpha$-embedding $C_*\to D_*$ with ${\dim(D_0)<\varepsilon}$.
	\end{enumerate}
	\begin{proof}
		We show that (i) implies (ii).
		Let~$\varepsilon\in \IR_{>0}$.
		Since~$\Gamma$ is infinite, there exists a measurable subset~$A$ of~$X$ and a finite subset~$F$ of~$\Gamma$ with $\mu(A)<\varepsilon/2$ and $\mu(X\setminus F\cdot A)<\varepsilon/2$~\cite[Proposition~1]{Levitt95}.
		Set~$B\coloneqq X\setminus F\cdot A$.
		The $\LinftyX\ast \Gamma$-module $D_0\coloneqq \spann{A}\oplus \spann{B}$ satisfies $\dim (D_0)<\varepsilon$.
		Let $\eta\colon D_0\to \LinftyX$ be the $\LinftyX\ast\Gamma$-linear map that sends~$\chi_A$ to~$\chi_A$ and~$\chi_B$ to~$\chi_B$.
		We construct an element~$x\in D_0$ with $\eta(x)=1$ as follows.
		Denote the elements of the finite set~$F$ by~$\gamma_1,\ldots,\gamma_k$.
		Set $A_1\coloneqq \gamma_1\cdot A$ and $A_j\coloneqq \gamma_j\cdot A\setminus \bigcup_{m=1}^{j-1}A_{m}$ for~$j\in \{2,\ldots,k\}$.
		Then $\sqcup_{j=1}^k A_j=F\cdot A$.
		The element
		\[
			x\coloneqq \sum_{j=1}^k\gamma_j\cdot \chi_{\gamma_j^{-1}A_j}\cdot \chi_A + \chi_B\in D_0
		\]
	 	satisfies~$\eta(x)=1$ by construction.
		Then there is an $\alpha$-embedding (in low degrees) given by
		\[\begin{tikzcd}
			\cdots\ar{r}			
			& \bigoplus_{s\in S}\spann{X}\cdot e_s\ar{r}{\partial^D_1}
			& \spann{A}\oplus \spann{B}\ar{r}{\eta}
			& \LinftyX
			\\
			\cdots\ar{r}
			& \bigoplus_{s\in S}\IZ\Gamma\cdot e_s\ar{r}{\partial^C_1}\ar{u}{f_1}
			& \IZ\Gamma\ar{r}{\zeta}\ar{u}{f_0}
			& \IZ\ar[hook]{u}
		\end{tikzcd}\]
		where
		\begin{align*}
			\partial^D_1(e_s) &= (1_\Gamma-s)x; \\
			f_0(1_\Gamma) &= x; \\
			f_1(e_s) &= e_s.
		\end{align*}
		Clearly, (ii) implies (iii).
		We show that (iii) implies (i).
		Suppose that~$\Gamma$ is finite.
		We show that every $\alpha$-embedding $(C_*,\zeta)\to (D_*,\eta)$ satisfies $\dim D_0\ge 1/\#\Gamma$.
		Indeed, let $D_0=\bigoplus_{i\in I}\spann{A_i}$ and let~$x\in D_0$ with $\eta(x)=\chi_X$.
		We write $x=\sum_{i\in I}g_i\cdot \chi_{A_i}\cdot e_i$ for some~$g_i\in \LinftyX\ast \Gamma$.
		Then
		\[
			\chi_X=\eta(x)=\sum_{i\in I}g_i
\cdot \eta(\chi_{A_i}\cdot e_i).
		\]
		Since~$\supp(\eta(\chi_{A_i})\cdot e_i)\subset A_i$, we conclude
		\[
			1\le \sum_{i\in I}\#\Gamma\cdot \mu(A_i)
		\]
		and the claim follows.
	\end{proof}
\end{prop}

\begin{cor}\label{cor:deg0infiniteme}
  Let $\Gamma$ be a finitely generated infinite group and
  let $\alpha$ be a standard action of~$\Gamma$.
  Then
  \[ 
  \medim^Z_0 (\alpha) = 0
  \qand
  \mevol_0 (\alpha) = 0.
  \]
\end{cor}
\begin{proof}
	By Lemma~\ref{lem:integers_field}, we may assume that $Z=\IZ$.
  Then this is a direct consequence of Proposition~\ref{prop:deg0infinite} and
   the definition of the measured embedding dimension/volume. 
\end{proof}

\section{Amenable groups have cheap embeddings}\label{sec:ame:have:CEP}

We show that every standard action of the integers and, more generally, of an infinite amenable group has~$\medim$ and~$\mevol$ equal to zero in all degrees.

\subsection{The integers}

\begin{setup}\label{setup:integers}
  We consider the infinite cyclic group~$\Gamma \coloneqq \Z = \spann{t}$
  and a standard action $\alpha\colon \Gamma \actson (X,\mu)$.
Given $\delta\in \R_{>0}$ and $N\in \N$, by the Rokhlin lemma~\cite[Theorem~7.5]{kechrismiller}, there exist measurable subsets~$A,B\subset X$ with $\mu(B)<\delta$ such that
\[
	X=A\sqcup tA\sqcup t^2A\sqcup \cdots\sqcup t^{N-1}A\sqcup B .
\]
Clearly, $\mu(A)\le 1/N$.
Since $X=tX$, we have
\begin{equation}
	A\sqcup B = t^NA\sqcup tB,
\end{equation}
a fact that will be used repeatedly in the sequel.

Consider the marked projective $\LinftyX\ast \Gamma$-module $\spann{A}\oplus \spann{B}$ and define the element
\[
	x\coloneqq \sum_{j=0}^{N-1} t^j\chi_A + \chi_B \in\spann{A}\oplus\spann{B}.
\]
\end{setup}

\begin{prop}\label{prop:integers resolution}
	In the situation of Setup~\ref{setup:integers}, there is an $\LinftyX\ast \Gamma$-resolution of~$\LinftyX$ of the form
	\[
		0\to D_1\xrightarrow{\partial_1} D_0\xrightarrow{\eta} \LinftyX\to 0,
	\] 
	where $D_0=D_1=\spann{A}\oplus \spann{B}$ and the $\LinftyX\ast \Gamma$-linear maps $\eta$ and $\partial_1$ are given on generators by
	\begin{align*}
		\eta(\chi_A) &= \chi_A; \\
		\eta(\chi_B) &= \chi_B; \\
		\partial_1(\chi_A) &= \chi_A(t^0-t^1)x; \\
		\partial_1(\chi_B) &= \chi_B(t^0-t^1)x.
	\end{align*}
	\begin{proof}
		First of all, $(D_*,\eta)$ is an augmented chain complex, since we have
		\begin{align*}
			\eta(\partial_1(\chi_A)) &= \eta\bigl(\chi_A(\chi_A+\chi_B-t^N\chi_A-t\chi_B)\bigr) \\
			&= \chi_A(\chi_{A\sqcup B}-\chi_{t^NA\sqcup tB}) =0
		\end{align*}
		and similarly $\eta(\partial_1(\chi_B))=0$.
		Moreover, $\eta$ is surjective since $\eta(x)=\chi_X$.
			
		We show that $(D_*,\eta)$ is a resolution by exhibiting an $\LinftyX$-linear chain contraction~$c_*\colon D_*\to D_{*+1}$.
		\[\begin{tikzcd}
			0\ar{r} & \spann{A}\oplus \spann{B}\ar{r}{\partial_1} & \spann{A}\oplus \spann{B}\ar{r}{\eta}\ar[bend left]{l}{c_0} & \LinftyX\ar{r}\ar[bend left]{l}{c_{-1}} & 0
		\end{tikzcd}\]
		Define the $\LinftyX$-linear maps~$c_0$ and~$c_1$ on generators by
		\begin{align*}
			c_{-1}(\chi_X) &= x; \\
			c_0(t^m\chi_A) &={\begin{cases} -\chi_{t^mA} \sum_{j=0}^{m-1}t^j\chi_A - \chi_{t^mA}\sum_{j=0}^{m-1}t^j\chi_B & \text{if }m\ge 0; \\
			\chi_{t^mA} \sum_{j=m}^{-1}t^j\chi_A + \chi_{t^mA}\sum_{j=m}^{-1}t^j\chi_B & \text{if }m< 0;
			\end{cases}} \\
			c_0(t^m\chi_B) &={\begin{cases} -\chi_{t^mB} \sum_{j=0}^{m-1}t^j\chi_A -  \chi_{t^mB}\sum_{j=0}^{m-1}t^j\chi_B &\text{if }m\ge 0; \\
			\chi_{t^mB} \sum_{j=m}^{-1}t^j\chi_A +  \chi_{t^mB}\sum_{j=m}^{-1}t^j\chi_B &\text{if }m< 0.
			\end{cases}}
		\end{align*}
		
		We verify that $c_*$ is a chain contraction:
		First, we clearly have $\eta\circ c_{-1}=\id_{\LinftyX}$.
		Second, we have to show that $\partial_1\circ c_0=\id_{\spann{A}\oplus \spann{B}}- c_{-1}\circ\eta$.
		Indeed, for~$m\ge 0$ we have
		\begin{align*}
			\partial_1(c_0(t^m\chi_A))
			&= -\chi_{t^mA} \biggl(\sum_{j=0}^{m-1}t^j\biggr) (\chi_A+\chi_B) (t^0-t^1)x \\
			&= -\chi_{t^mA} \biggl(\sum_{j=0}^{m-1}t^j\biggr) (t^0-t^1)x \\
			&= -\chi_{t^mA}(t^0-t^m)x \\
			&= -\chi_{t^mA}x + \chi_{t^mA}t^m x \\
			&= -\eta(t^m\chi_A)x + t^m\chi_{A}x \\
			&= -c_{-1}(\eta(t^m\chi_A)) + t^m\chi_{A}
		\end{align*}
		and similarly for~$m<0$.
		The calculation that $\partial_1(c_0(t^m\chi_B))=t^m\chi_B-c_{-1}(\eta(t^m\chi_B))$ for all~$m\in \IZ$ is analogous.
		
		Third, we have to show that $c_0\circ \partial_1=\id_{\spann{A}\oplus \spann{B}}$.
		Indeed, for~$m\ge 0$ we have
		\begin{align*}
			&c_0\bigl(\partial_1(t^m\chi_A)\bigr)
			= c_0\bigl(t^m\chi_A(\chi_A+\chi_B-t^N\chi_A-t\chi_B)\bigr) \\
			&= \chi_{t^mA}\bigl(c_0(t^m\chi_A)-c_0(t^{m+N}\chi_A)-c_0(t^{m+1}\chi_B)\bigr) \\
			&= \chi_{t^mA}\biggl(-\chi_{t^mA}\sum_{j=0}^{m-1}t^j(\chi_A+\chi_B) + \chi_{t^{m+N}A}\sum_{j=0}^{m+N-1}t^j(\chi_A+\chi_B)
			\\
			&\quad + \chi_{t^{m+1}B}\sum_{j=0}^{m}t^j(\chi_A+\chi_B)\biggr) \\
			&= \chi_{t^mA}\biggl(-\chi_{t^mA}\sum_{j=0}^{m-1}t^j(\chi_A+\chi_B) + \chi_{t^{m+N}A}\sum_{j=m+1}^{m+N-1}t^j(\chi_A+\chi_B) 
			\\
			&\quad + (\chi_{t^{m+N}A}+\chi_{t^{m+1}B})\sum_{j=0}^{m}t^j(\chi_A+\chi_B)\biggr) \\
			&= \chi_{t^mA}\biggl(-\chi_{t^mA}\sum_{j=0}^{m-1}t^j(\chi_A+\chi_B) + \chi_{t^{m+N}A}t^m\sum_{j=1}^{N-1}t^j(\chi_A+\chi_B) 
			\\
			&\quad + (\chi_{t^{m}A}+\chi_{t^{m}B})\sum_{j=0}^{m}t^j(\chi_A+\chi_B)\biggr) \\
			&= \chi_{t^mA}\biggl(\chi_{t^{m+N}A}t^m\sum_{j=1}^{N-1}t^j\chi_B + t^m\chi_A\biggr) \\
			&=t^m\chi_A,
		\end{align*}
		where for the last equality we use the following Lemma~\ref{lem:disj:sets:rokhlin}, and similarly for~$m<0$.
		The calculation that $c_0(\partial_1(t^m\chi_B))=t^m\chi_B$ for all~$m\in \IZ$ is analogous.
		This finishes the proof.
	\end{proof}
\end{prop}

\begin{lem}\label{lem:disj:sets:rokhlin}
	For all~$j\in \{1,\ldots,N-1\}$, we have 
	\begin{align*}
		A\cap t^NA\cap t^jB &=\emptyset; \\
		B\cap t^NA\cap t^jB &= \emptyset.
	\end{align*}
	\begin{proof}
		We only prove the first statement, as the second is proved similarly.
		We proceed by induction on~$j$.
		For $j=1$, we have $t^NA\cap tB=\emptyset$.
		Assume for all~$j\in \{1,\ldots,N-2\}$ that $A\cap t^NA\cap t^jB=\emptyset$.
		Let~$b\in B$.
		We have to show that $t^{N-1}b\not\in A\cap t^NA$.
		We have $tb\in tB\subset A\sqcup B$.
		If $tb\in A$, then $t^{N-1}b\in t^{N-2}A$ and hence $t^{N-1}b\not\in A\cap t^NA$.
		If $tb\in B$, then $t^{N-1}b\in t^{N-2}B$ and hence by induction $t^{N-1}b\not\in A\cap t^NA$.
	\end{proof}
\end{lem}

\begin{cor}\label{cor:integers:medim:mevol:zero}
	Let~$\alpha$ be a standard action of~$\IZ$.
	For every~$n\in \IN$, we have $$\medim^Z_n(\alpha)=0 \qand \mevol_n(\alpha)=0.$$
	\begin{proof}
	By Lemma~\ref{lem:integers_field}, we may assume that~$Z=\IZ$.
	  Since the group~$\Gamma\coloneqq \IZ$ is of type~F, there exists a finite free $\IZ\Gamma$-resolution~$C_*$ of the trivial $\IZ\Gamma$-module~$\IZ$.
          The $\LinftyX * \Gamma$-resolution~$(D_*,\eta)$ from Proposition~\ref{prop:integers resolution}
          satisfies
          \[ \dim(D_0)=\dim(D_1)=\mu(A\cup B)\le 1/N+\delta
          \qand \|\partial_1\|\le 2.
          \]
		By the fundamental lemma of homological algebra, there exists an $\LinftyX\ast \Gamma$-chain map $f_*\colon \LinftyX\otimes_\IZ C_*\to D_*$ extending~$\id_{\LinftyX}$.
		(Since both $\LinftyX\otimes_\IZ C_*$ and~$D_*$ are projective resolutions, the map~$f_*$ is in fact a chain homotopy equivalence.)
		Then the composition
		\[
			C_*\to \LinftyX\otimes_\IZ C_*\xrightarrow{f_*} D_*
		\]
		is an $\alpha$-embedding.
	\end{proof}
\end{cor}

We describe an explicit $\alpha$-embedding for a standard action of~$\IZ$.

\begin{ex}
\label{ex:integers}
	For $\Gamma\coloneqq \IZ = \spann{t}$, we consider the usual $\IZ\Gamma$-resolution~$C_*$
	\[
		0\to \IZ\Gamma\to \IZ\Gamma\to \IZ\to 0
	\]
	and the induced $\LinftyX\ast \Gamma$-resolution $\LinftyX\otimes_\IZ C_*$
	\[
		0\to \LinftyX\otimes_\IZ \IZ\Gamma\xrightarrow{\partial^C_1} \LinftyX\otimes_\IZ \IZ\Gamma\xrightarrow{\zeta} \LinftyX\otimes_\IZ \IZ\to 0,
	\]
	where
	\begin{align*}
		\zeta(t^0) &= \chi_X; \\
		\partial^C_1(t^0) &= t^0-t^1.
	\end{align*}
	In the situation of Setup~\ref{setup:integers}, let $(D_*,\eta)$ be the $\LinftyX\ast \Gamma$-resolution from Proposition~\ref{prop:integers resolution}.
	We exhibit chain maps $f_*\colon \LinftyX\otimes_\IZ C_*\to D_*$ and $r_*\colon D_*\to \LinftyX\otimes_\IZ C_*$ and a chain homotopy $h_*\colon \LinftyX\otimes_\IZ C_*\to \LinftyX\otimes_\IZ C_{*+1}$ between $r_*\circ f_*$ and $\id_{\LinftyX\otimes_\IZ C_*}$:
	\[\begin{tikzcd}
		0\ar{r} & \spann{A}\oplus \spann{B}\ar{r}{\partial^D_1}\ar[bend left]{d}{r_1} & \spann{A}\oplus \spann{B}\ar{r}{\eta}\ar[bend left]{d}{r_0} & \LinftyX\ar{r}\ar{d}{\cong} & 0 \\
		0\ar{r} & \LinftyX\otimes_\IZ \IZ\Gamma\ar{r}{\partial^C_1}\ar{u}{f_1} & \LinftyX\otimes_\IZ \IZ\Gamma\ar{r}{\zeta}\ar{u}{f_0}\ar[bend left]{l}{h_0} & \LinftyX\otimes_\IZ \IZ\ar{r}\ar{u}\ar[bend left]{l}{h_{-1}} & 0
	\end{tikzcd}\]
	The $\LinftyX\ast \Gamma$-chain maps~$f_*$ and $r_*$ are given on generators by
	\begin{align*}
		f_0(t^0) &= x \quad \text{ (where $x$ is defined as in Setup~\ref{setup:integers})}; \\
		f_1(t_0) &= \chi_A+\chi_B; \\
		r_0(\chi_A) &= \chi_A\otimes t^0; \\
		r_0(\chi_B) &= \chi_B\otimes t^0; \\
		r_1(\chi_A) &= \chi_A \widetilde{x}; \\
		r_1(\chi_B) &= \chi_B \widetilde{x},
	\end{align*}
	where $\widetilde{x}\in \LinftyX\otimes_\IZ \IZ\Gamma$ is defined as
	\[
		\widetilde{x}\coloneqq \sum_{j=0}^{N-1} \chi_{t^NA}\otimes t^j + \chi_{tB}\otimes t^0.
	\]
	Indeed, $f_*$ and $r_*$ are chain maps, the only non-obvious identity being the following:
	\begin{align*}
		\partial^C_1(r_1(\chi_A)) &= \partial^C_1(\chi_A\widetilde{x}) \\
		&= \chi_A\Bigl(\sum_{j=0}^{N-1}\chi^{t^NA}\otimes (t^j-t^{j+1})+\chi_{tB}\otimes (t^0-t^1)\Bigr) \\
		&= \chi_A(\chi_{t^NA}\otimes t^0-\chi_{t^NA}\otimes t^N+\chi_{tB}\otimes t^0-\chi_{tB}\otimes t^1) \\
		&= \chi_A(\chi_A\otimes t^0 + \chi_B\otimes t^0-\chi_{t^NA}\otimes t^N - \chi_{tB}\otimes t^1) \\
		&= \chi_Ar_0(\chi_A+\chi_B-t^N\chi_A-t\chi_B) \\
		&= r_0(\chi_A(t^0-t^1)x) \\
		&= r_0(\partial_1^D(\chi_A))
	\end{align*}
	and similarly $\partial^C_1(r_1(\chi_B))=r_0(\partial^D_1(\chi_B))$.
	The $\LinftyX\ast \Gamma$-chain homotopy~$h_*$ is given by~$h_{-1}=0$ and
	\[
		h_0(t^0) = -\sum_{j=0}^{N-1}\sum_{k=0}^{j-1} \chi_{t^jA}\otimes t^k.
	\]
	We have
	\begin{align*}
		\partial^C_1(h_0(t^0)) 
		&= -\sum_{j=0}^{N-1}\sum_{k=0}^{j-1}\chi_{t^jA}\otimes (t^k-t^{k+1}) \\
		&= -\sum_{j=0}^{N-1}\chi_{t^jA}\otimes(t^0-t^j) \\
		&= \sum_{j=0}^{N-1}t^j\chi_A\otimes t^0 - \sum_{j=0}^{N-1}\chi_{t^jA}\otimes t^0 \\
		&= \sum_{j=0}^{N-1}t^j\chi_A\otimes t^0 + \chi_B\otimes t^0 - t^0 \\
		&= r_0(x)-t^0 \\
		&= r_0(f_0(t_0))-t^0
	\end{align*}
	and
	\begin{align*}
		h_0(\partial^C_1(t^0)) 
		&= h_0(t^0-t^1) \\
		&= -\sum_{j=0}^{N-1}\sum_{k=0}^{j-1}\chi_{t^jA}\otimes t^k + t^1\sum_{j=0}^{N-1}\sum_{k=0}^{j-1}\chi_{t^jA}\otimes t^k \\
		&= -\sum_{j=1}^{N-1}\chi_{tjA}\otimes t^0 + \sum_{k=1}^{N-1}\chi_{t^NA}\otimes t^k \\
		&= \widetilde{x}-t^0 \\
		&= r_1(f_1(t^0))-t^0.
	\end{align*}
	Finally, we note that the operator-norms of the above maps satisfy the following estimates
	\begin{align*}
		\|f_0\|, \|f_1\|, \|r_0\| &\le1; \\
		 \|r_1\| &\le N; \\
		 \|h_0\| &\le N^2.
	\end{align*}
\end{ex}

\subsection{Amenable groups}
  
We prove that standard actions~$\alpha$ of infinite amenable groups have medim and mevol equal to zero in all degrees.
We do so by constructing $\alpha$-embeddings with arbitrarily small dimension and lognorm, to which we refer as ``cheap'' $\alpha$-embeddings.
We first develop some general preparations.
Given a matrix~$\Lambda$ over~$\IZ\Gamma$ and a marked projective module~$D_1$, we construct a marked projective module~$D_2$ and a map $D_2\to D_1$ given by right multiplication with~$\Lambda$ such that $\dim(D_2)$ is controlled by~$\dim(D_1)$.

\begin{rem}\label{rem:matrix bound}
	Let $\Lambda=(\lambda_{ij})_{(i,j)\in I\times J}$ be a matrix with entries in~$\IZ\Gamma$.
	We set
	\[
		\kappa(\Lambda)\coloneqq \max_{i,j} |\lambda_{ij}|_1.
	\]
	Then $|\lambda_{ij}|_1\ge \max\{|\lambda_{ij}|_\infty, \#\supp_\Gamma(\lambda_{ij})\}$.
	Let $f\colon \bigoplus_{i\in I}\spann{A_i}\to \bigoplus_{j\in J}\spann{B_j}$ be a map between marked projective modules given by right multiplication with the matrix~$\Lambda$.
	Then by Lemma~\ref{lem:normN1est} we have
	\[
		\|f\|\le \Ntmax(f)\cdot \|f\|_\infty\le \kappa(\Lambda)^2\cdot \#J.
	\]
	The upper bound~$\kappa(\Lambda)^2\cdot \#J$ for~$\|f\|$ is very coarse but depends only on the matrix~$\Lambda$ and not on the sets~$(A_i)_i$ and~$(B_j)_j$.
\end{rem}

\begin{lem}\label{lem:supp1 extension}
	Let~$\Gamma$ be a group and let $\alpha\colon \Gamma\actson (X,\mu)$ be a standard action.
	Let~$I$ and~$J$ be finite sets, let~$\Lambda=(\lambda_{ij})_{(i.j)\in I\times J}$ be a matrix with entries in~$\IZ\Gamma$, and let~$(B_j)_{j\in J}$ be a family of measurable subsets of~$X$.
	Then there exists a family~$(A_i)_{i\in I}$ of measurable subsets of~$X$ with	
	\[
		\mu(A_i)\le \kappa(\Lambda)\cdot \sum_{j\in J} \mu(B_j)
	\]
	and an $\LinftyX\ast\Gamma$-linear map $f\colon \bigoplus_{i\in I}\spann{A_i}\to \bigoplus_{j\in J}\spann{B_j}$ given by right multiplication with~$\Lambda$ satisfying
	\[
		\|f\|\le \kappa(\Lambda)^2\cdot \# J.
	\]
	\begin{proof}
		For~$i\in I$, we consider the element 
		\[
			y_i\coloneqq \sum_{j\in J} \lambda_{ij}\cdot \chi_{B_j}e_j \in \bigoplus_{j\in J}\spann{B_j}.
		\]
		By construction, the subset $A_i\coloneqq \supp_1(y_i)\subset X$ satisfies $\mu(A_i)\le \kappa(\Lambda)\cdot \sum_{j\in J}\mu(B_j)$.
		The map $f\colon \bigoplus_{i\in I}\spann{A_i}\to \bigoplus_{j\in J}\spann{B_j}$ defined by $f(\chi_{A_i}e_i)=y_i$ is a well-defined $\LinftyX\ast\Gamma$-linear map and is given by right multiplication with~$\Lambda$. 
		The map~$f$ satisfies~$\|f\|\le \kappa(\Lambda)^2\cdot \#J$ by Remark~\ref{rem:matrix bound}.
	\end{proof}
\end{lem}

We denote by~$C_{*\ge 1}$ a chain complex that is concentrated in degrees~$\ge 1$.

\begin{lem}\label{lem:supp1 chain extension}
	Let $C_{*\ge 1}$ be a free $\IZ\Gamma$-chain complex with $C_k\cong \bigoplus_{I_k}\IZ\Gamma$ and $\partial^C_k$ given by (right multiplication with) a matrix~$\Lambda_k$.
	Let~$n\in \IN$ and suppose that $I_k$ is finite for all~$k\le n$.
	Let $(B_{1,j})_{j\in I_1}$ be a family of measurable subsets of~$X$.
	Then there exists a $\LinftyX\ast\Gamma$-chain complex~$D_{*\ge 1}$ with $D_1 = \bigoplus_{j\in I_1}\spann{B_{1,j}}$ satisfying for all~$k\le n$
	\[
		\dim(D_k)\le \dim(D_1)\cdot \prod_{m=2}^k \kappa(\Lambda_{m})\cdot \#I_{m}
	\]
	and
	\[
		\|\partial^D_k\|\le \kappa(\Lambda_k)^2\cdot \#I_{k-1}
	\]
	and there exists an $\LinftyX\ast\Gamma$-chain map~$f_*\colon \LinftyX\otimes_\IZ C_*\to D_*$ given by~$f_k(e_i)=\chi_{B_{k,i}}e_i$.
	\begin{proof}
		We construct~$D_*$, $\partial^D_*$, and~$f_*$ inductively.
		Set~$D_1\coloneqq \bigoplus_{j\in I_1} \spann{B_{1,j}}$ and $f_1(e_j)=\chi_{B_{1,j}}e_j$.
		Lemma~\ref{lem:supp1 extension} yields a module~$D_2=\bigoplus_{i\in I_2}\spann{B_{2,i}}$ with $\dim(D_2)\le \dim(D_1)\cdot  \kappa(\Lambda_2)\cdot \#I_2$
		and a map $\partial^D_2\colon D_2\to D_1$ with $\|\partial^D_2\|\le \kappa(\Lambda_2)^2\cdot \#I_1$.
		Setting~$f_2(e_i)=\chi_{B_{2,i}}e_i$, we have $f_1\circ \partial^C_2=\partial^D_2\circ f_2$.
		We apply Lemma~\ref{lem:supp1 extension} inductively to obtain~$D_n$, $\partial^D_n$, and~$f_n$ with the desired properties.
		For~$k\ge n+1$ and~$i\in I_k$, we simply set~$B_{k,i}\coloneqq X$ and~$D_k\coloneqq \bigoplus_{i\in I_k}\spann{B_{k,i}}$.
		For~$\Lambda_k=(\lambda_{k,ij})$, we set $\partial^D_k(\chi_{B_{k,i}}e_i)=\sum_{j\in I_{k-1}}\lambda_{k,ij}\cdot \chi_{B_{k-1,j}}e_j$. 
		We have thus constructed a commutative diagram
                \[
                \begin{tikzcd}
			\cdots\ar{r}
			& \bigoplus_{I_{n+1}}\spann{X}\ar{r}{\partial^D_{n+1}}
			& \bigoplus_{i\in I_n}\spann{B_{n,i}}\ar{r}{\partial^D_n} 
			& \cdots\ar{r}{\partial^D_2}
			& \bigoplus_{j\in I_1}\spann{B_{1,j}} \\
			\cdots\ar{r}
			& \LinftyX\otimes_\IZ C_{n+1}\ar{r}{\partial^C_{n+1}}\ar{u}{f_{n+1}}
			& \LinftyX\otimes_\IZ C_n\ar{r}{\partial^C_n}\ar{u}{f_n}
			& \cdots\ar{r}{\partial^C_2}
			& \LinftyX\otimes_\IZ C_1\ar{u}{f_1}
		\end{tikzcd}
                \]
		Note that $(D_*,\partial^D_*)$ is indeed a chain complex; for every~$k\ge 2$ we have $\partial^D_{k-1}\circ \partial^D_k=0$ because~$f_k$ is surjective and $\partial^C_{k-1}\circ \partial^C_k=0$.
	\end{proof}
\end{lem}

The point of Lemma~\ref{lem:supp1 chain extension} is that to construct a cheap $\alpha$-embedding~$f_*\colon C_*\to D_*$, it suffices to construct~$D_{*\le 1}$ with $\dim(D_1)$ arbitrarily small and~$f_{*\le 1}$ with~$f_1$ being the obvious projection. 
For amenable groups this can be achieved using the following strong version of the Rokhlin lemma.

\begin{thm}[Rokhlin lemma, {\cite[Theorem~3.6]{CJKMSTD}}]
\label{thm:Rokhlin}
        Let~$\Gamma$ be a countable amenable group,
        let $\alpha\colon \Gamma \actson (X,\mu)$ be a standard action, 
	let~$F\subset \Gamma$ be a finite set, and
	let~$\delta\in \IR_{>0}$.
	Then there exists a $\mu$-conull $\Gamma$-invariant Borel set~$X'\subset X$,
	a finite set~$J$,
	a family~$(A_j)_{j\in J}$ of Borel subsets of~$X'$, and
	a family~$(T_j)_{j\in J}$ of $(F,\delta)$-invariant non-empty finite subsets of~$\Gamma$
	such that $(T_j\cdot x)_{j\in J,x\in A_j}$ partitions~$X'$.
\end{thm}

Here $T_j$ being \emph{$(F,\delta)$-invariant} means that 
\[
	\frac{\#(F\cdot T_j\triangle T_j)}{\# T_j}\le \delta.
\]

\begin{rem}\label{rem:Rokhlin fund domain}
	If $\Gamma$ is infinite and $F$ contains a generating set, we have $1/\#T_j \le \delta$ for all~$j\in J$ and hence
	\[
		\mu\Bigl(\bigcup_{j\in J} A_j\Bigr)\le \sum_{j\in J}\mu(A_j)
		=\sum_{j\in J} \frac{1}{\#T_j}\cdot  \mu(T_j\cdot A_j)
		\le \delta\cdot \sum_{j\in J}\mu(T_j\cdot A_j)
		= \delta.
	\]
\end{rem}

\begin{thm}
\label{thm:amenable}
	Let~$n\in \IN$ and let~$\Gamma$ be an infinite amenable group of type~$\FP_{n+1}$.
	Let~$\alpha$ be a standard $\Gamma$-action. 
	Then there exists~$K\in \IR_{>0}$ such that for every~$\varepsilon\in \IR_{>0}$, there exists an $\alpha$-embedding $C_*\to D_*$ such that for all~$r\in \{0,\ldots,n+1\}$, we have $\dim (D_r)<\varepsilon$ and $\|\partial^D_r\|\le K$.
	
	In particular, we have:
	\begin{align*}
    \fa{r \in \{0,\dots, n+1\}}
    &\medim^Z_r(\alpha)=0;
    \\
    \fa{r\in \{0,\dots, n\}}
    &\mevol_r(\alpha)=0.
  \end{align*}
	\begin{proof}
		By Lemma~\ref{lem:integers_field}, we may assume that~$Z=\IZ$.
		Let~$S$ be a finite generating set of~$\Gamma$.
		There exists a free $\IZ\Gamma$-resolution~$(C_*,\zeta)$ of~$\IZ$ of the form
		\[
			\cdots\to 
			\bigoplus_{I_2}\IZ\Gamma\xrightarrow{\partial^C_2}
			\bigoplus_{I_1}\IZ\Gamma\xrightarrow{\partial^C_1}
			\IZ\Gamma\xrightarrow{\zeta}
			\IZ\to 0
		\]
		with $I_1=S$, $\partial^C_1(e_s)=1_\Gamma-s$, and $I_k$ finite for~$k\le n+1$.
		We construct a chain map $\LinftyX\otimes_\IZ C_*\to D_*$ from which we then obtain a  ``cheap'' $\alpha$-embedding $C_*\to D_*$ by composition with the canonical map $C_*\to \LinftyX\otimes_\IZ C_*$.
		
		\textbf{Step~1:} $*\le 1$.
		This step uses the amenability of~$\Gamma$ and the Rokhlin lemma.
		Let $F\coloneqq S\cup S^{-1}\subset \Gamma$ and let~$\delta\in \IR_{>0}$.
		Theorem~\ref{thm:Rokhlin} yields a finite set~$J$, a family~$(A_j)_{j\in J}$ of measurable subsets of~$X$, and a family~$(T_j)_{j\in J}$ of finite $(F,\delta)$-invariant finite subsets of~$\Gamma$ such that $(T_j\cdot x)_{j\in J,x\in A_j}$ is a partition of a $\mu$-conull subset of~$X$.
		
		The set~$B_0\coloneqq \bigcup_{j\in J}A_j\subset X$ satisfies $\mu(B_0)< \delta$ by Remark~\ref{rem:Rokhlin fund domain}.
		Hence $D_0\coloneqq \spann{B_0}$ satisfies $\dim(D_0)\le \delta$.
		We consider the element
		\[
			x\coloneqq \sum_{j\in J}\sum_{t\in T_j} t\chi_{A_j} \in \spann{B_0}.
		\]
		For~$s\in S$, let $B_{1,s}\coloneqq \supp_1((1_\Gamma-s)x)\subset X$.
		We claim that $\mu(B_{1,s})\le 2\delta$.
		Indeed, let~$T_{j,s}\coloneqq s\cdot T_j\triangle T_j \subset \Gamma$ and observe that
		\begin{align*}
			\#T_{j,s} 
			&= \#(sT_j\setminus T_j) + \#(T_j\setminus sT_j) \\
			&= \#(sT_j\setminus T_j) + \#(s^{-1}T_j\setminus T_j) \\
			&\le 2\#(FT_j\setminus T_j) \\
			&\le 2\delta \#T_j.
		\end{align*}
		Since $B_{1,s}\subset \bigcup_{j\in J}\bigcup_{t\in T_{j,s}} tA_j$, we have
		\[
			\mu(B_{1,s}) 
			\le \sum_{j\in J}\sum_{t\in T_{j,s}} \mu(tA_j) 
			= \sum_{j\in J} \#T_{j,s}\cdot \mu(A_j) 
			\le 2\delta\cdot \sum_{j\in J} \#T_j\cdot \mu(A_j) = 2\delta.
		\]
		Hence $D_1\coloneqq \bigoplus_{s\in S} \spann{B_{1,s}}$ satisfies $\dim(D_1)\le 2\delta\cdot \#S$.
		
		Define the $\LinftyX\ast \Gamma$-linear map $\partial^D_1\colon D_1\to D_0$ by $\partial^D_1(\chi_{B_{1,s}})=(1_\Gamma-s)x$, which satisfies $\|\partial^D_1\|\le 2$. 
		
		Then the following diagram commutes
		\[\begin{tikzcd}
			\bigoplus_{s\in S} \spann{B_{1,s}}\ar{r}{\partial^D_1} 
			& \spann{B_0}\ar{r}{\eta}
			& \LinftyX \\
			\LinftyX\otimes_\IZ \bigoplus_S \IZ\Gamma\ar{u}{f_1}\ar{r}{\partial^C_1}
			& \LinftyX\otimes_\IZ \IZ\Gamma\ar{u}{f_0}\ar{r}{\zeta}
			& \LinftyX\ar{u}{\id} 
		\end{tikzcd}\]
		where $f_0(1_\Gamma)=x$ and $f_1(e_s)=\chi_{B_{1,s}}e_s$.
		We have $\eta(x)=\chi_X$ by construction, and $\eta\circ \partial^D_1=0$ because $f_1$ is surjective and $\zeta\circ \partial^C_1=0$.
		
		\textbf{Step~2:} $*\ge 2$.
		For~$k\ge 2$, suppose~$\partial^C_k$ is given by (right multiplication with) the matrix~$\Lambda_k$.
		Then Lemma~\ref{lem:supp1 chain extension} yields a chain complex~$D_{*\ge 1}$ satisfying
		\[
			\dim(D_n)\le \dim(D_1)\cdot \prod_{m=2}^n \kappa(\Lambda_{m})\cdot \#I_{m} \le 2\delta\cdot \#S\cdot  \prod_{{m}=2}^n \kappa(\Lambda_{m})\cdot \#I_{m}
		\]
		and
		\[
			\|\partial^D_{n+1}\|\le \kappa(\Lambda_{n+1})\cdot \#I_n
		\]
		and a chain map $f_{*\ge 1}\colon \LinftyX\otimes_\IZ C_*\to D_*$.
		Note that~$D_{*\ge 0}$ is indeed a chain complex; we have $\partial^D_1\circ \partial^D_2=0$ because $f_2$ is surjective and $\partial^C_1\circ \partial^C_2=0$.
	\end{proof}
\end{thm}

\begin{rem}
Since all standard actions of countable infinite amenable groups are orbit equivalent~\cite[Theorem~6]{ornsteinweiss}, if Theorem~\ref{thm:wbOE:intro} could be extended to all orbit equivalences (instead of only weak bounded orbit equivalences), then Theorem~\ref{thm:amenable} would be a direct consequence of Corollary~\ref{cor:integers:medim:mevol:zero}.
\end{rem}

\begin{rem}
	Let~$\Gamma$ be an infinite amenable group. By Example~\ref{ex:wc-amenable}, all \pmp\ actions
	of~$\Gamma$ are weakly equivalent (see Definition~\ref{def:wc}). Thus, Theorem~\ref{thm:wc} yields
	that it suffices to show the vanishing of~$\medim$ and~$\mevol$ for a single standard action 
	of~$\Gamma$. 
\end{rem}

\section{Amalgamated products}
\label{sec:amalg}

\subsection{Amalgamated products}\label{subsec:amalg:prod}

We prove inheritance properties of~$\medim$ and $\mevol$ for actions of amalgamated products.
Let~$Z$ be the integers (with the usual norm) or a finite field (with the trivial norm).

Recall that the \emph{mapping cone}~$\Cone(\varphi)_*$ of a chain map $\varphi_*\colon D_*\to E_*$ is the chain complex with chain modules \[\Cone(\varphi)_n=D_{n-1}\oplus E_n\] and differentials $\partial_n\colon \Cone(\varphi)_n\to \Cone(\varphi)_{n-1}$ given by 
\[
	\partial_n(x,y)=\bigl(-\partial^{D}_{n-1}(x), \partial^{E}_n(y)+\varphi_{n-1}(x)\bigr).
\]

\begin{lem}
\label{lem:cone}
	Let~$R$ be the ring~$\LinftyX\ast \Gamma$.
	Let $\varphi_*\colon D_*\to E_*$ be a map of marked projective $R$-chain complexes.
	For all~$n\in \IZ$, the following hold:
	\begin{enumerate}[label=\enum]
		\item\label{item:cone dim} $\dim_R (\Cone(\varphi)_n)= \dim_R(D_{n-1})+\dim_R(E_n)$;
		\item\label{item:cone lognorm} $\lognorm (\partial^{\Cone(\varphi)}_n)\le \dim_R(D_{n-1})\cdot\log_+(\|\partial^{D}_{n-1}\|+\|\varphi_{n-1}\|) + \lognorm(\partial^{E}_n)$.
	\end{enumerate}
	\begin{proof}
		(i) This is clear by the definition of~$\Cone(f)_n$.
		
		(ii) Using properties of~$\lognorm$ (Proposition~\ref{prop:lognorm}), we have
		\begin{align*}
			\lognorm(\partial^{\Cone(\varphi)}_n)
			&\le \lognorm(\partial^{\Cone(\varphi)}_n|_{D_{n-1}})+\lognorm(\partial^{\Cone(\varphi)}_n|_{E_n}) \\
			&\le \lognorm\bigl(D_{n-1}\to D_{n-2}\oplus E_{n-1}, x\mapsto (-\partial^{D}_{n-1}(x),\varphi_{n-1}(x))\bigr) \\
			&\quad +  \lognorm\bigl(E_n\to D_{n-2}\oplus E_{n-1},y\mapsto (0,\partial^{E}_n(y))\bigr) \\
			& \le \dim_R(D_{n-1})\cdot \log_+(\|\partial^{D}_{n-1}\|+\|\varphi_{n-1}\|) + \lognorm(\partial^{E}_{n})
		\end{align*}
		as claimed.
	\end{proof}
\end{lem}

An amalgamated product is a pushout of groups along injective structure maps. 

\begin{prop}
\label{prop:amalg}
	Let $\Gamma\cong \Gamma_1\ast_{\Gamma_0}\Gamma_2$ be an amalgamated product, where~$\Gamma_i$ is of type~$\sfFP_\infty$ for~$i\in \{0,1,2\}$. 
	Let $\alpha\colon \Gamma\actson (X,\mu)$ be a standard action.
	For~$i\in \{0,1,2\}$, we write $\alpha_i\coloneqq \alpha|_{\Gamma_i}\colon \Gamma_i\actson (X,\mu)$ for the restricted action.
	For all~$n\in \IN$, the following hold:
	\begin{enumerate}[label=\enum]
		\item $\medim^Z_n(\alpha)\le \medim^Z_n(\alpha_1)+\medim^Z_n(\alpha_2)+\medim^Z_{n-1}(\alpha_0)$;
		\item If $\Gamma_0\cong \{1\}$ and~$n\ge 1$, then $\mevol_n(\alpha)\le \mevol_n(\alpha_1)+ \mevol_n(\alpha_2)$;
		\item If $\Gamma_0\cong \IZ$, then $\mevol_n(\alpha)\le \mevol_n(\alpha_1)+\mevol_n(\alpha_2)$.
	\end{enumerate}
	\begin{proof}
		For~$i\in \{0,1,2\}$, let $f(i)_*\colon C(i)_*\to D(i)_*$ be an $\alpha_i$-embedding.
		It follows from the short exact sequence of $Z\Gamma$-modules
		\[
			0\to Z[\Gamma/\Gamma_0]\to Z[\Gamma/\Gamma_1]\oplus Z[\Gamma/\Gamma_2]\to Z\to 0
		\]
		that there exists a $Z\Gamma$-chain map 
		\[
			g_*\colon \ind_{\IZ\Gamma_0}^{\IZ\Gamma}C(0)_*\to \ind_{\IZ\Gamma_1}^{\IZ\Gamma}C(1)_*\oplus \ind_{\IZ\Gamma_2}^{\IZ\Gamma}C(2)_*
		\]
		such that~$C_*\coloneqq \Cone(g)_*$ is a free $Z\Gamma$-resolution of~$Z$.
		Consider the following diagram of $\LinftyX\ast\Gamma$-chain complexes:
		\[\begin{tikzcd}
			\ind_{L^\infty(\alpha_0)\ast\Gamma_0}^{\LinftyX\ast\Gamma}D(0)_*\ar[dashed]{r}{\varphi_*}\ar[bend left]{d}{r(0)_*}
			& \ind_{L^\infty(\alpha_1)\ast\Gamma_1}^{\LinftyX\ast\Gamma}D(1)_*\oplus \ind_{L^\infty(\alpha_2)\ast\Gamma_2}^{\LinftyX\ast\Gamma}D(2)_*\ar{r}
			& D_*
			\\\ind_{Z\Gamma_0}^{\LinftyX\ast\Gamma}C(0)_*\ar{r}{\ind g_*}\ar{u}{\ind f(0)_*}
			& \ind_{Z\Gamma_1}^{\LinftyX\ast\Gamma}C(1)_*\oplus \ind_{Z\Gamma_2}^{\LinftyX\ast\Gamma}C(2)_*\ar{r}\ar{u}{\ind f(1)_*\oplus\ind f(2)_*}
			& \ind_{Z\Gamma}^{\LinftyX\ast\Gamma}C_*\ar[dashed]{u}
		\end{tikzcd}\]
		Here
		\begin{itemize}
			\item $r(0)_*$ is a homotopy left-inverse of~$\ind f(0)_*$;
			\item $\varphi_*\coloneqq \ind f(1)_*\oplus \ind f(2)_*\circ \ind g_* \circ r(0)_*$;
			\item $D_*\coloneqq \Cone(\varphi)_*$;
			\item the right vertical map is induced by functoriality of the mapping cone (involving a homotopy making the left square commutative). 
		\end{itemize}
		Then the composition
		\[
			C_*\to \ind_{Z\Gamma}^{\LinftyX\ast \Gamma}C_*\to D_* 
		\]
		is an $\alpha$-embedding. 
		Since 
		\[
			\dim_{\LinftyX\ast \Gamma}\bigl(\ind_{L^\infty(\alpha_i)\ast \Gamma_i}^{\LinftyX\ast \Gamma} D(i)_n\bigr)=\dim_{L^\infty(\alpha_i)\ast \Gamma_i}(D(i)_n),
		\]
		part~(i) follows from Lemma~\ref{lem:cone}~\ref{item:cone dim}.
		
		For part~(ii), assume that~$\Gamma_0$ is the trivial group and~$n\ge 1$.
		For~$i\in \{1,2\}$, we fix $\alpha_i$-embeddings $f(i)_*\colon C(i)_*\to D(i)_*$.
		Let~$C(0)_*\to D(0)_*$ be the obvious $\alpha_0$-embedding concentrated in degrees~$\le 0$.
		For brevity, we write~$\ind$ instead of~$\ind_{L^\infty(\alpha_i)\ast \Gamma_i}^{\LinftyX\ast \Gamma}$.
		Then, since $\dim (D_n)=0$, Lemma~\ref{lem:cone}~\ref{item:cone lognorm} yields
		\begin{align*}
			\lognorm(\partial^{D}_{n+1})
			&\le \lognorm(\ind \partial^{D(1)}_{n+1}\oplus \ind \partial^{D(2)}_{n+1})
			\\
			&\le \lognorm(\ind \partial^{D(1)}_{n+1})+ \lognorm(\ind \partial^{D(2)}_{n+1})
			\\
			&\le \lognorm(\partial^{D(1)}_{n+1})+ \lognorm(\partial^{D(2)}_{n+1}).
		\end{align*}
		Here the last step uses Lemma~\ref{lem:lognorm_ind_res}~\ref{item:lognorm_ind}.
		
		For part~(iii), assume that~$\Gamma_0=\IZ$.
		For~$i\in \{1,2\}$, we fix $\alpha_i$-embeddings $f(i)_*\colon C(i)_*\to D(i)_*$.
		Let $C(0)_*$ be the usual free $\IZ\Gamma_0$-resolution of~$\IZ$ and let~$K\in \IN$ be large.
		By Example~\ref{ex:integers}, there exists an $\alpha_0$-embedding $C(0)_*\to D(0)_*$ satisfying $\dim_{L^\infty(\alpha_0)\ast \Gamma_0}(D(0)_j)<1/K$, $\|\partial^{D(0)}_j\|\le K$, and $\|r(0)_j\|\le K$.
		Then Lemma~\ref{lem:cone}~\ref{item:cone lognorm} yields
		\begin{align*}
			\lognorm(\partial^{D}_{n+1})
			&\le 1/K\cdot \log_+\bigl(K+(\|\ind f(1)_n\|+\|\ind f(2)_n\|)\cdot \|\ind g_n\|\cdot K\bigr)
			\\
			&\quad + \lognorm(\ind \partial^{D(1)}_{n+1}\oplus \ind\partial^{D(2)}_{n+1})
			\\
			&\le 1/K\cdot \log_+(LK) + \lognorm(\ind\partial^{D(1)}_{n+1})+\lognorm(\ind\partial^{D(2)}_{n+1})
			\\
			&\le 1/K\cdot \log_+(LK)+\lognorm(\partial^{D(1)}_{n+1})+\lognorm(\partial^{D(2)}_{n+1}).
		\end{align*}
		Here $L\coloneqq 1+(\|f(1)_n\|+\|f(2)_n\|)\cdot \|g_n\|$ is independent of~$K$ and~$D(0)_*$ and the last step uses Lemma~\ref{lem:lognorm_ind_res}~\ref{item:lognorm_ind}.
		Part~(iii) follows by sending $K\to \infty$.
	\end{proof}
\end{prop}

Proposition~\ref{prop:amalg} generalises to fundamental groups of finite graphs of groups (with infinite cyclic or trivial edge groups).

\subsection{Free groups}\label{subsec:freegroups}

We compute~$\medim$ and~$\mevol$ for standard actions of free groups. 

\begin{prop}\label{prop:freegroups}
  Let $d\in \N_{>0}$, let $F_d$ be a free group of rank~$d$,
  and let~$\alpha$ be a standard
  action of~$F_d$. Then:
  \begin{enumerate}[label=\enum]
  \item For all~$n \in \N$, we have~$\mevol_n (\alpha) = 0$;
  \item For all~$n \in \N \setminus \{1\}$, we have~$\medim_n^Z (\alpha) = 0$;
  \item We have~$\medim^\Z_1 (\alpha) \geq d-1$;
  \item If $\alpha$ is the profinite completion with respect to a directed system~$(\Gamma_i)_{i\in I}$ of finite index normal subgroups of~$F_d$ with $\bigcap_{i\in I}\Gamma_i=1$,
    then
    \[ \medim^Z_1 (\alpha) = d - 1.
    \]
  \end{enumerate}
\end{prop}
\begin{proof}
	Parts~(i) and~(ii) follow from Corollary~\ref{cor:deg0infiniteme} and Lemma~\ref{lem:cd}.
	Alternatively, one can apply Proposition~\ref{prop:amalg} by viewing~$F_d$ as a free product of copies of~$\IZ$.

 (iii)
  By Theorem~\ref{thm:l2upper}, we have
  \[
    \medim^\Z_1 (\alpha)
    \geq \ltb 1{F_d}
    = d - 1.
  \]
  
(iv)
  Let $(\Gamma_i)_{i \in I}$ be directed system
  with~$\alpha\colon F_d \actson \widehat \Gamma_*$.
  A direct computation shows that
  \[ \widehat b_1(F_d,\Gamma_*;Z) = d-1.
  \]
  Therefore, by the dynamical upper bound (Theorem~\ref{thm:dynupperproofsec})
  we have 
  \[ d - 1 = \widehat b_1(F_d,\Gamma_*;Z) \leq \medim^Z_1 (\alpha).  
  \]
  Conversely, we obtain~$\medim^Z_1 (\alpha) \leq d-1$ from Lemma~\ref{lem:finite_index}.
  Indeed, since~$\Gamma_i$ is a free group of rank $1+[F_d:\Gamma_i](d-1)$, Lemma~\ref{lem:finite_index}
  yields
  \[
  	\medim^Z_1(\alpha)\le \frac{1}{[F_d:\Gamma_i]}+d-1.
  \]
  Since~$[F_d:\Gamma_i]\to \infty$ as~$i\to \infty$, the claim follows.
\end{proof}

\subsection{Surface groups}
\label{sec:surface}

We compute~$\medim$ and~$\mevol$ for standard actions of surface groups.

\begin{prop}
\label{prop:surface_groups}
	Let~$\Sigma_g$ be a closed orientable surface of genus~$g$.
	Let~$\alpha$ be a standard action of~$\pi_1(\Sigma_g)$.
	Then:
	\begin{enumerate}[label=\enum]
		\item For all~$n\in \IN$, we have $\mevol_n(\alpha)=0$;
		\item For all~$n\in \IN\setminus \{1\}$, we have $\medim^Z_n(\alpha)=0$;
		\item We have $\medim^\IZ_1(\alpha)\ge 2g-2$;
		\item If $\alpha$ is the profinite completion with respect to a directed system~$(\Gamma_i)_{i\in I}$ of finite index normal subgroup of~$\pi_1(\Sigma_g)$ with $\bigcap_{i\in I}\Gamma_i=1$, then
		\[
			\medim^Z_1(\alpha)=2g-2.
		\]
	\end{enumerate}
	\begin{proof}
		Parts~(i) and~(ii) in degree~$0$ follow from Corollary~\ref{cor:deg0infiniteme} since~$\pi_1(\Sigma_g)$ is infinite.
		In positive degrees, the claims follow from Proposition~\ref{prop:amalg} and Proposition~\ref{prop:freegroups} by viewing~$\pi_1(\Sigma_g)$ as an iterated amalgamated product of free groups over infinite cyclic subgroups.
		
		(iii) By Theorem~\ref{thm:l2upper}, we have
		\[
			\medim^\IZ_1(\alpha)\ge b^{(2)}_1(\pi_1(\Sigma_g))=2g-2.
		\]
		
		(iv) Let $(\Gamma_i)_{i \in I}$ be a directed system
  with~$\alpha\colon \pi_1(\Sigma_g) \actson \widehat \Gamma_*$.
  A direct computation shows that
  \[ \widehat b_1(\pi_1(\Sigma_g),\Gamma_*;Z) = 2g-2.
  \]
  Therefore, by the dynamical upper bound (Theorem~\ref{thm:dynupperproofsec})
  we have 
  \[ 2g-2 = \widehat b_1(\pi_1(\Sigma_g),\Gamma_*;Z) \leq \medim^Z_1 (\alpha).  
  \]
  Conversely, we obtain~$\medim^Z_1 (\alpha) \leq 2g-2$ from Lemma~\ref{lem:finite_index}.
  Indeed, since~$\Gamma_i$ is the fundamental group of a surface of genus $1+[\pi_1(\Sigma_g):\Gamma_i](g-1)$, Lemma~\ref{lem:finite_index}
  yields
  \[
  	\medim^Z_1(\alpha)\le \frac{2}{[\pi_1(\Sigma_g):\Gamma_i]}+2g-2.
  \]
  Since~$[\pi_1(\Sigma_g):\Gamma_i]\to \infty$ as~$i\to \infty$, the claim follows.
	\end{proof}
\end{prop}

\section{Products with an amenable factor}
\label{sec:amenable_factor}

We prove that product actions groups with an amenable factor have~$\medim$ and~$\mevol$ equal to zero.
Let~$Z$ be the integers (with the usual norm) or a finite field (with the trivial norm).

\begin{prop}
	Let~$n\in \IN$.
	Let~$\Gamma_1$ be an infinite amenable group of type~$\sfFP_{n+1}$ and let~$\Gamma_2$ be a group of type~$\sfFP_{n+1}$.
	For~$i\in \{1,2\}$, let~$\alpha_i\colon \Gamma_i\actson (X_i,\mu_i)$ be a standard action.
	We denote by $\alpha_1\times \alpha_2\colon \Gamma_1\times \Gamma_2\actson (X_1\times X_2,\mu_1\otimes \mu_2)$ the product action.
	Then 
	\begin{align*}
    \fa{r \in \{0,\dots, n+1\}}
    &\medim^Z_r(\alpha_1\times \alpha_2)=0;
    \\
    \fa{r\in \{0,\dots, n\}}
    &\mevol_r(\alpha_1\times \alpha_2)=0.
  \end{align*}
  	\begin{proof}
		By Lemma~\ref{lem:integers_field}, we may assume that~$Z=\IZ$.
		For~$i\in \{1,2\}$, we denote $R_i\coloneqq L^\infty(\alpha_i)\ast \Gamma_i$.
		Since~$\Gamma_i$ is of type~$\sfFP_{n+1}$, there exists a free $\IZ\Gamma_i$-resolution~$C(i)_*$ of~$\IZ$ such that the $\IZ\Gamma_i$-module~$C(i)_r$ is finitely generated for all~$r\le n+1$.
		Then we have an $\alpha_2$-embedding $C(2)_*\to D(2)_*\coloneqq \ind_{\IZ\Gamma_2}^{R_2} C(2)_*$ with $\dim_{R_2} (D(2)_r)=\rk_{\IZ\Gamma_2}(C(2)_r)<\infty$.
		Since~$\Gamma_1$ is amenable, by Theorem~\ref{thm:amenable} there exists~$K\in \IR_{>0}$ such that for all~$\varepsilon\in \IR_{>0}$ there exists an $\alpha_1$-embedding $C(1)_*\to D(1)_*$ such that for all~$r\in \{0,\ldots,n+1\}$, we have $\dim_{R_1} (D(1)_r)<\varepsilon$ and $\|\partial^{D(1)}_r\|\le K$.
		Consider the ring extension
		\[
			R_1\otimes_{\IZ} R_2 = (L^\infty(\alpha_1)\ast \Gamma_1)\otimes_{\IZ} (L^\infty(\alpha_2)\ast \Gamma_1)\to L^\infty(\alpha_1\times \alpha_2)\ast (\Gamma_1\times \Gamma_2)\eqqcolon R.
		\]
		Then the composition
		\[
			C(1)_*\otimes_{\IZ} C(2)_*\to D(1)_*\otimes_\IZ D(2)_*\to \ind_{R_1\otimes_{\IZ} R_2}^R \bigl(D(1)_*\otimes_{\IZ} D(2)_*\bigr)
		\]
		is an $(\alpha_1\times \alpha_2)$-embedding.
		Indeed, given measurable subsets~$A_i\subset X_i$ for~$i\in \{1,2\}$, we have
		\[
			\ind_{R_1\otimes_{\IZ} R_2}^R\bigl(\spann{A_1}\otimes_{\IZ} \spann{A_2}\bigr) \cong \spann{A_1\times A_2}.
		\]
		Thus, for every~$r\in \{0,\ldots,n+1\}$, we have
		\begin{align*}
			\dim_{R}\bigl((\ind_{R_1\otimes_{\IZ}R_2}^R D(1)_*\otimes_{\IZ} D(2)_*)_r\bigr)
			&= 
			\sum_{p+q=r} \dim_{R_1}(D(1)_p)\cdot \dim_{R_2}(D(2)_q)
			\\
			&\le \varepsilon\cdot \sum_{q=0}^r \dim_{R_2}(D(2)_q).
		\end{align*}
		We can choose~$\varepsilon$ arbitrarily small and conclude $\medim^\IZ_r(\alpha_1\times \alpha_2)=0$.
		
		The boundary map of the chain complex $D(1)_*\otimes D(2)_*$ is given by
		\[
			\partial_r\colon \bigoplus_{p+q=r} D(1)_p\otimes_{\IZ} D(2)_q\to \bigoplus_{p+q=r-1} D(1)_p\otimes_{\IZ} D(2)_q,
		\]
		\[
			\partial_r = \bigoplus_{p+q=r} (\partial^{D(1)}_p\otimes \id_{D(2)_q}) \oplus (-1)^p (\id_{D(1)_p}\otimes \partial^{D(2)}_q).
		\]
		Then we have
		\begin{align*}
			&\lognorm(\ind_{R_1\otimes_\IZ R_2}^R \partial_{r+1})
			\\
			&\le \sum_{p+q=r+1} \dim_R(\ind_{R_1\otimes_\IZ R_2}^R D(1)_p\otimes_{\IZ} D(2)_q)
			\\
			&\quad \quad \quad \cdot  \log_+\bigl(\|\ind_{R_1\otimes_\IZ R_2}^R \partial^{D(1)}_p\otimes \id_{D(2)_*}\|+ \|\ind_{R_1\otimes_\IZ R_2}^R \id_{D(1)_*}\otimes \partial^{D(2)}_q\|\bigr)
			\\
			&\le \sum_{p+q=r+1} \dim_{R_1}(D(1)_p) \cdot \dim_{R_2}(D(2)_q)\cdot \log_+(\|\partial^{D(1)}_p\|+\|\partial^{D(2)}_q\|)
			\\
			&\le \varepsilon\cdot \sum_{q=1}^{r+1} \dim_{R_2}(D(2)_q)\cdot \log_+(K+\|\partial^{D(2)}_q\|).
		\end{align*}
		Again, we can choose~$\varepsilon$ arbitrarily small and conclude $\mevol_r(\alpha_1\times \alpha_2)=0$.
	\end{proof}
\end{prop}

\section{Finite index subgroups}\label{sec:finindex}

\subsection{Induction and restriction}

We prove proportionality results for~$\medim$ and~$\mevol$ of standard actions of finite index subgroups.
  
Let~$\Lambda$ be a finite index subgroup of~$\Gamma$.
If $\alpha\colon \Gamma\actson (X,\mu)$ is a standard $\Gamma$-action, then $\alpha|_\Lambda\colon \Lambda\actson (X,\mu)$ is a standard $\Lambda$-action.
The ring inclusion $\LinftyXLambda\ast \Lambda\to \LinftyX\ast \Gamma$ induces a restriction functor $\res^{\LinftyX\ast \Gamma}_{\LinftyXLambda\ast \Lambda}$ on module categories.
  
  If $\beta\colon \Lambda\actson (Y,\nu)$ is a standard $\Lambda$-action, then $\ind_\Lambda^\Gamma \beta\colon \Gamma\actson (\Gamma\times_\Lambda Y, \ind_\Lambda^\Gamma \nu)$ is a standard $\Gamma$-action.
  Note that we rescale $\ind_\Lambda^\Gamma \nu$ so that it is a probability measure.   

\begin{prop}
\label{prop:finite_index}
  Let $n \in \N$, let $\Gamma$ be a group of type~$\sfFP_{n+1}$,
  and let $\Lambda \subset \Gamma$ be a subgroup of finite index.
  \begin{enumerate}[label=\enum]
  \item
    If $\alpha\colon \Gamma\actson (X,\mu)$ is a standard $\Gamma$-action,
    then
    \begin{align*}
    \medim^Z_n (\alpha|_\Lambda)
    &\le [\Gamma:\Lambda]\cdot \medim^Z_n(\alpha)
    \\
    \mevol_n (\alpha|_\Lambda)
    &\le [\Gamma:\Lambda]\cdot \mevol_n(\alpha).
    \end{align*}
  \item\label{item:finite_index_ind}
    If $\beta\colon \Lambda\actson (Y,\nu)$ is a standard $\Lambda$-action,
    then
    \begin{align*}
    \medim^Z_n (\ind_\Lambda^\Gamma \beta)
    & = \frac1{[\Gamma:\Lambda]} \cdot \medim^Z_n (\beta)
    \\
    \mevol_n (\ind_\Lambda^\Gamma \beta)
    & = \frac1{[\Gamma:\Lambda]} \cdot \mevol_n (\beta).
    \end{align*}
  \end{enumerate}
\end{prop}
\begin{proof}
  (i) Let $C_*\to D_*$ be an $\alpha$-embedding.
  Then $\res^{Z\Gamma}_{Z\Lambda}C_*\to \res^{\LinftyX\ast \Gamma}_{\LinftyXLambda\ast \Lambda} D_*$ is a $\alpha|_\Lambda$-embedding with
  \begin{align*}
  	\dim_{\LinftyXLambda\ast \Lambda}\bigl(\res^{\LinftyX\ast \Gamma}_{\LinftyXLambda\ast \Lambda}D_n\bigr)
	&\le [\Gamma:\Lambda]\cdot \dim_{\LinftyX\ast \Gamma}(D_n)
	\\
	\lognorm\bigl(\res^{\LinftyX\ast \Gamma}_{\LinftyXLambda\ast \Lambda} \partial^D_n\bigr)
	&\le [\Gamma:\Lambda]\cdot \lognorm(\partial^D_n)
  \end{align*}
  by Lemma~\ref{lem:lognorm_ind_res}~\ref{item:lognorm_res}.
  
  (ii) Since~$\ind_\Lambda^\Gamma\beta$ is weakly bounded orbit equivalent to~$\beta$ (Definition~\ref{def:wbOE}), this follows from Theorem~\ref{thm:wbOE}.
\end{proof}

\begin{lem}
\label{lem:finite_index}
	Let~$n\in \IN$, let~$\Gamma$ be a residually finite group of type~$\sfFP_{n+1}$, and let~$\Gamma_*=(\Gamma_i)_{i\in I}$ be a directed system of finite index normal subgroups of~$\Gamma$ with $\bigcap_{i\in I} \Gamma_i=1$.
	Let~$\alpha\colon \Gamma\actson (\widehat{\Gamma}_*,\mu)$ be the profinite completion of~$\Gamma$ with respect to~$\Gamma_*$.
	Fix~$i\in I$.
	Let~$C_*$ be a free $Z\Gamma_i$-resolution of~$Z$.
	Let~$D_*$ be a free $Z\Gamma_i$-chain complex augmented over~$Z$ and let~$f_*\colon C_*\to D_*$ be a $Z\Gamma_i$-chain map.
	Then we have
	\begin{align*}
		\fa{r\in \{0,\ldots,n+1\}}
		\medim^Z_r(\alpha)
		&\le \frac{1}{[\Gamma:\Gamma_i]}\cdot \rk_{Z\Gamma_i}(D_r)
		\\
		\fa{r\in \{0,\ldots,n\}}
		\mevol_{r}(\alpha)
		&\le \frac{1}{[\Gamma:\Gamma_i]}\cdot \rk_{\IZ\Gamma_i}(D_{r+1})\cdot \log_+\|\partial^{D}_{r+1}\|.
	\end{align*}
	\begin{proof}
		The free $Z\Gamma_i$-chain complex~$D_*$ has chain modules of the form~$D_r\cong \bigoplus_{I_r}Z\Gamma_i$ for some index set~$I_r$ and differentials~$\partial^D_r$. 
		We define an associated marked projective chain complex~$\widehat{D}_*$ over~$R\coloneqq \LinftyX\ast \Gamma$ as follows.
		Consider the map $p_i\colon \widehat{\Gamma}_*\to \Gamma/\Gamma_i$ and set $A\coloneqq p_i^{-1}(1_\Gamma\Gamma_i)$.
		The subset~$A$ of~$\widehat{\Gamma}_*$ is $\Gamma_i$-invariant and has measure $\mu(A)=1/[\Gamma:\Gamma_i]$.
		 We define the $R$-chain modules $\widehat{D}_r\coloneqq \bigoplus_{I_r}\spann{A}_{R}$ and differentials~$\partial^{\widehat{D}}_r$ given by the same matrix (with entries in~$Z\Gamma_i$) as~$\partial^D_r$.
		 Then~$\widehat{D}_*$ is a marked projective $R$-chain complex augmented over~$L^\infty(\alpha)$ with
		 \begin{align*}
		 	\dim_R(\widehat{D}_r)
			&= \mu(A)\cdot \# I_r
			= \frac{1}{[\Gamma:\Gamma_i]}\cdot \rk_{Z\Gamma_i}(D_r)
			\\
			\lognorm(\partial^{\widehat{D}}_{r+1})
			&\le \dim_R(D_{r+1})\cdot \log_+\|\partial^{\widehat{D}}_{r+1}\|
			\le \frac{1}{[\Gamma:\Gamma_i]}\cdot \rk_{\IZ\Gamma_i}(D_{r+1})\cdot \log_+\|\partial^D_{r+1}\|.
		 \end{align*}
		 The composition of $Z\Gamma_i$-chain maps $C_*\to D_*\to \widehat{D}_*$ extends to a $Z\Gamma$-chain map $Z\Gamma\otimes_{Z\Gamma_i}C_*\to \widehat{D}_*$.
		 Let~$C'_*$ be a free $Z\Gamma$-resolution of~$Z$.
		 Then we obtain an $\alpha$-embedding as the composition
		 \[\begin{tikzcd}
		 	C'_*\ar{r}\ar[two heads]{d}
			& Z\Gamma\otimes_{Z\Gamma_i}C_*\ar{r}\ar[two heads]{d}
			& \widehat{D}_*\ar[two heads]{d}
			\\
			Z\ar{r}
			& Z[\Gamma/\Gamma_i]\ar{r}
			& L^\infty(\alpha)
		 \end{tikzcd}\]
		 where the $Z\Gamma$-chain map $C'_*\to Z\Gamma\otimes_{Z\Gamma_i}C_*$ is induced by the $Z\Gamma$-map $Z\to Z[\Gamma/\Gamma_i]$ that sends~$1\in Z$ to the (finite) sum of all $Z$-basis elements of~$Z[\Gamma/\Gamma_i]$.
	\end{proof}
\end{lem}	
	
	\begin{rem}
		We outline an alternative proof of Lemma~\ref{lem:finite_index} using results from Part~\ref{part:dyn}.
		However, it involves actions that are not essentially free to which our setup could be extended.
		
		In the situation of Lemma~\ref{lem:finite_index},
		we denote by $\alpha_i\colon \Gamma\actson \Gamma/\Gamma_i$ the translation action.
		Since~$\alpha_i$ is weakly contained in~$\alpha$ (Example~\ref{ex:wc-profin}), by monotonicity of~$\medim$ and~$\mevol$ (Theorem~\ref{thm:wc}) we have
		\begin{align*}
			\medim^Z_r(\alpha)
			&\le \medim^Z_r(\alpha_i)
			\\
			\mevol_{r}(\alpha)
			&\le \mevol_{r}(\alpha_i).
		\end{align*}
		We denote by $\beta_i\colon \Gamma_i\actson \pt$ the trivial action.
		Then $\alpha_i\cong \ind_{\Gamma_i}^\Gamma \beta_i$ and by Proposition~\ref{prop:finite_index}~\ref{item:finite_index_ind} we have
		\begin{align*}
			\medim^Z_r(\alpha_i)
			&= \frac{1}{[\Gamma:\Gamma_i]}\cdot \medim^Z_r(\beta_i)
			\\
			\mevol_{r}(\alpha_i)
			&= \frac{1}{[\Gamma:\Gamma_i]}\cdot \mevol_{r}(\beta_i).
		\end{align*}
	Under the ring isomorphism $Z\Gamma_i\cong L^\infty(\beta_i)\ast \Gamma_i$, the $Z\Gamma_i$-chain map~$f_*\colon C_*\to D_*$ is a $\beta_i$-embedding and hence
	\begin{align*}
		\medim^Z_r(\beta_i)
		&\le \rk_{Z\Gamma_i}(D_r)
		\\
		\mevol_{r}(\beta_i)
		&\le \lognorm(\partial^{D}_{r+1})\le \rk_{\IZ\Gamma_i}(D_{r+1})\cdot \log_+\|\partial^{D}_{r+1}\|.
	\end{align*}
	Combining the above inequalities proves Lemma~\ref{lem:finite_index}.
	\end{rem}

\subsection{The cheap rebuilding property}
\label{sec:CRP}

We prove that the profinite completion action of groups satisfying (an equivariant version of) the algebraic cheap rebuilding property has~$\medim$ and~$\mevol$ equal to zero.
The algebraic cheap rebuilding property was introduced by the authors~\cite{LLMSU} modeled after the (geometric) cheap rebuilding property of Abert--Bergeron--Fraczyk--Gaboriau~\cite{ABFG21}.

\begin{defn}[{\cite[Definition~4.18]{LLMSU}}]
\label{defn:CRP}
	Let~$\Gamma$ be a group and let~$n\in \IN$.
	We say that~$\Gamma$ satisfies the \emph{algebraic cheap $n$-rebuilding property} if there exist a free $\IZ\Gamma$-resolution~$C_*$ of~$\IZ$ with $C_j$ finitely generated for all~$j\le n$ and~$\kappa\in \IR_{\ge 1}$ such that for all~$T\in \IR_{\ge 1}$ and every residual chain~$\Lambda_*$ in~$\Gamma$ there exists~$i_0\in \IN$ such that for all~$i\ge i_0$, the $\IZ$-chain complex of $\Lambda_i$-coinvariants~$(C_*)_{\Lambda_i}$ is a $\IZ$-chain homotopy retract of a free $\IZ$-chain complex~$C'_*$ via $\IZ$-chain maps $f_*\colon (C_*)_{\Lambda_i}\to C'_*$ and $g_*\colon C'_*\to (C_*)_{\Lambda_i}$ and $\IZ$-chain homotopy~$H_*\colon g_*\circ f_*\simeq \id_{(C_*)_{\Lambda_i}}$ satisfying the following for all~$j\le n$:
	\begin{align*}
		\rk_\IZ(C'_j)
		&\le \frac{\kappa}{T}\cdot \rk_\IZ\bigl((C_j)_{\Lambda_i}\bigr)
		\\
		\|\partial^{C'}_j\|,\|f_j\|, \|g_j\|, \|H_j\|
		&\le \exp(\kappa)\cdot T^\kappa.
	\end{align*}
\end{defn}

The key facts about the algebraic cheap rebuilding property are that it implies the vanishing of torsion homology growth~\cite[Lemma~4.22]{LLMSU}, admits a bootstrapping theorem~\cite[Proposition~4.23]{LLMSU}, and is satisfied by the group of integers~$\IZ$~\cite[Examples~4.21]{LLMSU}.

In order to establish a relationship to the vanishing of~$\medim$ and~$\mevol$, we define an equivariant version of the algebraic cheap rebuilding property.
The difference is that in Defintion~\ref{defn:equivariant_CRP} we require the existence of a $\IZ\Lambda_i$-chain homotopy retraction of~$\res^{\IZ\Gamma}_{\IZ\Lambda_i} C_*$, while in Definition~\ref{defn:CRP} we require only the existence of a (non-equivariant) $\IZ$-chain homotopy retraction of the $\Lambda_i$-coinvariants~$(C_*)_{\Lambda_i}$.
	
\begin{defn}
\label{defn:equivariant_CRP}	
	Let~$\Gamma$ be a group and let~$n\in \IN$.
	We say that~$\Gamma$ satisfies the \emph{algebraic cheap equivariant $n$-rebuilding property} ($\sfCERP_n$ for short) if there exists a free $\IZ\Gamma$-resolution~$C_*$ of~$\IZ$ with $C_j$ finitely generated for all~$j\le n$ and~$\kappa\in \IR_{\ge 1}$ such that for all~$T\in \IR_{\ge 1}$ and every residual chain~$\Lambda_*$ in~$\Gamma$ there exists~$i_0\in \IN$ such that for all~$i\ge i_0$, the $\IZ\Lambda_i$-chain complex $\res^{\IZ\Gamma}_{\IZ\Lambda_i}C_i$ is a $\IZ\Lambda_i$-chain homotopy retract of a free $\IZ\Lambda_i$-chain complex~$E_*$ via $\IZ\Lambda_i$-chain maps $f_*\colon \res^{\IZ\Gamma}_{\IZ\Lambda_i}C_*\to E_*$ and $g_*\colon E_*\to \res^{\IZ\Gamma}_{\IZ\Lambda_i}C_*$ and a $\IZ\Lambda_i$-chain homotopy $H_*\colon g_*\circ f_*\simeq \id_{\res^{\IZ\Gamma}_{\IZ\Lambda_i}C_*}$ satisfying the following for all~$j\le n$:
	\begin{align*}
		\rk_{\IZ\Lambda_i}(E_j)
		&\le \frac{\kappa}{T}\cdot \rk_{\IZ\Lambda_i}(\res^{\IZ\Gamma}_{\IZ\Lambda_i}C_*)
		\\
		\|\partial^{E}_j\|, \|f_j\|, \|g_j\|, \|H_j\|
		&\le \exp(\kappa)\cdot T^\kappa
	\end{align*}
\end{defn}

	Clearly, the algebraic cheap equivariant rebuilding property implies the non-equivariant one.
	We do not know if the converse holds.

\begin{rem}
	The algebraic cheap equivariant rebuilding property is a bootstrappable property of residually finite groups in the sense of~\cite[Definition~3.4]{LLMSU}.
	The proof is analogous to the non-equivariant case~\cite[Proposition~4.23~(i)]{LLMSU}.
	Hence this property admits a bootstrapping theorem~\cite[Theorem~3.6]{LLMSU}.
	Since the group of integers~$\IZ$ satisfies~$\sfCERP_n$ for all~$n\in \IN$, see~\cite[Example~4.21]{LLMSU}, repeated applications of the bootstrapping theorem show that many groups satisfy~$\sfCERP_n$ for suitable~$n$.
	For example, infinite elementary amenable groups of type~$\sfFP_\infty$ satisfy~$\sfCERP_n$ for all~$n$ and the special linear group~$\mathrm{SL}_d(\IZ)$ satisfies~$\sfCERP_{d-2}$ for~$d\ge 3$.
\end{rem}

\begin{thm}
	Let~$n\in \IN$ and let~$\Gamma$ be a group satisfying~$\sfCERP_{n+1}$.
	Let~$\Lambda_*$ be a residual chain in~$\Gamma$ and let~$\alpha\colon \Gamma\actson (X,\mu)$ be the profinite completion of~$\Gamma$ with respect to~$\Lambda_*$.
	Then
	\begin{align*}
		\fa{r\in \{0,\ldots,n+1\}}
		&\medim^Z_r(\alpha)=0
		\\
		\fa{r\in \{0,\ldots,n\}}
		&\mevol_{r}(\alpha)=0.
	\end{align*}
	\begin{proof}
		By Lemma~\ref{lem:integers_field}, we may assume that~$Z=\IZ$.
		We use the notation $C_*,\kappa,T,i_0$ from Defintion~\ref{defn:equivariant_CRP}.
		For~$i\ge i_0$, we apply Lemma~\ref{lem:finite_index} to the $\IZ\Lambda_i$-chain map $f_*\colon \res^{\IZ\Gamma}_{\IZ\Lambda_i}C_*\to E_*$, where the free $\IZ\Lambda_i$-chain complex~$E_*$ satisfies
		\begin{align*}
			\rk_{\IZ\Lambda_i}(E_r)
			&\le \frac{\kappa}{T}\cdot [\Gamma:\Lambda_i]\cdot \rk_{\IZ\Gamma}(C_r)
			\\
			\|\partial^{E}_r\|
			&\le \exp(\kappa)\cdot T^\kappa.
		\end{align*}
		Hence Lemma~\ref{lem:finite_index} yields
		\begin{align*}
			\medim^{\IZ}_r(\alpha)
			&\le \frac{\kappa}{T}\cdot \rk_{\IZ\Gamma}(C_r)
			\\
			\mevol_{r}(\alpha)
			&\le \frac{\kappa^2(1+\log(T))}{T}\cdot \rk_{\IZ\Gamma}(C_{r+1})
		\end{align*}
	Sending $T\to \infty$ finishes the proof.
	\end{proof}
\end{thm}

%% file: wc.tex
\section{Weak containment}
\label{sec:weak_containment}

In this section, we prove monotonicity of measured embedding dimension and measured embedding volume under 
weak containment of actions (Theorem~\ref{thm:wc}). After recalling the definition of weak containment, we 
introduce an upper bound on the norm and a way of translating chain complexes over a crossed product ring
to a different action.

\subsection{Preliminaries on weak containment}

We briefly recall the definition of weak containment and a few examples.
\begin{defn}[weak containment, {{\cite[p.~64]{Kechris-globalaspects}}}]
	\label{def:wc}
	Let~$\Gamma$ be a group and~$\alpha\colon \Gamma\actson (X,\mu)$ and~$\beta\colon \Gamma \actson (Y,\nu)$ 
	be \pmp\ actions of~$\Gamma$ on standard probability spaces. 
	We say that~$\alpha$ is \emph{weakly contained} in~$\beta$ (in symbols~$\alpha\wkcont\beta$)
	if for all~$n\in \N$, measurable sets~$A_1,\dots,A_n\subseteq X$, finite sets $F\subseteq \Gamma$, 
	and~$\epsilon >0$, there are measurable sets~$B_1,\dots,B_n\subseteq Y$ such that
	\indexw{weak containment}\indexnot{<}{$\wkcont$}{weak containment}
	\[
		\forall_{\gamma\in F}\ \forall_{i,j\in \{1,\dots,n\}}\  
		\bigl|\mu(\gamma^\alpha(A_i) \cap A_j) - \nu (\gamma^\beta (B_i) \cap B_j)\bigr| < \epsilon.
	\]
	Here we write~$\gamma^\alpha$ (resp.~$\gamma^\beta$) when~$\gamma \in \Gamma$ is acting on~$(X, \mu)$ via $\alpha$ (resp.\ on~$(Y,\nu)$ via~$\beta$).
\end{defn}

Weak containment is transitive on \pmp\ actions of a given group.

\begin{ex}[]
	Let~$\Gamma$ be a group and~$\alpha\colon \Gamma\actson (X,\mu)$ and~$\beta\colon \Gamma \actson (Y,\nu)$ 
	be \pmp\ actions of~$\Gamma$ on standard probability spaces. 
	It is straightforward to check that~$\alpha\wkcont\alpha\times \beta$,
	where~$\alpha\times \beta \colon  \Gamma\actson(X\times Y, \mu\otimes \nu)$ is the product action.
\end{ex}

\begin{ex}[]
	\label{ex:wc-profin}
	Let~$\Gamma$ be a residually finite group and~$\Lambda_*$ be a residual chain of~$\Gamma$.
	Let~$(X,\mu)$ be the inverse limit of the system~$(\Gamma/\Lambda_i)_{i\in \N}$,
	equipped with the Haar measure. Then, we have an action 
	$\Gamma\actson X$ via left translation. This action is weakly contained in the 
	profinite completion action~$\Gamma\actson\widehat\Gamma$~\cite[Proposition~2.3]{Kechris-space-free-group}. 
\end{ex}

For countably infinite groups, there is a smallest action with respect to weak containment.

\begin{ex}[Bernoulli shift]
	\label{ex:Bernoulli}
	Let~$(X,\mu)$ be a non-trivial probability space (i.e., $\mu$ is not concentrated in one point)
	 and~$\Gamma$ be a countably infinite group.
	The \emph{Bernoulli shift} of~$\Gamma$ on~$X$
	is the action of~$\Gamma$ on~$\prod_\Gamma X$ (endowed with the product measure)
	via shifting of the factors. 
	Ab\'ert and Weiss proved that the Bernoulli shift 
	is weakly contained in every free \pmp\ action of~$\Gamma$ \cite[Theorem~1]{Abert-Weiss-Bernoulli}.
	\indexw{Bernoulli shift}
\end{ex}

\begin{ex}[amenable groups]
	\label{ex:wc-amenable}
	Let~$\Gamma$ be an infinite amenable group. Then all free \pmp\ actions of~$\Gamma$ on standard probability spaces are 
	weakly contained in the Bernoulli shift~\cite[p.~91]{Kechris-globalaspects}. As a consequence in this situation all free \pmp\ actions on standard probability spaces 
	are \emph{weakly equivalent} (Example~\ref{ex:Bernoulli}).
\end{ex}

\begin{defn}[ergodic action]
An action~$\Gamma\actson (X,\mu)$ is \emph{ergodic}\indexw{ergodic} if for every 
	measurable subset~$A\subseteq X$ with~$\Gamma\cdot A = A$, we have
	\[
		\mu(A) = 0 \qor \mu(X\setminus A) = 0.
	\]
\end{defn}

\begin{defn}[EMD*, {{\cite[Definition~4.4, Proposition~4.5]{Kechris-space-free-group}}}]
	\label{def:EMD}	
	An infinite countable residually finite group~$\Gamma$ satisfies~EMD* if every ergodic standard probability
	action of~$\Gamma$ is weakly contained in the profinite completion action~$\Gamma\actson \widehat{\Gamma}$.
	\indexw{EMD*}
	\indexnot{EMD}{EMD*}{Property EMD*}
\end{defn}

For the examples below, note that Tucker-Drob proved that for all groups property EMD* is equivalent to a similarly 
defined property MD \cite[Theorem~1.4]{tuckerdrob-weak}.

\begin{ex}[EMD*]
	\label{ex:EMD}
	The following groups satisfy EMD*:
	\begin{itemize}
		\item countable free groups \cite[Theorem~3.1, Proposition~4.5]{Kechris-space-free-group};
		\item residually finite infinite amenable groups \cite[Proposition~13.2]{Kechris-globalaspects};
		\item free products of non-trivial groups $\Gamma * \Lambda$, where each is either finite or has property EMD*~\cite[Theorem~4.8]{tuckerdrob-weak}
		\item subgroups of groups with property~EMD*~\cite[p.~486]{Kechris-space-free-group};
		\item finite index extensions of groups with property~EMD*~\cite[p.~486]{Kechris-space-free-group};
		\item extensions $1 \to N \to \Gamma \to Q \to 1$ where $N$ is a finitely generated group with property~EMD* and $Q$ is a residually finite amenable group~\cite[Theorem~1.4]{Bowen-TD-coinduction}.
		\end{itemize}
		The previous properties show that also the following geometric families of groups have EMD*:
		\begin{itemize}
		\item fundamental groups of connected closed surfaces~\cite[Theorem~1.4]{Bowen-TD-coinduction};
		\item fundamental groups of connected compact hyperbolic $3$-manifolds with emp\-ty or toroidal boundary~\cite[Proposition~5.2]{FLMQ} (see also~\cite[Corollary~3.11]{FLPS}).
	\end{itemize}
	More examples can be found in the survey by Burton and Kechris \cite[pp.~2698f]{Burton-Kechris-weak-containment}.
\end{ex}

Many dynamical invariants are monotone under weak containment (including, e.g., 
cost~\cite[Corollary~10.14]{Kechris-globalaspects} and integral foliated simplicial volume~\cite[Theorem~1.5]{FLPS},
see Sections~\ref{sec:cost} and~\ref{sec:ifsv} for the definitions). In Theorem~\ref{thm:wc}, we will prove monotonicity
of measured embedding dimension and measured embedding volume under weak containment. 
In particular, for groups satisfying~EMD*, we can bound 
$\medim$ and~$\mevol$ of the profinite completion by the corresponding invariant of 
any action (Corollary~\ref{cor:EMDreduction}).

Instead of working directly with the definition, we will often employ the 
following characterisation of weak containment using weak neighbourhoods.

\begin{defn}[weak neighbourhoods, {{\cite[Section~1(B)]{Kechris-globalaspects}}}]
	\label{def:wk-nbhds}
	Let $\Gamma$ be a group and $(X,\mu)$ be a standard probability space.
	The \emph{weak topology} on the space of \pmp\ actions~$\Gamma\actson (X,\mu)$
	is defined by the following basic open neighbourhoods:
	Let $\alpha\colon  \Gamma \actson X$, $F\subseteq \Gamma$ be finite, $n\in \N$, $A_1, \dots, A_n\subseteq X$ be measurable, and~$\varepsilon \in \R_{>0}$. Then, \indexw{weak topology}
	\[
		\bigl\{\beta\colon  \Gamma\actson X \bigm| \forall_{\gamma\in F} \ \forall_{i \in \{1,\dots, n\}}\ 
		\mu \bigl( (\gamma^\alpha A_i) \symmdiff (\gamma^\beta A_i)\bigr) < \varepsilon\bigr\}
	\] 
	is open in the weak topology.
\end{defn}

\begin{prop}[{{\cite[Proposition~10.1]{Kechris-globalaspects}}}]
	\label{prop:wc-char}
	Let~$\alpha\colon  \Gamma \curvearrowright (X,\mu)$ and~$\beta\colon \Gamma\curvearrowright (Y,\nu)$
	 be \pmp\ actions. 
	Then, $\alpha$ is weakly contained in~$\beta$
	if and only if in every weak neighbourhood~$U$ of~$\alpha$, there is~$\beta'\in U$
	such that~$\beta'$ is isomorphic to~$\beta$ as actions on standard probability spaces, i.e., 
	there is an isomorphism~$\varphi\colon (X,\mu) \to (Y,\nu)$ of measure spaces such
	 that~$\varphi(\gamma^\alpha x) = \gamma^\beta \varphi(x)$ for all~$x\in X$ and~$\gamma\in \Gamma$. 
\end{prop}

\input{Q}

\input{transl}

\input{strict-transl}

\subsection{Proof of monotonicity}
\label{subsec:wc}

The main goal of this section is to prove the following monotonicity result
for measured embedding dimension and volume under weak containment of actions.

\begin{thm}[weak containment]\label{thm:wc}
	\indexw{weak containment}
  Let $n\in \N$, let $\Gamma$ be a group of type~$\FP_{n+1}$, and let~$\alpha\colon \Gamma\actson (X,\mu)$, 
  $\beta\colon \Gamma\actson (Y,\nu)$ be free \pmp\ actions of~$\Gamma$ on standard probability spaces.
  Let~$\alpha \prec \beta$. Let~$Z$ denote the integers (equipped with the standard norm) 
  or a finite field (equipped with the trivial norm).
  Then, we have
  \begin{align*}
     \medim_n^Z (\beta)
     & \leq \medim_n^Z(\alpha),
     \\
     \mevol_n (\beta)
     & \leq \mevol_n(\alpha).
  \end{align*}  
\end{thm}

\begin{rem}[]
	The proof of this theorem consists of several steps. We give a roadmap to the proof
	outlining the main ideas: 
	\begin{enumerate}
		\item We fix an~$\alpha$-embedding from~$C_*^\alpha$ to an augmented~$\Linfalpha$-chain complex~$D_*$
		 with~$\dim_\alpha (D_n)$ 
		 ``close'' to~$\medim_n^Z (\alpha)$ 
		 (resp.\ $\lognorm_\alpha (\partial_{n+1}^D)$ ``close'' to $\mevol_n(\alpha)$).
		\item Because~$\alpha$ is weakly contained in~$\beta$, the action~$\beta$ is (isomorphic to) an action
		$\beta'\colon \Gamma\actson X$
		``close'' to~$\alpha$ in the weak topology (see Proposition~\ref{prop:wc-char}).
		\item We can translate~$D_*$ from~$\alpha$ to~$\beta'$ (see Section~\ref{subsec:transl})
		and obtain~$(D_*)_{\beta'}$.
		However, in general, $(D_*)_{\beta'}$ will no longer be a chain complex, but only an almost chain complex 
		over~$\Linfbetap$
		and the $\alpha$-embedding~$C_*^\alpha \to D_*$ translates to an almost chain map~$C_*^{\beta'} \to
		 (D_*)_{\beta'}$ over~$\Linfbetap$. The error depends on the previous distances.
		\item We can strictify~$(D_*)_{\beta'}$ to obtain a (strict)
		$\Linfbetap$-chain complex~$\widehat D_*$. We obtain an almost $\beta'$-embedding to this complex.
		\item We can also strictify the almost $\beta'$-embedding to obtain a strict 
		$\Linfbetap$-chain map~$C_*^{\beta'} \to \widetilde D_*$ to a 
		(different) strict chain complex~$\widetilde D_*$.
		\item If~$\beta'$ is ``close'' enough to~$\alpha$, the strictified chain complex~$\widetilde D_*$ is
		``close'' to~$D_*$, thus $\dim_{\beta'} (\widetilde D_n)$
		is ``close'' to $\dim_\alpha (D_n)$ 
		(resp.\ $\lognorm_\beta (\partial_{n+1}^{\widetilde D}$) is ``close'' to $\lognorm_\alpha (\partial_{n+1}^{D})$).
		 \item We can make the error arbitrarily small, thus proving the claim.
	\end{enumerate}
	The main difficulty is making the notions of closeness precise and controlling the distances. These
	distances
	 depend on one another, thus some work needs to be done to obtain global control on the errors.
\end{rem}

The key approximation is contained in the following lemma.

\begin{lem}[]
	\label{lem:wc-tech}
	Let~$n\in \N$, let~$\Gamma$ be a group of type~$\FP_{n+1}$, 
	and let~$\alpha\colon \Gamma \actson (X,\mu)$ be a free probability measure preserving standard action.
	Let~$f_*\colon C_*^\alpha \to D_*$ be an~$\alpha$-embedding and~$\epsilon > 0$. Then, 
	there exists a weak neighbourhood~$U$ of~$\alpha$ such that for all~$\beta\in U$,  
	there exists a~$\beta$-embedding~$C_*^\beta \to \widetilde D_* $ satisfying
	\[
		\dim_\beta (\widetilde D_n) \le \dim_\alpha (D_n) + \epsilon \qand
		\lognorm_\beta (\partial_{n+1}^{\widetilde D}) \le \lognorm_\alpha (\partial_{n+1}^D) + \epsilon.
	\]
\end{lem}

\begin{proof}	
	We fix a free~$Z\Gamma$-resolution~$C_*\onto Z$ of the trivial~$Z\Gamma$-module~$Z$ with finitely
	generated~$Z\Gamma$-modules in degrees~$\le n+1$.
	We can additionally 
	assume that~$C_0 = Z\Gamma$ and that the augmentation map~$\eta\colon  Z\Gamma\to Z$
	is given by mapping all~$\gamma\in \Gamma$ to~$1\in Z$.
	As in the definition of~$\medim_n^Z$ and $\mevol_n$ (Definition~\ref{def:alpha-emb}),
	 we set~$C_*^\alpha \coloneqq \Linfalpha\otimes_{Z\Gamma} C_*$.
	Let~$f_*\colon C_*^\alpha \to D_*$ be the given~$\alpha$-embedding, i.e., 
	an $\Linfalpha$-chain map extending the identity on~$\LinftyXwc$.
	By Lemma~\ref{lem:transl-lognorm}, we can pick a neighbourhood~$U$ such that for all $\beta \in U$, we have
	\[
		\lognorm_\beta ((\partial_{n+1}^D)_\beta) \le \lognorm_\alpha (\partial_{n+1}^D) + \frac{\epsilon}{2}.
	\]
	We fix a translation-invariant constant~$\kappa$, and monotone increasing maps~$K$ and~$p$ 
	as in Theorem~\ref{thm:strict-cplx}. Moreover, we fix a translation-invariant constant
	and monotone increasing constant as in Theorem~\ref{thm:strict-map}.
	By  taking the maximum, we can also
	denote the latter by~$\kappa$ resp.\ $K$.
	We restrict the neighbourhood~$U$ to a potentially smaller neighbourhood that additionally satisfies
	the translation-invariance condition in Definition~\ref{def:transl-inv-const}
	(for $\delta \coloneqq \epsilon \slash 2$). We define
	\[
	Q(f,n) \coloneqq \max \bigl\{ Q(f_0), \dots Q(f_n)\bigr\}
	\] and
	\[M = \Bigl( K\bigl(\max \bigl\{
			\kappa_\delta(f_*), 2\cdot p(\kappa_\delta(D_*))
		\bigr\}\bigr) \cdot Q(f,n) + 1\Bigl) \cdot K(\kappa_{\delta/2}(D_*)+1).\]
	and choose~$\delta\in \R_{>0}$ such that
	\[
		M\cdot \delta \le \epsilon \qand
		M\cdot \logp (M\cdot \delta) \le \frac{\epsilon}{2} \qand \delta\le \epsilon.
	\]
	By Corollary~\ref{cor:transl-chainmaps}, we can restrict~$U$ to a 
	neighbourhood of~$\alpha$ such that for all~$\beta\in U$  the translated chain complex~$(D_*)_\beta$ is 
	a \almostcc{\delta}{n} and~$(f_*)_\beta\colon  (C_*^\alpha)_\beta \to (D_*)_\beta$
	is a \almostcm{\delta}{n}.
	By Lemma~\ref{lem:transl-C}, we have~$(C_*^\alpha)_\beta \cong C_*^\beta$.	
	
	Let~$\beta\in U$. We first strictify~$(D_*)_\beta$. By Theorem~\ref{thm:strict-cplx}, we obtain a strict marked 
	projective
	$\Linfbeta$-chain complex~$\widehat{D}_*$ such that the inclusion~$i_*\colon (D_*)_\beta \hookrightarrow \widehat D_*$
	is a \almostcm{(K(\kappa_\delta((D_*)_\beta))\cdot \delta)}{n} and
	\[
		\dgh {K(\kappa_\delta((D_*)_\beta))}{\widehat{D}_*, (D_*)_\beta, n}\le
		 K\bigl(\kappa_\delta((D_*)_\beta)\bigr)\cdot \delta.
	\]  
	By translation-invariance and because of our choice of~$U$, 
	we have \[\kappa_{\delta}((D_*)_\beta) \le \kappa_{\delta/2}(D_*)+1.\] 
	Because~$K$ is monotone increasing, we can directly
	use~$K(\kappa_{\delta/2}(D_*)+1)$ as an upper bound in the following. 
	Then, by Lemma~\ref{lem:comp-almost-chain-maps} and 
	Remark~\ref{rem:normleQ}, we have that the 
	composition $i_*\circ (f_*)_\beta\colon  C_*^\beta \to (\widehat D_*)_\beta$
	is a \almostcm{(Q(f_\beta,n)\cdot K(\kappa_{\delta/2}(D_*)+1)\cdot \delta + \delta)}{n}. Note that~$Q(f_\beta,n) 
	\le Q(f,n)$.
	
	Now, we can strictify the chain map: By Theorem~\ref{thm:strict-map}, 
	there exists a strict marked projective~$\Linfbeta$-chain complex~$\widetilde D_*$ admitting a 
	(strict) $\Linfbeta$-chain map $C_*^\beta \to \widetilde D_*$ extending the identity on~$\LinftyXwc$ with
	\[
		\dgh {K(\kappa_\delta(i_*\circ (f_*)_\beta))}{\widetilde D_*, \widehat D_*, n} \le K\bigl(\kappa_\delta(i_*\circ (f_*)_\beta)\bigr)\cdot
		Q(f,n)\cdot K(\kappa_{\delta/2}(D_*)+1)\cdot \delta.
	\]
	 
	Thus, $C_*^\beta \to \widetilde D_*$ is a~$\beta$-embedding.
	From the explicit descriptions of the transla\-tion-invariant constants 
	(see the proofs of Theorem~\ref{thm:strict-cplx} and Theorem~\ref{thm:strict-map}), 
	we obtain that
	\[
		\kappa_\delta(i_*\circ (f_*)_\beta) \le \max \bigl\{
			\kappa_\delta(f_*), 2\cdot \kappa_\delta(\widehat D_*)
		\bigr\} \le \max \bigl\{
			\kappa_\delta(f_*), 2\cdot p(\kappa_\delta(D_*))
		\bigr\}.
	\]
	Because the Gromov-Hausdorff distance satisfies the triangle inequality
	(Proposition~\ref{prop:dghchaincomplex}), we obtain
	\[
		\dgh {M}{\widetilde D_*, (D_*)_\beta,n} \le M\cdot \delta
	\]
	with \[M = \Bigl( K\bigl(\max \bigl\{
			\kappa_\delta(f_*), 2\cdot p(\kappa_\delta(D_*))
		\bigr\}\bigr) \cdot Q(f,n) + 1\Bigl) \cdot K(\kappa_{\delta/2}(D_*)+1),\]
	as defined at the beginning of this proof.	 
	For the dimension, we thus obtain that
	\begin{align*}
		\dim_\beta (\widetilde D_n) &\le \dim_\beta ((D_n)_\beta) + M\cdot \delta 
							&(\text{Proposition~\ref{prop:lognorm}~\ref{itm:lognorm-GH}})\\
		&= \dim_\alpha (D_n) + M\cdot \delta & (\text{Remark~\ref{rem:dim-transl}})\\
		&\le \dim_\alpha (D_n) +\epsilon. & (\text{choice of }\delta)
	\end{align*}
	For the lognorm, we have
	\begin{align*}
		\lognorm_\beta (\partial_{n+1}^{\widetilde D})
		&\le \lognorm_\beta ((\partial_{n+1}^D)_\beta) + M\cdot \logp (M\cdot \delta) &
		(\text{Proposition~\ref{prop:lognorm}~\ref{itm:lognorm-GH}})\\
		&\le \lognorm_\beta ((\partial_{n+1}^D)_\beta) + \frac{\epsilon}{2}
		&(\text{choice of }\delta)\\
		&\le \lognorm_\alpha (\partial_{n+1}^D) + \epsilon.
		&(\text{choice of }U) & &\qedhere
	\end{align*}
\end{proof}

We can now prove the theorem that measured embedding dimension and volume
 are monotone under weak embeddings.

\begin{proof}[Proof of Theorem~\ref{thm:wc}]	
	We show how to deduce the statement for~$\mevol_n$. The proof for~$\medim_n^Z$ works similarly by replacing
	every occurrence of ``$\lognorm_*( \partial_{n+1})$'' by ``$\dim_* (D_n)$''.
	
	Without loss of generality, we assume that~$\mevol_n(\alpha) < \infty$.
	Let~$\epsilon \in \R_{>0}$. By definition of~$\mevol_n$, there is an~$\alpha$-embedding~$C_*^\alpha \to D_*$
	with \[\lognorm_\alpha (\partial_{n+1}^{D}) \le \mevol_n (\alpha) +\epsilon.\]
	Because~$\alpha\prec\beta$, in every weak neighbourhood~$U$ of~$\alpha$,
	there is~$\beta'\in U$ such that~$\beta' \cong \beta$ (Proposition~\ref{prop:wc-char}).
	Thus, Lemma~\ref{lem:wc-tech} yields a weak neighbourhood~$U$ of~$\alpha$ 
	and~$\beta'\in U$ with~$\beta'\cong\beta$ such that there is a~$\beta'$-embedding~$C_*^{\beta'}\to \widetilde D_*$ 
	with 
	\[
		\lognorm_{\beta'} (\partial_{n+1}^{\widetilde D}) \le \lognorm_\alpha (\partial_{n+1}^D) +\epsilon.
	\]
	As~$\beta' \cong \beta$, this defines a~$\beta$-embedding~$C_*^\beta \to \widetilde D_*$. Thus,
	\begin{align*}
		\mevol_n (\beta) &\le \lognorm_\beta (\partial_{n+1}^{\widetilde D})\\
		&\le \lognorm_\alpha (\partial_{n+1}^D) +\epsilon\\
		&\le \mevol_n (\alpha) +2\cdot\epsilon.
	\end{align*}
	Taking the limit~$\epsilon \to 0$ yields the claim.
\end{proof}

%% file: Q.tex
\subsection{An upper bound on the norm}

In the following, we will often use the following upper bound to the norm.

Let~$\alpha\colon \Gamma \actson (X,\mu)$ be a \pmp\ action. 
Since multiple actions will be involved,
we stress the action in the notation (we write, e.g., $L^\infty(X)\ast_\alpha \Gamma\coloneqq \LinftyX\ast \Gamma$, $\spann{A}_\alpha$, $\dim_\alpha$, $\rk_\alpha$, $\lognorm_\alpha$).
We restrict to the special case that
$R = \Linfalpha\coloneqq \linf X *_\alpha \Gamma$.

\begin{defn}[]
	\label{def:Q}
	Let~$f\colon M\to N$ be an~$\Linfalpha$-homomorphism between marked projective~$\Linfalpha$-modules. 
	Let~$P$ be a reduced presentation of~$f$ as specified in Setup~\ref{setup:opnormalt}. 
	We define 
	\[
		Q(f,P) \coloneqq \sum_{(i,j,k,\gamma)\in I\times J\times K\times F} |a_{i,j,k,\gamma}| 
	\]
	and 
	\[
		Q(f) \coloneqq \min_P Q(f,P),
	\]
	where the minimum is taken over all possible reduced presentations of~$f$. Note that this is indeed a minimum,
	as norms of elements in~$Z$ lie in~$\N$. 
	If~$z\in \Linfalpha$, define~$Q(z) \coloneqq Q(f_z)$, where $f_z\colon \Linfalpha \to \Linfalpha$
	is the map given by right multiplication with~$z$.
	
	If~$\eta\colon\gen A_\alpha \to \LinftyXwc$ is an $\Linfalpha$-linear map, we define
	\[
		Q(\eta) \coloneqq Q\bigl(\iota(\eta(\chi_A,1))\bigr),
	\]
	where~$\iota\colon \LinftyXwc \hookrightarrow \Linfalpha$ is the canonical inclusion into the summand indexed 
	by~$1\in \Gamma$.
	Finally, if~$\eta\colon M = \bigoplus_{i\in I} \gen{A_i}_\alpha \to \LinftyXwc$ is an $\Linfalpha$-linear map,
	we define
	\[
		Q(\eta) \coloneqq \sum_{i\in I} Q\bigl(\eta|_{\gen{A_i}_\alpha}\bigr).
	\]
\end{defn}

\begin{rem}
	\label{rem:normleQ}
	Proposition~\ref{prop:opnormalt} shows that for every $R$-homomorphism~$f$ between marked projective $R$-modules,
	we have $\|f\| \le Q(f)$. 
	Straightforward calculations also show that~$Q$ is an upper bound to the~$\infty$-norm
	$\|\cdot\|_\infty$ (Definition~\ref{defn:infty-norm}), to~$\Nsum$, and to~$\Ntmax$ (Definition~\ref{def:size}).
\end{rem}

\begin{lem}[] 
	\label{lem:Q-basic} We record a few basic properties of this quantity.
	\begin{enumerate}[label=\enum]
		\item Let~$f\colon M\to N$ be an~$\Linfalpha$-homomorphism between marked projective $\Linfalpha$-modules
		and~$M = M_1\oplus M_2$ be a marked decomposition. Then,
		\[
			Q(f) \le Q(f|_{M_1}) + Q(f|_{M_2}).
		\]
		\item Let $f\colon M\to N$ and~$g\colon N\to P$ be~$\Linfalpha$-homomorphisms between marked projective~$\Linfalpha$-modules. Then,
		\[
			Q(g\circ f) \le Q(g)\cdot Q(f).
		\]
		\item Let~$f_1\colon M\to N_1$ and $f_2\colon M \to N_2$ be~$\Linfalpha$-homomorphisms between marked projective~
		$\Linfalpha$-modules. Then, $(f_1,f_2)\colon  M\to N_1\oplus N_2$ satisfies
		\[
			Q((f_1,f_2)) \le Q(f_1)+Q(f_2).
		\]
	\end{enumerate}
\end{lem}

\begin{proof}
This follows from straightforward computations.
\end{proof}

%% file: transl.tex
\subsection{Translating actions}
\label{subsec:transl}

Let~$\alpha$ and~$\beta$ be \pmp\ actions of a group~$\Gamma$ on 
a standard probability space~$(X,\mu)$.
In this section, we consider modules and maps defined over~$\Linfalpha$, and we 
produce ``corresponding''
$\Linfbeta$-modules and $\Linfbeta$-maps. 

We define the translation of modules and homomorphisms as follows.

\begin{defn}[translation of modules]
	Let~$M = \bigoplus_{i\in I} \gen{A_i}_\alpha$ be a marked projective $\Linfalpha$-module. 
	We define the \emph{translation of~$M$ to~$\beta$} as the marked projective $\Linfbeta$-module~$M_\beta$ 
	via
	\indexw{translation}
	\indexnot{Mb}{$M_\beta$}{translation of a module}
	\[
		M_\beta \coloneqq \bigoplus_{i\in I} \gen{A_i}_\beta.
	\]
\end{defn}

\begin{rem}[]
	\label{rem:dim-transl}
	Since the definition of dimension (Definition~\ref{def:marked-proj-mod}) does not depend on the action, 
	we have that~$\dim_\beta (M_\beta) = \dim_\alpha (M)$.
\end{rem}

We can also translate maps to the action~$\beta$. 

\begin{defn}[]
	Let~$f\colon \Linfalpha^m \to \Linfalpha^n$ be a linear map between marked free $\Linfalpha$-modules.
	Recall from Setup~\ref{setup:opnormalt} that~$f$ is given by right multiplication with a matrix~$z$
	over~$\Linfalpha$, where 
	\[
		z_{i,j} = \sum_{(k,\gamma)\in K\times F} a_{i,j,k,\gamma} \cdot (\chi_{\gamma^\alpha U_k}, \gamma),
	\]
	and $(U_k)_{k\in K}$ is a finite family of disjoint subsets of~$X$, the set $F\subseteq \Gamma$
	is finite, and~$a_{i,j,k,\gamma}\in Z$.
	We define the \emph{translation of~$f$ to~$\beta$} to be the~$\Linfbeta$-linear 
	map $f_\beta\colon \Linfbeta^m \to \Linfbeta^n$ 
	defined by right multiplication with the matrix $z_{\beta} = ( (z_\beta)_{i,j})_{i\in I,j\in J}$
	that is defined by 
	\begin{align*}
		(z_\beta)_{i,j} &= \sum_{(k,\gamma)\in K\times F} a_{i,j,k,\gamma} \cdot (\chi_{\gamma^\beta U_k}, \gamma).
	\end{align*}
	It is straightforward to show that~$f_\beta$ is well-defined
	and does not depend on the chosen presentation of~$z$ in Setup~\ref{setup:opnormalt}.
	
	More generally, let~$f\colon M = \bigoplus_{i\in I} \gen{A_i}_\alpha \to N = \bigoplus_{j\in J} \gen{B_j}_\alpha$ be a 	linear map between
	 marked projective~$\Linfalpha$-modules. We define its \emph{translation to~$\beta$} by 
	 \[
	 		f_\beta \coloneqq  \pi_{N_\beta} \circ(\iota_N\circ f\circ \pi_M)_\beta \circ \iota_{M_\beta}
	 			\colon M_\beta \to N_\beta,
	 \] 
	 where~$\iota_N\colon N \to \Linfalpha^{\# J}$ and~$\iota_{M_\beta}\colon M_\beta \to \Linfbeta^{\# I}$ denote
	 the canonical marked inclusions and~$\pi_{N_\beta}\colon \Linfbeta^{\# J} \to N_\beta$ 
	 and~$\pi_M\colon \Linfalpha^{\# I} \to M$ denote the canonical marked projections.
\end{defn} 

\begin{rem}[]
	The action~$\alpha$ is replaced by~$\beta$ in three places:
	\begin{enumerate}
		\item In the generation of modules: We generated an~$\Linfbeta$-module instead of one over~$\Linfalpha$;
		\item In the multiplication of the matrix: We multiply over~$\Linfbeta$;
		\item In the coefficients: We multiply with~$\chi_{\gamma^\beta U_k}$, where~$\gamma$ now acts via~$\beta$
		on~$U_k$. Previously, we considered the action via~$\alpha$.
	\end{enumerate}
\end{rem}

\begin{rem}[]
	\label{rem:Q-transl}
	From Definition~\ref{def:Q}, it follows that~$Q$ is translation-invariant, i.e., $Q(f_\beta) \le Q(f)$.
\end{rem}

For complexes obtained from~$Z\Gamma$-chain complexes by tensoring, there is an easy description of the
translation.

\begin{lem}[]
	\label{lem:transl-C}
	Let~$C_* \onto Z$ be an augmented free~$Z\Gamma$-chain complex. Then, there is a canonical isomorphism of 
	$\Linfbeta$-chain complexes
	\[
		\bigl(\Linfalpha\otimes_{Z\Gamma}C_*\bigr)_\beta \cong \Linfbeta\otimes_{Z\Gamma} C_*.
	\]
	If~$C_0 = Z\Gamma$ and the augmentation map~$\eta\colon Z\Gamma\to Z$ is given by sending all~$\gamma\in \Gamma$
	to~$1\in Z$, then we can define an augmentation map~$\Linfalpha\otimes_{Z\Gamma}C_0 \to \LinftyXwc$
	and the previous isomorphism can be extended by the identity on~$\LinftyYwc$ in degree~$-1$.
\end{lem}

\begin{proof}
	The tensor product and the translation are compatible with direct sums, so we can work componentwise.
	Without loss of generality, we assume~$C_j \cong_{Z\Gamma} Z\Gamma$. Then, we have isomorphisms
	of~$\Linfbeta$-modules
	\begin{align*}
		\bigl(\Linfalpha\otimes_{Z\Gamma} C_j\bigr)_\beta &\cong \bigl(\Linfalpha\bigr)_\beta
		= \bigl( \gen{X}_\alpha \bigr)_\beta 
		= \gen X_\beta
		\cong \Linfbeta\otimes_{Z\Gamma} C_j.
	\end{align*}
	For the boundary maps, because all of the above isomorphisms are compatible with direct sums,
	we can suppose that~$\partial_{j+1}\colon  Z\Gamma \to Z\Gamma$ is given by 
	right multiplication with 
	$
		\sum_{\gamma\in \Gamma} a_\gamma \cdot \gamma.
	$
	Then, $\Linfalpha\otimes_{Z\Gamma}\partial_{j+1}$ is given by right multiplication with
	\[
		\sum_{\gamma\in \Gamma} a_\gamma \cdot (\chi_X, \gamma) = 
		\sum_{\gamma\in \Gamma} a_\gamma \cdot (\chi_{\gamma^{\highlight \alpha} X}, \gamma).
	\]
	Thus, its translation $\bigl(\Linfalpha\otimes_{Z\Gamma}\partial_{j+1}\bigr)_\beta$ to~$\beta$ is given by
	\[
	\sum_{\gamma\in \Gamma} a_\gamma \cdot (\chi_{\gamma^{\highlight \beta} X}, \gamma).
	\]
	Because~$\gamma^\beta X = \gamma^\alpha X$, this agrees with~$\Linfbeta\otimes_{Z\Gamma}\partial_{j+1}$. 
	For the extension to degree~$-1$, we define~$\eta\colon  \Linfalpha\otimes_{Z\Gamma}C_0 \cong \Linfalpha \to \LinftyXwc$
	as the~$\LinftyXwc$-linear extension of~$\eta((\lambda,\gamma)) \coloneqq \gamma\lambda$ for all~$\lambda\in \LinftyXwc$ and~$\gamma\in \Gamma$.
\end{proof}

Compositions behave well under translation in the following sense:

\begin{lem}[composition estimate]
	\label{lem:comp-estimate}
	Let~$f\colon M\to N$ and~$g\colon N\to P$ be linear maps over~$\Linfalpha$ and~$\delta\in\R_{>0}$.
	 Then, there exists a weak 
	neighbourhood~$U$ of~$\alpha$ such that for all~$\beta \in U$, we have
	\[
		(g\circ f)_\beta =_\delta g_\beta \circ f_\beta.
	\]
\end{lem}

\begin{proof}
	By Lemma~\ref{lem:alm-eq-split}, Lemma~\ref{lem:almosteqinherit}~\ref{itm:almeq-sum}, and the 
	definition of the translation via the free case,
	 we can assume that~$M=N=P=\Linfalpha$.
	We fix presentations as in Setup~\ref{setup:opnormalt}, 
	i.e., $f$ and~$g$ are given by multiplication with
	 elements~$z_f,z_g\in\Linfalpha$, respectively.
	Since sums behave well with almost equality (Lemma~\ref{lem:almosteqinherit}~\ref{itm:almeq-sum}), we can assume without loss of generality 
	that
	\[
		z_f = a \cdot (\chi_{\gamma^\alpha W},\gamma) \qand z_g = b \cdot (\chi_{\lambda^\alpha V},\lambda),
	\]
	where~$a,b\in Z$, $\gamma,\lambda\in \Gamma$ and~$W,V \subseteq X$ are measurable subsets.
	We define~$U$ to be the weak neighbourhood of~$\alpha$
	defined by setting $F =\{\lambda^{-1}\}, {n = 2}$, ${A_1=W}, {A_2= V}$,
	and~$\epsilon = \delta$ (in the notation used in Definition~\ref{def:wk-nbhds}).
	
	Then,~$g\circ f$ is given by right multiplication with 
	\begin{align*}
		z_f\cdot z_g &= ab \cdot (\chi_{\gamma^\alpha W\cap \gamma^\alpha\lambda^\alpha V}, \gamma\lambda)\\
		&= ab \cdot (\chi_{(\gamma\lambda)^\alpha ((\lambda^{-1})^\alpha W\cap V)}, \gamma\lambda).
	\end{align*}
	Thus, for all~$\beta \in U$, the translation~$(g\circ f)_\beta$ is given by right multiplication with
	\begin{equation}
		\label{eq:comp-est-1}
		ab \cdot (\chi_{(\gamma\lambda)^\beta ((\lambda^{-1})^{\highlight\alpha} W\cap V)}, \gamma\lambda).
	\end{equation}
	Similarly,~$g_\beta\circ f_\beta$ is given by right multiplication with 
	\begin{equation}
		\label{eq:comp-est-2}
		ab \cdot (\chi_{(\gamma\lambda)^\beta ((\lambda^{-1})^{\highlight\beta} W\cap V)}, \gamma\lambda).
	\end{equation}
	Note that the expressions in Equation~\eqref{eq:comp-est-1} and Equation~\eqref{eq:comp-est-2} differ only by an
	~$\alpha$ resp.\ $\beta$ in the exponent of~$\lambda^{-1}$. 
	 Thus, by Example~\ref{ex:alm-eq}, $(g\circ f)_\beta$ and~$g_\beta\circ f_\beta$ are almost equal with 
	error at most
	\begin{align*}
		& \mu \bigl(
		(\gamma\lambda)^\beta ((\lambda^{-1})^\alpha W\cap V)
		\symmdiff
		(\gamma\lambda)^\beta ((\lambda^{-1})^\beta W\cap V)
		\bigr)\\
		&= \mu \bigl(
		((\lambda^{-1})^\alpha W\cap V)
		\symmdiff
		((\lambda^{-1})^\beta W\cap V)
		\bigr)
		\le \delta,
	\end{align*}
	where the last inequality is given by the choice of the weak neighbourhood~$U$.
\end{proof}

\begin{cor}[translation of chain complexes]
	\label{cor:transl-complex}
	Let~$n\in \N$, $\delta\in \R_{>0}$, and~$(D_*,\eta)$ be a marked projective $\Linfalpha$-chain complex
	with an augmentation map~$\eta\colon  D_0 \to \LinftyXwc$.
	Then, there exists a weak neighbourhood~$U$ of~$\alpha$ such that for all~$\beta\in U$,  the translated
	sequence~$((D_*)_\beta, \eta_\beta)$ is a
	marked projective \almostcc{\delta}{n} over~$\Linfbeta$.
\end{cor}

\begin{proof}
	We apply Lemma~\ref{lem:comp-estimate} multiple times. Define~$U$ to be the intersection of the neighbourhoods where
	the estimate holds for~$\partial_r\circ \partial_{r+1}$ for~$r\in \{0,\dots,n\}$. Moreover, because~$\eta$
	is an augmentation, there exists~$z\in D_0$ with~$\eta(z) = 1$. Thus, in a suitable neighbourhood,
	we have~$\eta_\beta(z_\beta) =_\delta 1_\beta = 1_\alpha = 1$.
\end{proof}

\begin{cor}[translation of chain maps]
	\label{cor:transl-chainmaps}
	Let~$n\in \N$, let $\delta\in \R_{>0}$, and let $f_*\colon  C_*\to D_*$ be an $\Linfalpha$-chain map between marked projective
	chain complexes~$(C_*,\zeta)$ and~$(D_*, \eta)$ extending the identity on~$\LinftyXwc$.
	Then, there exists a weak neighbourhood~$U$ of~$\alpha$ such that for all~$\beta\in U$, 
	the 
	translations of the
	chain complexes~$(C_*)_\beta$ and~$(D_*)_\beta$ are marked 
	projective \almostccs{\delta}{n} over~$\Linfbeta$ and moreover, the map
	$(f_*)_\beta\colon (C_*)_\beta \to (D_*)_\beta$ is a \almostcm{\delta}{n}.
\end{cor}

\begin{proof}
	We apply the above Corollary~\ref{cor:transl-complex} to~$C_*$ and~$D_*$ and intersect the resulting neighbourhoods.
	Moreover, we apply Lemma~\ref{lem:comp-estimate} with error term~$\delta \slash 2$ to the compositions
	$\eta \circ f_0$, $\partial_r^D\circ f_r$, and~$f_{r-1}\circ \partial_r^C$ for $r\in \{1,\dots, n+1\}$. We intersect 
	all resulting neighbourhoods of~$\alpha$ to obtain a new neighbourhood~$U$ of~$\alpha$. For all~$\beta\in U$, we have
	\[
		\eta_\beta \circ (f_0)_\beta =_{\delta/2} (\eta\circ f_0)_\beta = \zeta_\beta.
	\]	
	Moreover, for~$r\in \{1,\dots, n+1\}$, we have 
	\begin{align*}
		(\partial_r^D)_\beta \circ (f_r)_\beta 
	&=_{\delta/2} (\partial_r^D \circ f_r)_\beta\\
	&= (f_{r-1}\circ \partial_r^C)_\beta \\
	&=_{\delta/2} (f_{r-1})_\beta \circ (\partial_r^C)_\beta
	\end{align*}
	Thus by Lemma~\ref{lem:almosteqinherit}~\ref{itm:composition:additivity}, we obtain 
	that~$(\partial_r^D)_\beta \circ (f_r)_\beta =_\delta (f_{r-1})_\beta \circ (\partial_r^C)_\beta$.
\end{proof}

We show that the marked rank, $\lognorm$ (Definition~\ref{def:lognorm}),
and the norm change continuously in the action
in the following sense:

\begin{lem}[translation and marked rank]
	\label{lem:transl-mrk}
	Let~$f\colon  M\to N$ be an $\Linfalpha$-homo\-morphism between marked projective $\Linfalpha$-modules and~$\delta\in \R_{>0}$. Then,
	there exists a weak neighbourhood~$U$ of~$\alpha$ such that for all~$\beta\in U$, we have
	$\mrk_\beta (f_\beta) \le \mrk_\alpha (f) + \delta$.
\end{lem}

\begin{proof}
	By Lemma~\ref{lem:mrk-expl}, we have  (in the notation of Setup~\ref{setup:opnormalt})
	\[
		\mrk_\alpha (f) = \sum_{j\in J} \mu \biggl(
			\bigcup_{\substack{(i, k,\gamma)\in I\times K\times F,\\ a_{i,j,k,\gamma} \neq 0}}
			((\gamma^{-1})^\alpha A_i \cap U_k)
		\biggr).\\
	\]
	Thus, for every action~$\beta$, we have 
	\begin{align*}
		\mrk_\beta (f_\beta) - \mrk_\alpha (f) &\le \sum_{\substack{(i,j, k,\gamma)\in I\times J\times K\times F,\\ a_{i,j,k,\gamma} \neq 0}} \mu \bigl(((\gamma^{-1})^\alpha A_i \cap U_k) \symmdiff 
			((\gamma^{-1})^\beta A_i \cap U_k)
		\bigr)\\
		&\le \sum_{(i,j, k,\gamma)\in I\times J\times K\times F} \mu \bigl(((\gamma^{-1})^\alpha A_i \cap U_k) \symmdiff 
			((\gamma^{-1})^\beta A_i \cap U_k)
		\bigr)\\
		&\le\delta,
	\end{align*}
	where the last inequality holds in a suitable weak neighbourhood~$U$ that is defined as in Definition~\ref{def:wk-nbhds}
	with error term~$\epsilon \coloneqq \delta \slash \big(\# I\cdot \# J\cdot \# K\cdot \# F\big)$ (or $\epsilon \coloneqq 1$
	if the denominator is zero) and the~$(A_i)_{i\in I}$ as the
	test sets.
\end{proof}

\begin{lem}[translation and norm]
	\label{lem:transl-norm}
	Let~$f\colon  M\to N$ be an~$\Linfalpha$-ho\-mo\-mor\-phism between marked projective 
	$\Linfalpha$-modules and~$\delta\in \R_{>0}$. Then,
	there exists a weak neighbourhood~$U$ of~$\alpha$ such that for all~$\beta\in U$, the following holds:
	There is an~$\Linfbeta$-homomorphism~$f'_\beta \colon  M_\beta \to N_\beta$ such that 
	\[
		f_\beta =_{\delta, Q(f)} f'_\beta \qand \|f'_\beta\|\le \|f\|
		\qand f'_\beta(M_\beta) \subseteq f_\beta(M_\beta).
	\]
\end{lem}

\begin{proof}
	By an analogue of Lemma~\ref{lem:alm-eq-split}, 
	we can assume that~$M=\gen A_\alpha$. We will thus drop~$i\in I$ from the notation.
	Let~$P$ be a presentation representing~$f$ as in Setup~\ref{setup:opnormalt}
	such that~$Q(f,P) = Q(f)$. 
	Fix the notation of Setup~~\ref{setup:opnormalt}.
	Pick a weak neighbourhood~$U$ of~$\alpha$ where for all~$\gamma\in F, k\in K$, and~$\beta \in U$, we have
	\[
		\mu \bigl(
			\gamma^\alpha U_k \symmdiff \gamma^\beta U_k
		\bigr) \le \frac{\delta}{\# K\cdot \# F}.
	\] 
	Let~$\beta \in U$. We define
	\[
		A'' \coloneqq \bigcup_{(k,\gamma)\in K\times F} \gamma^\alpha U_k \symmdiff \gamma^\beta U_k
	\]
	and~$A'\coloneqq A\setminus A''$. We set~$M'\coloneqq \gen{A'}_\beta$
	and~$M''\coloneqq \gen{A''\cap A}_\beta$. Hence, we have an isomorphism of 
	$\Linfbeta$-modules $M_\beta \cong M' \oplus M''$.
	By construction, we have that $\dim_\beta (M'') \le \mu(A'') \le \delta$.
	We define~$f'_\beta \coloneqq f_\beta \circ \iota_{A'}\circ \pi_{A'}$, 
	where~$\pi_{A'}$ is the projection
	onto the marked summand~$M'$ and~$\iota_{A'}\colon M' \to M$ is the canonical marked inclusion. In particular, 
	\[f'_\beta|_{M'} = f_\beta|_{M'} \qand f'_\beta|_{M''} = 0\]
	and $f'_\beta(M_\beta) \subseteq f_\beta(M_\beta)$.
	Moreover, note that for~$L\subseteq K\times F$, we have
	\begin{equation*}
	\label{eq:cap-symmdiff}
		A'\cap \bigcap_{(k,\gamma)\in L} \gamma^{\highlight\beta} U_k \subseteq \bigcap_{(k,\gamma)\in L} \gamma^{\highlight\alpha} U_k.
	\end{equation*}
	 Thus, by the explicit description of the operator norm 
	(Proposition~\ref{prop:opnormalt}), we have 
	\begin{align*}
		\|f'_\beta\| 
		 &= \max\Bigl\{
			\sum_{j\in J, (k,\gamma)\in L} |a_{j,k,\gamma}| \Bigm| L\subseteq K\times F \text{ with }
			\mu\Bigl( A'\cap \bigcap_{(k,\gamma)\in L} \gamma^{\beta} U_k \Bigr) > 0
		\Bigr\}
		\\
		&\le \max\Bigl\{
			\sum_{j\in J, (k,\gamma)\in L} |a_{j,k,\gamma}| \Bigm| L\subseteq K\times F \text{ with }
			\mu\Bigl( \bigcap_{(k,\gamma)\in L} \gamma^{\alpha} U_k \Bigr) > 0
		\Bigr\}
		\\
		&= \|f\|. &
	\end{align*}
	It remains to estimate~$\|f_\beta|_{M''}\|$: we have
	\begin{align*}
		\|f_\beta|_{M''}\|
		&\le \|f_\beta\|\\
		&\le Q(f_\beta) & (\text{Remark~\ref{rem:normleQ}})\\
		&\le Q(f). & (\text{Remark~\ref{rem:Q-transl}})&\qedhere
	\end{align*}
\end{proof}

We use these two estimates to show that also $\lognorm$ is continuous in the action.

\begin{lem}[translation and lognorm]
	\label{lem:transl-lognorm}
	Let~$f\colon  M\to N$ be an~$\Linfalpha$-ho\-mo\-mor\-phism 
	between marked projective $\Linfalpha$-modules and~$\epsilon\in \R_{>0}$. Then,
	there exists a weak neighbourhood~$U$ of~$\alpha$ such that for all~$\beta\in U$, we have
	\[\lognorm_\beta (f_\beta) \le \lognorm_\alpha (f) + \epsilon.\]
\end{lem}

\begin{proof}
	We set~$\delta \coloneqq \epsilon \slash \big(4\cdot \logp Q(f)\big) > 0$ (or~$\delta \coloneqq 1$ if~$\logp Q(f) = 0$).
	By definition of~$\lognorm_\alpha$, there exists a marked decomposition~$(M_i)_{i\in I}$
	of~$M$ over $\Linfalpha$
	with
	\begin{equation}
	\label{eq:lognorm'}
	\sum_{i\in I} \lognorm_\alpha'(f|_{M_i} \colon  M_i \to N) 
	= \lognorm_\alpha' \bigl(
		f, (M_i)_{i\in I}
	\bigr)
	< \lognorm_\alpha (f) + \frac{\epsilon}{2}.
	\end{equation}
	To simplify notation, we will assume that~$M = M_i$ consists of a single summand 
	and estimate the change of~$\lognorm_\alpha'(f|_{M_i})$ under translation. (The general case will then follow 
	by dividing the allowed error by~$\# I$, which stays constant during this proof.)
	
	By Lemma~\ref{lem:transl-mrk}, there exists a weak neighbourhood~$U_1$ of~$\alpha$,
	 such that for all \mbox{$\beta\in U_1$},
	we have~$\mrk_\beta (f_\beta) \le \mrk_\alpha (f) +\delta$. Moreover, by Lemma~\ref{lem:transl-norm}, there exists a 
	weak neighbourhood~$U_2$ of~$\alpha$, such that for all~$\beta\in U_2$, 
	there exists~$f'_\beta \colon  M_\beta \to N_\beta$
	with 
	\[f_\beta =_{\delta, Q(f)} f'_\beta \qand \|f'_\beta\| \le \|f\| \qand
	f'_\beta(M_\beta) \subseteq f_\beta(M_\beta).\]
	We define~$U\coloneqq U_1\cap U_2$. Then, for all~$\beta\in U$, we have
	\begin{align*}
		& \lognorm_\beta (f_\beta)\\ 
		&\le \lognorm_\beta (f'_\beta) + \delta\cdot \logp Q(f) 
		&(\text{Lemma~\ref{prop:lognorm}~\ref{itm:lognorm-almeq}})\\
		&\le \lognorm_\beta' (f'_\beta) + \delta\cdot \logp Q(f)
		&(\text{Def.\ of }\lognorm)\\
		&= \min \bigl\{\dim_\beta M_\beta, \mrk_\beta (f'_\beta)\bigr\} \cdot \logp \|f'_\beta\|+ \delta\cdot \logp Q(f)\\
		&= \min \bigl\{\dim_\alpha M, \mrk_\beta (f'_\beta)\bigr\} \cdot \logp \|f'_\beta\|+ \delta\cdot \logp Q(f)
			&(\text{Remark~\ref{rem:dim-transl}})\\	
		&\le \min \bigl\{\dim_\alpha M, \mrk_\beta (f_\beta)\bigr\} \cdot \logp \|f'_\beta\|	+ \delta\cdot \logp Q(f)
			&(f'_\beta(M_\beta) \subseteq f_\beta(M_\beta))\\
		&\le \min \bigl\{\dim_\alpha M, \mrk_\alpha f +\delta \bigr\} \cdot \logp \|f'_\beta\| + \delta\cdot \logp Q(f)
			&(\beta \in U_1)\\
		&\le \min \bigl\{\dim_\alpha M, \mrk_\alpha f +\delta \bigr\} \cdot \logp \|f\| + \delta\cdot \logp Q(f)
			&(\|f'_\beta\|\le\|f\|)\\
		&\le \min \bigl\{\dim_\alpha M, \mrk_\alpha f\bigr\} \cdot \logp \|f\| + 2\delta\cdot \logp Q(f)
			&(\text{Remark~\ref{rem:normleQ}})\\
		&\le \lognorm_\alpha'(f) + \frac{\epsilon}{2}
		&(\textrm{Def.\ of }\delta)\\
		&\le \lognorm_\alpha (f) + \epsilon. & (\text{Equation~\eqref{eq:lognorm'}}) &\qedhere
	\end{align*}
\end{proof}

%% file: strict-transl.tex
\subsection{Strictification and translation}
\label{subsec:strict-transl}

We revisit the results for strictification of chain complexes and chain maps from Section~\ref{sec:strictification}.
We establish more abstract upper bounds and investigate how these change under translation. 

As before, the two main results of this section are the following:
\begin{itemize}
	\item Every almost chain complex is ``close'' to a strict chain complex (Theorem~\ref{thm:strict-cplx}).
	\item Every almost chain map is ``close'' to a strict chain map (Theorem~\ref{thm:strict-map}).
\end{itemize}

Let~$\alpha\colon  \Gamma \actson (X,\mu)$ be an essentially free \pmp\ action 
on a standard probability space. We consider modules over the crossed product ring 
$\Linfalpha\coloneqq \Linf X *_\alpha \Gamma$. 

We bound the complexity of the input data by the following notion:

\begin{defn}[translation-invariant constant]
\label{def:transl-inv-const}
	Let~$(X,\mu)$ be a standard probability space, $\Gamma$ be a group and $n\in \N$. 
	A \emph{translation-invariant constant} is a family of maps~$\kappa = (\kappa_\delta)_{\delta\in \R_{>0}}$ that,
	given $\delta\in\R_{>0}$, a \pmp\ action~$\alpha\colon \Gamma\actson (X,\mu)$, and an (augmented) 
	marked projective \almostcc{\delta}{n}~$(D_*,\eta)$
	over~$\Linfalpha$, assigns a positive real number~$\kappa_\delta(D_*, \eta)$ such that the following holds:
	
	 Let~$\delta\in \R_{>0}$, $\alpha\colon \Gamma\actson X$  and~$(D_*, \eta)$ be a marked projective 
	\almostcc{\delta/2}{n} over~$\Linfalpha$. Then, there exists a neighbourhood~$U$ of~$\alpha$
	(see Definition~\ref{def:wk-nbhds})	
	such that for every~$\beta\in U$, we have that~$((D_*)_\beta,\eta_\beta)$ is a \almostcc{\delta}{n} and
	\[
		\kappa_\delta((D_*)_\beta,\eta_\beta) \le \kappa_{\delta/2}(D_*, \eta) +1.
	\]
	We often write~$\kappa_\delta(D_*)$ instead of~$\kappa_\delta(D_*,\eta)$.
	 \indexw{translation-invariant constant}
	\indexnot{kd}{$(\kappa_\delta)_\delta$}{translation-invariant constant}
\end{defn}

In a similar fashion, we can define translation-invariant constants of almost chain maps 
	between almost chain complexes.
	
\begin{defn}[translation-invariant constant of chain complexes]
	Let~$(X,\mu)$ be a standard probability space, $\Gamma$ be a group and~$n\in \N$. 
	A \emph{translation-invariant constant} is a family of maps~$(\kappa_\delta)_{\delta\in \R_{>0}}$
	that, given $\delta\in \R_{>0}$, a \pmp\ action~$\alpha\colon\Gamma\actson (X,\mu)$,
	and a \almostcm{\delta}{n}~$f_*\colon C_*\to D_*$ between augmented marked projective \almostccs{\delta}{n} over~$\Linfalpha$, assigns a positive real number~$\kappa_\delta(f_*)$ such that the following holds:
	
	Let~$\delta\in \R_{>0}$, $\alpha\colon \Gamma\actson X$  and 
	$f_*\colon C_*\to D_*$ be a \almostcm{\delta/2}{n} between marked projective
	\almostccs{\delta/2}{n} over~$\Linfalpha$. Then, there exists a neighbourhood~$U$ of~$\alpha$
	such that for every~$\beta\in U$, we have that~$(f_*)_\beta\colon (C_*)_\beta \to (D_*)_\beta$ is 
	a \almostcm{\delta}{n} between \almostccs{\delta}{n} over~$\Linfbeta$ and
	\[
		\kappa_\delta((f_*)_\beta) \le \kappa_{\delta/2}(f_*) +1.
	\]
\end{defn}

\begin{thm}[]
	\label{thm:strict-cplx}
	Let~$(X,\mu)$ be a standard probability space,~$\Gamma$ be a group, and $n\in \N$. Then,
	there exist monotone increasing functions~$K, p\colon \R_{>0} \to \R_{>0}$  and
	a translation-invariant constant~$\kappa = (\kappa_\delta)_{\delta\in \R_{>0}}$ such that 
	for every \pmp\ action~$\alpha\colon  \Gamma\actson (X,\mu)$, $\delta\in \R_{>0}$, 
	and every marked projective \almostcc{\delta}{n}~$(D_*,\eta)$ over~$\Linfalpha$, 
	there exists a marked projective (strict) $\Linfalpha$-chain complex~$(\widehat D_*, \widehat\eta)$ (up to 
	degree~$n+1$) such that
	\[
		\dgh {K(\kappa_\delta(D_*))} {\widehat D_*, D_*, n} \leq K(\kappa_\delta(D_*)) \cdot \delta.
	\]
	Moreover,~$\widehat D_*$ can be chosen such that the following hold:
	\begin{enumerate}[label=\enum]	
	\item For each~$j\in \{0,\dots,n\}$, the module $D_j$ is a submodule of~$\widehat D_j$ 
	and the inclusion map~$D_*\hookrightarrow \widehat D_*$ is a \almostcm{(K(\kappa_\delta(D_*))\cdot \delta)}{n}.
	\item We have~$\kappa_\delta(\widehat D_*) \le p(\kappa_\delta(D_*))$.
	\end{enumerate}\indexw{strictification!of an almost chain complex}
\end{thm}

Before giving the proof, we recall from Definition~\ref{def:crazyconsts} that 
\begin{align*}
    \kappa_n(D_*)
    & \coloneqq \max \bigl\{\|\eta\|, \|\partial_1^D\|, \dots, \|\partial_{n+1}^D\|\bigr\}
    \\
    \numax_n(D_*)
    & \coloneqq \max \bigl\{\|\eta\|_\infty, \Nmax(\partial_1^D), \dots, \Nmax(\partial_{n+1}^D) \bigr\}
    \\
    \nusum_n(D_*)
    & \coloneqq \max \bigl\{\|\eta\|_\infty, \Nsum(\partial_1^D), \dots,  \Nsum(\partial_{n+1}^D) \bigr\}
    .
  \end{align*}
 Moreover, if $\kappa \in \R_{>0}$, then we say that $\overline \kappa_n(D_*) < \kappa$
  if 
  \[
  \max \bigl\{
  \rk (D_1), \dots, \rk (D_{n+1}),
    \kappa_n(D_*),
    \numax_n(D_*)
    \bigr\}
    < \kappa
  \]
  and there exists a~$z \in D_0$ with
  \[ \eta(z)=_\delta 1
  , \quad
  \Nbasic(z) < \kappa
  , \quad
  \Ntbasic(z) < \kappa
  , \quad
  |z|_\infty < \kappa.
  \]

\begin{proof}
	We employ Theorem~\ref{thm:strictifycomplex} 
	to define~$K$: For the fixed~$n\in\N$ and~$x\in \R_{>0}$, we define~$K(x)$ 
	to be one of the~$K\in \R_{>0}$, for which Theorem~\ref{thm:strict-cplx}
	 holds (with~$\kappa\coloneqq x$). Without loss of
	generality, we can assume that the function~$K$ is monotone increasing.
	We define~$\kappa$ as follows: Let~$\alpha\colon \Gamma\actson (X,\mu)$ be a \pmp\ action,
	$\delta\in \R_{>0}$, and~$(D_*,\eta)$ be 
	a marked projective \almostcc{\delta}{n} over~$\Linfalpha$. We define
	\[
		\kappa_\delta' \coloneqq \inf_z Q(z),
	\]
	where~$z$ ranges over all~$z\in D_0$ with~$\eta(z) =_\delta 1$.
	We then define
	\[
		\kappa_\delta(D_*, \eta) \coloneqq \max \bigl\{
			\kappa_\delta', Q(\eta), Q(\partial_1^D), \dots, Q(\partial_{n+1}^D),
			\rk (D_1), \dots, \rk (D_{n+1})
		\bigr\} +1.
	\]
	Note that~$Q$ and~$\rk$ are invariant under translation (see Remark~\ref{rem:Q-transl}). For 
	the \mbox{(almost-)}invariance
	of~$\kappa'$, note that if $z\in D_0$ such that~$\eta(z) =_{\delta/2} 1$, then the composition estimate
	(Lemma~\ref{lem:comp-estimate}) shows that in a suitable neighbourhood~$U$, we have for all~$\beta\in U$
	that~$\eta_\beta(z_\beta) =_\delta 1$. Moreover, $Q(z_\beta) \le Q(z)$. Thus, we can find a neighbourhood as in the translation-invariance
	condition. Note that because~$\kappa'$ is defined as an infimum, we add the ``$+1$'' in the definition of 
	a translation-invariant constant to obtain an open neighbourhood.
	Thus, $\kappa$ defines a translation-invariant constant.
	
	We show that the functions~$K$ and~$\kappa$ satisfy the desired conditions:
	Let~$\alpha\colon \Gamma\actson (X,\mu)$, $\delta\in \R_{>0}$, 
	and~$(D_*,\eta)$ be a marked projective \almostcc{\delta}{n}. 
	We write \mbox{$\kappa_D \coloneqq \kappa_\delta(D_*)$} and~$K_D\coloneqq K(\kappa_\delta(D_*))$. We have
	$\overline\kappa_n(D_*) <\kappa_D$, because~$\Nsum(\cdot)$, $\Ntmax(\cdot)$, $|\cdot|_\infty$,
	 and~$\|\cdot\|$ are bounded from above by~$Q(\cdot)$
	(see Remark~\ref{rem:normleQ}). Thus, Theorem~\ref{thm:strict-cplx} yields
	 a marked projective~$\Linfalpha$-chain complex~$(\widehat D_*, \widehat \eta)$ (up to degree~$n+1$)
	with 
	\[
		\dgh {K_D} {\widehat D_*, D_*, n} \le K_D \cdot \delta.
	\]
	Moreover, Theorem~\ref{thm:strict-cplx}
	states that the inclusion map~$D_*\hookrightarrow \widehat D_*$ is a \almostcm{(K_D\cdot \delta)}{n}.
	For the second statement, we have to dive into the details of the proof of 
	Theorem~\ref{thm:strictifycomplex}
	and proceed by induction over the degrees.
	The proof of Lemma~\ref{lem:almostsurjective} yields that~$\kappa_\delta'(\widehat D_*) \le
	 \kappa_\delta'(D_*)$. Furthermore,
	that proof shows that~$\widehat D_0 = D_0 \oplus \gen{B}_\alpha$ for some~$B\subseteq X$ and thus,
	\[
		\rk (\widehat D_0) = \rk (D_0) +1.
	\]
	Lemma~\ref{lem:Q-basic} then shows that
	\begin{align*}
		Q(\widehat \eta) &\le Q(\eta) + Q\bigl(\widehat\eta|_{\langle B\rangle}\bigr)\\
		&= Q(\eta) + Q(1-\eta(z)) & (\text{proof of Lemma~\ref{lem:almostsurjective}})\\
		&\le Q(\eta) + (1 + Q(\eta) \cdot Q(z)) & (\text{Lemma~\ref{lem:Q-basic}})\\
		&\le p_0(\kappa_\delta(D_*))
	\end{align*}
	for the function~$p_0\colon x\mapsto 1+x + x^2$. 
	
	For the inductive step, assume that $\widehat D_{r-1}$ and~$\partial_r^{\widehat D}$ 
	have been constructed and satisfy the 
	theorem with the function~$p_{r-1}$.
	The proof of Lemma~\ref{lem:almostsurjective} constructs~$\widehat D_r$ as~$D_r\oplus E_r$, 
	where $\rk (E_r) \le \rk (D_r)$, 
	and $\dim_\alpha (E_r)$ is bounded by a function in~$\kappa_\delta(\widehat D_*)$.
	Moreover, Lemma~\ref{lem:Q-basic} yields that
	\begin{align*}
		Q(\widehat{\partial_r}) &\le Q(\widetilde \partial_r) + \rk E_r \cdot Q(\widetilde \partial_r) \cdot Q(\partial_{r+1})\\
		&\le Q(\widetilde \partial_r) \cdot (1 +\rk D_r \cdot Q(\partial_{r+1}))\\
		&\le (Q(\partial_{r+1}) + \rk E_r) \cdot (1+ \rk D_r \cdot Q(\partial_{r+1}))\\
		&\le (Q(\partial_{r+1}) + \rk D_r) \cdot (1+ \rk D_r \cdot Q(\partial_{r+1})),
	\end{align*}
	which is bounded from above by a function in~$\kappa_\delta(\widehat D_*)$.
	Moreover, in degree~$n+1$, we have
	\begin{align*}
		Q(\widehat {\partial}_{n+1}) &= Q(\widetilde{\partial}_{n+1})\\
		&\le Q(\partial_{n+1}) + \rk (E_n)\\
		&\le Q(\partial_{n+1}) + \rk (D_n),
	\end{align*}	 
	which is also bounded from above by $2\cdot\kappa_\delta(\widehat D_*)$.
	Finally, by construction, we have $\rk (\widehat{D}_r) \le 2\cdot \rk (D_r)$.
	Altogether, we obtain that
	\[
		\kappa_\delta(\widehat D_*) \le p(\kappa_\delta(D_*)),
	\]
	where~$p$ is defined to be the maximum of all the upper bounds encountered so far.	
	Note that we can assume~$p$ to be monotone increasing, otherwise we set
	\[
		p'(x) \coloneqq \max_{y\le x} p(y). \qedhere
	\]
\end{proof}

We can also strictify chain maps between (strict) chain complexes.

\begin{thm}[]
	\label{thm:strict-map}
	Let~$(X,\mu)$ be a standard probability space, let~$\Gamma$ be a group, and let~$n\in \N$.
	Then, there exists a monotone increasing function~$K\colon  \R_{>0} \to \R_{>0}$ and 
	a translation-invariant constant~$\kappa$ such that for all~$\alpha\colon \Gamma\actson X$,
	$\delta\in \R_{>0}$,
	and for every \almostcm{\delta}{n}~$f_*\colon (C_*,\zeta)\to (\widehat D_*, \widehat{\eta})$ between 
	marked projective (strict) $\Linfalpha$-chain complexes, extending the identity on 
	~$\LinftyXwc$, there exists a marked projective strict $\Linfalpha$-chain complex~$\widetilde D_*$
	with 
	\[
		\dgh {K(\kappa_\delta(f_*))} {\widetilde D_*, \widehat D_*, n} < K(\kappa_\delta(f_*))\cdot \delta
	\] 
	that admits a chain map~$\widetilde f_*\colon  C_* \to \widetilde D_*$ extending the identity on~$\LinftyXwc$.
	\indexw{strictification!of an almost chain map}
\end{thm}

\begin{proof}
	Given~$\alpha\colon \Gamma\actson (X,\mu)$, $\delta\in \R_{>0}$, 
	and a \almostcm{\delta}{n}~$f\colon C_*\to \widehat D_*$, we define
	\begin{align*}
		\kappa_\delta(f_*)\coloneqq \max\bigl\{
			Q(\zeta), Q(\partial_1^C), \dots, Q(\partial_{n+1}^C),
			Q(\widehat\eta), Q(\partial_1^{\widehat D}), \dots, Q(\partial_{n+1}^{\widehat D}),&\\
			Q(f_0), \dots, Q(f_n)
		&\bigr\}+1.
	\end{align*}
	Because~$Q$ is translation-invariant (Remark~\ref{rem:Q-transl}), $\kappa$ is a translation-invariant constant.
	We employ Theorem~\ref{thm:strict-cplx} to define the map~$K\colon \R_{>0} \to \R_{>0}$, which 
	can be assumed to be monotone.
	
	Because the norm of~$f$ is bounded by~$Q$ (Remark~\ref{rem:normleQ}),  we 
	have the following estimates~$\max\{\kappa_n(C_*), \nusum_n(C_*)\} < \kappa_\delta(f_*)$
	and~$\kappa_n(\widehat D_*) < \kappa_\delta(f_*)$. Moreover,~$\kappa_n(f_*)\le \kappa_\delta(f_*)$. Thus, 
	Theorem~\ref{thm:strictifychmap} yields a marked projective~$\Linfalpha$-chain complex~$(\widetilde D_*, \widetilde \eta)$
	with
	\[
		\dgh {K(\kappa_\delta(f_*))} {\widetilde D_*, \widehat D_*, n} < K(\kappa_\delta(f_*))\cdot \delta
	\]
	that admits a chain map~$\widetilde f_*\colon  C_* \to \widetilde D_*$ extending the identity on~$\LinftyXwc$ such that 
	\[
	\dgh {K(\kappa_\delta(f_*))} {\widetilde f_r, f_r} < K(\kappa_\delta(f_*))\cdot \delta
	\]
	for all $r \in \{0, \ldots, n+1\}$.
\end{proof}

%% file: disintegration.tex
\section{The disintegration estimate}\label{sec:disint}

We show a basic disintegration estimate for measured embedding
dimension and measured embedding volume
(Proposition~\ref{prop:disint}).  This is useful in the
context of orbit equivalence and property~$\EMD^*$. As usual $Z$ denotes~$\IZ$ or a finite field.

\begin{defn}[disintegration]
  Let $\alpha \colon \Gamma \actson (X,\mu)$ be a standard action. We
  write $\Prob(\alpha)$ for the set of all probability measures on the
  measurable space~$X$ that are invariant under the measurable action
  underlying~$\alpha$.  A \emph{disintegration of~$\alpha$} is a
  map~$m \colon X \to \Prob(\alpha)$ with the following properties:
  \begin{itemize}
  \item For every measurable subset~$A \subset X$, the evaluation
    map
    \begin{align*}
      X & \to [0,1]
      \\
      x & \to m_x(A)
    \end{align*}
    is measurable and $\mu(A) = \int_X m_x(A) \;d\mu(x)$.
  \item For all~$x \in X$ and all~$\gamma \in \Gamma$, we
    have~$m_{\gamma \cdot x} = m_x$.
  \item For all~$\nu \in \Prob(\alpha)$, the preimage~$X_\nu \coloneqq
    m^{-1}(\{\nu\})$ is measurable and $\nu(X_\nu) \in \{0,1\}$.
  \end{itemize}
  Such a disintegration of~$\alpha$ is an \emph{ergodic decomposition
    of~$\alpha$} if~$m_x$ is ergodic for all~$x \in X$ (with respect
  to the measurable action underlying~$\alpha$).
\end{defn}

Ergodic decompositions always exist~\cite[Section~4]{varadarajan}.
If $m$ is a disintegration of a standard action $\alpha \colon \Gamma \actson (X,\mu)$,
then for $\mu$-almost every~$x \in X$, the underlying action
of~$\Gamma$ on~$X$ is essentially free with respect to~$m_x$~\cite[Remark~3.7]{loehsartori}. We then
also write~$(\alpha,m_x)$ for the induced standard action~$\Gamma
\actson (X,m_x)$.

For convenience, we introduce the following dual of the abbreviation
``almost every'': Given a probability space~$(X,\mu)$ and a
property~$P \colon X \to \text{\textsf{Bool}}$ (where \textsf{Bool}
denotes the Booleans), we say that \emph{there $\mu$-exists an~$x \in
  X$ with property~$P$} if there exists a measurable subset~$A \subset
X$ with the property that~$\mu(A) > 0$ and that $P(x)$ holds for
every~$x \in A$.

\begin{prop}\label{prop:disint}
  Let $n \in \N$, let $\Gamma$ be a group of type~$\FP_{n+1}$, 
  let $\alpha \colon \Gamma \actson (X,\mu)$ be a standard
  $\Gamma$-action, and let $m \colon X \to \Prob(\alpha)$
  be a disintegration of~$\alpha$. Then:
  \begin{enumerate}[label=\enum]
  \item For every~$\varepsilon \in \R_{>0}$, there
    $\mu$-exists an~$x\in X$ with
    \[ \medim_n^Z (\alpha,m_x) \leq \medim_n^Z (\alpha) + \varepsilon.
    \]
  \item For every~$\varepsilon \in \R_{>0}$, there
    $\mu$-exists an~$x \in X$ with
    \[ \mevol_n (\alpha,m_x) \leq \mevol_n (\alpha) + \varepsilon.
    \]
  \end{enumerate}
\end{prop}

To prepare the proof of Proposition~\ref{prop:disint}, we introduce
the following notation: We write~$R \coloneqq \linf{X,\mu,Z} * \Gamma$ and
$R(x) \coloneqq \linf{X,m_x,Z} *\Gamma$ whenever $x \in X$. Implicitly, we
only speak of those~$x \in X$ for which the $\Gamma$-action on~$X$ is
essentially free with respect to~$m_x$; this is satisfied for $\mu$-almost
every~$x \in X$. From marked projective $R$-modules~$M$ and
$R$-linear maps~$f \colon M \to N$, we obtain associated
marked projective $R(x)$-modules~$M(x)$ and $R(x)$-linear
maps~$f(x) \colon M(x) \to N(x)$ for $\mu$-almost every~$x \in X$.
We record basic observations on dimensions and operator norms:

\begin{lem}\label{lem:disintdimnorm}
  Let $\alpha \colon \Gamma \actson (X,\mu)$ be a standard
  $\Gamma$-action and let $m \colon X \to \Prob(\alpha)$ be a
  disintegration of~$\alpha$.  Let $M,N$ be marked projective
  $R$-modules and let $f\colon M \to N$ be an $R$-linear
  map.
  Then:
  \begin{enumerate}[label=\enum]
  \item We have $\dim_R (M) = \int_X \dim_{R(x)} (M(x)) \;d\mu(x)$. 
    In particular, there $\mu$-exists an~$x \in X$ with
    \[ \dim_{R(x)} (M(x)) \leq \dim_R (M).
    \]
  \item
    For $\mu$-almost every~$x \in X$, we have
    \[ \| f(x)\|_{R(x)} \leq \|f\|_R.
    \]
  \item For every~$\varepsilon \in \R_{>0}$,
    there $\mu$-exists an~$x \in X$ with
    \[ \lognorm_{R(x)} (f(x)) \leq \lognorm_{R} (f) + \varepsilon.
    \]
  \end{enumerate}
\end{lem}
\begin{proof}
  (i)
  As the dimension is additive with respect to marked decompositions,
  it suffices to consider the case that $M = \gen A _R$ for some
  measurable subset~$A \subset X$. In this case, by definition, we have
  \begin{align*}
    \dim_R (M)
  &
  = \dim_R (\gen A _R)
  = \mu(A)
  = \int_X m_x(A) \;d\mu(x)
  \\
  & 
  = \int_X \dim_{R(x)} \bigl(\gen A_{R(x)}\bigr) \;d\mu(x)
  = \int_X \dim_{R(x)} \bigl(M(x)\bigr) \;d\mu(x).
  \end{align*}
  
  (ii)
  This is a consequence of the explicit description of the operator
  norm (Proposition~\ref{prop:opnormalt}) and the following observation:
  if $U \subset X$ is a measurable subset such that
  there $\mu$-exists an~$x \in X$ with~$m_x(U) > 0$,
  then this already implies that~$\mu(U) > 0$.
    
  Indeed, taking into account that 
  \[\{x \in X \mid m_x(U) > 0 \}
  = \bigcup_{n \in \N} \{x \in X \mid m_x(U) > 1/n \},\]
   we see that there exists a~$\delta \in \R_{>0}$ such that
  there $\mu$-exists an~$x \in X$ with measure~$m_x(U) > \delta$. 
  Let $A \coloneqq \{x \in X \mid m_x(U) > \delta\}$. Then $A$
  is measurable and $\mu(A) > 0$. By construction, we have
  \[ \mu(U)
  = \int_X m_x(U)\;d\mu(x)
  \geq \int_A m_x (U) \;d\mu(x)
  >\delta \cdot \mu(A)
  > 0.
  \] 
  
  (iii)
  As $\lognorm$ is defined as an infimum over all marked
  decompositions and the different branches of~$\lognorm'$,
  it suffices to consider the following situation:
  Let $M = \bigoplus_{i \in I} M_i \oplus \bigoplus_{j \in J} M'_j$
  be a marked decomposition of~$M$, let $N_j$ be marked direct
  summands of~$N$ with~$f(M_j') \subset N_j$ and let
  \[ \ell
  \coloneqq \sum_{i \in I} \dim_R (M_i) \cdot \log_+ \|f|_{M_i}\|_R
  + \sum_{j \in J} \dim_R (N_j) \cdot \log_+ \|f|_{M'_j}\|_R.
  \]
  Similarly, for~$x \in X$, we define~$\ell(x)$ over~$R(x)$.
  Using the second and the first part, we obtain
  \begin{align*}
    \int_X \ell(x) \; d\mu(x)
    &
    \leq 
    \int_X \biggl(\sum_{i \in I} \dim_{R(x)} (M_i(x)) \cdot \log_+ \|f|_{M_i}\|_R
    \\
    &\quad + \sum_{j \in J} \dim_{R(x)} (N_j(x)) \cdot \log_+ \|f|_{M'_j}\|_R\biggr) \; d\mu(x)
    \\
    & = \sum_{i \in I} \dim_R (M_i) \cdot \log_+ \|f|_{M_i}\|_R
    + \sum_{j \in J} \dim_R (N_j) \cdot \log_+ \|f|_{M'_j}\|_R
    \\
    & = \ell.
  \end{align*}
  In particular, there
  $\mu$-exists an~$x \in X$ with~$\ell(x) \leq \ell$;
  therefore,
  \[ \lognorm_{R(x)} (f(x)) \leq \ell(x) \leq \ell.
  \]
  Considering all marked situations as in the definition
  of~$\ell$ (and thus of~$\lognorm_R (f)$) proves the claim.
\end{proof}

\begin{proof}[Proof of Proposition~\ref{prop:disint}]
  Let 
  \[ \begin{tikzcd}
    C_*
    \ar{r}{f_*}
    \ar[two heads]{d}
    &
    D_*
    \ar[two heads]{d}
    \\
    Z
    \ar{r}
    &
    \linf{X,\mu,Z}
    \end{tikzcd}
  \]
  be an $\alpha$-embedding (over~$R$).
  Then, for $\mu$-almost every~$x \in X$, we obtain a corresponding
  $(\alpha,m_x)$-embedding (over~$R(x)$) of the following form:
  \[ \begin{tikzcd}
    C_*
    \ar{r}{f_*(x)}
    \ar[two heads]{d}
    &
    D_*(x)
    \ar[two heads]{d}
    \\
    Z
    \ar{r}
    &
    \linf{X,m_x,Z}
    \end{tikzcd}
  \]
  By Lemma~\ref{lem:disintdimnorm}, there $\mu$-exists
  an~$x\in X$ with~$\dim_{R(x)} (D_n(x)) \leq \dim_R (D_n)$. 
  Analogously, for each~$\varepsilon \in \R_{>0}$, by Lemma~\ref{lem:disintdimnorm},
  there $\mu$-exists an~$x \in X$ that satisfies~$\lognorm_{R(x)} (\partial^{D(x)}_{n+1})
  \leq \lognorm_R (\partial^D_{n+1}) + \varepsilon$. 

  Therefore, considering all $\alpha$-embeddings, we obtain the claimed
  approximative disintegration estimates.
\end{proof}

As straightforward consequences of the disintegration
estimate (Proposition~\ref{prop:disint}), we obtain:

\begin{cor}[ergodic actions suffice]\label{cor:ergsuffices}
  Let $n \in \N$, let $\Gamma$ be a group of type~$\FP_{n+1}$, and let
  $\alpha$ be a standard $\Gamma$-action.
  Moreover, let $\varepsilon \in \R_{>0}$. Then:
  \begin{enumerate}[label=\enum]
  \item
    There exists an ergodic standard
    $\Gamma$-action~$\beta$ with
    \[ \medim_n^Z (\beta) \leq \medim_n^Z (\alpha) + \varepsilon.
    \]
  \item
    There exists an ergodic standard
    $\Gamma$-action~$\beta$ with
    \[ \mevol_n (\beta) \leq \mevol_n (\alpha) + \varepsilon.
    \]
  \end{enumerate}
\end{cor}
\begin{proof}
  The standard $\Gamma$-action~$\alpha$ admits an ergodic
  decomposition~\cite{varadarajan}. Applying the disintegration
  estimate (Proposition~\ref{prop:disint}) to such an
  ergodic decomposition proves the claim.
\end{proof}

\begin{cor}\label{cor:EMDreduction}
  Let $n \in \N$, let $\Gamma$ be a residually finite group of type~$\FP_{n+1}$
  that satisfies~$\EMD^*$, 
  and let
  $\alpha$ be a standard $\Gamma$-action.
  Then
  \[
  \medim^Z_n (\Gamma \actson \widehat \Gamma) \leq \medim^Z_n (\alpha)
  \qand
  \mevol_n (\Gamma \actson \widehat \Gamma) \leq \mevol_n (\alpha).
  \]
\end{cor}
\begin{proof}
  We only give the proof for~$\mevol$; the proof for~$\medim$ works in
  the same way.  We write~$\gamma \colon \Gamma \actson \widehat
  \Gamma$ for the profinite completion action.  Let $\varepsilon \in
  \R_{>0}$. By Corollary~\ref{cor:ergsuffices}, there exists an
  ergodic standard $\Gamma$-action~$\beta$ with~$\mevol_n (\beta)
  \leq \mevol_n (\alpha) + \varepsilon$. Because $\Gamma$
  satisfies~$\EMD^*$ and $\beta$ is ergodic, we have~$\beta \prec
  \gamma$.  Therefore, the weak containment estimate
  (Theorem~\ref{thm:wc}) gives
  \[ \mevol_n (\gamma) \leq \mevol_n (\beta) \leq \mevol_n (\alpha) + \varepsilon.
  \]
  Taking~$\varepsilon \to 0$ shows that $\mevol_n (\gamma) \leq \mevol_n (\alpha)$.
\end{proof}

\begin{cor}\label{cor:trivialproduct}
  Let $n \in \N$, let $\Gamma$ be a group of type~$\FP_{n+1}$, let
  $\alpha \colon \Gamma \actson (X,\mu)$ be a standard $\Gamma$-action,
  and let $\beta \colon \Gamma \actson (Y,\nu)$ be the trivial $\Gamma$-action
  on a standard Borel probability space~$(Y,\nu)$. Then, the diagonal
  action~$\alpha \times \beta \colon \Gamma \actson (X\times Y, \mu \otimes \nu)$
  is a standard $\Gamma$-action and
  \[ \medim_n^Z (\alpha \times \beta)
  = \medim_n^Z (\alpha)
  \qand
  \mevol_n (\alpha \times \beta)
  = \mevol_n (\alpha).
  \]
\end{cor}
\begin{proof}
  We obtain~``$\leq$'' from the weak containment estimate (Theorem~\ref{thm:wc})
  and the fact that~$\alpha \prec \alpha \times \beta$.
  More elementarily, one can also show this directly, by taking the product
  with~$(Y,\nu)$ at every step.

  For the estimate~``$\geq$'', we use the disintegration
  \begin{align*}
    m \colon X \times Y & \to \Prob(\alpha \times \beta)
    \\
    (x,y) & \mapsto
    \bigl( A \mapsto \mu (A_y)\bigr)
  \end{align*}
  of~$\alpha \times \beta$. Here, $A_y \subset X$ denotes the image
  of~$A \cap (X \times \{y\})$ under the canonical bijection~$X \times
  \{y\} \to X$. For $\mu \otimes \nu$-almost every~$(x,y) \in X \times
  Y$, the standard action~$(\alpha \times \beta,m_{(x,y}))$ is
  canonically isomorphic in the measured sense to~$\alpha$ (through
  the canonical projection~$X \times Y \to X$).

  By the disintegration estimate, for every~$\varepsilon \in \R_{>0}$,
  there $\mu \otimes \nu$-exists an~$(x,y) \in X \times Y$ with
  \[ \mevol_n (\alpha)
  = \mevol_n (\alpha \times \beta, m_{(x,y)})
  \leq \mevol_n (\alpha \times \beta)
  + \varepsilon.
  \]
  Taking~$\varepsilon \to 0$ shows that $\mevol_n (\alpha) \leq \mevol_n
  (\alpha \times \beta)$.  The argument for~$\medim_n^Z$ can be
  carried out in the same way.
\end{proof}

More generally, we expect that measured embedding dimension
and measured embedding volume of a disintegrated action is
equal to the integral of the corresponding measured embedding
dimensions/volumes.  

%% file: eqrelring.tex
\section{Working over the equivalence relation ring}\label{sec:overZR}

Measured embeddings over the equivalence relation ring can be
approximated in a controlled way by measured embeddings over the
crossed product ring (Corollary~\ref{cor:overZR}).  The compatibility of
marked projective dimensions and logarithmic norms with the
Gromov--Hausdorff distance (Proposition~\ref{prop:dGH},
Proposition~\ref{prop:lognorm}) shows that the measured embedding
dimension and the measured embedding volume can alternatively be
computed via measured embeddings over the equivalence relation ring.

\begin{cor}\label{cor:overZR}
  Let $Z$ denote $\Z$ (with the usual norm) or a finite field (with the trivial norm).
  Let $n \in \N$, let $\Gamma$ be a group of type~$\FP_{n+1}$, and let $\alpha$
  be a standard $\Gamma$-action. Let $C_*$ be a free $Z \Gamma$-resolution
  of~$Z$ that is of finite rank up to degree~$n+1$. Moreover, let $\Rrel$
  denote the orbit relation of~$\alpha$, let $D_*$
  be a marked projective $Z \Rrel$-complex, and let $f_* \colon C_* \to D_*$ be a
  $Z\Gamma$-chain map extending the canonical inclusion~$Z \hookrightarrow \linf {\alpha,Z}$.

  Then there exists a $K \in \R_{>0}$ such that: For every~$\delta \in \R_{>0}$, there
  exists a marked projective $\linf{\alpha,Z}*\Gamma$-chain complex~$\widehat D_*$
  and a $Z\Gamma$-chain map~$\widehat f_* \colon C_* \to \widehat D_*$ extending
  the canonical inclusion~$Z \hookrightarrow \linf{\alpha,Z}$ with
  \begin{align*}
    \dgh K {\Ind {\linf{\alpha,Z} * \Gamma}{Z\Rrel}\widehat D_*, D_*,n}
    & < \delta
  \\
  \fa{r \in \{0,\dots, n+1\}}
  \dgh K {\Ind {Z\Gamma}{Z\Rrel} (\widehat f_r), \Ind {Z\Gamma}{Z\Rrel} (f_r)}
  & < \delta 
  .
  \end{align*}
\end{cor}
\begin{proof}
  This is the special case of Theorem~\ref{thm:adaptedembgen}, where the
  subalgebra is the ring~$Z\Rrel$ and where
  the subalgebra~$S$ is the algebra of all measurable subsets.
\end{proof}

\begin{prop}[small resolutions over the equivalence relation ring]\label{prop:smallres}
  Let $Z$ denote~$\Z$ (with the standard norm) or a finite field
  (with the trivial norm). 
  Let $\Rrel$ be a measured standard equivalence relation on
  a standard Borel probability space~$(X,\mu)$, let $n \in \N$,
  and let $D_*$ be a marked projective $Z\Rrel$-resolution
  of~$\linf {\alpha,Z}$ (up to degree~$n+1$). Then:
  If $\Gamma$ is a countable group of type~$\FP_{n+1}$ and if $\alpha$
  is standard probability action of~$\Gamma$ on~$(X,\mu)$ that
  induces~$\Rrel$, then
  \begin{align*}
    \medim^Z_n(\alpha)
    & \leq \dim_{Z\Rrel} (D_n)
    \\
    \mevol_n(\alpha)
    & \leq \lognorm (\partial^{D}_{n+1})
    \quad\text{if $Z = \Z$}.
  \end{align*}
\end{prop}
\begin{proof}
  Let $C_*$ be a free $Z\Gamma$-resolution of~$Z$ of finite type (up
  to degree~$n+1$).  By Corollary~\ref{cor:overZR} and
  Propositions~\ref{prop:dGH}/\ref{prop:lognorm}, it suffices to
  find a $Z\Gamma$-chain map~$C_* \to D_*$ extending the
  canonical inclusion~$Z \hookrightarrow \linf {\alpha,Z}$.

  Let $\widehat C_* \coloneqq \Ind {Z\Gamma}{Z\Rrel} (C_*)$. Then $\widehat
  C_*$ is a marked projective $Z \Rrel$-complex (augmented over~$\linf
  {\alpha,Z}$, up to degree~$n+1$). Because $D_*$ is a
  $Z\Rrel$-resolution of~$\linf {\alpha,Z}$, by the fundamental lemma
  of homological algebra, there exists a $Z\Rrel$-chain map~$\widehat
  f_* \colon \widehat C_* \to D_*$ extending~$\id_{\linf
    {\alpha,Z}}$. Hence, the composition of~$\widehat f_*$ with the
  canonical chain map~$C_* \to \widehat C_*$ induced by the
  inclusion~$Z\Gamma \hookrightarrow Z\Rrel$ has the desired
  properties.
\end{proof}

Unfortunately, it is not clear whether/how these considerations lead
to a meaningful version of measured embedding dimension/volume that is
invariant under orbit equivalence. For instance, it is not clear in
how many cases $Z\Rrel$-resolutions (satisfying the implicit
finiteness conditions in marked projectivity) as in
Proposition~\ref{prop:smallres} actually exist. The case of weakly
bounded orbit equivalence is accessible and treated in
Section~\ref{sec:wbOE}.

%% file: wbOE.tex
\section{Weak bounded orbit equivalence}\label{sec:wbOE}

Because $L^2$-Betti numbers are compatible with orbit
equivalence~\cite{gaboriaul2}, the homology growth
over~$\Q$ shares the same property for residually finite
groups of finite type (via the approximation
theorem~\cite{Lueck94_approx}).
It is an open problem to determine how (vanishing of)
homology gradients over finite fields or torsion homology
growth behaves under orbit equivalence. As a step towards
this problem, we show that measured embedding dimension
and measured embedding volume are compatible with weak
bounded orbit equivalences. In particular, these invariants
provide upper bounds for homology growth over finite finite fields
and for torsion homology growth that are compatible with
weak bounded orbit equivalences.

\begin{setup}
  In this section, 
  let $Z$ denote $\Z$ with the standard norm or a finite
  field with the trivial norm.
\end{setup}

\begin{thm}[weak bounded orbit equivalence and $\medim$, $\mevol$]\label{thm:wbOE}
  Let $n \in \N$, 
  let $\Gamma$ and $\Lambda$ be groups of type~$\FP_{n+1}$, and let
  $\alpha$ and $\beta$ be standard actions of $\Gamma$ and $\Lambda$,
  respectively, that are weakly bounded orbit equivalent of index~$c$.
  Then, we have
  \begin{align*}
    \medim^Z_n (\alpha)
    & = c \cdot \medim^Z_n (\beta), \\
    \mevol_n (\alpha)
    & = c \cdot \mevol_n (\beta).
  \end{align*}
\end{thm}

The theorem will be derived from a corresponding statement
on measured embeddings over truncated crossed product rings.
At the moment, the case of general orbit equivalence is out
of reach, because the equivalence relation rings do not
exhibit the same level of exactness and finiteness properties
as the crossed product rings.

As a preparation for the proof, we discuss dimensions and norms over
truncated crossed product rings (Section~\ref{subsec:truncring}) as
well as truncated version of measured embeddings, measured embedding
dimension, and measured embedding volume (Section~\ref{subsec:truncemb}).

\subsection{Truncated crossed product rings}\label{subsec:truncring}

Let $R$ be a unital ring. If $p \in R$ is idempotent, then $pRp$ is a
unital ring. The idempotent~$p$ is called \emph{full} if the
multiplication homomorphism~$Rp \otimes_{pRp} pR \to R$ is surjective.

If $p \in R$ is a full idempotent, then $R$ and $pRp$ are Morita
equivalent through the functors~$pR \otimes_R \args$ and $Rp
\otimes_{pRp} \args$. In particular, we can translate freely
between projective resolutions over~$R$ and projective resolutions
over~$pRp$.

For convenience, we prove the following well-known fact that we need repeatedly. 

\begin{lem}\label{lem:fullidempotent}
If $p$ is a full idempotent of~$R$, then the multiplication homomorphism~$Rp \otimes_{pRp} pR \to R$ is bijective. 
\end{lem}

\begin{proof}
By definition, the multiplication homomorphism $m$ is surjective. Consider 
a pre-image $\sum_{i\in I} r_ip\otimes ps_i$ of $1\in R$. That is, $\sum_{i\in I} r_ips_i=1$. There is a homomorphism $f\colon R\to Rp\otimes_{pRp}pR$ of left $R$-modules that maps $1$ to $\sum_{i\in I}r_ip\otimes ps_i$. We claim that $f$ is the inverse of $m$. 
It is obvious that $m\circ f=\id$. The other composition $f\circ m=\id$ follows from: 
\begin{align*}
f(m(xp\otimes py)) &= f(xpy)\\
                   &= xpy\cdot\sum_{i\in I}r_ip\otimes ps_i \\
                   &= \sum_{i\in I}xp(pyr_ip)\otimes ps_i \\
                   &= \sum_{i\in I} xp\otimes pyr_ips_i\\
                   &= xp\otimes py\bigl(\sum_{i\in I}r_ips_i\bigr)\\
                   &= xp\otimes py. \qedhere                   
\end{align*}	
\end{proof}

In the context of measured embeddings, we need more refined information:
We need to preserve marked projectivity (instead of projectivity)
and control over the dimensions and norms of maps.

\begin{defn}
  Let $\alpha \colon \Gamma \actson (X,\mu)$ be a standard action.
  A subset~$A \subset X$ is \emph{$\alpha$-cofinite} if it is measurable
  and if there exists a finite set~$F \subset \Gamma$ with~$F \cdot A = X$
  (up to $\mu$-measure zero).
\end{defn}

Clearly, cofinite sets in this sense have non-zero measure. 

\begin{rem}[existence of small cofinite subsets]\label{rem:cofinitesmall}
  Let $\alpha \colon \Gamma \actson (X,\mu)$ be a standard action of
  an infinite group~$\Gamma$ and let $\varepsilon \in \R_{>0}$.  Then,
  there exists an $\alpha$-cofinite subset~$A \subset X$ with~$\mu(A)
  < \varepsilon$: Indeed, there exists a measurable subset~$A' \subset
  X$ with~$\Gamma \cdot A' = X$ and~$\mu(A') <
  \varepsilon/2$~\cite[Proposition~1]{Levitt95}. We choose a large
  enough finite subset~$F \subset \Gamma$ with $\mu(X \setminus F
  \cdot A') < \varepsilon/2$. Then $A \coloneqq A' \cup (X \setminus F \cdot
  A')$ is $\alpha$-cofinite and~$\mu(A) < \varepsilon$.
\end{rem}

\begin{defn}[marked projectives, dimensions, $\lognorm$ over truncated crossed product rings]
  Let $\alpha \colon \Gamma \actson (X,\mu)$ be a standard action and
  let $A \subset X$ be $\alpha$-cofinite. We write~$R \coloneqq \linf {\alpha, 
  Z} * \Gamma$ and $p \coloneqq (\chi_A,1) \in R$.
  \begin{itemize}
  \item If $B \subset A$ is a measurable subset, we write
    \[  \gen B _{pRp} \coloneqq pRp \cdot \chi_B.
    \]
  \item A \emph{marked projective $pRp$-module} is a triple~$(M,
    (B_i)_{i\in I}, \varphi)$, consisting of a $pRp$-module~$M$, a
    finite family~$(B_i)_{i \in I}$ of measurable subsets of~$A$, and
    a $pRp$-isomorphism~$\varphi \colon M \to \bigoplus_{i \in I}
    \gen{B_i}_{pRp}$.
  \item The \emph{dimension} of a marked projective~$pRp$-module~$(M, (B_i)_{i \in I}, \varphi)$
    is given by
    \[ \dim_{pRp} (M) \coloneqq \sum_{i \in I} \frac{\mu(B_i)}{\mu(A)} \in \R_{\geq 0}.
    \]
  \item If $f \colon M \to N$ is a $pRp$-linear map between marked
    projective $pRp$-modules, then we define~$\lognorm_{pRp} (f)$
    analogously to the non-truncated case.
  \end{itemize}
  As in the case of marked projective modules in the non-truncated
  case, we will usually leave~$\varphi$ implicit.
\end{defn}
  
\begin{prop}[truncation and marked projectivity]\label{prop:truncmarkedproj}
  Let $\alpha \colon \Gamma \actson (X,\mu)$ be a standard action and
  let $A \subset X$ be $\alpha$-cofinite. We write~$R \coloneqq \linf {\alpha, 
  Z} * \Gamma$ and $p \coloneqq (\chi_A,1) \in R$. Then:
  \begin{enumerate}[label=\enum]
  \item The idempotent~$p$ is full in~$R$.
  \item If $M$ is a marked projective $R$-module,
    then $pM$ inherits a decomposition as a marked projective $pRp$-module
    with 
    \[ \dim_{pRp} (pM) = \frac1{\mu(A)} \cdot \dim_R (M).
    \]
  \item
    If $f \colon M \to N$ is an $R$-linear map between marked
    projective $R$-modules, then (with respect to a marked decomposition
    as in the previous item)
    \[ \lognorm_{pRp} (pf) \leq \frac1{\mu(A)} \cdot \lognorm_R (f).
    \]
  \item If $M$ is a marked projective $pRp$-module,
    then $Rp \otimes_{pRp} M$ inherits a canonical decomposition
    as a marked projective $R$-module
      with
      \[ \dim_R (Rp\otimes_{pRp} M)
      = \mu(A) \cdot \dim_{pRp} (M).
      \]
  \item
    If $f \colon M \to N$ is a $pRp$-linear map between marked
    projective $pRp$-modules, then (with respect to the marked decomposition
    as in the previous item)
    \[ \lognorm_R (\id \otimes_{pRp} f) \leq \mu(A) \cdot \lognorm_{pRp} (f).
    \]
  \end{enumerate}
\end{prop}

The marked projective structure on~$pM$ depends on certain
choices, but these will not be problematic in our situations.

\begin{proof}
  Because $A$ is $\alpha$-cofinite there exists a finite set~$F
  \subset \Gamma$ with~$F \cdot A =X$ (up to measure~$0$). Through
  inductive removal, we obtain a family~$(B_\gamma)_{\gamma \in F}$ of
  measurable subsets of~$A$ with (up to measure~$0$)
  \[ X = \bigsqcup_{\gamma \in F} \gamma \cdot B_\gamma.
  \]

  (i)
  With this decomposition, we can write the unit in~$R$ as
  \[
  1 = \sum_{\gamma \in F} (\chi_{\gamma \cdot B_\gamma},1)
  = \sum_{\gamma \in F} (1,\gamma) \cdot (\chi_{B_\gamma},1) \cdot (1,\gamma^{-1})
  = \sum_{\gamma \in F} (1,\gamma) \cdot p \cdot (\chi_{B_\gamma},1) \cdot (1,\gamma^{-1}),
  \]
  which shows that the idempotent~$p$ is full in~$R$.

  (ii)
  We only need to consider the case that $M$ has rank~$1$. Thus, let $M = \gen B _R$.
  Then (up to measure~$0$)
  \[ B = \bigsqcup_{\gamma \in F} \gamma \cdot B_\gamma \cap B
  = \bigsqcup_{\gamma \in F} \gamma \cdot (B_\gamma \cap \gamma^{-1} \cdot B),
  \]
  and we obtain
  \begin{align*}
    p \gen B _R
    & \cong_{pRp}
    \bigoplus_{\gamma \in F} p \cdot R \cdot (1,\gamma) \cdot (\chi_{B_\gamma \cap \gamma^{-1} \cdot B},1) \cdot (1,\gamma^{-1})
    \\
    & \cong_{pRp}
    \bigoplus_{\gamma \in F} p \cdot R \cdot (1,\gamma) \cdot (\chi_{B_\gamma \cap \gamma^{-1} \cdot B},1)
    & \text{($(1,\gamma^{-1})$ is a unit in~$R$)}
    \\
    & \cong_{pRp}
    \bigoplus_{\gamma \in F} p \cdot R \cdot \chi_{B_\gamma \cap \gamma^{-1} \cdot B}
    & \text{($(1,\gamma)$ is a unit in~$R$)}
    \\
    & = 
    \bigoplus_{\gamma \in F} p \cdot R \cdot p \cdot \chi_{B \gamma \cap \gamma^{-1} \cdot B}
    & \text{($B \gamma \cap \gamma^{-1} \cdot B \subset A$)}
    \\
    & =
    \bigoplus_{\gamma \in F} \gen{ B_\gamma \cap \gamma^{-1} \cdot B}_{pRp}.
  \end{align*}
  This provides a marked projective decomposition of the $pRp$-module~$p \gen B_R$.
  Moreover, with respect to this marked projective decomposition, we have
  \begin{align*}
    \dim_{pRp} (p \gen B_R)
    & = \sum_{\gamma \in F} \frac1{\mu(A)} \cdot \mu(B_\gamma \cap \gamma^{-1} \cdot B)
    = \frac1{\mu(A)} \cdot \sum_{\gamma \in F} \mu(\gamma \cdot B_\gamma \cap B)
    \\
    & = \frac1{\mu(A)} \cdot \mu \Bigl( \bigcup_{\gamma \in F} \gamma \cdot B_\gamma \cap B\Bigr)
    = \frac1{\mu(A)} \cdot \mu(B).
  \end{align*}

  (iii)
  This follows from the dimension estimate in the previous part
  and the definition of the logarithmic norms.
  
  (iv)
  We only need to consider the case that $M$ has rank~$1$. Thus, let $M = \gen B _{pRp}$
  with $B \subset A$ measurable.
  Because $p$ is an idempotent in~$R$ and $(\chi_B, 1) \in pRp$, the canonical $R$-homomorphism
  \[ Rp \otimes_{pRp} \gen B _{pRp} \to \gen B _R
  \]
  is an isomorphism.
  In particular, this gives a canonical marked projective
  structure on~$Rp \otimes_{pRp} M$ and
  \[ \dim_R (Rp \otimes_{pRp} M) = \mu(A) \cdot \dim_{pRp} (M).
  \]

  (v)
  This follows from the dimension estimate in the previous part
  and the definition of the logarithmic norms.
\end{proof}

\subsection{Measured embeddings over truncated crossed product rings}\label{subsec:truncemb}

Let $R$ be a unital ring, let $P \subset R$ be a set of idempotents.
A \emph{marked projective $(R,P)$-module} is an $R$-module~$M$, together
with a decomposition
\[ M \cong_R \bigoplus_{i \in I} R p_i,
\]
where $(p_i)_{i \in I}$ is a finite family in~$P$.

\begin{defn}[embedding]\label{def:genemb}
  Let $R$ be a unital ring and let $P \subset R$ be a set of idempotents.
  Let $L$ be an $R$-module. An \emph{$(R,P)$-embedding (up to degree~$n$)} of~$L$ is a
  triple~$((C_*, \zeta), f_*, (D_*,\eta))$, consisting of
  \begin{itemize}
  \item a free $R$-resolution~$(C_*,\zeta)$ of~$L$,
  \item an $(R,P)$-marked projective augmented $R$-chain complex~$(D_*,\eta)$
    augmenting~$L$,
  \item and an $R$-chain map~$f \colon C_* \to D_*$ up to degree~$n+1$ extending~$\id_L \colon L \to L$.
  \end{itemize}
  \[\begin{tikzcd}
  C_*
  \ar{r}{f_*}
  \ar{d}{\zeta}
  & D_*
  \ar{d}{\eta}
  \\
  L
  \ar[equal]{r}
  & L
  \end{tikzcd}
  \]
\end{defn}

\begin{rem}
  We call these constellations embeddings because they
  are similar to embeddings in a derived sense:
  In Definition~\ref{def:genemb}, 
  if $((C_*,\zeta), f_*, (D_*,\eta))$
  is an $(R,P)$-embedding of the $R$-module~$L$, 
  then the fundamental lemma of homological algebra
  provides us with an $R$-chain map~$g_* \colon D_* \to C_*$
  extending~$\id_L$ that satisfies $g_* \circ f_* \simeq_R \id$.
\end{rem}
  
\begin{rem}[change of domain resolution]\label{rem:embeddingindepres}
  In the situation of Definition~\ref{def:genemb}, if $((C_*,\zeta), f_*,
  (D_*,\eta))$ is an $(R,P)$-embedding of~$L$ and $(C'_*, \zeta')$ is
  a free $R$-resolution of~$L$, then there also
  exists an $(R,P)$-embedding~$((C'_*,\zeta'), f'_*, (D_*,\eta))$ with
  the same target complex~$D_*$ (by the fundamental lemma of
  homological algebra). However, the quantitative properties of the
  maps~$f'_*$ and $f_*$ might be different; in the following
  discussion, we will only take quantitative aspects of the target
  complex into account and therefore the choice of the domain
  resolution of~$L$ will be immaterial.
  The same argument applies if $C'_*$ consists of (finitely generated)
  projective $R$-modules instead of (finitely generated) free $R$-modules.
\end{rem}

\begin{defn}[measured embeddings, truncated case]\label{def:embtrunc}
  Let $\alpha \colon \Gamma \actson (X,\mu)$ be a standard action,
  let $R \coloneqq \linf{\alpha,Z} * \Gamma$, 
  let $A \subset X$ be $\alpha$-cofinite,
  and let $p \coloneqq (\chi_A,1) \in R$.
  We write~$P_A \coloneqq \{ (\chi_B,1) \mid \text{$B \subset A$ measurable} \}$.
  \begin{itemize}
  \item
    An \emph{$(\alpha,A,Z)$-measured embedding (up to degree~$n$)} is a $(pRp, P_A)$-embedding up to degree~$n$
    of~$p\linf{\alpha,Z}$ in the sense of Definition~\ref{def:genemb}.
  \item
    We write~$\Aug_n(\alpha,A,Z)$ for the class of all augmented
    target complexes arising in $(\alpha,A,Z)$-measured embeddings up to degree~$n$.
  \end{itemize}
\end{defn}

In the situation of Definition~\ref{def:embtrunc}, it might be helpful
to understand that the $pRp$-module~$p\linf{\alpha,Z}$ is canonically
isomorphic to~$\linf{A,Z}$. The $pRp$-structure on~$\linf{A,Z}$ is given
by
\[ p (f,\gamma) p \cdot (\chi_B,1)
= \chi_{A \cap \gamma \cdot B} \cdot f
\]
for all~$f \in \linf {\alpha,Z}$, $\gamma \in \Gamma$, and all measurable
subsets~$B \subset A$ (Remark~\ref{rem:linfX-ZR}).

\begin{defn}[$\medim$ and $\mevol$, truncated case]
  Let $\alpha \colon \Gamma \actson (X,\mu)$ be a standard action,
  let $R \coloneqq \linf{\alpha,Z} * \Gamma$, 
  let $A \subset X$ be $\alpha$-cofinite,
  and let $p \coloneqq (\chi_A,1) \in R$. Let $n \in \N$
  and let $\Gamma$ be of type~$\FP_{n+1}$. We then
  set
  \begin{align*}
    \medim^Z_n(\alpha,A)
    & \coloneqq \inf_{D_* \in \Aug_n(\alpha,A,Z)} \dim_{pRp} (D_n)
    \\
    \mevol_n(\alpha,A)
    & \coloneqq \inf_{D_* \in \Aug_n(\alpha,A,\Z)} \lognorm_{pRp} (\partial^{D}_{n+1})
    \quad \text{if $Z = \Z$}.
  \end{align*}
\end{defn}

\begin{rem}[starting from resolutions over the group ring]
  Let $n \in \N$, let $\Gamma$ be of type $\FP_\infty$, and let $(C_*,\zeta)$
  be a free $Z \Gamma$-resolution of the trivial $Z\Gamma$-module~$Z$.  
  Let $\alpha \colon \Gamma \actson (X,\mu)$ be a standard
  action and let $R \coloneqq \linf {\alpha,Z} * \Gamma$. 
  \begin{enumerate}
  \item Then $(R \otimes_{Z\Gamma} C_*, \id \otimes_Z \zeta)$ is a
    free $R$-resolution of the $R$-module~$\linf {\alpha,Z} \cong_R
    R\otimes_{Z \Gamma} Z$ (with the canonical
    action; Remark~\ref{rem:linfX-ZR}),
    because $R$ is flat over~$Z \Gamma$ (Proposition~\ref{prop:flat}).
  \item If $A \subset X$ is $\alpha$-cofinite and $p \coloneqq (\chi_A, 1)
    \in R$ is the associated full idempotent, then $(pR
    \otimes_{Z\Gamma} C_*, \id \otimes_{Z\Gamma}Z)$ is a projective
    $pRp$-resolution of the $pRp$-module~$pR\otimes_{Z \Gamma} Z \cong_{pRp}
    p\linf{\alpha,Z}$ (with the canonical action).
  \end{enumerate}
  Therefore, we may start from resolutions over the group ring~$Z\Gamma$
  when computing or estimating measured embedding dimensions/volumes
  (Remark~\ref{rem:embeddingindepres}).

  In particular, for $A = X$, the notions of embeddings and measured
  embedding dimension/volume coincide with the ones from
  Section~\ref{subsec:introsetup}.
  
  A similar remark applies when $\Gamma$ is of type $\FP_{n+1}$. 
\end{rem}

\subsection{Weak bounded orbit equivalence}

We recall the notion of weak bounded orbit equivalence, which
is a more restrictive type of stable orbit equivalence.

\begin{defn}[weak bounded orbit equivalence]\label{def:wbOE}
  Standard actions $\alpha \colon \Gamma \actson (X,\mu)$
  and $\beta \colon \Lambda \actson (Y,\nu)$ are
  \emph{weakly bounded orbit equivalent of index~$c \in \R_{>0}$}
  if there exists an $\alpha$-cofinite set~$A \subset X$,
  a $\beta$-cofinite set~$B \subset Y$, and a measurable
  isomorphism~$\varphi \colon A \to B$ with the following properties:
  \begin{itemize}
  \item We have
    \[ \frac1{\mu(A)} \cdot \varphi_* \mu|_A
    = \frac1{\nu(B)} \cdot \nu|_B
    \qand
    c = \frac{\mu(A)}{\nu(B)}.
    \]
  \item For every~$\gamma \in \Gamma$, there is a finite
    set~$F(\gamma)\subset \Lambda$ such that for $\mu$-almost
    every~$x \in A \cap \gamma^{-1} \cdot A$,
    we have~$\varphi(\gamma \cdot x) \in F(\gamma)\cdot \varphi(x)$.
  \item For every~$\lambda \in \Lambda$, there is a finite
    set~$E(\lambda) \subset \Gamma$ such that for $\nu$-almost
    every~$y \in B \cap \lambda^{-1} \cdot B$,
    we have~$\varphi^{-1}(\lambda \cdot y) \in E(\lambda) \cdot \varphi^{-1}(y)$.
  \end{itemize}
\end{defn}

\begin{ex}[uniform lattices]\label{ex:uniform:lattices:wbOE}
  Let $G$ be a locally compact second countable Hausdorff topological
  group and let $\Gamma, \Lambda \subset G$ be uniform lattices
  in~$G$. Then the standard actions~$\Gamma \actson G/\Lambda$ and
  $\Lambda \actson G/\Gamma$ (with respect to the normalised Haar
  measure) are weakly boundedly orbit equivalent~\cite[Example~2.31]{sauer2002l2}.
  For instance, fundamental groups of closed Riemannian manifolds
  are uniform lattices in the isometry group of the Riemannian universal
  covering~\cite[Example~2.31 and Theorem~2.36]{sauer2002l2}.
\end{ex}

\begin{ex}[amenable groups]\label{ex:ame:}
  Let $\Gamma$ and $\Lambda$ be infinite finitely generated amenable groups.
  Then $\Gamma$ and $\Lambda$ admit weakly boundedly orbit equivalent
  standard actions if and only if $\Gamma$ and $\Lambda$ are
  quasi-isometric~\cite[Lemma~2.25 and Theorem~2.38]{sauer2002l2}.

 On the other hand, Gaboriau describes examples of pairs of groups $\Gamma\times F_n$ 
 and $\Gamma\times F_m$ for $n\ne m$ that are non-amenable and quasi-isometric but 
 not weakly orbit equivalent~\cite[Section~2.3]{gaboriau-survey}. 
\end{ex}

\begin{ex}[weak bounded orbit equivalence vs.~weak orbit equivalence]
  Let $\Gamma$ be a cocompact lattice in $\mathrm{SL}_n(\IR)$. Then $\Gamma$ and $\mathrm{SL}_n(\IZ)$ are measure equivalent, hence weakly orbit equivalent~\cite[Theorem~2.33]{sauer2002l2}. However, $\Gamma$ and $\mathrm{SL}_n(\IZ)$ are not quasi-isometric and therefore not weakly boundedly orbit equivalent~\cite[Lemma~2.25]{sauer2002l2}. This can be deduced from the invariance of the cohomological dimension under quasi-isometry~\cite[Corollary~1.2]{sauer-homological}. 
\end{ex}

\begin{rem}[invariants under weak (bounded) orbit equivalence]
The following group invariants are invariants -- or invariants up to scaling by the index -- under weak bounded orbit equivalence: 
\begin{itemize}
\item the Novikov--Shubin invariants among groups of type $\FP_\infty$~\cite{sauer-homological};  
\item the $L^2$-torsion among groups of type $\mathrm{F}$ and actions that satisfy the measure-theoretic determinant conjecture~\cite{lueck+sauer+wegner}; 
\item the integral foliated simplicial volume among fundamental groups of closed aspherical manifolds~\cite{loeh_pagliantini}.  
\end{itemize}

The $L^2$-torsion and the integral foliated simplicial volume might be invariants of weak orbit equivalence for all we know. The Novikov--Shubin invariants are definitely not as the example of $\IZ$ and $\IZ^2$ shows. 
Finally, $L^2$-Betti numbers are invariants up to scaling by the index of weak orbit equivalence by Gaboriau's theorem~\cite{gaboriaul2}.  
\end{rem}

\subsection{Proof of Theorem~\ref{thm:wbOE}}

The proof of Theorem~\ref{thm:wbOE} consists of two components:
\begin{enumerate}
\item We relate measured embeddings over the full crossed product
  ring to measured embeddings over truncated crossed product rings
  (Proposition~\ref{prop:truncemb}).
\item We use the given weak bounded orbit equivalence to compare
  measured embeddings over the corresponding truncated crossed product
  rings of the two standard actions
  (Proposition~\ref{prop:wbOEemb}).
\end{enumerate}

\begin{prop}\label{prop:truncemb}
  Let $\alpha \colon \Gamma \actson (X,\mu)$ be a standard action
  and let $A \subset X$ be $\alpha$-cofinite. Let $n \in \N$ and let $\Gamma$
  be of type~$\FP_{n+1}$. Then:
  \begin{align*}
    \medim^Z_n (\alpha)
    & = \mu(A) \cdot \medim^Z_n (\alpha,A)
    \\
    \mevol_n (\alpha)
    & = \mu(A) \cdot \mevol_n (\alpha,A)
    \quad \text{if $Z = \Z$.}
  \end{align*}
  More precisely (where $R \coloneqq \linf{\alpha,Z} * \Gamma$ and $p \coloneqq (\chi_A,1)$):
  \begin{enumerate}[label=\enum]
  \item If $(D_*,\eta) \in \Aug_n(\alpha,Z)$, then $(pD_*,p\eta) \in \Aug_n(\alpha,A,Z)$.
  \item If $(D_*,\eta) \in \Aug_n(\alpha,A,Z)$, then $(Rp \otimes_{pRp} D_*, \id \otimes_{pRp} \eta)
    \in \Aug_n(\alpha,Z)$.
  \end{enumerate}
\end{prop}
\begin{proof}
  In view of the definition of $\medim$ and $\mevol$ in terms of
  measured embeddings and the compatibility of $pR \otimes_R \args$
  and $Rp \otimes_{pRp}\args$ with dimensions and $\lognorm$
  (Proposition~\ref{prop:truncmarkedproj}), it suffices to show the two
  claims on measured embeddings.

  We abbreviate~$L \coloneqq \linf {\alpha,Z}$.
  
  (i)  Let $(D_*,\eta) \in \Aug_n(\alpha,Z)$. We choose a free
  $R$-resolution~$(C_*,\zeta)$ of~$L$ of finite type. Then there
  exists an $R$-chain map~$f_* \colon C_* \to D_*$ extending~$\id_L$
  (Remark~\ref{rem:embeddingindepres}).

  Applying the functor~$pR \otimes_R \args$, we obtain a $pRp$-chain
  map~$pf_* \colon pC_* \to pD_*$ extending~$\id_{pL}$.  The complexes
  $pC_*$ and $pD_*$ consist of marked projective $pRp$-modules
  (Proposition~\ref{prop:truncmarkedproj}). Moreover, $pR
  \otimes_R \args$ is exact, because $pR$ is projective (as $p$ is
  idempotent); hence, $(pC_*,p\zeta)$ is a projective $pRp$-resolution
  of~$pL$ and $(pD_*,p\eta)$ augments to~$pL$.

  Therefore, $pf_*$ witnesses that there also exists an
  $(\alpha,A,Z)$-measured embedding with target~$(pD_*,p\eta)$
  (Remark~\ref{rem:embeddingindepres}); i.e., $(pD_*,p\eta) \in
  \Aug_n(\alpha,A,Z)$.
  
  (ii) Conversely, let $(D_*,\eta) \in \Aug_n(\alpha,A,Z)$. Let
  $(C_*,\zeta)$ be a free $R$-resolution of~$L$. Because $pR$ is
  projective over~$R$ (whence flat), the induced
  complex~$(pC_*,p\zeta)$ is a projective $pRp$-resolution of~$pL$.
  Hence, there exists a $pRp$-chain map~$f_* \colon pC_* \to D_*$
  extending~$\id_{pL}$.

  Applying the functor~$Rp \otimes_{pRp} \args$ and the canonical
  natural isomorphism~$Rp \otimes_{pRp} pR \otimes_R \args \cong \id$
  from Lemma~\ref{lem:fullidempotent}, we thus obtain an $R$-chain
  map~$\id \otimes_{pRp} f_* \colon C_* \cong_R Rp \otimes_{pRp} pC_* \to Rp \otimes_{pRp} D_*$
  extending the identity on~$L \cong_R Rp\otimes_{pRp} pL$.

  Moreover, $Rp \otimes_{pRp} D_*$ consists of marked projective
  $pRp$-modules (Proposition~\ref{prop:truncmarkedproj}) and $(Rp
  \otimes_{pRp} D_*, \id \otimes_{pRp} \eta)$ augments to~$L \cong_R
  Rp\otimes_{pRp} pL$ (by right-exactness of $Rp \otimes_{pRp}
  \args$). 

  Therefore, $\id \otimes_{pRp} f_*$ witnesses that there also exists
  an $(\alpha,Z)$-measured embedding with target~$(Rp \otimes_{pRp} D_*, \id \otimes_{pRp} \eta)$
  (Remark~\ref{rem:embeddingindepres});
  i.e., $(Rp \otimes_{pRp} D_*, \id \otimes_{pRp} \eta)$
   lies in~$\Aug_n(\alpha,Z)$.
\end{proof}

\begin{prop}\label{prop:wbOEemb}
  Let $n \in \N$ and let $\Gamma$, $\Lambda$ be groups
  of type~$\FP_{n+1}$. Suppose that there exist
  standard actions~$\alpha \colon \Gamma \actson (X,\mu)$
  and $\beta \colon \Lambda \actson (Y,\nu)$ that
  admit a weak bounded orbit equivalence~$\varphi \colon A \to B$.
  Then:
  \begin{align*}
    \medim^Z_n (\alpha, A)
    & = \medim^Z_n(\beta,B)
    \\
    \mevol_n (\alpha, A)
    & = \mevol_n(\beta,B)
    \quad \text{if $Z = \Z$.}
  \end{align*}
  More precisely: The maps~$\varphi$ and $\psi \coloneqq \varphi^{-1}$ induce
  mutually inverse isomorphisms
  \[ \varphi^* \colon
  (\chi_B,1) \cdot \bigl(\linf{\beta,Z} * \Lambda\bigr) \cdot (\chi_B,1)
  \leftrightarrow
  (\chi_A,1) \cdot \bigl(\linf{\alpha,Z} * \Gamma\bigr) \cdot (\chi_A,1)
  :\! \psi^*
  \]
  of unital rings, which allow to pull back and push forward
  module structures, such that:
  \begin{enumerate}[label=\enum]
  \item
    \label{it:wbOEembphi}
    If $(D_*,\eta) \in \Aug_n(\alpha,A,Z)$, then
    $(\varphi^*D_*, \varphi^*\eta) \in \Aug_n(\beta,B,Z)$.
  \item
    \label{it:wbOEembpsi}
    If $(D_*,\eta) \in \Aug_n(\beta,B,Z)$, then
    $(\psi^*D_*,\psi^*\eta) \in \Aug_n(\alpha,A,Z)$.
  \end{enumerate}
\end{prop}
\begin{proof}
  We abbreviate~$R \coloneqq \linf{\alpha,Z} * \Gamma$, $p\coloneqq (\chi_A,1) \in R$,
  $S \coloneqq \linf{\beta,Z}*\Lambda$, and $q \coloneqq (\chi_B,1)$.
  For~$\lambda \in \Lambda$, let $E(\lambda) \subset \Gamma$ be
  a finite set for~$\varphi$ as in Definition~\ref{def:wbOE}.
  Then \begin{align*}
    \varphi^* \colon 
    qSq & \to pRp
    \\
    q \cdot (g,\lambda) \cdot q
    & \mapsto \sum_{\gamma \in E(\lambda)}
    p \cdot (g \circ \varphi|_{A_{\gamma,\lambda}},\gamma) \cdot p
  \end{align*}
  with~$A_{\gamma,\lambda} \coloneqq \{x \in A \cap \gamma^{-1}\cdot A \mid \varphi(\gamma \cdot x)
  = \lambda \cdot \varphi(x)\}$ 
  gives a well-defined, unital ring homomorphism. 
      The corresponding
  map~$\psi^*$ for~$\psi$ witnesses that $\varphi^*$ is an isomorphism.
  To see this one passes to the (restricted) equivalence relation rings. Let $R_{\alpha, A}\coloneqq \{ (\gamma\cdot x,x)\mid x\in A, \gamma\cdot x\in A, \gamma\in\Gamma\}$ be the 
  equivalence relation ring of the orbit equivalence relation of~$\alpha$ restricted to $A\times A$. Similarly, we define $\calR_{\beta, B}$. 
  Clearly, $\varphi$ induces a measure preserving isomorphism $\calR_{\alpha, A}\cong \calR_{\beta, B}$ of equivalence relations and thus a ring isomorphism $Z\calR_{\beta,B}\xrightarrow{\cong}Z\calR_{\alpha, A}$, which restricts to $\varphi^\ast$ in the following way. Consider the following commutative square of ring homomorphisms. 
  \[\begin{tikzcd}
  qSq\ar[d, hook]\ar[r, "\varphi^\ast"] & pRp\ar[d, hook]\\
  Z\calR_{\beta, B}\ar[r, "\cong"]& Z\calR_{\alpha, A}
  \end{tikzcd}
    \]
  The left vertical map maps $q(g,\lambda)q=(qg(\lambda\cdot q), \lambda)$ to the function $f\colon \calR_{\beta, B}\to Z$ such that 
  \[ f(\lambda'\cdot y, y)=\begin{cases}
                              q(\lambda\cdot y)g(\lambda\cdot y)q(y)& \text{ if $\lambda'=\lambda$;}\\
                              0                                     & \text{ otherwise.}
  \end{cases}
    \]
  The left vertical map is a ring homomorphism (cf.~Subsection~\ref{subsec:rings}). The right vertical map is defined similarly. We verify that the diagram 
  commutes. The element~$f$ is mapped under the lower horizontal map (induced by~$\varphi$) to $f'\in Z\calR_{\alpha, A}$ with 
  \[ f'(\gamma\cdot x, x)=\begin{cases}
                              q(\lambda\cdot \varphi(x))g(\lambda\cdot \varphi(x))q(\varphi(x))& \text{ if $\varphi(\gamma\cdot x)=\lambda\cdot\varphi(x)$;}\\
                              0                                     & \text{ otherwise.}
  \end{cases}
    \]
  Hence $f'$ can be expressed as a sum of functions $f'_\gamma$ ranging over $\gamma\in E(\lambda)$ such that $f'_\gamma$ is supported on $A_{\gamma, \lambda}$. The function $f'_\gamma$ is the image of the $\gamma$-summand in the formula for $\varphi^\ast(q\cdot (g,\lambda)\cdot q)$. 
  So the square commutes. 

  Pulling back the module structure along~$\varphi^*$ defines an exact 
  functor (even an equivalence) from the category of $pRp$-modules to
  the category of $qSq$-modules. By construction,
  $\varphi^* \gen{A}_{pRp} \cong_{qSq} \gen{B}_{qSq}$; 
  therefore, for all measurable subsets~$\widetilde A \subset A$, we obtain
  \[ \varphi^* \gen{\widetilde A}_{pRp}
  \cong_{qSq} \gen{\varphi(\widetilde A)}_{qSq}
  \]
  and
  \[
  \dim_{pRp} (\gen{\widetilde A}_{pRp})
  = \frac{\mu(\widetilde A)}{\mu(A)}
  = \frac{\nu(\varphi(\widetilde A))}{\nu(B)}
  = \dim_{qSq} (\gen{\varphi(\widetilde A)}_{qSq}).
  \]
  Hence, this functor canonically turns marked projective
  $pRp$-modules into marked projective $qSq$-modules with
  the same dimensions; similarly, the logarithmic norm
  of homomorphisms between marked projective modules is preserved.
  Furthermore, we have $\varphi^* p\linf{\alpha,Z} \cong_{qSq} q\linf{\beta,Z}$.

  Analogous statements hold for~$\psi^*$.
  
  This shows the claims~\ref{it:wbOEembphi} and~\ref{it:wbOEembpsi}.
  In view of the compatibility with the dimensions and the logarithmic
  norms, also the statements on measured embedding dimension and
  measured embedding volume follow.
\end{proof}

\begin{proof}[Proof of Theorem~\ref{thm:wbOE}]
  Let $\alpha \colon \Gamma \actson (X,\mu)$ and $\beta \colon \Lambda \actson
  (Y,\nu)$ be the given actions and let $A \subset X$ and $B \subset Y$
  be cofinite sets for which there exists a weak bounded orbit
  equivalence~$\varphi \colon A \to B$. Furthermore, let $n \in \N$.
  We then obtain
  \begin{align*}
    \mevol_n (\alpha)
    & = \mu(A) \cdot \mevol_n(\alpha, A)
    & \text{(Proposition~\ref{prop:truncemb})}
    \\
    & = \mu(A) \cdot \mevol_n(\beta,B)
    & \text{(Proposition~\ref{prop:wbOEemb})}
    \\
    & = \frac{\mu(A)}{\nu(B)} \cdot \nu(B) \cdot \mevol_n(\beta,B)
    \\
    & =\frac{\mu(A)}{\nu(B)} \cdot \mevol_n (\beta)
    & \text{(Proposition~\ref{prop:truncemb})}.
  \end{align*}
  By definition, the first quotient is the index of~$\varphi$.

  The proof for $\medim^Z_n$ works in the same way.
\end{proof}

\subsection{Example: Hyperbolic $3$-manifolds}
\label{sec:ex_hyperbolic}

Let $M$ be a hyperbolic $3$-manifold of finite volume, let $\Gamma \coloneqq
\pi_1(M)$, and let $\Gamma_*$ be a residual chain in~$\Gamma$.
Conjecturally, it is expected that $\widehat t_1(\Gamma,\Gamma_*) =
\vol(M) / 6 \pi$ holds~\cite{bergeron_venkatesh},
where $\vol(M)$
denotes the hyperbolic volume of~$M$. L\^e~\cite[Theorem~1.1]{le_hyperbolic}
proved that indeed
\[ \widehat t_1(\Gamma,\Gamma_*) \leq \frac{\vol(M)}{6 \cdot \pi}
\]
holds (even more generally in the context of homology torsion growth
of orientable, irreducible, compact $3$-manifolds with empty or
toroidal boundary).

In the following, we reproduce the existence of a volume-linear upper
bound for torsion homology growth of closed hyperbolic $3$-manifolds
via the dynamical approach. We follow the strategy for the dynamical
computation of stable integral simplicial volume of
$3$-manifolds~\cite{loeh_pagliantini,FLMQ}.

\begin{thm}\label{thm:mevolhyp3}
  Let $M$ and $N$ be oriented closed connected hyperbolic $3$-manifolds
  and let $\Gamma \coloneqq \pi_1(M)$, $\Lambda \coloneqq \pi_1(N)$. Then
  \begin{align*}
      \frac{\mevol_1 (\Gamma \actson \widehat \Gamma)}{\vol (M)}
      & = \frac{\mevol_1 (\Lambda \actson \widehat \Lambda)}{\vol (N)}
      \\
      \frac{\medim^Z_1 (\Gamma \actson \widehat \Gamma)}{\vol (M)}
      & = \frac{\medim^Z_1 (\Lambda \actson \widehat \Lambda)}{\vol (N)}.
  \end{align*}
\end{thm}
\begin{proof}
  We denote the profinite completion standard actions
  by $\alpha_\Gamma \colon \Gamma \actson \widehat \Gamma$
  and $\alpha_\Lambda \colon \Lambda \actson \widehat \Lambda$,
  respectively. 
  Let $\overline \alpha_\Lambda \colon \Gamma \actson \widehat \Lambda$
  denote the trivial action of~$\Gamma$ on~$\widehat \Lambda$, 
  which is not essentially free. 

  We compare $\Gamma$ and~$\Lambda$ through their actions
  on hyperbolic $3$-space: 
  Let $\beta_\Gamma \colon \Gamma \actson G/\Lambda$
  and $\beta_\Lambda \colon \Lambda \actson G/\Gamma$
  be the canonical standard actions on~$G \coloneqq \Isom^+(\hyp^3)$
  associated with the
  hyperbolic $3$-manifolds $M$ and~$N$. These actions are
  mixing~\cite[Theorem~III.2.1]{bekkamayer} 
  and 
  weakly bounded orbit equivalent with index~$\vol(M)/\vol(N)$ (Example~\ref{ex:uniform:lattices:wbOE}).
  
  As the group~$\Gamma$ satisfies~$\EMD^*$~(Example~\ref{ex:EMD})
  we obtain
  \begin{align*}
    \mevol_1 (\alpha_\Gamma)
    & \leq \mevol_1 (\overline\alpha_\Lambda \times \beta_\Gamma)
    & \text{(Corollary~\ref{cor:EMDreduction})}
  \end{align*}
  for the diagonal action~$\overline \alpha_\Lambda \times
  \beta_\Gamma$ of~$\Gamma$ on~$\widehat \Lambda \times G/\Lambda$. 
  The standard actions~$\overline\alpha_\Lambda \times \beta_\Gamma$
  and $\alpha_\Lambda \times \beta_\Lambda$ are weakly bounded orbit
  equivalent with index~$\vol(M)/\vol(N)$. 
  This follows from that fact that the index of the measure coupling $\widehat\Lambda\times G$  
  that gives rise to this weak bounded orbit equivalence has index $\vol(M)/\vol(N)$ and from~\cite[Lemma~3.2]{furman}. 
  Therefore, Theorem~\ref{thm:wbOE} shows that
  \begin{align*}
    \mevol_1 (\alpha_\Gamma)
    & \leq \mevol_1 (\overline \alpha_\Lambda \times \beta_\Gamma)
    \\
    & \leq \frac{\vol(M)}{\vol(N)}
    \cdot \mevol_1 (\alpha_\Lambda \times \beta_\Lambda)
    & \text{(Theorem~\ref{thm:wbOE})}
    \\
    & \leq 
    \frac{\vol(M)}{\vol(N)}
    \cdot \mevol_1 (\alpha_\Lambda).
    & \text{(Lemma~\ref{lem:productaction})}
  \end{align*}
  Symmetrically, we obtain~$\mevol_1 (\alpha_\Lambda) \leq
  \vol(N)/\vol(M) \cdot \mevol_1(\alpha_\Gamma)$.
  The argument for~$\medim^Z_1$ works in the same way.
\end{proof}

\begin{lem}\label{lem:productaction}
  Let $\Gamma$ be a countable group, let $\alpha \colon \Gamma \actson
  (X,\mu)$ be a standard action and let $\beta \colon \Gamma \actson
  (Y,\nu)$ be a probability measure preserving action (not necessarily
  essentially free) on a standard Borel probability space. We
  write~$\alpha \times \beta \colon \Gamma \actson (X \times Y, \mu
  \otimes \nu)$ for the associated diagonal standard action.
  Let $n \in \N$ and let $\Gamma$ be of type~$\FP_{n+1}$.
  Then
  \begin{align*}
    \medim^Z_n (\alpha \times \beta)
    & \leq \medim^Z_n (\alpha)
    \\
    \mevol_n (\alpha \times \beta)
    & \leq \mevol_n (\alpha).
  \end{align*}
\end{lem}
\begin{proof}
  We have~$\alpha \prec \alpha \times \beta$ and thus the
  estimates are a consequence of monotonicity under weak containment
  (Theorem~\ref{thm:wc}).
  Alternatively, one could also prove these estimates by straightforward
  direct constructions of measured embeddings. 
\end{proof}

\begin{cor}\label{cor:torsionhyp3}
  There exists a constant~$K \in \R_{>0}$ with the following property:
  For all oriented closed connected hyperbolic $3$-manifolds~$M$,
  the fundamental group~$\Gamma \coloneqq \pi_1(M)$ satisfies
  \[ \widehat t_1(\Gamma) \leq K \cdot \vol(M).
  \]
\end{cor}
\begin{proof}
  This is a direct consequence of Theorem~\ref{thm:mevolhyp3} and the
  dynamical upper bound for logarithmic torsion growth
  (Theorem~\ref{thm:dynupperproofsec}).
\end{proof}

In order to obtain the constant~$\vol/6\pi$, one would need a single
calculation of the measured embedding volume for some closed
hyperbolic $3$-manifold. Comparison with the simplicial volume
estimate gives constant~$6 \cdot \log 3/v_3$
(Example~\ref{exa:ifsvhyp3}), which is not optimal.

%% file: cost.tex
\section{The cost estimate}\label{sec:cost}

The measured embedding dimension in degree~$1$ is compatible with
cost, a dynamical version of the rank of groups~\cite{gaboriaucost,
  kechrismiller}. Because of Lemma~\ref{lem:integers_field} we state 
  the result only over the integers.

\begin{thm}\label{thm:cost}
  Let $\Gamma$ be an infinite group of type~$\FP_2$ and let $\alpha$ be
  a standard action of~$\Gamma$. Then
  \[ \medim^\Z_1 (\alpha) \leq \cost (\alpha) - 1.
  \]
\end{thm}

In combination with Theorem~\ref{thm:l2upper}, we obtain the
sandwich
\[ \ltb 1 \Gamma \leq \medim^\Z_1 (\alpha) \leq \cost (\alpha) -1.
\]
Gaboriau asked whether the two outer terms are equal for infinite groups~\cite[p.~129]{gaboriaul2}; it is thus natural to raise the following question:

\begin{question}
  Let $\Gamma$ be an infinite group of type~$\FP_2$ and let $\alpha$
  be a standard action of~$\Gamma$. Do we always have
  $\medim^\Z_1 (\alpha) = \cost (\alpha) -1$\;?
\end{question}

The basic idea to prove the cost estimate (Theorem~\ref{thm:cost}) is
to construct the low degrees of a resolution over~$\linf{\alpha,\Z} *
\Gamma$ from graphings of the orbit relation of~$\alpha$. The
fundamental theorem of homological algebra then provides
$\alpha$-embeddings with such a target.  In order to achieve the
additive correction term~$-1$, we do this in the slightly more general
case of restricted actions (Proposition~\ref{prop:costtrunc}). The
scaling properties of cost and measured embedding dimension then give
the claimed upper bound from Theorem~\ref{thm:cost}.
In principle, this method works for all countable groups (not
only those of type~$\FP_2$), but our setting is not optimised for that
level of generality.

\begin{prop}\label{prop:costtrunc}
  Let $\Gamma$ be an infinite group of type~$\FP_2$, let $\alpha
  \colon \Gamma \actson (X,\mu)$ be a standard action of~$\Gamma$, and
  let $A \subset X$ be $\alpha$-cofinite. Then
  \[ \medim^\Z_1 (\alpha,A)
  \leq \cost \Bigl(\Rrel_\alpha |_A, \frac1{\mu(A)} \cdot \mu|_A\Bigr).
  \]
  Here, $\Rrel_\alpha|_A \coloneqq (A \times A) \cap \{(x,\gamma \cdot x)
  \mid x \in X, \gamma \in \Gamma\}$ denotes the restriction of the
  orbit relation of~$\alpha$ to~$A$.
\end{prop}
\begin{proof}
  We follow the proof of~$\ltb 1 \Gamma \leq \cost (\alpha) -1$ via
  resolutions~\cite[Chapter~4.3.2]{loeh2020ergodic}.
  Let $R \coloneqq \linf{\alpha,\Z} *\Gamma$ and let $p \coloneqq (\chi_A,1) \in R$.
  Building on the analogy of graphings of equivalence relations as
  dynamical versions of generating sets of groups, we construct the
  low degrees of resolutions of~$p \linf{\alpha,\Z}$ from graphings
  of~$\Rrel_\alpha |_A$: Let $\Phi$ be a graphing
  of~$\Rrel_\alpha|_A$; without loss of generality we may assume that
  $\Phi$ is given by a family~$\Phi = (\varphi_i \coloneqq \gamma_i \cdot - \colon A_i \to 
  B_i)_{i \in I}$ of translation maps, where $I$ is countable and where for each~$i \in
  I$, we have~$\gamma_i \in \Gamma$ and measurable subsets~$A_i
  \subset A$ with~$B_i \coloneqq \gamma_i \cdot A_i \subset A$.

  We define
  \begin{align*}
    D_0 & \coloneqq pR
    \qand
    P_1 \coloneqq \bigoplus_{i \in I} \gen{A_i} _{pRp}   
  \end{align*}
  as well as
  \begin{align*}
    \partial_0^D \colon D_0 & \to p \linf{\alpha,\Z}
    \\
    p \cdot (f,\gamma) & \mapsto p \cdot f
    \\
    \partial_1^P \colon P_1 & \to D_0
    \\
    \chi_{A_i} \cdot e_i
    & \mapsto (\chi_{A_i}, 1) \cdot \bigl((1,1) - (1,\gamma_i)\bigr).
  \end{align*}
  By construction, $\partial_0^P$ is surjective and $\partial_0^D
  \circ \partial_1^P = 0$. However, in general, we may not have
  that~$\im \partial_1^P$ reaches all of~$\ker
  \partial_0^D$. Therefore, we introduce the following correction
  term:
  
  Let $\varepsilon \in \R_{>0}$ and let $(g_k)_{k \in \N}$ be an
  enumeration of~$\Gamma$. For~$k, n \in \N$, we set
  \[ A(k,n)
  \coloneqq \bigl\{
  x \in A
  \bigm|
  \exi{i_1,\dots, i_n \in I}
  \exi{\varepsilon_1, \dots, \varepsilon_n \in \{-1,1\}}
  g_k \cdot x = \varphi_{i_n}^{\varepsilon_n} \circ \dots \circ \varphi_{i_1}^{\varepsilon_1}(x)
  \bigr\}.
  \]
  Each $A(k,n)$ is a measurable subset of~$A$. Because $\Phi$ is a
  graphing of~$\Rrel_\alpha|_A$, we obtain for all~$k \in\N$ that
  $\bigcup_{n \in \N} A(k,n) = A$. Hence, for each~$k \in \N$, there
  is an~$n_k \in \N$ such that
  \[ C_k \coloneqq A \setminus \bigcup_{n=0}^{n_k} A(k,n) 
  \]
  satisfies~$\mu(C_k) \leq 1/2^{k+1} \cdot \varepsilon \cdot \mu(A)$. 
  We then set
  \[ E_1 \coloneqq \bigoplus_{k \in \N} \gen{C_k}_{pRp}
  \]
  and
  \begin{align*}
    \partial_1^E \colon E_1 & \to D_0
    \\
    \chi_{C_k} \cdot e_k
    & \mapsto (\chi_{C_k},1) \cdot \bigl((1,1) - (1,g_k)\bigr).
  \end{align*}
  Finally, we define
  \begin{align*}
    D_1 & \coloneqq P_1 \oplus E_1
    \qand
    \partial_1^D \coloneqq \partial_1^P \oplus \partial_1^E \colon D_1 \to D_0.
  \end{align*}
  By construction~$\partial_0^D \circ \partial_1^D = 0$.
  Moreover, the correction term ensures that~$\ker\partial_0^D = \im \partial_1^D$;
  indeed the inclusion $\ker\partial_0^D \subseteq \im \partial_1^D$  can be shown by a straightforward adaptation of the
  case~$A = X$~\cite[Lemma~4.3.11]{loeh2020ergodic} via an inductive
  argument.

  The modules~$D_0$ and $D_1$ are projective. We may extend the
  low-degree sequence
  \[ \begin{tikzcd}
    D_1
    \ar{r}{\partial_1^D}
    &
    D_0
    \ar[two heads]{r}{\partial_0^D}
    &
    p\linf{\alpha,\Z}
    \end{tikzcd}
  \]
  to a $pRp$-resolution~$D_*$ of~$p
  \linf{\alpha,\Z}$ that consists of free $pRp$-modules in
  degrees greater than or equal to~$2$.  Let $C_*$ be a free $R$-resolution
  of~$\linf{\alpha,\Z}$ that is of finite type (in degrees~$\leq 2$).
  By the fundamental theorem of homological algebra, there exists a
  $pRp$-chain map~$f_* \colon pC_* \to D_*$ extending the identity
  on~$p\linf{\alpha,\Z}$.  Because $pC_*$ is finitely generated
  over~$pRp$ in degrees~$\leq 2$, the images of~$f_1$ and~$f_2$ touch
  only finitely many of the marked summands in~$D_1$ and~$D_2$,
  respectively. Therefore, we can find an $(\alpha,A,Z)$-embedding to
  a marked projective target complex~$\widehat D_*$ with
  \begin{align*}
    \dim_{pRp} (\widehat D_1)
    & \leq \dim_{pRp} (D_1)
    = \sum_{i \in I} \frac1{\mu(A)}\cdot \mu|_A(A_i)
    + \sum_{k \in \N} \frac1{\mu(A)}\cdot \mu|_A(C_k)
    \\
    & 
    \leq \sum_{i \in I} \frac1{\mu(A)} \cdot \mu(A_i)
    + \sum_{k \in \N} \frac1{\mu(A)} \cdot \frac1{2^{k+1}} \cdot \varepsilon \cdot \mu(A)
    \\
    & = \cost \Bigl(\Phi, \frac1{\mu(A)} \cdot \mu|_A\Bigr)
    + \varepsilon.
  \end{align*}
  Taking~$\varepsilon \to 0$ and then taking the infimum over all graphings~$\Phi$
  of~$\Rrel_\alpha|_A$ shows that
  \[ \medim^\Z_1 (\alpha,A)
  \leq \cost \Bigl(\Rrel_\alpha |_A, \frac1{\mu(A)} \cdot \mu|_A\Bigr),
  \]
  as claimed.
\end{proof}

\begin{proof}[Proof of Theorem~\ref{thm:cost}]
  Let $\varepsilon \in \R_{>0}$. Because $\Gamma$ is infinite,
  there exists an $\alpha$-cofinite subset~$A \subset X$
  with~$\mu(A) < \varepsilon$ (Remark~\ref{rem:cofinitesmall}).
  Combining the cost estimate for~$\Rrel_\alpha|_A$ and the scaling
  properties of cost~\cite[Theorem~21.1]{kechrismiller} and~$\medim_1^\Z$
  (Proposition~\ref{prop:truncemb}), we obtain
  \begin{align*}
    \medim^\Z_1 (\alpha)
    & = \mu(A) \cdot \medim^\Z_1(\alpha, A)
    & \text{(Proposition~\ref{prop:truncemb})}
    \\
    & \leq \mu(A) \cdot \cost \Bigl(\Rrel_\alpha|_A, \frac1{\mu(A)} \cdot \mu|_A\Bigr)
    & \text{(Proposition~\ref{prop:costtrunc})}
    \\
    & = \cost (\Rrel_\alpha|_A, \mu|_A)
    \\
    & = \cost (\alpha) - \mu(X \setminus A)
    & \text{(scaling of cost)}
    \\
    & = \cost (\alpha) - 1  + \mu(A)
    \\
    & \leq \cost (\alpha) -1 + \varepsilon.
  \end{align*}
  Taking~$\varepsilon \to 0$ gives the claimed estimate.
\end{proof}

In the proof of Theorem~\ref{thm:cost}, we do not obtain any control
on~$\partial_2^D$ and thus no upper estimate for~$\mevol_1$ in terms
of cost. This is compatible with the expectation that no such upper
bound for~$\mevol_1$ (and whence for~$\widehat t_1$) should exist, as
predicted by the conjectures on logarithmic torsion homology growth
(in degree~$1$) of closed hyperbolic $3$-manifolds~\cite{bergeron_venkatesh, Lueck13, le_hyperbolic} and the
computation of cost of their fundamental groups~\cite[Theorem~8.5]{Agolrankgradients}.

\begin{ex}
If $\Gamma$ is a lattice in a higher rank semisimple real Lie group, or a lattice in a product of at least two automorphism groups of trees, then $\Gamma$ has fixed price~$1$~\cite[Theorem~D]{FMW}. Hence, Theorem~\ref{thm:cost} shows that $\medim_1^Z(\alpha) = 0$ for every standard action~$\alpha$ of~$\Gamma$. Examples of such groups include, e.g., Burger--Mozes groups~\cite[Corollary~1.1]{FMW}.
\end{ex}

%% file: simvol.tex
\section{The simplicial volume estimate}\label{sec:ifsv}

The stable integral simplicial volume of closed manifolds gives upper
bounds on logarithmic torsion homology growth and Betti number
growth~\cite{sauer:volume:homology:growth}. Dynamically, the stable
integral simplicial volume can be expressed as integral foliated
simplicial volume of the profinite completion~\cite{loeh_pagliantini, loeh2020ergodic} 
and integral foliated simplicial volume provides upper
bounds on the $L^2$-Betti numbers~\cite{mschmidt} and
cost~\cite{loeh_cost}. The following estimates of measured embedding
dimension and measured embedding volume complement these connections: 

\begin{thm}\label{thm:simvolestimate}
  Let $M$ be an oriented closed connected aspherical
  $n$-manifold with fundamental group~$\Gamma$, let $\alpha$
  be a standard $\Gamma$-action, and let $k \in \{0,\dots, n\}$.
  Then
  \begin{align*}
    \medim^\Z_k (\alpha)
    & \leq {{n+1} \choose {k+1}} \cdot \ifsv M^\alpha
    \\
    \mevol_k (\alpha)
    & \leq \log(k+2) \cdot {{n+1} \choose {k+1}} \cdot \ifsv M^\alpha.
  \end{align*}
\end{thm}

Here, $\ifsv M^\alpha$ is the \emph{$\alpha$-parametrised (integral foliated)
  simplicial volume},
i.e.,
\begin{align*}
  \ifsv M ^\alpha
  \coloneqq \inf
  \bigl\{ |c|_1 \bigm|
  \;
  & c \in \linf{\alpha,\Z} \otimes_{\Z \Gamma} C_n(\widetilde M;\Z)
  \\
  & \text{is an $\alpha$-parametrised fundamental cycle of~$M$}
  \bigr\}.
\end{align*}
We refer to the literature for further details on this
definition~\cite{mschmidt,loeh_pagliantini}.

\begin{rem}
  More generally, one can also define corresponding simplicial volumes
  with finite field coefficients (with respect to the trivial norm on
  the coefficients)~\cite[Section~4.6]{loeh_cost}. The arguments below
  work verbatim in that setting and thus give estimates of the
  measured embedding dimension over finite fields in terms of the
  corresponding parametrised simplicial volume with the same
  coefficients.
\end{rem}
  
\subsection{Proof of Theorem~\ref{thm:simvolestimate}}

For the proof of Theorem~\ref{thm:simvolestimate}, we construct
$\alpha$-embeddings to marked projective chain complexes defined
out of $\alpha$-parametrised fundamental cycles. For the construction
of such $\alpha$-embeddings, the equivariant chain-level version
of Poincar\'e duality is essential. 

\begin{rem}[Poincar\'e duality]\label{rem:PD}
  Let $M$ be an oriented closed connected $n$-manifold with
  fundamental group~$\Gamma$, let $Z$ be $\Z$ (with the usual norm)
  or a finite field (with the trivial norm),
  and let $\alpha$ be a standard $\Gamma$-action. Let $L = Z$ (as
  trivial $Z\Gamma$-module) or $L = \linf{\alpha,Z}$ and let $R \coloneqq L *
  \Gamma$ (i.e., $R= Z\Gamma$ or $R = \linf{\alpha,Z} * \Gamma$,
  respectively). On~$R$, we consider the involution~$\overline{\args}$
  induced by the inversion map on~$\Gamma$.  Let $c = \sum_{j=1}^m a_j
  \otimes \sigma_j$ be a cycle in~$L \otimes_{Z \Gamma} C_*(\widetilde
  M;Z)$ with~$a_j \in L$ and $\sigma_j \in \map(\Delta^n,\widetilde
  M)$.
  Then the cap-product
  \begin{align*}
    \args\cap c
    \colon \Hom_{Z\Gamma}\bigl( C_{n-*}(\widetilde M;Z), Z \Gamma\bigr)
    & \to
    R \otimes_{Z\Gamma} C_*(\widetilde M;Z)
    \\
    f
    & \mapsto
    \sum_{j=1}^m \overline{a_j \cdot f(\sigma_j\rfloor_{n-*})} \cdot {}_{*}\lfloor \sigma_j 
  \end{align*}
  is a well-defined $Z\Gamma$-chain map with the following additional
  properties~\cite[Chapter~5.6]{lueckmacko_surgerytheory}\cite[Section~3.3]{braunsauer_macroscopic}:
  \begin{itemize}
  \item If $[c] = [c']$ in~$H_n(L \otimes_{Z \Gamma} C_*(\widetilde M;Z))$,
    then $\args \cap c \simeq_{Z \Gamma} \args \cap c'$.
  \item If $L = Z$ and $c$ represents a $Z$-fundamental cycle of~$M$
    (under the canonical chain isomorphism~$Z \otimes_{Z\Gamma}
    C_*(\widetilde M;Z) \cong_Z C_*(M;Z)$), then $\args \cap c$ is a
    $Z\Gamma$-chain homotopy equivalence.
  \end{itemize}
\end{rem}

\begin{proof}[Proof of Theorem~\ref{thm:simvolestimate}]
  Let $c \in C_n(M; \linf {\alpha})$ be an $\alpha$-parametrised
  fundamental cycle of~$M$, say~$c = \sum_{j=1}^m a_j \otimes \sigma_j$
  with~$a_j \in \linf {\alpha}$ and $\sigma_j \in \map(\Delta^n, \widetilde M)$;
  without loss of generality, we may assume that the $\sigma_j$ all belong
  to different $\Gamma$-orbits so that $|c|_1 = \sum_{j=1}^m |a_j|_1$.
  By definition of the integral foliated simplicial volume and
  $\medim$/$\mevol$, it suffices to construct an $\alpha$-embedding
  to a complex~$D_*$, whose ``size'' is controlled well enough in
  terms of~$|c|_1$. We abbreviate~$R \coloneqq \linf{\alpha} * \Gamma$.
  
  \emph{Construction of the target complex.}
  For~$j \in \{1,\dots, m\}$, we write~$A_j \coloneqq \supp (a_j)$. Let $k \in \N$.
  Let $S_k(\sigma_j)$ denote the set of all~$k$-faces of~$\sigma_j$. We
  define the marked projective $R$-module
  \[ D_k \coloneqq \bigoplus_{\tau \in \bigcup_{j=1}^m S_k(\sigma_j)} \gen{A_j} \cdot \tau
  \]
  and, for~$k \in \N_{>0}$, we set
  \begin{align*}
    \partial^D_k \colon D_k
    & \to D_{k-1}
    \\
    \chi_{A_j} \cdot \tau
    & \mapsto \sum_{r=0}^k (-1)^r \cdot \chi_{A_j} \cdot \partial_r \tau
    \quad\text{for~$\tau \in S_k(\sigma_j)$}.
  \end{align*}
  Moreover, we define
  \begin{align*}
    \eta \colon D_0 & \to \linf{\alpha}
    \\
    \chi_{A_j} \cdot \tau & \mapsto \chi_{A_j}.
  \end{align*}
  Viewing the marked generators of~$D_*$ as actual singular simplices
  on~$\widetilde M$ produces a canonical $\Z \Gamma$-chain map~$s_* \colon D_*
  \to R \otimes_{\Z \Gamma} C_*(\widetilde
  M;\Z)$, which extends~$\id_{\linf{\alpha}}$ with respect to~$\eta$ and the
    canonical augmentation~$R \otimes_{\Z\Gamma}
    C_0(\widetilde M;\Z) \to \linf{\alpha}$.
  We will see below that $\eta \colon D_0 \to \linf{\alpha}$ indeed
  is surjective.
  
  By construction, for each~$k \in \N$, we have
  \begin{align*}
    \dim (D_k)
    &
    \leq \sum_{j=1}^m {{n+1} \choose {k+1}} \cdot \mu(A_j)
    \leq {{n+1} \choose {k+1}} \cdot |c|_1,
    \\
    \lognorm (\partial_{k+1}^D)
    &
    \leq \logp \| \partial_{k+1}^D\| \cdot \dim (D_k)
    \leq \log (k+2) \cdot {{n+1} \choose {k+1}} \cdot |c|_1.
  \end{align*}
  
  \emph{Construction of the chain map.}  Let $C_*$ be a free
  $\Z\Gamma$-resolution of~$\Z$ (that is of finite rank up to
  degree~$n+1$). Because $M$ is aspherical, there exists a $\Z
  \Gamma$-chain map~$f_* \colon C_* \to C_*(\widetilde M;\Z)$
  extending~$\id_\Z$.
  Let $E_* \coloneqq \Hom_{\Z\Gamma}(C_{n-*}(\widetilde M;\Z), \Z \Gamma)$.  
  By equivariant Poincar\'e duality,
  there is a $\Z\Gamma$-chain homotopy inverse~$g_* \colon C_*(\widetilde M;\Z)
  \to E_*$
  of the map~$\args\cap c_{\Z}$ induced by an integral
  fundamental cycle~$c_{\Z} \in \Z \otimes_{\Z\Gamma} C_n(\widetilde M ;\Z)$
  (Remark~\ref{rem:PD}). 
  Finally, the cap-product map
  \begin{align*}
    h_* \coloneqq \args \cap c
  \colon E_* 
  & \to D_*
  \\
  f
  & \mapsto \sum_{j=1}^m
  \overline{a_j \cdot f(\sigma_j\rfloor_{n-*})} \cdot {}_{*}\lfloor\sigma_j 
  \end{align*}
  is well-defined and 
  a $\Z \Gamma$-chain map. Indeed, by Remark~\ref{rem:PD}, this holds
  for the target~$R \otimes_{\Z \Gamma} C_*(\widetilde M;\Z)$;
  by construction of~$D_*$, this map factors over~$s_*$.
  
  We will now explain why $h_* \circ g_* \circ f_* \colon C_* \to D_*$
  is an $\alpha$-embedding, i.e., that this composition extends
  the inclusion~$\Z \to \linf {\alpha}$ as constant functions
  and that $\eta \colon D_0 \to \linf \alpha$ is surjective: 
  We consider the following diagram of $\Z \Gamma$-chain maps:
  \[
  \begin{tikzcd}
    C_*(\widetilde M;Z)
    \ar[equal]{d}
    \ar{r}{g_*}
    &
    E_*
    \ar[equal]{d}
    \ar{r}{\args\cap c}
    &
    D_*
    \ar{r}{s_*}
    &
    R \otimes_{\Z \Gamma} C_*(\widetilde M;\Z)
    \\
    C_*(\widetilde M;Z)
    \ar{r}[swap]{g_*}
    &
    E_*
    \ar{rr}[swap]{\args\cap c_\Z}
    &
    &
    \Z\Gamma \otimes_{\Z \Gamma} C_*(\widetilde M;\Z)
    \ar{u}[swap]{\text{canonical map}}
  \end{tikzcd}
  \]
  The right hand square commutes up to $\Z\Gamma$-chain homotopy
  because $c$ and $c_\Z$ are cycles in~$\linf {\alpha} \otimes_{\Z
    \Gamma} C_*(\widetilde M;\Z)$ that, by definition of
  $\alpha$-parametrised fundamental cycles, represent the same class in
  homology (Remark~\ref{rem:PD}). The lower composition is $\Z\Gamma$-chain
  homotopic to the identity (by choice of~$g_*$).
  Taking~$H_0$ of this diagram thus results in the following
  commutative diagram of $\Z$-modules:
  \[
  \begin{tikzcd}
    \Z
    \ar[equal]{d}
    \ar{r}{\cong}
    &
    H_0(\widetilde M;\Z)
    \ar[equal]{d}
    \ar{rr}{H_0(h_* \circ g_*)}
    &
    &
    H_0(D_*)
    \ar{rr}{H_0(s_*)}
    &
    &
    H_0(M;\linf{\alpha})
    \ar{r}{\cong}
    &
    \linf{\alpha}
    \\
    \Z
    \ar{r}[swap]{\cong}
    &
    H_0(\widetilde M;\Z)
    \ar[equal]{rrrr}
    &
    &
    &
    &
    H_0(M;\Z)
    \ar{r}[swap]{\cong}
    \ar{u}[swap]{\text{canonical map}}
    &
    \Z
    \ar{u}
  \end{tikzcd}
  \]
  In particular, we see that $H_0(s_*)$ is surjective (because $H_0(s_*)$ is
  $\linf{\alpha}$-linear and $1$ lies in the image). By construction
  of~$\eta$, this shows that also $\eta$ is surjective and that $h_* \circ g_*$
  extends the canonical inclusion~$\Z \to \linf{\alpha}$.
  
  Hence, $h_* \circ g_* \circ f_* \colon C_* \to D_*$ 
  is an $\alpha$-embedding and we obtain
  \begin{align*}
    \medim^\Z_k (\alpha)
    &
    \leq \dim (D_k)
    \leq {{n+1}\choose {k+1} } \cdot |c|_1
    \\
    \mevol_k (\alpha)
    &
    \leq \lognorm (\partial^D_{k+1})
    \leq \log (k+2) \cdot {{n+1}\choose {k+1}} \cdot |c|_1.
  \end{align*}
  Taking the infimum over all $\alpha$-parametrised fundamental
  cycles~$c$ of~$M$ gives the claimed estimates.
\end{proof}

\subsection{Examples}
\label{sec:ex_simvol}

We combine the simplicial volume estimate
(Theorem~\ref{thm:simvolestimate}) with known computations of integral
foliated simplicial volume. While the resulting upper bounds for
(torsion) homology growth are not new, they can now be combined with
other inheritance results for measured embedding dimension/volume to
obtain new results. 

\begin{ex}\label{exa:ifsvhyp3}
  Let $\Gamma$ be the fundamental group of an oriented closed connected aspherical $3$-manifold~$M$.
  Then, we have
\begin{align*}
\widehat t_1(\Gamma,\Gamma_*)
    & \leq \mevol_1 (\Gamma \actson \widehat \Gamma_*)
    & \text{(Theorem~\ref{thm:dynupper})}
    \\
    & \leq 6 \cdot \log (3) \cdot \ifsv M ^{\Gamma \actson \widehat \Gamma_*}
    & \text{(Theorem~\ref{thm:simvolestimate})}
    \\
    & = 6 \cdot \log (3) \cdot \frac{\textup{hypvol}(M)}{v_3},
    & \text{\cite[Theorem~1.7]{FLMQ}} 
    \end{align*}
    where $\textup{hypvol}$ denotes the total volume of the hyperbolic pieces in the JSJ decomposition of $M$.
    
  In the case of hyperbolic $3$-manifolds, this is a coarser version of the estimate obtained in
  Theorem~\ref{thm:mevolhyp3}.  It should be noted that the estimates
  for~$\mevol_1$ in the proof of Theorem~\ref{thm:mevolhyp3} and the
  computation of~$\ifsv M ^{\Gamma \actson \widehat \Gamma_*}$ are
  based on the same, dynamical, principles (weak bounded orbit equivalence
  and approximation).
\end{ex}

\begin{ex}\label{ex:riem:ifsv}
  More generally, one obtains also upper bounds for the measured
  embedding dimension and measured embedding volume in terms of the
  Riemannian volume and a bound on the volume of small
  balls~\cite[Theorem~1.5]{braunsauer_macroscopic}: Let $V_1 \in
  \R_{>0}$ and $n \in \N_{>0}$. Then, there exists a
  constant~$\const{}(n,V_1) \in \R_{>0}$ with the following property:
  For every aspherical oriented closed connected Riemannian $n$-manifold
  that satisfies~$\vol_{\widetilde M} (B) \leq V_1$ for all balls~$B \subset \widetilde M$
  of radius at most~$1$ 
  and every standard $\pi_1(M)$-space~$\alpha$, we have for
  all~$k \in \{0,\dots,n\}$: 
  \begin{align*}
    \mevol_k (\alpha)
    & \leq \log(k+2) \cdot {{n+1} \choose {k+1}} \cdot \ifsv M ^\alpha
    & \text{(Theorem~\ref{thm:simvolestimate})}
    \\
    & \leq \log(k+2) \cdot {{n+1} \choose {k+1}}
    \cdot \const{}(n,V_1)
    \cdot \vol(M).
    & \text{\cite[Theorem~1.5]{braunsauer_macroscopic}}
  \end{align*}
  For~$\medim_k^\Z (\alpha)$, an analogous version holds.
\end{ex}

\begin{ex}
\label{ex:amenable_cover}
  Let $M$ be an oriented closed connected aspherical $n$-manifold
  (with~$n>0$) with fundamental group~$\Gamma$. Suppose that there
  exists an open cover of~$M$ by amenable subsets of multiplicity at
  most~$n$. Then, for all standard $\Gamma$-spaces~$\alpha$ and
  all~$k \in \N$, we have
  \[ \medim_k^\Z (\alpha)
  = 0
  \qand
  \mevol_k (\alpha)
  = 0,
  \]
  because $\ifsv M ^\alpha = 0$~\cite{loehmoraschinisauer}.  This type
  of vanishing arises in many geometric
  situations, e.g., manifolds with amenable fundamental group, graph $3$-manifolds, smooth manifolds that admit smooth circle actions without fixed points and manifolds admitting an $F$-structure~\cite[Section~1.1]{loehmoraschinisauer}\cite{Sauer09}.
\end{ex}